\documentclass[reqno]{amsart}
\usepackage[top=2cm,bottom=2cm,left=3cm,right=3cm]{geometry}
\usepackage{amssymb,amsthm,amsmath,amsxtra,dsfont,appendix,stix}
\usepackage{xcolor}
\definecolor{darkred}{RGB}{139,0,0}
\usepackage{hyperref} 
\hypersetup{
	colorlinks=true, 	
	citecolor=blue, 
	linkcolor=darkred,
	urlcolor=darkred
}

\usepackage[
backend=biber,
style=alphabetic-verb,
sortlocale=auto,
natbib=true,
url=false, 
doi=true,
eprint=true,
safeinputenc=true,
giveninits=true,
maxbibnames=10
]{biblatex}
\usepackage[nameinlink]{cleveref}
\usepackage[disable]{todonotes}

\usepackage{xparse,mathtools}
\usepackage{enumitem}
\usepackage{cancel}

\usepackage{mathrsfs}

\usepackage{tcolorbox}

\newcommand{\1}{\mathds{1}}
\renewcommand{\i}{\mathbf{i}}
\renewcommand{\a}{\mathrm{a}}
\renewcommand{\d}{\mathrm{d}}
\newcommand{\A}{\mathcal{A}}
\newcommand{\B}{\mathcal{B}}
\newcommand{\C}{\mathds{C}}
\newcommand{\E}{\mathds{E}}

\newcommand{\F}{\mathcal{F}}
\newcommand{\G}{\mathrm{G}}
\newcommand{\W}{\mathrm{W}}

\renewcommand{\O}{\mathcal{O}}
\renewcommand{\L}{\mathfrak{L}}
\newcommand{\M}{\mathrm{M}}
\newcommand{\N}{\mathbb{N}}

\newcommand{\Q}{\mathcal{Q}}
\renewcommand{\P}{\mathds{P}}
\newcommand{\X}{\mathrm{X}}
\newcommand{\R}{\mathds{R}}

\newcommand{\Ai}{\operatorname{Ai}}
\newcommand{\SAi}{\operatorname{SAi}}

\newcommand{\sgn}{\operatorname{sgn}}
\newcommand{\Z}{\mathbb{Z}}
\newcommand{\D}{\mathds{D}}

\newcommand{\dif}{\mathrm{d}}
\newcommand{\bpsi}{\boldsymbol{\psi}}

\newcommand{\tphi}{\boldsymbol{\chi}}
\newcommand{\Exp}{\mathbb{E}}

\numberwithin{equation}{section}

\newtheorem{theorem}{Theorem}[section]

\newtheorem{proposition}[theorem]{Proposition}
\newtheorem{corollary}[theorem]{Corollary}
\newtheorem{lemma}[theorem]{Lemma}
\newtheorem{remark}[theorem]{Remark}

\theoremstyle{definition} 
\newtheorem{definition}{Definition}[section]
\newtheorem{assumption}[definition]{Assumptions}

\renewcommand{\Re}{\operatorname{Re}}
\renewcommand{\Im}{\operatorname{Im}}

\renewcommand\o{
\mathchoice
{{\scriptstyle\mathcal{O}}}
{{\scriptstyle\mathcal{O}}}
{{\scriptscriptstyle\mathcal{O}}}
{\scalebox{.7}{$\scriptscriptstyle\mathcal{O}$}}
}

\newcommand{\filt}{\mathscr{F}}

\DeclareDocumentCommand{\Pto} {o} {
\IfNoValueTF {#1}
{\overset{\P}{\longrightarrow}}
{ \xrightarrow[ #1 \to \infty]{\P }}
}
\DeclareDocumentCommand{\Asto} {o} {
\IfNoValueTF {#1}
{\overset{\operatorname{a.s.}}{\longrightarrow}}
{
\xrightarrow[ #1 \to \infty]{\operatorname{a.s.} }
}
}
\DeclareDocumentCommand{\Mgfto} {o} {
\IfNoValueTF {#1}
{\overset{\operatorname{mgf}}{\longrightarrow}}
{ \xrightarrow[ #1 \to \infty]{\operatorname{mgf} }}
}

\DeclareDocumentCommand{\Wkto} {o} {
\IfNoValueTF {#1}
{\overset{\rm law}{\longrightarrow}}
{ \xrightarrow[ #1 \to \infty]{\operatorname{law}}}
}

\DeclareDocumentCommand \LPto { O{1} }
{\overset{\operatorname{\LP^{#1}}}{\longrightarrow}}

\title[Bulk asymptotics of the G$\beta$E characteristic polynomial]{Bulk asymptotics of the Gaussian $\beta$-ensemble characteristic polynomial}

\author[G. Lambert]{Gaultier Lambert}
\address{Department of Mathematics, KTH Royal Institute of Technology, Stockholm, Sweden}
\email{glambert@kth.se}

\author[E. Paquette]{Elliot Paquette}
\address{Department of Mathematics, McGill University, Montreal, QC, Canada}
\email{elliot.paquette@mcgill.ca}

\date{August 1, 2025}

\begin{document}

\maketitle

\begin{abstract}
The Gaussian $\beta$-ensemble (G$\beta$E) is a fundamental model in random matrix theory.
In this paper, we provide a comprehensive asymptotic description of the characteristic polynomial of the G$\beta$E anywhere in the bulk of the spectrum that simultaneously captures both local-scale fluctuations (governed by the Sine-$\beta$ point process) and global/mesoscopic log-correlated Gaussian structure, which is accurate down to vanishing errors as $N\to\infty$.

As immediate corollaries, we obtain several important results: (1) convergence of characteristic polynomial ratios to the stochastic zeta function, extending known results from \cite{ValkoViragZeta} to the G$\beta$E; (2) a martingale approximation of the log-characteristic polynomial which immediately recovers the central limit theorem from \cite{BourgadeModyPain}; (3) a description of the order one correction to the martingale in terms of the stochastic Airy function.
\end{abstract}


\tableofcontents

\todo[inline]{Fix the parabolic paper (most of the appendix goes) }
\todo[inline]{check correlations + write down martingale CLT}

\section{Introduction}

\subsection{Gaussian $\beta$-ensembles.}
For $\beta>0$ and $N\in\N$, the \emph{Gaussian $\beta$-ensemble} (G$\beta$E) is a distribution on $\R^N$ given by 
\begin{equation}\label{eq:GbE}
(\lambda_1, \lambda_2, \dots,\lambda_N) \mapsto
\frac{1}{\mathcal{Z}(N,\beta)} e^{- \sum_{i=1}^N \beta N \lambda_i^2} \prod_{i > j} |\lambda_i - \lambda_j|^\beta. 
\end{equation}
This is the subject of a long line of literature in random matrix theory, see e.g.~\cite{Forrester}. 
The most traditional investigation of this point process is through its bulk local limit, which is described by the Sine-$\beta$ point process introduced in \cite{KillipStoiciu,ValkoVirag}, and which generalize the classical determinantal/Pfaffian point processes for $\beta \in \{1,2,4\}$ \cite{mehta2004random}. 

A second, more recent direction of interest is the study of the distributions of the characteristic polynomial of \eqref{eq:GbE}.  Specifically, in this paper, we focus on the the normalized characteristic polynomial
\begin{equation} \label{charpoly}
\Phi_N(z) \coloneqq
w_N(z)
{\textstyle \prod_{i=1}^N} ( z - \lambda_i),
\quad\text{where}\quad
w_N(z) \coloneqq (\tfrac{N}{2\pi})^{1/4}e^{-N z^2}
{\textstyle \prod_{k=1}^N \sqrt{\tfrac{4N}{k}}},
\qquad z\in\R. 
\end{equation}
In particular, the normalization is chosen so that $\int_\R\big(\E \Phi_N(z)\big)^2 \d z=1/2$ and the empirical measure $\frac1N \sum_{i=1}^N \delta_{\lambda_i}$ converges (in a large deviation sense) to the semicircle law $\varrho$  on $[-1,1]$; see Section~\ref{sec:tri}.  

The study of characteristic polynomials of random matrices has focused on its connection to log-correlated Gaussian fields. 
In particular, $\log |\Phi_N(z)|$ converges in distribution to a log-correlated Gaussian field $\Re\mathfrak{X}(z)$ for $z \in \C \setminus [-1,1]$ (this is originally due to \cite{Johansson}); see \eqref{eq:X_field} for the definition of the limit $\mathfrak{X}$.  While harmonic in the upper-half plane, the field is not pointwise defined on $[-1,1]$, but can be formalized as a random Schwartz distribution.  
Nevertheless, by suitable approximations, it is possible to define the exponential of $\Re \mathfrak{X}(z)$ on $[-1,1]$; the resulting random measures are instances of \emph{Gaussian multiplicative chaos} (GMC) measures \cite{Berestycki}.  Then the connection between $|\Phi_N(x)|^{\gamma}$ and GMC measures has been shown only in case $\beta=2$ \cite{CFLW} (see also \cite{BWW} for the $L^2$--regime and \cite{Pax} for some related results for $\beta=1,4$). 

For general $\beta$-ensembles (regular one-cut potential and any fixed $\beta > 0$), the log-correlated field structure has been established, in the sense of finite-dimensional marginals and in the sense of exponential moments in \cite{BourgadeModyPain} (see also \cite{Augeri} for a related CLT). A closely related problem is the convergence of the leading order behavior for the maximum of the recentered log-characteristic polynomial, which was established in \cite{LambertPaquette01} for $\beta=2$ and \cite{BLZ} for general $\beta>0$.
In fact, \cite{BLZ} establishes that the $\O(1)$ behavior of the maximum of characteristic polynomials of many large random Hermitian matrix models is universal and matches that of G$\beta$E.
For the circular $\beta$-ensemble (C$\beta$E), the asymptotic picture is much more complete and the convergence in distribution of the maximum of the characteristic polynomial has been established in \cite{CMN,PaquetteZeitouni}.
The convergence of powers of the C$\beta$E characteristic polynomial have also been obtained throughout the subcritical phase in \cite{ChhaibiNajnudel,LambertNajnudel}. These results rely on the theory of \emph{orthogonal polynomial on the unit circle} by studying  the  asymptotics of the \emph{Szeg\"o recursion}. This method is specific to circular $\beta$-ensembles but it bears some resemblance with the \emph{Pr\"ufer phase} recursion investigated in Section~\ref{sec:ell}.

\smallskip

In this paper, we aim to give a bridge between these two pictures, by giving a description of the characterisitic polynomial at multiple points $\{z_j\}$ in the bulk $(-1,1)$ which simultaneously recovers the local-scale $z_1-z_2=\Theta(N^{-1})$ fluctuations, governed by the Sine-$\beta$ point process, and the global/mesoscopic $|z_1-z_2| \gg N^{-1}$ log-correlated field structure. Our description is furthermore accurate, in a distributional sense, down to vanishing errors in the bulk as $N\to\infty$.
As an illustration of the usefulness of the description, it will be an immediate consequence that for a fixed $z\in(-1,1)$, the ratio $\Phi_N(z+\lambda/N\varrho(z))/\Phi_N(z)$ converges in the sense of finite dimensional marginals to a random analytic function, called the \emph{stochastic zeta function} \cite{NajnudelNikeghbali,Assiotis,ValkoViragZeta}. 
The limit object was introduced in \cite{ValkoViragZeta} as the local limit of analogous ratios for circular-$\beta$-ensemble. The local convergence of ratios of G$\beta$E characteristic polynomials to the stochastic zeta function is new.

\subsection{Tridiagonal models.} \label{sec:tri}

Our starting point is the tridiagonal matrix model or random \emph{Jacobi matrix}   for the Gaussian $\beta$-ensemble \cite{DumitriuEdelman},
\begin{equation} \label{def:trimatrix}
\mathbf{A} =
\left[ \begin{array}{cccc} 
b_1 & a_1 & &\\
a_1 & b_2 & a_2 & \\
& a_2 & b_3 & \ddots  \\
&& \ddots & \ddots
\end{array} \right],
\end{equation}
where $b_k \sim \mathcal{N}(0,2)$ and  $a_k \sim \chi_{\beta k}$ are independent random variables for $k\ge 1$.  
The eigenvalues of the principal minor $[\mathbf{A}/\sqrt{4N\beta}]_{N\times N}$ of this matrix are distributed according to \eqref{eq:GbE}, and consequently $\Phi_N(z) = w_N(z)\det([z-({4}{N\beta})^{-1/2}\mathbf{A}]_{N\times N}).$
Our method does apply to a class of random Jacobi matrices which generalize the Gaussian $\beta$-ensembles (after an appropriate truncation). 
Our main results are formulated under the following assumptions and notations.

\begin{definition} \label{def:noise}
The entries of the tridiagonal matrix model $\mathbf{A}$ are independent random variables which depend on a parameter $\beta>0$. 
We define for $k\in\N$, 
\begin{equation} \label{noise}
X_k \coloneqq \frac{b_{k+1}}{\sqrt 2}, 
\qquad\qquad
Y_k \coloneqq \frac{a_{k}^2 - \beta k}{\sqrt{2\beta k}}  .  
\end{equation}
We assume that for fixed $\mathfrak{K}, \mathfrak{S}\in \N$, it holds for $k \ge \mathfrak{K}$, 
\begin{equation}\label{mean}
\E X_k = \E Y_k = 0 , \qquad \E X_k^2 =  \E Y_k^2  = 1
\qquad \text{and}\qquad
\| X_k \|_2 ,  \| Y_k \|_2 \le \mathfrak{S}.
\end{equation}
Here and in the sequel of this paper, $\|\cdot\|_q$ refers to the Orlicz norm defined in Appendix~\ref{sec:concentration}. 
In the sequel, all constants are allowed to depend on the fixed parameters $\beta,\mathfrak{S}>0$.

We define $\sigma$-algebras $\F_0 = \sigma(b_1)$ and
\(
\F_{n} \coloneqq \sigma\{ X_k,Y_k : k\le n\},
\)
Then, $\{\F_{n}\}_{n\in\N_0}$ is a filtration.  
\end{definition}

\begin{remark}
The G$\beta$E fits this framework only after a mild truncation of the entries.  For every $\epsilon > 0$, there are $\mathfrak{K}$ and $\mathfrak{S}$ sufficiently large (depending on $(\beta, \epsilon)$) and a matrix model $\tilde{\mathbf{A}}$ satisfying Definition \ref{noise} so that $\P( \mathbf{A} \neq \tilde{\mathbf{A}}) \geq 1-\epsilon$.  In particular any convergence statement as $N\to\infty$ we formulate under Definition~\ref{def:noise} also applies to G$\beta$E.
\end{remark}

We will abuse the notation \eqref{charpoly} and define for $n\in\N$, $z\in\R$, 
\begin{equation} \label{def:Phi}
\Phi_n(z) \coloneqq  w_n(z)\det[z-({4}{N\beta})^{-1/2}\mathbf{A}]_{n,n} , \qquad 
w_n(z) \coloneqq (\tfrac{N}{2\pi})^{1/4}e^{-N z^2} {\textstyle \prod_{k=1}^n \sqrt{\tfrac{4N}{k}}},
\end{equation}
for the rescaled characteristic polynomials of successive minors of the random matrix $\mathbf{A}$. Note this agrees with \eqref{charpoly} for $n=N$ but we do not emphasize the dependence of $\Phi_n$ on $N$ throughout the paper. 

\subsection{Hermite polynomials.}
For comparison, it is of interest to consider the properties of the deterministic matrix $\E\mathbf{A}$. 
For G$\beta$E, this also corresponds to the weak limit as $\beta \to\infty$. 
This case motivates our choice of normalization \eqref{def:Phi} as well as the choice of \eqref{noise}--\eqref{mean} for the characteristic polynomial, as this leads to the identity:
\begin{equation*} 
\E \Phi_n(z) = h_n(z), 
\end{equation*}
where $\{h_n(z); z\in\C\}_{n\ge 0}$ are the \emph{Hermite functions}, which are orthonormal with respect to the Gaussian measure $(\tfrac{2N}{\pi})^{1/2}e^{-2N x^2}dx$ on $\R$, and which have zeros asymptotically distributed according to the semicircle law $\varrho$ on $[-1,1]$. 

\medskip

It will be advantageous to compare our main result (Theorem~\ref{thm:main}) with the classical \emph{Plancherel-Rotach} asymptotics for the Hermite polynomials \cite{PR29}
for $z \in \big[-1+\tfrac{c_N}{N^{2/3}},1-\tfrac{c_N}{N^{2/3}}\big]$ and $\lambda\in\R$, it holds as $n\to\infty$,
\begin{equation} \label{PR1}
h_N\big(z + \tfrac{\lambda}{N \varrho(z)}\big) 
= \sqrt{1/\pi}\,(1-z^2)^{-1/4}
\Re
\bigl[
\exp
\bigl(\i\pi\big(N F(z) - \tfrac{\arcsin(z)}{2\pi} +\lambda\big)
+\o(1)
\bigr)
\bigr],
\end{equation}
where $F(z) = \int_z^{1} \varrho(x) \d x$ is an antiderivative of the semicircle and the error goes to $0$ locally uniformly in $\lambda$ and for $z$ in this range  if $c_N \to \infty.$ 
In contrast, at the edges, one has \emph{Airy-type} asymptotics, it holds locally uniformly in $\lambda\in\R$ as $N\to\infty$,   
\begin{equation} \label{PR2}
h_N\big(\pm 1+\tfrac{\lambda}{2 N^{2/3}}\big) = (\pm1)^N \sqrt{N^{1/3}}\Ai(\pm \lambda)\big(1+\O(N^{-1/3})\big) . 
\end{equation}
Both regimes are consistent and these asymptotics are universal for orthonormal polynomials with respect to varying weight on $\R$ in the one-cut regime, \cite{Deiftstrong}.

\subsection{Limiting stochastic processes.}
The Hermite polynomials describe the mean behavior of the G$\beta$E characteristic polynomials. 
To describe the fluctuations and present our main theorem, we need to introduce two stochastic processes; 
\begin{itemize}  \setlength\itemsep{0em}
\item  A Gaussian analytic function $\mathfrak{X}=\big\{\mathfrak{X}(z) :  z\in \C \setminus [-1,1]\big\}$ which describes the macroscopic fluctuations of the log characteristic polynomial. 
\item The Sine-$\beta$ point process and the \emph{stochastic zeta function} which describes the microscopic fluctuations of the G$\beta$E eigenvalues and the scaling limit of the characteristic polynomial inside the bulk. 
\end{itemize}


\paragraph{Macroscopic Gaussian landscape -- log-correlated field.}
We introduce a map, sometimes called the  \emph{inverse Joukowsky transform},
\begin{equation} \label{Jouk}
J : \C \setminus [-1,1] \ni w \mapsto  w- \sqrt{w^2-1}
\end{equation}
where the branch of $\sqrt{\cdot}$ is chosen so that $J:  \C \setminus [-1,1] \to \D$ is conformal. 
This function describes the asymptotics of Hermite polynomials outside of the cut $[-1,1]$ as $-2J(z)$ corresponds to the Stieltjes transform of the semicircle distribution.
More relevant here, it gives the exact correlation structure of the (harmonic) Gaussian field $\mathfrak{X}$ which describes the fluctuations of $z \in \C \setminus [-1,1] \mapsto \log \Phi_N(z)$. 
We define $\mathfrak{X} : \C \setminus [-1,1] \to \C$ to be a mean-zero Gaussian field such that $\mathfrak{X}(\overline{z}) = \overline{\mathfrak{X}(z)}$ and
\begin{equation}\label{eq:X_field}
\E[\mathfrak{X}(x)\mathfrak{X}(z)]= -2\log\big(1-J(x)J(z)\big),
\qquad
x,z \in \C \setminus [-1,1].
\end{equation}
This corresponds to the pull-back of the GAF, $z\in\D\mapsto \sum_{k \geq 1} \xi_k z^k/\sqrt{k}$ with i.i.d.~standard real Gaussian coefficients $\{\xi_k\}_{k\in\N}$, under the map \eqref{Jouk}. 
We refer to \cite[Section~1.4]{LambertPaquette02} for further properties of this complex-valued log-correlated field.
Then, by \cite[Theorem~1.4]{LambertPaquette02}, in the topology of locally uniform convergence,
\begin{equation}\label{gobalgmc}
\biggl\{\tfrac{\Phi_N(z)}{h_N(z)} : z \in \C \setminus [-1,1] \biggr\}
\Wkto[N]
\biggl\{
\exp\Big( \sqrt{\tfrac{1}{\beta}} \mathfrak{X}(z) - \tfrac{1}{2\beta} \Exp \mathfrak{X}(z)^2 \Big)
: z \in \C \setminus [-1,1] 
\biggr\}. 
\end{equation}

Then, we can define a generalized field $\big\{\mathfrak{X}(z) ; z\in\R \big\}$ by continuity from the upper-half plane. This is a log-correlated Gaussian field with correlation structure; for $x,z\in\R$ with $x\neq z$, 
\begin{equation} \label{corrW}
\E[\mathfrak{X}(x)\mathfrak{X}(z)]= -2\log\big(1-J(x)J(z)\big) ,\qquad 
\E[\mathfrak{X}(x)\overline{\mathfrak{X}(z)}]= -2\log\big(1-J(x)\overline{J(z)}\big), 
\end{equation} 
where $J(x) = \displaystyle\lim_{\eta\to0^+} J(x+\i\eta)$ is given by \eqref{def:J} below. 
Then  $\big\{\mathfrak{X}(z) : z\in [-1,1] \big\}$ is a complex-valued Gaussian generalized field and $\big\{\mathfrak{X}(z) : z\in\R\setminus [-1,1] \big\}$ is a real-valued smooth Gaussian field. 

\paragraph{Microscopic landscape -- the stochastic zeta function.} 
To define the stochastic zeta function of \cite{ValkoViragZeta}, we introduce the \emph{complex sine equation}. 
Let $\{Z_t :t\in\R_+\}$ be a complex Brownian motion with normalization $[Z_t,Z_{t}]=0$ and $[Z_t,\overline{Z}_t] = 2t$ for $t\ge 0$.
We consider the coupled solutions of the stochastic differential equation (SDE) for $\lambda\in\C$  and $t\ge 0$, 
\begin{equation}\label{eq:csse}
\begin{aligned}
\dif \omega_t(\lambda)
&= \i\frac{\pi \lambda}{\sqrt{t}}\dif t
+ \sqrt{\frac{2}{\beta t}}
\biggl( \bigl(1-e^{-\i\Im\omega_t(\lambda)}\bigr) \dif Z_t \biggr),
\qquad
\omega_0(\lambda) = 0 . 
\end{aligned}
\end{equation}
This equation is singular as $t \to 0$, but there is a unique continuous strong solution $\big\{ \omega_t(\lambda) : \lambda\in\C , t\in\R_+\big\}$ with the property that $\omega_0=0$ (see Lemma \ref{lem:sine0}).  We note that this differs slightly from existing formulations (\cite{KillipStoiciu} and \cite{ValkoViragZeta}), by simple changes of time and space (see Appendix~\ref{app:Sine} for details).

The resulting solution $\lambda\in\C \mapsto \omega_t(\lambda)$ has many properties: in particular, it is an entire function and the map $\lambda\in\R \mapsto \Im\omega_t(\lambda)$ is non-decreasing.
This equation was in a sense introduced in \cite{KillipStoiciu} and one can define the Sine-$\beta$ point process:
\begin{equation}
\big\{\lambda\in\R ;\,  \Im \omega_1(\lambda) + \alpha \in 2\pi \Z \big\}
\end{equation}
where the random variable $\alpha$ is uniform in $[0,2\pi]$, independent of the Brownian motion $\{Z_t\}$.
Hence, the function $\lambda\in\R \mapsto \lfloor\Im\omega_1(\lambda) + \alpha\rfloor_{2\pi}$, where $\lfloor\cdot\rfloor_{2\pi}$ denotes floor function\footnote{
For $x\in\R$, we denote $\lfloor x\rfloor_{2\pi} = k$ if $x\in[2\pi k , 2\pi(k+1))$ for $k\in\Z$ and $\{ x\}_{2\pi} = x-2\pi k$ so that $\{ x\}_{2\pi}\in[0,2\pi)$ and  $\lfloor x\rfloor_{2\pi}\in\Z$.} 
mod-$2\pi$, is $2\pi$ multiplied by the counting function of the Sine-$\beta$ point process. 
The equation \eqref{eq:csse} can also be used to construct the scaling limit of the characteristic polynomial.
Following \cite{ValkoViragZeta}, we define \emph{stochastic $\zeta$ function};
\begin{equation} \label{eq:zeta}
{\zeta}_\beta(\lambda)
\coloneqq
\frac{\Re( e^{i \alpha + \omega_1(\lambda)})}
{\Re( e^{i \alpha})},
\qquad
\lambda \in \C. 
\end{equation}
The properties of this function, in particular its relationship to certain Dirac operators, are studied in \cite{ValkoViragZeta}. 
By a coupling argument going back to \cite{ValkoViragOperatorLimit}, it is also known that $\boldsymbol{\zeta}_\beta$ is the limit of microscopic ratios of the circular $\beta$-ensemble characteristic polynomial  \cite[Theorem 41]{ValkoViragZeta}. 
We obtain a similar description for G$\beta$E (Corollary~\ref{cor:zeta}). 

\subsection{Pr\"ufer phases.}
Our main result (Theorem~\ref{thm:main}) can be viewed as a type of probabilistic version of the Plancherel-Rotach asymptotics \eqref{PR1} for the Hermite polynomials, which hold in the case of $\beta=\infty.$  
These asymptotic are obtained by analyzing the recursion for the characteristic of the random tridiagonal matrix model from Section~\ref{sec:tri}. 
In this section, we review the basic properties of this recursion and we define a type of \emph{Pr\"ufer phase} which is convenient to study the \emph{elliptic part} of the recursion.

The sequence of characteristic polynomials $\{\Phi_n\}_{n\ge 0}$, \eqref{def:Phi}, satisfies a 3-term recurrence, or equivalently a $2\times2$ matrix recurrence \eqref{rec1}. If the spectral parameter $z\in [-1,1]$, this recurrence exhibits a \emph{turning point} at step $N_0(z) = \lfloor Nz^2 \rfloor$ where the fundamental solutions of the 3-term recurrence transition from exponential type to oscillatory, or equivalently where the transfer matrices transition from having distinct real eigenvalues (i.e.\ hyperbolic matrices) to complex conjugate pairs (i.e.\ elliptic matrices). These different behaviors also arise for the Hermite polynomials ($\beta=\infty$) and the different regimes are explained in details in \cite[Section 1.2]{LambertPaquette02}. In particular, the transition window (called the \emph{parabolic regime}) around the turning point is of size $\O( \lfloor Nz^2 \rfloor^{1/3})$. 

At generic $z \in (-1,1)$ we see all these behaviors, but there are two special cases:
\begin{itemize}[leftmargin=*]  
\item the edges, $z\in\{\pm 1\}$, where the whole recurrence is hyperbolic, save for a parabolic regime of size $\O(N^{1/3})$ at the end of the recurrence.
\item $z$ in a $\O(N^{-1/2})$-neighborhood of 0 where the whole recurrence is elliptic. 
\end{itemize}

We have already studied the edge cases in \cite{LambertPaquette03}, and we established that the scaling limit of the characteristic polynomial is given in terms of the \emph{stochastic Airy function}, see Section~\ref{sec:edge}.
In particular, Theorem~\ref{thm:PR2} should be compared to \eqref{PR2} for the Hermite polynomials in case of $\beta=\infty.$  
These asymptotics also occur in the transition window and they will be instrumental to prove our main Theorem~\ref{thm:main}. 

\smallskip

In this paper, we focus on the \emph{elliptic part of the recursion} which encodes the bulk asymptotics of the characteristic polynomials.  Let $\mathcal{I}_n \coloneqq \big(-\sqrt{n/N}, \sqrt{n/N}\big)$ so that $\{z \in \mathcal{I}_n\}$ is equivalent to $\{n>N_0(z)\}$. 
To describe the evolution of the characteristic polynomials for $n>N_0(z)$, we introduce a new process $\big\{\bpsi_n(z) : z\in \mathcal{I}_n \big\}$ by a linear combination;
\begin{equation} \label{def:psi}
\exp\big(\bpsi_n(z)\big) \coloneqq
\i\sqrt{\tfrac{n}{n-Nz^2}}
\Big( e^{-\i\theta_n(z)}\Phi_{n}(z)-
\sqrt{\tfrac{n+1}{n}}\Phi_{n+1}(z)\Big) , 
\qquad
\theta_n(z)   \coloneqq \arccos\big(z\sqrt{N/n}\big).
\end{equation}
This definition may seem ad hoc, but it comes naturally from the transfer matrix recursion and we verify that for $z\in (-1,1)$ and $n>N_0(z)$,  
\begin{equation} \label{polychar}
\Phi_{n}(z) = \Re\big( \exp\bpsi_n(z)\big). 
\end{equation}
In the sequel,  $\big\{\bpsi_n(z) : n\ge N_0(z) \big\}$ will be called the \emph{(complex) Pr\"ufer phase} and we decompose
\begin{equation} \label{def:phase}
\bpsi_n(z)\eqqcolon \rho_n(z) + \i \phi_n(z) , 
\qquad\qquad
\begin{pmatrix}
\rho_n \\ \phi_n
\end{pmatrix} : z\in\mathcal{I}_n \mapsto \R^2
\quad\text{are smooth functions.}
\end{equation}
The process $\big\{\bpsi_n(z) : z\in \mathcal{I}_n \big\}$ is well-defined because of the interlacing property of the zeros of $\Phi_{n+1}(z),\Phi_n(z)$ and the phase $\big\{\phi_n(z) : z\in \mathcal{I}_n \big\}$ is properly constructed in the  Appendix~\ref{app:charpoly}.
In particular, it satisfies multiple approximate monotonicity properties, most significantly for $z\in\mathcal{I}_n$,
\begin{equation}\label{eq:N_n}
\left\lfloor\phi_{n+1}(z) - \tfrac{\pi}{2}\right\rfloor_{\pi} =  \mathrm{N}_n([z,\infty))
\quad \text{where} \quad \mathrm{N}_n([z,\infty)) \coloneqq \#\{ \lambda \geq z : \Phi_n(\lambda) = 0 \}.
\end{equation}
Here, $\mathrm{N}_n:\R \mapsto [0,n]$ is the (non-increasing) counting function for the eigenvalues of the matrix $[({4}{N\beta})^{-1/2}\mathbf{A}]_{n}$; see Proposition \ref{prop:prufer} for a proof as well as other detailed properties. 

\paragraph{Main theorem.}
We now state our main result: 

\begin{theorem} \label{thm:main}
Suppose $z=z(N) \in (-1,1)$ is such that $N^{1/3}\varrho(z) \to \infty$.
Then for $\lambda \in \R$,
\[
\Phi_N\big(z + \tfrac{\lambda}{N \varrho(z)}\big)
=\Re\Big[\exp( \bpsi_N(z) + \tfrac12\varphi_N(\lambda;z))\Big]
=(1-z^2)^{-\mathrm{c}_\beta}
\Re\Big[\exp\big(\i\pi N F(z)  + \tfrac12\varphi_N(\lambda;z)- \tfrac{\M_N(z)}{\sqrt{\beta}} 
+\Omega_N(z)\big)\Big] 
\]
where $\mathrm{c}_\beta = \frac{1}{4}-\frac{1}{2\beta}$, $\varphi_N(0;z)=0$, and
where $F$, $\varphi_N(\lambda;z)$, $\{\M_n\}$ and $\Omega_N$ satisfy the following:
\begin{enumerate}[leftmargin=*]  
\item $F(z) = \int_z^1 \varrho(x)\dif x$ is the antiderivative of the semicircle law.
\item The pair $\big( \{\phi_N(z)\}_{2\pi}, \{ \varphi_N(\lambda;z) : \lambda \in \R \big)$ converges in distribution in the sense of finite dimensional marginals as $N\to\infty$ to $\big( \alpha, \omega_1(\lambda) : \lambda \in \R \big)$ where $\alpha$ is uniform in $[0,2\pi]$, independent of $\omega$, which is a solution of the \emph{complex sine equation} \eqref{eq:csse}. This extends to locally uniform convergence  when restricting to $\lambda\in\R \mapsto \Im \varphi_N(\lambda;z)$.
\item The process $\{\M_n :n \in \N\}$ is a martingale adapted to $\{\filt_n :n \in \N\}$ and it matches the correlation structure of the Gaussian field $\mathfrak{X}$; 
if $x = x(N) \in \R$, then as $N\to\infty$, 
\begin{equation} \label{corrM}
\begin{cases}
[\M_N(x), \M_N(z)] =  -2\log_{\epsilon_N(z)}\big(1-J(x)J(z)\big) + \O(1),\\
[\M_N(x), \overline{\M_N(z)}] = -2\log_{\epsilon_N(z)}\big(1-J(x)\overline{J(z)}\big) + \O(1),
\end{cases}
\end{equation}
with $\log_\epsilon(1-z) \coloneqq -\sum_{k \epsilon < 1} \frac{z^k}{k}$ and $\epsilon_N(z)^{-1} \coloneqq \max\{ N^{1/3}, N\varrho(z)^2\}$.  The errors $\O(1)$ terms are tight families of random variables. Moreover, if $|x-z|/\epsilon_N(z) \to \infty$, then the errors tend to $0$ in probability.
\item The error term $\big\{ \Omega_N(z) : N \in \N\big\}$ is tight and further converges in law as $N \to \infty$ provided either $Nz^2 = \lambda$ for fixed $\lambda \in \R$ or $Nz^2 \to \infty$.
\end{enumerate}
\end{theorem}


\paragraph{Discussion and corollaries.}
We note that there are two scaling regimes in Theorem~\ref{thm:main}, one where $Nz^2$ is fixed and another away from $0$ where $Nz^2 \to \infty$; indeed they differ in multiple qualitatively distinct ways.  
In particular, if $Nz^2=\lambda$ for $\lambda\in\R$, it is possible to entirely remove the parameter $N$ from the definition of the characteristic polynomial and we have formulated in Theorem \ref{thm:0} a version of Theorem \ref{thm:main} which is special to this regime. 

The representation of $\Phi_N(z)$ in Theorem~\ref{thm:main} is a generalization of the Plancherel-Rotach asymptotics \eqref{PR1} for the Hermite polynomials, which hold in the case of $\beta=\infty$. In particular, 
$\E \omega_1(\lambda) = 2\pi \i \lambda $ by \eqref{eq:csse} and the deterministic leading behavior is captured by the semicircle law, through $F(z)$, for all $\beta \in(0,\infty]$.  In what follows, we discuss in order the remaining $\varphi_N$, ${\M_N}$, and $\Omega_N$ terms.


\smallskip

To begin, $\varphi_N(\lambda;z)$ is an approximate solution of the complex sine equation, and it encodes the limiting Sine-$\beta$ point process.  As an immediate corollary of Theorem~\ref{thm:main}, we observe the convergence to the stochastic zeta function: 
\begin{corollary}\label{cor:zeta}
	Suppose $z=z(N)$ is such that $N^{1/3}\varrho(z) \to \infty$.
	Then 
	\[
	\bigl\{
	{
	\Phi_N\bigl(z + \tfrac{\lambda}{N \varrho(z)}\bigr)
	}/{
	\Phi_N(z)
	}
	:\lambda \in \R
	\bigr\}
	\Wkto[N]
	\bigl\{ \zeta_\beta(\lambda) : \lambda \in \R\bigr\},
	\]
	in the sense of finite dimensional marginals.
\end{corollary}
\begin{proof}
	Using \eqref{polychar}, we have that 
	\[
		\frac{\Phi_N\bigl(z + \tfrac{\lambda}{N \varrho(z)}\bigr)}{\Phi_N(z)}
		=\frac{\Re \exp( \i \phi_N(z) + \varphi_N(\lambda;z))}{\Re \exp( \i \phi_N(z))}.
	\]
	Hence from Theorem~\ref{thm:main} and the representation \eqref{eq:zeta}, the conclusion is immediate. 
\end{proof}

We recall that $\Phi_N$ is normalized by a deterministic weight \eqref{def:Phi}. If we instead consider the ratio of monic characteristic polynomials, we deduce that at any fixed $z \in (-1,1)$, in the sense of finite-dimensional marginals in $\lambda \in \R$,
\[
\prod_{j=1}^N \left( 1 - \frac{\lambda}{N \varrho(z)(z-\lambda_j)}\right)
= \frac{w_N(z)}{w_N(z+\tfrac{\lambda}{N \varrho(z)})}\frac{\Phi_N\bigl(z + \tfrac{\lambda}{N \varrho(z)}\bigr)}{\Phi_N(z)} 	\Wkto[N] \exp\left( \frac{2\lambda z}{\varrho(z)} \right) {\zeta}_\beta(\lambda).
\]
This extends the convergence in \cite{NajnudelNikeghbali} from the GUE to the G$\beta$E, and extends the convergence in \cite{ValkoViragZeta} from the C$\beta$E to the G$\beta$E.

\smallskip
	
Under the hypothesis from Definition~\ref{def:noise}, the martingale $\{\M_n : n \in \N\}$ have uniformly small increments, hence from the standard martingale central limit theorem, we have:
\begin{corollary}\label{cor:cltM}
	Suppose $z_j=z_j(N) \in \R$ for $j=1,\dots, k$ and  
	\[
	\frac{\log_{\epsilon_N(z_j)} \big(1-J(z_j)J(z_k)\big)}{\log N} \to \alpha_{j,k} 
	\qquad\text{and}\qquad
	\frac{\log_{\epsilon_N(z_j)} \big(1-J(z_j)\overline{J(z_k)}\big)}{\log N} \to \widetilde{\alpha}_{j,k} 
	\text{ as } N\to\infty.
	\]
	Then the coefficients $\alpha_{j,k}$ and $\widetilde{\alpha}_{j,k}$ are necessarily real-valued, and furthermore
	\[
	\biggl\{
	\tfrac{\M_N(z_j)}{\sqrt{ \log (N)}}
	: j=1,\dots,k
	\biggr\}
	\Wkto[N]
	\biggl\{ \mathfrak{X}'_j : j=1,\dots,k \biggr\},
	\]
	a family of centered complex normal random variables with $\Exp \mathfrak{X}'_j \mathfrak{X}'_k = \alpha_{j,k}$ and $\Exp \mathfrak{X}'_j \overline{\mathfrak{X}'_k} = \widetilde{\alpha}_{j,k}$. In particular, $\{ \Re\mathfrak{X}'_j : j=1,\dots,k\}$ and $\{ \Im\mathfrak{X}'_j : j=1,\dots,k\}$ are independent and the convergence also holds in the sense of exponential moments.
\end{corollary}

\begin{proof}
From the Definition \ref{def:M} below of the martingale $\M_N$, it is easily verified the sum of the fourth moments of the increments of the martingale is bounded independently of $N$.  Hence the Lypaunov CLT condition is satisfied, and the conclusion follows from the standard martingale central limit theorem using the estimates \eqref{corrM}.  The exponential integrability of the martingale follows as the sub-gaussian norm of $\M_N$ is control by its standard deviation $\Theta(\sqrt{\log N})$.  The independence of the real and imaginary parts of $\{\mathfrak{X}'_j : j=1,\dots,k\}$ is a consequence of the limit $\{\alpha_{j,k}\}$ and $\{\widetilde{\alpha}_{j,k}\}$ being real-valued, which is also a consequence of \eqref{corrM}.
For comparison, we  also record that 
	\[
	\biggl\{
	\tfrac{\mathfrak{X}\big(z_j+ \i \epsilon_N(z)\big)}{\sqrt{ \log (N)}}
	: j=1,\dots,k
	\biggr\}
	\Wkto[N]
	\biggl\{ \mathfrak{X}'_j : j=1,\dots,k \biggr\} . \qedhere
	\]
\end{proof}

The martingale $\{\Re\M_N(z)\}$ and $\{\Im\M_N(z)\}$ can be directly compared to the real part of the logarithm of the characteristic polynomial and to the recentered eigenvalue counting function\footnote{The counting function \eqref{eq:N_n}, $\mathrm{N}_n([z,\infty))$ can also be connected to the imaginary part of the logarithm of the characteristic polynomial $\frac{1}{\pi}\Im \log(\Phi_n(z))$, when the $\log(\cdot)$ is defined by continuity from the upper half plane (branch cuts to the left).}, respectively.  Specifically, we have the following relations:

\begin{corollary}\label{cor:ReIm}
	Suppose $z=z(N)$ is such that $N^{1/3}\varrho(z) \to \infty$.  Then
	\begin{equation}\label{tight}
	\begin{aligned}
	&\left\{
	\Re \log \Phi_N(z) -  \left( c_\beta \log (1-z^2) -\frac{\Re \M_N(z)}{\sqrt{\beta}}\right) : N \in \N
	\right\}
	\quad\text{and}\quad \\
	&\left\{
	\pi \mathrm{N}_N([z,\infty)) -  \left( \pi N F(z) - \frac{\Im \M_N(z)}{\sqrt{\beta}}\right)
	:N \in \N
	\right\}
	\end{aligned}
	\end{equation}
	are tight families of random variables.  Hence the  CLT  shown in Corollary~\ref{cor:cltM} holds with $\M_N(z_j)$ replaced by
	\[
	-\log|\Phi_N(z_j)|-\i \pi \mathrm{N}_N([z_j,\infty))
	+\left( c_\beta \log (1-z_j^2) + \i \pi N F(z_j) \right).
	\]
\end{corollary}
\begin{proof}
	The real part of the logarithm of the characteristic polynomial is given by 
	\[
	\Re \log \Phi_N(z) = \log | \Re \exp(\bpsi_N(z)) | =
	c_\beta \log(1-z^2) - \frac{\Re \M_N(z)}{\sqrt{\beta}} + \Re \Omega_N(z) + \log \cos( \phi_N(z) ).
	\]
	The $\Re \Omega_N(z)$ term is tight from Theorem \ref{thm:main}, and the last $\log \cos( \phi_N(z) )$ term converges in law.  Hence the tightness follows. For the imaginary part, we have from \eqref{eq:N_n} that
	\[
	\pi \mathrm{N}_N([z,\infty)) = \lfloor \phi_N(z) - \tfrac{\pi}{2} \rfloor_{\pi} = \bigg\lfloor \pi N F(z) - \frac{\Im \M_N(z)}{\sqrt{\beta}} - \frac{\pi}{2} \bigg\rfloor_{\pi}.
	\]
	Hence the claimed tightness follows.
\end{proof}

These corollaries immediately recovers a central limit theorem for the real and imaginary parts of the characteristic polynomial from \cite[Theorem 1.8]{BourgadeModyPain} in in the case of the G$\beta$E (the results of \cite{BourgadeModyPain} follow from an optimal local law and they hold for general regular one-cut $\beta$-ensembles). This builds on a large literature of related central limit theorems: this result is well-known in the determinantal case $\beta=2$ and is essentially due to \cite{Gustavsson} (which is formulated for the quantile function). For general $\beta>0$, the CLT for $\Re \Phi_N(x) $ with $x \in (-1,1)\setminus \{0\}$ is obtained in \cite{Augeri}.
The situation at $0$ is special and it has been considered in \cite{TaoVu,Duy17}.
The edge CLT following from Theorem~\ref{thm:PR2} has already been studied in \cite{Johnstone, LambertPaquette03}.

\begin{remark}
	The convergence of Theorem \ref{thm:main} and Corollary~\ref{cor:cltM} hold jointly in the sense that the process
	\[
	\big( \{\phi_N(z)\}_{2\pi}, \tfrac{\M_N(z)}{\sqrt{\log N}}, \Omega_N(z), \big\{\varphi_N(\lambda;z)  : \lambda\in \R \big\} \big)
	\]
	converges in the sense of finite dimensional distributions, and the limiting random variables are all independent.
	
Moreover, combining Theorem \ref{thm:main} with \cite[Theorem~1.1]{LambertPaquette03} and 
\cite[Theorem 1.7]{LambertPaquette02}, we also obtain the tightness of the families of random variables \eqref{tight} indexed by $z
\in\mathcal{\R}$. 
\end{remark}

Finally we add some detail on $\Omega_N(z)$ and its limit $\Omega(z)$.  
As a consequence of \cite{LambertPaquette03}, the recentred martingale $\{\M_{N_0 + t(N_0)^{1/3}} - \M_{N_0} : t \ge 1\}$ converges to a diffusive limit : $\{\mathfrak{m}_t^{-} : t \ge 1\}$  as $N_0(z) \to \infty$, driven by a 2-sided real Brownian motion $\{B(t) : t \in \R\}$.  With respect to this Brownian motion, we can construct a version of the Stochastic Airy function $\SAi_t = \SAi_t(0)$, which is solution of a second-order diffusion with respect to $\{B(t)\}$ (see Section \ref{sec:HP} for precise definitions).  This is a stochastic process whose mean is the classical Airy function $\Ai(t)$.  Moreover, $\SAi$ is the scaling limit of the characteristic polynomial of G$\beta$E at the edge (Theorem~\ref{thm:PR2}).

We show in Proposition \ref{cor:mho} that in the limit $T \to -\infty$, there is a complex random variable $\hat{\boldsymbol{\mho}}_\beta^{-}$ so that 
\[
\SAi_{-T}  = \Re\big\{ \exp\big(\i \big(\tfrac23 T^{3/2} - {\mathrm{c}_\beta}\pi \big) \ + \tfrac1{\sqrt\beta} \mathfrak{m}_T^{-}    + {\mathrm{c}_\beta} \log T + \hat{\boldsymbol{\mho}}_\beta^{-} + \o_\P(1)\big)\big\} ,\qquad {\mathrm{c}_\beta} =\tfrac14-\tfrac{1}{2\beta},
\]
where the error converges in probability as $T\to\infty$. Then the limit random variable $\Omega(z)$ is given in law by 
\[
\Omega_N(z)  \Wkto[N] \Omega(z)
= \hat{\boldsymbol{\mho}}_\beta^{-} 
-\i2 c_\beta \arcsin(z) 
-\frac{\log 2}{c_\beta} 
+ \frac{{\rm g}}{\sqrt{\beta}}-\frac{\Exp {\rm g}^2}{2\beta}
\]
where ${\rm g}$ is an explicit Wiener integral of $\{B(t) : t > 0\}$; see Proposition \ref{prop:limit}.  Hence $\Omega(z)$ is a functional of a scaling window of the driving noise $\{(X_n,Y_n)\}$ for $|n-N_0| \leq N_0^{1/3}T$, where $N_0 \to \infty$ followed by $T \to \infty$.  Thus one can see  $\Omega(z)$ as statistics typically associated to the edge of the G$\beta$E characteristic polynomial.

\smallskip

With more effort (and we do not go into the details), one can show for the G$\beta$E that \emph{all} non-Gaussian behavior $\bpsi_N(z)$ is captured by a window of this turning point $N_0(z)$.  This is to say, for any $\epsilon > 0$, there is a probability space supporting $\bpsi_N(z)$, a 2-sided Brownian motion $\{B^z(t) :t \in \R\}$ and a Gaussian random variable $\mathfrak{G}_{N,T}$, so that (for a distance compatible with the topology of weak convergence), if $T$ is sufficiently large,
\begin{equation} \label{alldressed}
\operatorname{dist}\big( \bpsi_N(z), \mathfrak{G}_{N,T} + \mathfrak{m}_T^{-} + \hat{\boldsymbol{\mho}}_\beta^{-} + g_T\big) \leq \epsilon, 
\quad \text{ with } \mathfrak{G}_{N,T} \text{ independent of } \mathfrak{m}_T^{-} + \hat{\boldsymbol{\mho}}_\beta^{-} + g_T.
\end{equation} 
Here ${\rm g}_T$ is an approximation of ${\rm g}$ and the law of $(\mathfrak{m}_T^{-}, \hat{\boldsymbol{\mho}}_\beta^{-} ,g_T)$ does not depend on $N$ or $z$, provided that $N_0(z) \to\infty$.  Hence we conjecture the Mellin transform of $|\Phi_N(z)|$ at a bulk point $z \in (-1,1) \setminus \{0\}$ converges to a functional of $\SAi$ and its driving Brownian motion $B(t)$ in the sense that
\begin{equation} \label{mellin}
\left(\Exp |\Phi_N(z)|^{s}\right) N^{-\frac{s^2}{2\beta}} \to \mathfrak{f}(s, \SAi, B),
\end{equation}\todo{Make better}
an explicit functional $\mathfrak{f}(s,\cdot)$ which is implicit from the representation \eqref{alldressed}.  We note that this Mellin transform has been identified in the work of \cite{DeiftItsKrasovsky} for the GUE, and the same factor appears in the CUE \cite{widom1973toeplitz}, in terms of a Barnes G-function which is independent of the bulk point $z \in (-1,1)$. For both the C$\beta$E and the G$\beta$E at $0,$ the Mellin transform is explicit \cite[Proposition 4.3]{bourgade2009circular} (see also \cite{ForresterFrankel}), and so we expect that \eqref{mellin} coincides with these cases.

\todo[inline]{Organization}

\subsection{Notations.}

Throughout this paper, we rely on the following conventions. 
Some of these notations are consistent with our previous work \cite{LambertPaquette02,LambertPaquette03}. 

\begin{definition}\label{def:Z}
The spectral parameter $z$  is allowed to depend on the dimension $N\in\N$ of the underlying matrix. We consider a 
sequence $z=z(N)$ with a limit point also denoted $z\in[-1,1]$. The \emph{turning point} of the transfer matrix recursion is $N_0(z)\coloneqq \lfloor Nz(N)^2\rfloor$ and we set
\[
\Q \coloneqq  \big \{ z(N)  \in (-1,1) ; \,  \liminf_{N\to\infty}  Nz(N)^2=\infty \text{ and } 
\liminf_{N\to\infty}  N^{1/3}\varrho(z(N)) =\infty \big\} .
\]
To describe the transition window (called the \emph{parabolic regime}), we introduce the following time units, for $T\ge 0$, 
\begin{equation} \label{para}
N_T(z) \coloneqq \lfloor Nz^2 +T \mathfrak{L}(z)\rfloor  , \qquad\qquad \L(z) =  \lceil Nz^2\rceil^{1/3} 
\end{equation}
for $z\in\R$.
The first condition in $\Q$ guarantees that $\L(z(N)) \to\infty$ as $N\to\infty$. 
The second condition in $\Q$ guarantees for any $T\ge 0$, $N_T(z(N)) \ll N$  as $N\to\infty$ so that the spectral parameter is away from the edge of the semicircle law.
\end{definition}

In the sequel, we will need to distinguish two asymptotic regimes: $z\in \Q$ or $z=\tfrac{\mu}{2\sqrt N}$  for $\mu \in \mathcal{K}$ where $\mathcal{K}\Subset\R$ is any compact. 
In the second case, the whole recursion is elliptic and according to \eqref{def:Phi}, 
\begin{equation} \label{def:Phi0}
\Phi_n(z) = N^{1/4} \widehat\Phi_n(\mu)\sqrt{e^{-\mu^2/2} /\sqrt{2\pi}} , 
\qquad 
\widehat\Phi_n(\mu) \coloneqq \det[\mu-\beta^{-1/2}\mathbf{A}]_{n,n} \sqrt{\textstyle \prod_{k=1}^n k^{-1}},  . 
\end{equation}
In particular, the sequence $\{\widehat\Phi_n(\mu)\}_{n\in\N}$ is independent of the parameter $N$. 
In contrast, if $z\in\Q$, the initial part of the transfer matrix recurrence is not elliptic and we need to import the asymptotics of $\Phi_n(z)$ for $n$ in a neighborhood of the turning point from our previous work \cite{LambertPaquette03}. We review the relevant results in Section~\ref{sec:HP}. 


\medskip

In terms of Definition~\ref{def:noise}, the random variables which naturally arise in the characteristic polynomial recursion are given by\footnote{For $w\in\C\setminus[-1,1]$, the map $J$ is given by \eqref{Jouk}. 
The expression \eqref{def:J} follows by continuity form the upper-half plane. 
We also note that $J$ has the reflection symmetries; $J(\overline{w}) = \overline{J(w)}$ for $w \in \C \setminus (-1,1)$ and 
$J(-w) =- \overline{J(w)}$ for $w\in\R$.}, for $z\in\R$ and  $n\in\N$, 
\begin{equation} \label{def:J}
Z_n(z) \coloneqq  \frac{X_n + J\big(z\sqrt{N/n}\big) Y_n}{\sqrt{2}}  , 
\qquad \qquad
J(w) \coloneqq \begin{cases}
w \mp \sqrt{w^2-1} , &\pm w \ge 1\\
e^{-\i\arccos(w)} , & w\in [-1,1] 
\end{cases} . 
\end{equation}

In the sequel, we will also use the following conventions: 
\begin{itemize}[leftmargin=*] \setlength\itemsep{0em}
\item Given two positive sequences $\{a(N)\}_{N\in\N}, \{b(N)\}_{N\in\N}$, we write  $a \gg b$ if $\displaystyle\lim_{N\to\infty} b(N)/a(N) =0$.
\item Similarly, we write $b \lesssim a$ if there is a constant $C=C(\beta,\mathfrak{S},\mathfrak{K})$ such that $\displaystyle\limsup_{N\to\infty} b(N)/a(N) \le C$.
$C$ is also allowed allowed to depend on other parameters independent of $(z,N)$. 

\item For a random field $\X =\big\{ \X_N(x) : x\in\mathcal{S}_N , N\in\N\big\}$, we write $\X_N(x)= {\o_\P(1)}$ if 
\[
\sup_{\epsilon>0}\limsup_{N\to\infty}\sup_{x\in\mathcal{S}_N} \P\big[|\X_N(x)|\ge \epsilon\big] =0 .
\]
That is, if for  all $x\in \mathcal{S}_N $,  $ \X_N(x)\to 0$ in probability as $N\to\infty$. 
\item Similarly, we write $\X= {\O_\P(1)}$ if 
\[
\lim_{R\to\infty}\limsup_{N\to\infty}\sup_{x\in\mathcal{S}_N} \P\big[|\X_N(x)|\le R\big] =1.
\]
That is if the random field $\X$ is \emph{tight}. 
\item Let $\varrho(x) \coloneqq \frac2\pi\sqrt{1-x^2}\1\{|x|\le 1\}$ be the  semicircle law on $[-1,1]$ and $F(z) = \int_z^{1} \varrho(x) \d x$ for $z\in[-1,1]$. 
\end{itemize}

\subsection{Martingale noise.}
Corollary~\ref{cor:cltM} is a consequence of the log-correlated structure of the martingale $\big\{\M_n(z) : z\in[-1,1]\big\}_{n\ge 1}$ which describes the \emph{macroscopic fluctuations} of the characteristic polynomial.
In this section, we give an explicit description of the martingale from Theorem~\ref{thm:main} and the asymptotics for its bracket process. 
The martingale can be decomposed in two processes (Definition~\ref{def:GW}), which have small correlations as $N\to\infty$. 

\begin{definition}[Martingale noise] \label{def:GW}
In terms of the random variables \eqref{def:J} and \eqref{para}, we define for $z\in\R$ and $n\le N$, 
\begin{equation*} 
\G_n(z) \coloneqq \sum_{0< k\le n} \1\{k\notin \Gamma(z)\} \frac{Z_k(z)}{\sqrt{k}\sqrt{Nz^2/k-1}} , \qquad
\Gamma(z)  \coloneqq \big\{k\in\N: |k-Nz^2| < \mathfrak{L}(z)\big\} , 
\end{equation*}
where $\sqrt{\cdot}$ is chosen as in \eqref{Jouk}\footnote{For $w\in[-1,1]$, $\sqrt{w^2-1}$ is imaginary and defined by continuity from the upper-half plane. Moreover, the $\sqrt{\cdot}$ is consistent with $J$ in~\eqref{def:J}.}. 
Similarly, we define 
\begin{equation*}
\W_{n}(z)   \coloneqq \sum_{N_0(z)<k \le n}   \1\{k\notin \Gamma(z)\}\frac{Z_k(z)e^{2\i(\theta_k(z)+\phi_{k-1}(z))}}{{\sqrt{k}\sqrt{Nz^2/k-1}}} , \qquad \text{the process $\{\phi_n(z) : n> N_0(z)\}$ is given by \eqref{def:phase}.} 
\end{equation*}
In particular, both processes $\{\G_n\}$, $\{\W_n\}$ are  $\{\F_n\}$-martingales and  we define  for $z\in\R$ and $n\le N$, 
\begin{equation}\label{def:M}
\M_n(z) \coloneqq \G_n(z) + \overline{\W_n(z)} .
\end{equation}
\end{definition}

We make the following remarks about these definitions
\begin{itemize}[leftmargin=*] \setlength\itemsep{0em}
\item We exclude the set $\Gamma(z)$ from this sum because the noise becomes singular around the turning point. This singularity is responsible for the log-correlated structure of $\{\G_N(z),\W_N(z)\}$. 
\item The martingale $\{\G_n\}$ is a sum of independent random variables, so its brackets are deterministic sums. In contrast, because of the rapid growth of the phase $\{\phi_n\}$,  the brackets of $\{\G_n\}$ are random \emph{oscillatory sums}. 
In fact, because of  these \emph{oscillations}, the field $z\in (-1,1) \mapsto \W_N(z)$ behaves like a \emph{white noise}. 
Hence, the long-range covariance structure of $z\in\R \mapsto \M_N(z)$ coincide with that of $z\in\R \mapsto\G_N(z)$.
\item For $n\le N_0(z)$, $\M_n(z)=\G_n(z)$ is real-valued and this contribution comes from the \emph{hyperbolic part} of the recursion. 
In particular, the field $z \in\R \mapsto \M_N(z)$ is real-valued for $z\in\R\setminus (-1,1)$ and log-correlated on the spectrum, for $z\in[-1,1]$. 
\end{itemize}

The next proposition collects the precise asymptotics of the martingale brackets and we distinguish two different regimes;  

\begin{proposition}[Correlation structure] \label{prop:M}
Let $x,z\in\R$ and suppose, without loss of generality, that $|x|\le |z|$. 
Let $[z]_N \coloneqq |z| \vee N^{-1/2}$ and $\epsilon_N(z) \coloneqq \big(N\varrho(z)^2 \vee N^{1/3}\big)^{-1}$ for $z\in[-1,1]$. 
The following asymptotics hold as $N\to\infty$, 
\begin{itemize}[leftmargin=*] \setlength\itemsep{0em}
\item[\normalfont 1.] {\normalfont (Global regime)} If  $\big(|z|-1\big)\gg N^{-2/3}$ or if $z\in[-1,1]$ with  $ |x-z| \gg  N^{-2/3} [z]_N^{-1/3}$,    
\[\begin{aligned}
&\big[\M_N(z),\M_N(x)\big]
= \big[\G_N(z),\G_N(x)\big] + {\o_\P(1)} , 
\qquad \big[\M_N(z),\overline{\M_N(x)}\big]=  \big[\G_N(z),\overline{\G_N(x)}\big] + \o_\P(1)
\\
&= - 2\log\big(1-J(z)J(x)\big)  + {\o_\P(1)}, \qquad
\qquad\qquad\, = - 2 \log\big(1- J(z)\overline{J(x)}\big) +{\o_\P(1)}.
\end{aligned}\]
\item[\normalfont 2.] {\normalfont (Local regime)} 
For a constant $C\ge 1$, if  $ |x-z| \le CN^{-2/3}[z]_N^{-1/3}$, 
\[
\big[\M_N(z),\M_N(x)\big]  =-2 \log\big(\varrho(z) \vee \epsilon_N(z)\big)  +\O_\P(1), \qquad 
\big[\M_N(z),\overline{\M_N(x)}\big] = -2\log\big( \tfrac{|x-z|}{\varrho(z)} \vee  \epsilon_N(z)\big)+ \O_\P(1).
\]
\end{itemize}
\end{proposition}

Proposition~\ref{prop:M} is proved in Section~\ref{sec:LCF}. 
The two asymptotic regimes are consistent and we recover the correlation structure from claim 3 of Theorem~\ref{thm:main}. 

\medskip

These two different regimes depend on whether the  \emph{turning point are merging} (the global regime corresponds to the case where $|N_0(x)-N_0(z)| \gg \mathfrak{L}(z)$ as $N\to\infty$). This phenomena plays a surprising role in producing the log-correlated behavior of the characteristic polynomial and the local regime requires an extensive analysis to control the effect of the oscillations of the phase. In this regime, our estimates hold up to errors which are tight random variables (there is a non-trivial part of the covariance coming from the \emph{stochastic-Airy}-type behavior near the common turning point). In fact, this \emph{transition} is not apparent from the log-correlations  \eqref{corrM} and it does not appear when studying the logarithm of the G$\beta$E characteristic polynomial by other methods; for instance \cite{BourgadeModyPain} using loop equations or \cite{CFLW} using the determinantal structure.

\begin{remark}[Symmetry around 0] \normalfont\label{rk:sym0}
Let $\mathbf{A}^\dagger$ be the random  Jacobi matrix associate with the sequence  $\{a_k,-b_k\}_{k\in\N}$, \eqref{def:trimatrix}, and $\{\Phi_n^\dagger(z)\}_{n\in\N}$  the correspdonding sequence of characteristic polynomials, \eqref{def:Phi}. 
By construction, we have the relationships for any $z\in(-1,1)$ and $n\in\N$, 
\[
\Phi_n^\dagger(-z) = (-1)^n\Phi_n(z) , \qquad\qquad \bpsi_n^\dagger(-z)= \overline{\bpsi_n(z)}+\i n\pi .
\]
This is consistent with the fact that the map $\mathbf{A}\mapsto\mathbf{A}^\dagger$ transforms for $k\in\N$, 
\[
(X_k,Y_k) \mapsto (-X_k,Y_k) , \qquad\qquad
Z_k(z) \mapsto Z_k^\dagger(z) = - \overline{Z_k(-z)} .   
\]
Then, under this map, we verify that $\G_N^\dagger(-z) = \overline{\G_N(z)}$
and $\W_N^\dagger(-z) = -\overline{\W_N(z)}$. 
In particular, if the coefficients $(b_k)_{k\in\N}$ have symmetric laws as  it is the case for the Gaussian $\beta$-ensembles, then the martingale satisfies $\M_N(-z) \overset{\rm law}{=} \overline{\M_N(z)}$ for $z\in\R$. 
\end{remark}

\todo{Paper plan}

\section{Parabolic regime} \label{sec:HP}



\subsection{Edge asymptotics.} \label{sec:edge}
In this section, we review the main results from \cite{LambertPaquette02,LambertPaquette03}  which give an approximation of the characteristic polynomial in a transition window around the turning point. This is a crucial input in this paper which provides the asymptotics of the characteristic polynomial at the beginning of the elliptic part of the recurrence. 
First, we recall the definition of the \emph{Stochastic Airy function}. 

\begin{definition}[Stochastic Airy function] \label{def:SA}
Let  $\{B(t); t\in\R\}$ be a standard (two-sided) Brownian motion.
Let $\big\{\SAi_t(\lambda)
: t \in \R, \lambda \in \R\big\}$ be the unique\footnote{By the general theory, for any $\lambda\in\R$, the SDE has a unique solution in $L^2(\R_+)$ up to a multiplicative constant. This solution is constructed in \cite{LambertPaquette03} and it is fixed by the condition $\E\SAi_t(\lambda) = \Ai(t+\lambda)$ for $\lambda,t \in\R$ and $\beta>0$. Moreover, the zeros of $t\mapsto \SAi_t(\lambda)$ and $t\mapsto \partial_t\SAi_t(\lambda)$ interlace, so that the process $(t,\lambda)\in\R_+\times\R\mapsto\boldsymbol{\varpi}_\lambda(t)$ is well-defined (up to a multiple of $2\pi$) and continuous. Using \cite{LambertPaquette03}, Proposition 1.4 and Proposition 6.4,  as $\lambda\to+\infty$, $\partial_t\SAi_{-1} (\lambda) \sim -\sqrt{\lambda} \SAi_{-1} (\lambda) $ with $\SAi_{-1} (\lambda)>0$ so that we can fix the complex phase by $\Im\boldsymbol{\varpi}_\lambda(1) \to \mp\pi/2$ as $\lambda\to+\infty$.} strong solution in $H^1(\R_+)$ of the equation
\begin{equation*}
\partial_{tt}\phi_t(\lambda)=  \phi_t(\lambda)\big(t + \lambda + \tfrac{2}{\sqrt\beta} \d B(t)\big)  . 
\end{equation*}
In terms of this stochastic Airy function, we define 
\begin{equation} \label{SAistruct}
\exp\big(\boldsymbol{\varpi}_t^{\pm}(\lambda)\big)  \coloneqq  \SAi_{-t}(\lambda)
\pm \i t^{-1/2} \SAi_{-t}'(\lambda)  , \qquad\lambda\in\R, \, t>0. 
\end{equation} 
where $(t,\lambda) \mapsto \SAi_{t}'(\lambda)=\partial_t \SAi_{t}(\lambda)$ is a continuous function on $\R^2$.

The processes $\big\{\boldsymbol{\varpi}^\pm_t(\lambda) ; \lambda\in\R , t>0\big\}$ are continuous, smooth with respect to $\lambda\in\R$, and satisfy $\Im\boldsymbol{\varpi}^\pm_t(\lambda) \to \mp\pi/2$ as $\lambda\to+\infty$, for $t>0$ fixed. 
Moreover, $\boldsymbol{\varpi}_t^{-}(\lambda) = \overline{\boldsymbol{\varpi}_t^{+}(\lambda)}$ 
for $t>0$ and $\lambda\in\R$. 
\end{definition}

The main result of \cite{LambertPaquette03} is the counterpart of Theorem~\ref{thm:main} at the edge of the spectrum. 

\begin{theorem}[\cite{LambertPaquette03}, Theorem~1.1]\label{thm:PR2}
Let $\pm=\pm1$ and $\M_N(\pm) =  \G_{N_0}(\pm1)$ according to \eqref{def:M}. 
There are two independent stochastic Airy functions $\SAi^\pm$, so that 
\[
(\pm)^N\Phi_N\big(\pm\big(1 + \tfrac{\lambda}{2N^{2/3}}\big)\big)
= \sqrt{N^{1/3}}\exp\Big(\tfrac{\M_N(\pm)+\mathrm{g}_{\pm}}{\sqrt{\beta}} - \tfrac{\E \M_N(\pm)^2+\E\mathrm{g}_{\pm}^2}{2\beta}\Big)\big(\SAi_0^{\pm}(\lambda) +\underset{N\to\infty}{\o_\P(1)}\big)
\]
where $\mathrm{g}_{\pm}$  are (identically distributed) real Gaussian  random variables with mean zero and, for any compact $\mathcal{K} \subset \R$,  the error converges to 0 in probability uniformly for $\lambda\in\mathcal{K}$. 
\end{theorem}

These asymptotics should be compared to the Hermite polynomial asymptotics \eqref{PR2} (case $\beta=\infty$). 
In particular, $\SAi$ is the (random) counterpart of the Airy function for the edge asymptotics of the G$\beta$E characteristic polynomial.
Theorem~\ref{thm:PR2} is proved by using an explicit coupling and, in particular, the scaling limits of the characteristic polynomial and the eigenvalues at the edges $\pm1$ are independent.

\todo[inline]{What is missing from the Airy paper \\
- discrete approximation of the noise?
- Behavior near zero (including in the hyperbolic)? Change $N$ by $N_0$, etc
- results for the derivative \\
- $\pm$?\\
- indepdendence?}

\subsection{Asymptotics around the turning point.}
Throughout this section, we assume that $z\in\Q$ (Definition~\ref{def:Z}), otherwise the characteristic polynomial recursion has no turning point. 
We also proved in \cite{LambertPaquette03} that the asymptotics of the characteristic polynomial around the turning point are also described by the stochastic Airy function. 
This result is somewhat expected from the \emph{scale invariance} property of the random matrix \eqref{def:trimatrix}. The following result is a reformulation of \cite[Theorem 1.6]{LambertPaquette03}.

\begin{theorem}\label{thm:Airy}
Let $z\in\Q$, $\pm =\sgn(z)$, $\mathfrak{L}=\mathfrak{L}(z)$ and $N_t =N_t(z)$ for $t\in\R$. 
Recall that $\M_{N_0}(z) =\G_{N_0}(z)$ and  define
\[\begin{cases}
\widetilde{\Phi}_t(\lambda;z) \coloneqq  (\pm)^{N_t} \Phi_{N_t}\big(z\big(1 + \tfrac{\lambda}{2\mathfrak{L}^2}\big)\big) \cdot \big(\tfrac{N}{\L}\big)^{-1/4} \exp\Big(\tfrac{\G_{N_0}(z)}{\sqrt{\beta}} + \tfrac{\E[\G_{N_0}(z)^2]}{2\beta}\Big) 
\\
\widetilde{\Phi}_t'(\lambda;z) \coloneqq  (\pm)^{N_t}
\mathfrak{L}\Big(\Phi_{N_t}\mp \Phi_{N_t+1}\Big)\big(z\big(1 + \tfrac{\lambda}{2\mathfrak{L}^2}\big)\big)
\cdot\big(\tfrac{N}{\L}\big)^{-1/4}\exp\Big(\tfrac{\G_{N_0}(z)}{\sqrt{\beta}} + \tfrac{\E[\G_{N_0}(z)^2]}{2\beta}\Big) .
\end{cases}\]
For compact sets $\mathcal{K},\mathcal{T} \subset\R$, it holds in distribution $($in the $C^1(\mathcal{K}) \times C^0(\mathcal{T})$ topology$)$ as $N\to\infty$, 
\[
\big\{\widetilde{\Phi}_t(\lambda;z), \widetilde{\Phi}_t'(\lambda;z) ;  \lambda\in\mathcal{K}, t\in\mathcal{T} \big\} \to
\big\{ \SAi_{-t}(\lambda) , \SAi_{-t}^{\prime}(\lambda) ; \lambda\in\mathcal{K}, t\in\mathcal{T} \big\}\cdot \exp\big(\tfrac{\mathrm{g}}{\sqrt{\beta}} -\tfrac{\E(\mathrm{g}^2)}{2\beta} \big)
\]
where $\SAi$ is a stochastic Airy function and $\mathrm{g}$ is a mean-zero real Gaussian. 
\end{theorem}

The proof of  \cite[Theorem 1.6]{LambertPaquette03} proceeds by an explicit coupling of the noise from Definition~\ref{def:noise} with a Brownian motion $B^z = B = \{B_t , t\in\R\}$. 
In particular, the random variables $\{\SAi,\mathrm{g}\}$ from  Theorem~\ref{thm:Airy} are both defined in terms of $B^z$ and they are not independent.

In the sequel, we will use the following consequence of Theorem~\ref{thm:Airy}.

\begin{proposition}\label{prop:Airy}
Let $z\in\Q$, $\pm =\sgn(z)$, $\mathfrak{L}=\mathfrak{L}(z)$ and $N_t =N_t(z)$ for $t>0$.
Recall \eqref{def:psi} and define $\mho_N^1(\lambda, t;z)$ for $t>0$ and  $\lambda\in\R$ (implicitly) by  
\begin{equation} \label{err1}
\bpsi_{N_t}\big(z\big(1 + \tfrac{\lambda}{2\mathfrak{L}^2}\big)\big)
= \i\pi \1\{z<0\}N_t+ \tfrac14 \log\big(\tfrac{N}{\L}\big) - \tfrac{\G_{N_0}(z)}{\sqrt{\beta}} - \tfrac{\E\G_{N_0}(z)^2}{2\beta}+ \mho_N^1(\lambda, t;z) .
\end{equation}
For compact sets $\mathcal{K} \subset\R$ and $\mathcal{T} \subset\R_+$, the following limits hold jointly in  distribution as $N\to\infty$, 
\begin{enumerate}[leftmargin=*]  
\item  $\big\{\mho_N^1(\lambda,t;z) ; \lambda\in\mathcal{K}, t\in\mathcal{T} \big\} \to \big\{ \frac{\mathrm{g}}{\sqrt{\beta}} -\frac{\E(\mathrm{g}^2)}{2\beta}+ \boldsymbol{\varpi}^\pm_t(\lambda); \lambda\in\mathcal{K}, t\in\mathcal{T}  \big\} $ in the $C^1(\mathcal{K}) \times C^0(\mathcal{T})$ topology\footnote{This means that the processes $\mho_N^1(\lambda,t;z)$ and $\partial_\lambda\mho_N^1(\lambda,t;z)$ both converge uniformly for $(\lambda,t)\in \mathcal{K}\times\mathcal{T}$. Actually, the convergence of \cite{LambertPaquette03} holds in $C^k(\mathcal{K})$ for any $k\ge 1$ but we will need such a fact.}, where $ \boldsymbol{\varpi}^\pm$  are independent processes and 
$\boldsymbol{\varpi}^-\overset{\rm law}{=} \overline{\boldsymbol{\varpi}^+}$. 
\item the martingale satisfies $\big\{\M_{N_t,N_1}(z) \to  \mathfrak{m}_t^{\pm} , t\in\mathcal{T}\big\}$ in the $C^0(\mathcal{T})$ topology,where $\big\{ \mathfrak{m}_t^{\pm} , t\in[1,\infty)\big\}$ is a continuous martingale, $\mathfrak{m}^{\pm}$ are independent with $\mathfrak{m}^- \overset{\rm law}{=} \overline{\mathfrak{m}^+}$. 
\end{enumerate}
\end{proposition}

\begin{proof}
We can rewrite  \eqref{def:psi}; for $z\in (-1,1)$ and $n>N_0(z)$,  
\begin{equation*}
\exp\big(\bpsi_n(z)\big) 
= \Phi_{n}(z)-\tfrac{\i}{\sin\theta_n(z)}\Big(
\sqrt{\tfrac{n+1}{n}}\cdot\Phi_{n+1}(z)- \cos \theta_n(z) \cdot\Phi_{n}(z)\Big) . 
\end{equation*}
Moreover, around the turning point, 
\[\begin{aligned}
\exp\big(\i\theta_{N_t}\big(z\big(1 + \tfrac{\lambda}{2\mathfrak{L}^2}\big)\big)\big)
&= \pm \sqrt{Nz^2\big(1 + \tfrac{\lambda}{2\mathfrak{L}^2}\big)^2/N_t(z)} +\i\sqrt{1- Nz^2\big(1 + \tfrac{\lambda}{2\mathfrak{L}^2}\big)^2/N_t } \\
&= \pm1+\i \sqrt{t}\mathfrak{L}^{-1} +\O\big(\tfrac{|t|+|\lambda|}{\mathfrak{L}^2}\big)
\end{aligned}\]
so that with $n=N_t$
\[
\exp\big(\bpsi_n\big)\big(z\big(1 + \tfrac{\lambda}{2\mathfrak{L}^2}\big)\big)
= \Big(a_N^1\Phi_{n}\pm \tfrac{\i\mathfrak{L}}{\sqrt{t}}a_N^2\big(\Phi_{n}\mp\Phi_{n+1} \big)\Big)\big(z\big(1 + \tfrac{\lambda}{2\mathfrak{L}^2}\big)\big)
\]
where the coefficients $a_N^j$ are deterministic and  $a_N^j = 1+\o(1)$ as $N\to\infty$ uniformly for $\lambda\in\mathcal{K}, t\in\mathcal{T}$.

Then, using the asymptotics from Theorem~\ref{thm:Airy} and \eqref{SAistruct}, we obtain
\[\begin{aligned}
\exp\big( \mho_N^1(\lambda, t;z) \big) 
&=  (\pm)^{N_t} \exp \bpsi_{N_t}\big(z\big(1 + \tfrac{\lambda}{2\mathfrak{L}^2}\big)\big)
\Big(\big(\tfrac{N}{\L}\big)^{1/4} \exp\Big(\tfrac{\G^1}{\sqrt{\beta}} - \tfrac{\E(\G^1)^2}{2\beta}\Big)\Big)^{-1} \\
&= a_N^1 \widetilde{\Phi}_t(\lambda;z) \pm\i a_N^2 t^{-1/2}\widetilde{\Phi}_t'(\lambda;z)  \\
&\to \exp\big(\boldsymbol{\varpi}^\pm_t(\lambda)+ \tfrac{\mathrm{g}}{\sqrt{\beta}} -\tfrac{\E\mathrm{g}^2}{2\beta}\big) 
\end{aligned}\]
in distribution as $N\to\infty$ as $C^1\times C^0$ processes (that is, uniformly for $\lambda\in\mathcal{K}, t\in\mathcal{T}$).
This proves the first claim. In particular, the imaginary part of $\{\bpsi_n\}_{n>N_0}$ satisfies locally uniformly for  $t>0$, as $N\to\infty$, 
\[
\phi_{N_t}(z)- \pi \1\{z<0\}N_t \to  \boldsymbol{\chi}_t^{\pm} , \qquad\qquad
\boldsymbol{\chi}_t^{\pm} = \Im\boldsymbol{\varpi}^\pm_t(0) \text{ is a continuous real-valued process on $\R_+$}.
\]

For the second claim, recall the  Definition~\ref{def:GW} of the martingales $\G$ and $\W$.
In terms of the Brownian motion $B=B^z$ from the coupling of \cite[Theorem 1.6]{LambertPaquette03}, we have  
\begin{equation} \label{Mpara}
\G_{N_T,N_1}(z) \to \int_1^T \frac{\d B_t}{\i\sqrt{t}} , \qquad\qquad 
\W_{N_T,N_1}(z) \to \int_1^T e^{2\i \boldsymbol{\chi}_t^{\pm}} \frac{\d B_t}{\i\sqrt{t}}  ,
\end{equation}
in distribution as $N\to\infty$ as $C^0$ processes (indexed by $T\in\R_+$) and these limits hold jointly with that of Theorem~\ref{thm:Airy}.  
\eqref{Mpara} follows from the approximations $Z_n(z) \approx \tfrac{X_n+Y_n}{\sqrt 2}$  and $\d B_t^z \approx Z_n(z)/\sqrt{\L(z)}$
for $n=N_t(z)$ if $z\in\Q$. 
In this regime, $\sqrt{n}\sqrt{Nz^2/n-1} = \i \sqrt{t\L(z)}$, $e^{2\i\theta_n(z)} \simeq 1$ so that the sums $\G_{N_T,N_1}$ and  $\W_{N_T,N_1}$ converge to stochastic integrals without any further normalization (for $\W$, we use that the phase converges as a continuous process). This proves the second claim; for $z\in\Q$, in distribution  as $N\to\infty$, 
\[
\M_{N_T,N_1}(z) \to    \mathfrak{m}_T^{\pm}  = \int_1^T \big(1+e^{-2\i \boldsymbol{\chi}_t^{\pm}}\big) \frac{\d B_t}{\i\sqrt{t}}  . \qedhere
\]
\end{proof}

\begin{remark}[Independence at different points in the spectrum]
For another spectral parameter $x\in\Q$ with $|N_0(z)- N_0(x)|\gg \L(z)\vee \L(x)$, the limits from Proposition~\ref{prop:Airy} are independent. 
This follows from the fact that the coupling of \cite[Theorem 1.6]{LambertPaquette03} operates in a window of size $\O(\L(z))$ around the turning points, so we can choose the Brownian motions $B^z,B^x$ independently in this regime. 
\end{remark}

As a byproduct of our analysis of the characteristic polynomial in the elliptic regime, we can deduce the asymptotics of the stochastic Airy function in the oscillatory direction. 

\begin{proposition} \label{cor:mho}
There is a complex-valued random variable $\hat{\boldsymbol{\mho}}_\beta^{-}$, such that
\[
\SAi_{-T}(0)  = \Re\big\{ \exp\big(\i \big(\tfrac23 T^{3/2}- {\mathrm{c}_\beta}\pi \big) \ + \tfrac1{\sqrt\beta} \mathfrak{m}_T^{-}    + {\mathrm{c}_\beta} \log T + \hat{\boldsymbol{\mho}}_\beta^{-} + \o_\P(1)\big)\big\} ,\qquad {\mathrm{c}_\beta} =\tfrac14-\tfrac{1}{2\beta},
\]
where the error converges in probability as $T\to\infty$. 
\end{proposition}

\begin{proof}
According to the definition \eqref{SAistruct}, $\SAi_{-t}(\lambda) = \Re\big(\exp \boldsymbol{\varpi}_t^{\pm}(\lambda)\big)$.
We will obtain the asymptotics of random phase $\boldsymbol{\varpi}_t^{\pm}(0)$ as $t\to\infty$ in the proof of Proposition~\ref{prop:limit}, by \eqref{limitmho}, it holds  in distribution as $T\to\infty$, 
\[
\big( \boldsymbol{\varpi}^\pm_T(0) + \tfrac1{\sqrt\beta} \mathfrak{m}_T^{\pm}   \mp\i   \big(\tfrac23 T^{3/2}- {\mathrm{c}_\beta}\pi \big) + {\mathrm{c}_\beta} \log T\big) 
\to \hat{\boldsymbol{\mho}}_\beta^{\pm}  . 
\]
where $\hat{\boldsymbol{\mho}}_\beta^+ \overset{\rm law}{=} \overline{\hat{\boldsymbol{\mho}}_\beta^-}$ by Proposition~\ref{prop:Airy}. 
\end{proof}

\begin{remark}[Airy function asymptotics]\label{rk:Airy}
If $\beta=\infty$ and $z\in\Q$, similarly to the asymptotics \eqref{PR2}, it holds as $N\to\infty$, locally uniformly for $t\in\R$, 
\[
(\tfrac{\L}{N})^{1/4} h_{N_t}(z) = (\pm1)^{N_t} \Ai(-t)\big(1+\O(\L^{-1})\big) 
\]
with $\L=\L(z)$. 
Denote $\bpsi_{N_t}^\infty(z) = \bpsi_{N_t}(z)|_{\beta=\infty}$ for $t>0$. 
Since $\Phi_n|_{\beta=\infty} =h_n$ for all $n\in\N$, we deduce that as $N\to\infty$,
\[\begin{aligned}
\exp\big(\bpsi_{N_t}^\infty(z)\big) 
& \simeq h_{N_t}(z)\mp \tfrac{\i\mathfrak{L}}{\sqrt{t}}\big(
\pm  h_{N_t+1}(z) - h_{N_t}(z)\big)    \\
&\simeq  (\pm1)^{N_t}  (\tfrac{N}{\L})^{1/4}  \big( \Ai(-t) \pm \tfrac{\i}{\sqrt t} \Ai'(-t)\big) . 
\end{aligned}\]
Using the Airy function asymptotics in the oscillatory direction
\cite[Section 9.7]{DLMF}, we deduce that as $N\to\infty$ and then $T\to\infty$,
\[
\exp\big(\bpsi_{N_T}^\infty(z)\big) 
\simeq  (\pm1)^{N_T}  (\tfrac{N}{\pi^2 \L T})^{1/4}  \exp\big(\pm \i \big(\tfrac23 T^{3/2}- \tfrac\pi4\big)\big)
\]
This is consistent with Proposition~\ref{prop:Airy}, we obtain the asymptotics as $N\to\infty$ and $T\to\infty$,
\[
\mho_N^1(0,T;z)|_{\beta=\infty} \simeq  \boldsymbol{\varpi}^\pm_T|_{\lambda=0,\beta=\infty}  \simeq  - {\mathrm{c}_\infty}\log(\pi^2 T) \pm \i \big(\tfrac23 T^{3/2}- \tfrac\pi4\big) .
\]
In particular, this shows that $\hat{\boldsymbol{\mho}}_\infty = -\tfrac{\log \pi}{2}$. 
\end{remark}

\subsection{Continuity estimates.}

Let $\mathcal{I}_n= \big(-\sqrt{n/N}, \sqrt{n/N}\big)$. 
In the course of the proof, we will need some continuity estimates for the process $\big\{\bpsi_n(z) ; z\in\mathcal{I}_n\big\}$ at the beginning of the elliptic stretch, that is, for $n\in\N$ slightly after $N_0(z)$. 
If $z\in\Q$, these estimates are a direct consequence of Proposition~\ref{prop:Airy}.

\begin{proposition} \label{prop:entcont}
Let $\alpha>2$, $z\in[-1,1]$ and let $\L=\L(z)$, $N_T=N_T(z)$ for $T\ge 1$. For any $c>0$, we have
\[
\liminf_{\varepsilon\to0}
\liminf_{N\to\infty}
\P\bigg[\sup_{|z-w| \le \varepsilon/\sqrt{N\mathfrak{L}}}\bigg(\frac{|\bpsi_{N_T}(w)- \bpsi_{N_T}(z)|^\alpha}{N\mathfrak{L}|z-w|^{2}} \bigg)\leq c\bigg] =1. 
\]
\end{proposition}

\begin{proof}
Since $\alpha>2$, there is a $\delta>0$ and a numerical constant so that for any $\varepsilon\in(0,1]$, 
\[ 
\sup_{|z-w| \le \varepsilon/\sqrt{N\mathfrak{L}}}\bigg(\frac{|\bpsi_{N_T}(w)- \bpsi_{N_T}(z)|^\alpha}{N\mathfrak{L}|z-w|^{2}} \bigg)
\lesssim \epsilon^\delta \sup_{|\lambda| \le 2\varepsilon}\bigg| \frac{\bpsi_{N_T}\big(z\big(1 + \tfrac{\lambda}{2\mathfrak{L}^2}\big)\big)- \bpsi_{N_T}(z)}{\lambda} \bigg|^\alpha.
\]
By proposition~\ref{prop:Airy}, the random variable on the RHS converges in distribution as $N\to\infty$ (for a fixed $\varepsilon\in(0,1]$) and then as $\epsilon\to0$.
Indeed, using \eqref{err1}, we can replace the process $\bpsi_{N_T}\big(z\big(1 + \tfrac{\lambda}{2\mathfrak{L}^2}\big)\big)$ by  $\mho_N^1(\lambda, T;z)$ since all the other terms are independent of $\lambda$, then since $\L\to\infty$ as $N\to\infty$ and the limit process $\lambda\mapsto\boldsymbol{\varpi}_T(\lambda)$ is smooth on $\R$, we obtain  
\[
\sup_{|\lambda| \le 2\varepsilon}\bigg| \frac{\bpsi_{N_T}\big(z\big(1 + \tfrac{\lambda}{2\mathfrak{L}^2}\big)\big)- \bpsi_{N_T}(z)}{\lambda} \bigg|
\, \underset{N\to\infty}{\to}\,
\sup_{|\lambda| \le 2\varepsilon}\bigg| \frac{ \boldsymbol{\varpi}_T(\lambda)- \boldsymbol{\varpi}_T(0)}{\lambda} \bigg| 
\,\underset{\varepsilon\to0}{\to}\, |\partial_\lambda\boldsymbol{\varpi}_T(0)|. 
\]
By Slutsky's Lemma, this implies that in probability, 
\[
\limsup_{\varepsilon\to0}
\limsup_{N\to\infty}
\sup_{|z-w| \le \varepsilon/\sqrt{N\mathfrak{L}}}\bigg(\frac{|\bpsi_{N_T}(w)- \bpsi_{N_T}(z)|^\alpha}{N\mathfrak{L}|z-w|^{2}} \bigg) =0 .
\]
This proves the claim. 
\end{proof}

\section{Elliptic regime} \label{sec:ell}

The goal of this section is to prove that if the spectral parameter $z$ is inside the bulk, the random phase $\psi_{N}(z)$ which characterizes the characteristic polynomial, \eqref{polychar} can be decomposed as some deterministic terms, the martingale term $ \M_{N}(z)$ and an \emph{error} $\Omega_N^2(z)$ which forms a tight sequence of random variables as $N\to\infty$; see Proposition~\ref{prop:oneray1} below. 
The proof consists in analyzing the recursion for the sequence of Pr\"ufer phases $\{\psi_{n}(z)\}$ after the tuning point by using a linearization scheme.

\subsection{Elliptic recursion.}\label{sec:rec}
The goal of this section is to transform the $2\times2$ recursion for the characteristic polynomial in a scalar one using the transformation \eqref{xi1}.
Then, $\xi_n(z) \in\C$ does not vanish (because of the interlacing property of the zeros of $\Phi_n$) and $\xi_n(z) = e^{\bpsi_n(z)} $ according to \eqref{def:psi} (see Proposition~\ref{prop:prufer}). 
To describe the evolution of the process $\big\{\xi_{n}(z);z\in(-1,1) ,n >N_0(z)\big\}$, we rely on the following notation. 

\begin{definition} \label{def:ell}
Let  $z\in(-1,1)$ and $n>N_0(z)$, recall that 
$\theta_n(z) = \arccos\big(z\sqrt{N/n}\big)$ and, in terms of the random variables from Definition~\ref{def:noise}, define
\[
\delta_n(z) \coloneqq \frac1{\sqrt{n-Nz^2}}, \qquad 
\Delta_n(z) \coloneqq \frac{1}{2}\bigg(1-\frac{\delta_n(z)}{\delta_{n-1}(z)}\bigg), \qquad
Z_n' (z)\coloneqq \frac{\i\delta_n(z)}{\sqrt{2\beta}}   \big( \sqrt{\tfrac{n-1}{n}}  e^{\i \theta_{n-1}(z)} X_n + Y_n \big) e^{-\i \theta_{n}(z)}.
\]
\end{definition}

\begin{lemma} \label{lem:rec}
For $z\in(-1,1)$, the process $\{\xi_{n}(z)\}_{n>N_0(z)}$ is the (unique) solution of the equation 
\[
\xi_{n}e^{-\i\theta_n} 
=  \big(1-\Delta_n+   Z_n' \big) \xi_{n-1} + \big( \Delta_n  - \overline{Z_n'} e^{-2\i\theta_n} \big) \overline{\xi_{n-1}}  . 
\]
\end{lemma}

\begin{proof}
According to  \eqref{rec1} and \eqref{xi1}, the process $\{\xi_{n}(z)\}_{n>N_0(z)}$ satisfies the recursion
\begin{equation} \label{rec2}
\begin{pmatrix}
\xi_{n} \\
\overline{\xi_{n}}
\end{pmatrix}
= \sqrt{\tfrac{4N}{n}} V_{n}^{-1}  T_{n}^\beta  V_{n-1} \begin{pmatrix}
\xi_{n-1} \\
\overline{\xi_{n-1}}
\end{pmatrix}.
\end{equation}
Using \eqref{mnoise} and that $T_{n}^\infty = V_n \Lambda_n V_{n}^{-1}$, we split
\[\begin{aligned}
V_{n}^{-1}  T_{n}^\beta  V_{n-1}
&= V_{n}^{-1}T_{n}^\infty(z) V_{n-1}  -  \frac{1}{\sqrt{2\beta N}}  V_{n}^{-1}\begin{pmatrix} X_n & \sqrt{\frac{n}{4N}} Y_n\\  0 & 0  \end{pmatrix} V_{n-1} \\
&= \Lambda_n V_{n}^{-1}V_{n-1}  -  \frac{1}{\sqrt{2\beta N}}  V_{n}^{-1} \begin{pmatrix} Z_n'' & \overline{Z_n''} \\  0 & 0  \end{pmatrix}
\end{aligned}\]
where $Z_n'' = \big(\begin{smallmatrix} X_n \\  \sqrt{\frac{n}{4N}} Y_n\end{smallmatrix}\big)\cdot \big(\begin{smallmatrix}   \sqrt{\frac{n-1}{4N}} e^{\i\theta_{n-1}} \\ 1\end{smallmatrix}\big) = \sqrt{\frac{n}{4N}}\big(X_n \sqrt{\frac{n}{n-1}} e^{\i\theta_{n-1}} +Y_n\big)$. 
This expression follows from \eqref{conj} and we also have
\begin{equation}
\label{def:Delta}
V_{n}^{-1}  V_{n-1} 
= \mathrm{I} - \Delta_n 
\begin{pmatrix} 1 & -1 \\ -1 & 1 \end{pmatrix}.
\qquad\qquad
\Delta_n =\frac{\big(\sqrt{\frac{n}{4N}}e^{\i \theta_n}-\sqrt{\frac{n-1}{4N}}e^{\i \theta_{n-1}}\big) }{\i\sqrt{\frac nN-z^2}}.
\end{equation}
We easily check that this expression for $\Delta_n$ matches with that of Definition~\ref{def:ell}.
By \eqref{V-1}, this implies that
\[
\sqrt{\tfrac{4N}{n}} V_{n}^{-1}  T_{n}^\beta  V_{n-1} = \begin{pmatrix}  e^{\i\theta_n}  &0 \\  0 & e^{-\i\theta_n}\end{pmatrix} \bigg(\mathrm{I} - \Delta_n 
\begin{pmatrix} 1 & -1 \\ -1 & 1 \end{pmatrix}\bigg)
+ \frac{\i\delta_n}{\sqrt{2\beta}} \sqrt{\tfrac{4N}{n}}\begin{pmatrix} Z_n'' & \overline{Z_n''} \\  -Z_n'' &  -\overline{Z_n''}  \end{pmatrix}
\]
By \eqref{rec2}, multiplying by $\big(\begin{smallmatrix}  e^{\i\theta_n}  &0 \\  0 & e^{-\i\theta_n}\end{smallmatrix}\big)^{-1}$ on the left, the first row of this matrix gives the evolution for $\{\xi_{n}(z)\}_{n>N_0(z)}$ with $Z_n' =\i \delta_n e^{-\i\theta_n} \sqrt{\tfrac{4N}{n}}  Z_n'' /\sqrt{2\beta}
= \i \delta_n e^{-\i\theta_n} \big(X_n \sqrt{\frac{n}{n-1}} e^{\i\theta_{n-1}} +Y_n\big)/\sqrt{2\beta}
$
according to Definition~\ref{def:ell}.
\end{proof}

Then, one can approximate the evolution of the \emph{complex phase} $\{\bpsi_{n}(z)\}_{n>N_0(z)}$ by linearizing the evolution from Lemma~\ref{lem:rec}. 
In particular, the process $\xi_n(z)$  is subject to a large deterministic rotation (neglecting  both $\Delta_n, Z_n'$, noe has $\xi_{n} \approx e^{\i\theta_n}\xi_{n-1}$), this suggest to define a new process;  for $z\in(-1,1)$ and $n > m > N_0(z)$,
\begin{equation} \label{psi1}
\tilde{\psi}_{n,m}(z) \coloneqq \bpsi_n(z)-\bpsi_m(z) - \i\vartheta_{n,m}(z) 
\qquad\qquad
\vartheta_{n,m}(z) \coloneqq {\textstyle \sum_{k=m+1}^{n}} \theta_{k}(z) . 
\end{equation}

For our analysis, it will be crucial that the random variables  $\{Z_n'(z) : n > N_0(z)\}$ are independent, centered, sub-Gaussian, and we record the following estimates. 

\begin{lemma}\label{lem:Z}
For  $z\in(-1,1)$, let $Z_n(z) = \tfrac{X_n + Y_n  e^{-\i \theta_{n}(z)}}{\sqrt{2}}$ for $n>N_0(z)$ as in Definiton~\ref{def:Z}, then one has
\[
Z_n'(z)   =  \tfrac{\i \delta_n(z) Z_n(z)}{\beta^{1/2}} + \O(\delta_n^3|X_n|)  , \qquad 
\E |Z_n'(z)|^2 = \tfrac{\delta_n(z)\E |Z_n(z)|^2}{\beta} +\O(\delta_n^4) ,\qquad 
\E Z_n^{'2}(z) = -\tfrac{\delta_n(z)\E Z_n^{2}(z)}{\beta}  +\O(\delta_n^4) 
\]
and 
\begin{equation} \label{def:s}
\E |Z_n(z)|^2 =1 , \qquad\qquad
\E Z_n^2(z) =  \frac{1+e^{-2\i\theta_{n}(z)}}{2}  =  \cos\theta_n(z) e^{-\i\theta_{n}(z)}
\end{equation}
Moreover, one has
\[
0< \Delta_n - \delta_n^2/4 \le  \delta_n^4 . 
\]

\end{lemma}

\begin{proof}
We skip this elementary computations -- these estimates follow from the fact that the angle $|\theta_{n+1}- \theta_n| \le \delta_n^2/2 $, and the parameters $\delta_n$ satisfy  $0<\delta_n -\delta_{n+1} \le \delta_n^2/2$ and $\delta_n^2 \ge n^{-1}$. 
\end{proof}


\subsection{Linearization.}
To obtain asymptotics for $\{\tilde{\psi}_{n,m}(z)\}_{n\ge m}$, we proceed to linearize the evolution from Lemma~\ref{lem:rec} using that $\delta_n(z)$ are decreasing and small if $n$ is sufficiently far from the turning point $N_0(z)$. 
In particular, this requires to truncate the noise in order to control the linearization errors.
Fix a small $0<\epsilon < 1/9$ and define the events; for $m\in\N_{\ge \mathfrak{K}}$,  
\begin{equation} \label{eventA}
\A_m \coloneqq \big\{  |X_n|^2+|Y_n|^2 \le \beta n^\epsilon ;\, \forall\, n \ge m\big\} . 
\end{equation}

The parameter $\epsilon$ will play no role in the sequel, so we do not emphasize its dependence in $\A_m$.
We record that under the assumptions of Definition~\ref{def:noise}, by a direct union bound, there exists a constant $c= c(\mathfrak{S},\epsilon)$ so that 
\begin{equation} \label{PA}
\P[\A_m^{\rm c}] \lesssim \exp(-c \beta  m^{2\epsilon}) .
\end{equation}



\smallskip

The linearization errors are controlled deterministically and uniformly on the event \eqref{eventA}. 
In particular, Lemma~\ref{lem:linear} establishes that the process  $\{\tilde{\psi}_{n,m}(z)\}_{n\ge m}$  is \emph{varying slowly}, meaning that  $|\tilde{\psi}_{n+1,n}(z)|\ll 1$  away from the turning point of the recurrence.
We obtain the following decomposition:

\begin{lemma}[Linearization] \label{lem:linear}
Fix $z=z(N)\in(-1,1)$ and $N\in\N$ and assume that $m\ge N_1(z)$. 
On the event $\A_m$, it holds for $n\ge m$, 
\[
\tilde{\psi}_{n,m}(z)  
=  - \frac{\mathbf{Q}_{n,m}(z)}{4} - \frac{\M_{n,m}(z)}{\sqrt\beta}  + \frac{[\M_{n,m}(z)]+\mathbf{L}_{n,m}(z)}{2\beta} + \O(m^{(3\epsilon-1)/2})
\]
where $\{\M_{n,m}(z)\}_{n\ge m}, \{\mathbf{L}_{n,m}(z)\}_{n\ge m}$ are martingales and $\mathbf{Q}_{n,m}(z)\coloneqq\sum_{k=m+1}^n \delta_k^2(z)\big(1 -  e^{-2\i \phi_{k-1}(z)}\big)$. Moreover,  $\{\M_{n,m}(z)\}_{n\ge m}$ is as in terms of Definition~\ref{def:GW} and $\{\mathbf{L}_{n,m}(z)\}_{n\ge m}$ satisfies the  tail-bound \eqref{Lest} below. 
\end{lemma}

\begin{proof} 
Let $\tilde{\xi}_{n,m}(z)   \coloneqq \exp(\tilde{\psi}_{n,m}(z)) =  \xi_n(z) e^{-\i \vartheta_{n,m}(z)} \xi^{-1}_m(z)$  for $z\in(-1,1)$ and $n > m > N_0(z)$. 
By Lemma~\ref{lem:rec}, this process follows the evolution:
\begin{equation} \label{ell_rec}
\tilde{\xi}_{n,m}
=  \big(1-\Delta_n+   Z_n' \big) \tilde{\xi}_{n-1,m}+ \big( \Delta_n - \overline{Z_n'} e^{-2\i\theta_n} \big) e^{-2\i \phi_{n-1}}  \tilde{\xi}_{n-1,m} . 
\end{equation}
Observe that the noise satisfies  $|Z_n'(z)|   \le \delta_n(z) \sqrt{\tfrac{|X_n|^2+|Y_n|^2}{2\beta}}$ and, for any $\epsilon<1/2$, the map $n\mapsto \delta_n(z) n^{\epsilon}$ is decreasing  for $n>N_0(z)$. 
Then, on $\A_m$;
\[
\sup\big\{|Z_n'(z)| ; n\ge m  \big\} \le   m^{\epsilon/2} \delta_m(z) .
\]
It follows that if $m$ is sufficiently large, on $\A_m$;   
\begin{equation} \label{trunc1}
\sup\big\{|Z_n'(z)| ; n\ge m  ; z\in (-1,1) \text{ with } m\ge N_1(z) ; N\in\N\big\}  \le \epsilon,
\end{equation}
where $\epsilon$ is arbitrary (indeed, the condition $\{ m\ge N_1(z) \}$ guarantees that $\delta_m(z) \le m^{-1/6}$). 

Then, on $\A_m$, we can linearize  the RHS of \eqref{ell_rec} (the deterministic term $\Delta_n(z)$ are also small for $m\ge N_1(z)$ and $m$ sufficiently), so taking the principal branch of  $\log\big(\tilde{\xi}_{n,m}/\tilde{\xi}_{n-1,m}\big)$, by a Taylor expansion, we obtain
\begin{equation} \label{zeta3}
\begin{aligned}
\log\left(\tfrac{\tilde{\xi}_{n,m}}{\tilde{\xi}_{n-1,m}}\right)
&=\log\left( 1-\Delta_n+  Z_n'+ \big( \Delta_n - \overline{Z_n'}e^{-2\i\theta_n}  \big) e^{-2\i \phi_{n-1}} \right)
\\
& =\frac{\i \delta_n}{\sqrt\beta} \left( {Z_n} + \overline{Z_n}e^{-2\i \theta_{n}-2\i \phi_{n-1}}\right)
- \delta_n^2\frac{1 -  e^{-2\i \phi_{n-1}} }{4}
+ \delta_n^2\frac{\big( Z_{n}  + \overline{Z_{n}}  e^{-2\i \theta_{n}-2\i \phi_{n-1}}  \big)^2}{2\beta} 
+ \operatorname{EL}_n
\end{aligned}
\end{equation}
where the errors $\operatorname{EL}_n$ are defined implicitly by \eqref{zeta3}.
Here, we used Lemma~\ref{lem:Z} to replace the random variables $Z_n'$ by $Z_n$ in  \eqref{zeta3} and we check that for $n\ge m$, 
\begin{equation} \label{EL}
| \operatorname{EL}_n(z)|  \lesssim \delta_n(z)^3\big(1+|X_n|+|Y_n| \big)^3 .
\end{equation}
Thus, on $\A_m$,
\[
\sum_{n\ge m}  |\operatorname{EL}_n(z)|  \lesssim m^{3\epsilon/2} \delta_m(z) 
\]
and, choosing $\epsilon$ is small enough, the RHS converges to 0 as $m\to\infty$.

Summing \eqref{zeta3} and using that $\tilde{\psi}_{n,m} = \log(\tilde{\xi}_{n,m})$ with $\tilde{\psi}_{m,m}=0$, we obtain
\[
\tilde{\psi}_{n,m}   
=  - \frac{\M_{n,m}}{\sqrt\beta} - \frac{\mathbf{Q}_{n,m}}{4}   + \frac{\mathbf{S}_{n,m}}{2\beta} +\underset{m\to\infty}{\o(1)}
\]
where the martingale part is  $\M_{n,m} = -\i \sum_{k=m+1}^n \delta_{k} \big(Z_{k} + \overline{Z_{k}} e^{-2\i(\theta_k+\phi_{k-1})} \big) =  \G_n(z) + \overline{\W_n(z)}$ according to Definition~\ref{def:GW},  
and we define
\[
\mathbf{Q}_{n,m} \coloneqq \sum_{m<k\le n}  \delta_k^2\big(1 -  e^{-2\i \phi_{k-1}}  \big) , \qquad 
\mathbf{S}_{n,m} \coloneqq \sum_{m<k\le n} \ \delta_k^2 \big( Z_{k}  +\overline{Z_{k}}  e^{-2\i(\theta_k+\phi_{k-1})}  \big)^2 . 
\]

We can decompose
\[
\mathbf{S}_{n,m} = \mathbf{L}_{n,m} + [\M_{n,m}]
\]
where $\{\mathbf{L}_{n,m}\}_{n\ge m}$ is also a martingale and $\{[\M_{n,m}]\}_{n\ge m}$ denotes the bracket of the (complex) martingale $\{\M_{n,m}\}_{n\ge m}$; 
\begin{equation} \label{M2}
[\M_{n,m}]\coloneqq \sum_{m<k\le n} \delta_k^2 \E\big[ \big( Z_{k}  + \overline{Z_{k}}  e^{-2\i(\theta_k+\phi_{k-1})}  \big)^2 | \F_{k-1} \big] . 
\end{equation}

Finally, under the assumptions of Definition~\ref{def:noise}, the increments of the second martingale satisfy
$\|\mathbf{L}_{n,n-1}\|_{1}  \lesssim  \delta_{n}^2$ using the norm defined in the  Appendix~\ref{sec:concentration}. 
Hence, using that $\sum_{k>m} \delta_{n}^4 \lesssim \delta_{m}^2$, by Proposition~\ref{lem:conc1}, it holds for any $\lambda>0$, 
\begin{equation} \label{Lest}
\P\left[\max_{n>m}|  \mathbf{L}_{n,m}(z) |   \ge \lambda\right] \le 2 \exp\bigg(- \frac{c\lambda^2}{1+\lambda} \delta_{m}^{-2}(z) \bigg).
\qedhere
\end{equation}
\end{proof}

\subsection{Random oscillatory sums.} \label{sec:osc}
Recall that according to \eqref{psi1}, the phase $\phi_{n,m} =\Im(\bpsi_{n,m}) = \vartheta_{n,m}+ \boldsymbol{\chi}_{n,m}$. 
The goal of this section is to prove that certain oscillatory sums involving the phase $\{\phi_{n,m}(z)\}$ are small when $m\gg 1$, because of the fast variation of the deterministic part of the phase $\vartheta_{n,m}(z)$.

\paragraph{Continuity.}
The first step consists in showing that  the random part of the phase $\{\boldsymbol{\chi}_{n,m}(z)\}_{n\ge m}$ varies slowly as $n$ increases. 
The result is formulated  along suitable blocks and it will be crucial in the sequel of this paper.

\begin{proposition}[Smoothness of the phase]\label{lem:varphase}
Fix $z=z(N)\in(-1,1)$ and $N\in\N$.  Consider an increasing sequence $\{n_k\}_{k=T}^\infty$ such that $\delta_{n_k}^2(z)\cdot  (n_{k+1}-n_k) \lesssim 1/k $ for $k\in\N_{\ge T}$ and $n_T \ge N_1(z)$.  
Then, for $R\ge 1$, define the event
\begin{equation} \label{Asmooth}
\A_{\chi}(R,T; z) \coloneqq 
\bigcap_{k\ge T} \Big\{ \max_{\ell\in[n_k+1,n_{k+1}]} |   \boldsymbol{\chi}_{\ell,n_k}(z) |   \le  R k^{\frac{\epsilon-1}{2}}/T^{\frac\epsilon2} \Big\} .
\end{equation}
There exists a constant $c>0$ so that 
\begin{equation} \label{PAchi}
\P\big[ \A_{\chi}^{\rm c} (R,T; z) \cap \A_{n_T} \big] \lesssim\exp(-cR). 
\end{equation}
\end{proposition}

\begin{proof}
We use the notation from Lemma~\ref{lem:linear} and the Appendix~\ref{sec:concentration}. 
The increments of the martingale
satisfy $\|\M_{n,n-1}\|_{2}  \lesssim  \delta_{n}$ for $n > N_0$. 
So, by Proposition~\ref{lem:conc2}, we claim that for any $k\ge 1$,  and any $\lambda>0$, 
\begin{equation} \label{Freed}
\P \Big( \max_{n\in[n_k+1,n_{k+1}]} |  \M_{n,n_k}  |   \ge \lambda  \Big)
\le 2 \exp\big( - c k \lambda^2\big) ,
\end{equation}
where we used that $\sum_{n=n_k+1}^{n_{k+1}}  \delta_n^2 \le (n_{k+1}-n_k) \delta_{n_k}^2 \lesssim 1/k$. 
Hence, taking $\lambda = R k^{\frac{\epsilon-1}{2}}/T^{\frac\epsilon2}$ for a $0<\epsilon<1$, the RHS is summable over all integer $k\ge T$ and we obtain, for $R\ge 1$, 
\begin{equation} \label{tbvarM}
\P\bigg[ \bigcup_{k\ge T} \Big\{ \max_{n\in[n_k+1,n_{k+1}]} |  \M_{n,n_k}  |   \ge  R k^{\frac{\epsilon-1}{2}}/T^{\frac\epsilon2}\Big\} \bigg]
\lesssim \exp(-c R^2)
\end{equation}
where the implied constant depends only on $(\epsilon,\beta, \mathfrak{S})$.

\smallskip

By \eqref{M2} and using the deterministic bound 
$\E \big[ \big( Z_{k}  + \overline{Z_{k}}  e^{-2\i \phi_{k-1}}  \big)^2\big| \F_{k-1} \big] \lesssim 1$, we have 
\[
\max_{n\in[n_k+1,n_{k+1}]} |  \mathbf{Q}_{n,n_k}^0  |  ,  \max_{n\in[n_k+1,n_{k+1}]}|[\mathbf{M}_{n,n_k}^1]|
\lesssim \delta_{n_k}^2 (n_{k+1}-n_k ) \lesssim k^{-1} .
\]
In particular, the contributions from these terms are deterministically negligible.
Moreover, by \eqref{Lest}, it holds for any $k\ge T$
\begin{equation*}
\P \Big( \max_{n >n_k+1} |  \mathbf{M}^2_{n,n_k}  |   \ge  R k^{-\frac{1}{2}}     \Big)
\lesssim \exp\bigg(- \frac{c R}{k^{1/2}\delta_{n_k}^2}  \bigg)  . 
\end{equation*}
Therefore, by a union bound (using that $\big(k^{1/2}\delta_{n_k}^2 \lesssim k^{-1/2}$ by assumptions on the blocks); this implies that if $R\ge 1$,
\begin{equation} \label{tbvarL}
\P\left( \bigcup_{k\ge T} \Big\{ \max_{n\in[n_k+1,n_{k+1}]} |  \mathbf{M}^2_{n,n_k}  |   \ge  R k^{-\frac{1}{2}}\Big\} \right)
\lesssim \exp\big( - c R C_T\big) , \qquad C_T = \delta_{n_T}^{-2}(z)/\sqrt{T} .
\end{equation}

Finally, according  to \eqref{EL},   the linearization errors are controlled (deterministically) on the event $\A_{n_T}$; for every $k\ge T$, 
\[
{\textstyle \sum_{n=n_k+1}^{n_{k+1}}} |\operatorname{EL}_n| \lesssim      \delta_{n_k}^{3} n_k^{3\epsilon/2} (n_{k+1}-n_k) \lesssim  k^{-1} 
\]
where we used again that $(n_{k+1}-n_k) \delta_{n_k}^2 \lesssim 1/k$.
This shows that, for every block, the sum of the linearization errors are also negligible.

\smallskip

Hence, by Lemma~\ref{lem:linear} and  combining the estimates \eqref{tbvarM}--\eqref{tbvarL}, we obtain an analogous bound for the process $\{\tilde{\psi}_{n,m}\}_{n\ge m}$.
Adjusting the constants, we conclude that for any $R\ge 1$, 
\begin{equation} \label{TBimprovement} 
\P\bigg( \bigcup_{k\ge T} \Big\{ \max_{n\in[n_k+1,n_{k+1}]} |  \tilde{\psi}_{n,n_k}  |   \ge  R k^{\frac{\epsilon-1}{2}}/T^{\frac\epsilon2}\Big\} \cap \A_{n_T} \bigg)
\lesssim \exp\big( - c R(R \wedge C_T) \big)  
\end{equation}
with $C_T \gtrsim 1$. 
Since $\boldsymbol{\chi}_{n,m}=\Im\tilde{\psi}_{n,m}$, this completes the proof.
\end{proof}


\paragraph{Deterministic phase.} We also need  basic estimates about the growth of the deterministic part of the phase.

\begin{lemma} \label{lem:theta}
For $z\in [0,1)$, we have $ \theta_n(-z)=\pi- \theta_n(z)$ and the function $n \mapsto \theta_n(z) = \arccos\big(z\sqrt{N/n}\big)$ is non-decreasing  for $n\ge N_0(z)$.
Moreover, we have for any $L \in \N$, 
\[
|\theta_{n+L}(z) - \theta_n(z)| \le \frac{L \sqrt{N_0(z)}}{2n^{3/2}\sin\theta_n(z)}
\le L \delta_n^2(z)/2 .
\]
Moreover, if $|x| \le |z| <1$, it holds for $n\ge N_0(z)$, 
\[
|z-x| \sqrt{N} \delta_n(z) \ge |\theta_n(x)-\theta_n(z)| \ge |z-x| \sqrt{N}\delta_n(x)
\]
\end{lemma}

\begin{proof}
The function $t\in[1,\infty] \mapsto \theta(t)=\arccos(1/\sqrt t)$ is concave increasing, hence for $z\in [0,1)$,
\[
0\le \theta_{n+L} - \theta_n \le \frac{L}{N_0}  \theta'(\tfrac n{N_0}) 
= \frac{L}{2 n}\big(\tfrac n{N_0}-1\big)^{-1/2} 
= \frac{L \sqrt{N_0}}{2n^{3/2}\sin\theta_n} . 
\]
The next bound follows from the observation that $\sqrt{n}\sin \theta_n = \delta_n^{-1}$. 
The second claim follows from that  $-\partial_{z}\big(\arccos\big(z\sqrt{N/n}\big)\big)  = \sqrt{N}\delta_n(z)$. 
\end{proof}

\begin{lemma}\label{lem:osc}
There is a numerical constant so that for any $n \ge N_0(z)$ and any $L\in\N$, 
\[
\bigg| \sum_{j=n+1}^{n+L}  e^{\i 2 \vartheta_{j,n}(z)} \bigg|
\lesssim \sqrt{n} \delta_n(z) \big(1+\tfrac{L^3}{n}\big)  .
\]
\end{lemma}

\begin{proof}
Without loss of generality, we can assume that $z\in[0,1)$ for otherwise 
$\theta_n(z) +\pi = - \theta_n(-z)$.
By Lemma~\ref{lem:theta}, for any $n \ge N_0$ and $j\in\N$
\[
0\le  \vartheta_{n+j,n} -  j \theta_{n+1} \le \frac{j(j-1)}{4(n+1)\sin\theta_{n+1}} . 
\]
Hence, by decomposing
\[
\sum_{j=n+1}^{n+L}  e^{\i 2 \vartheta_{j,n}} = 
\sum_{j=1}^{L} \big(  e^{\i2(\vartheta_{n+j,n}-j\theta_{n+1})} -1 \big)  e^{\i 2 j \theta_{n+1}}  + \sum_{j=1}^{L}  e^{\i 2 j \theta_{n+1}}  ,
\]
we obtain
\[ \begin{aligned}
\bigg|
\sum_{j=n+1}^{n+L}  e^{\i 2 \vartheta_{j,n}}  \bigg|  & \le 
\bigg|\sum_{j=1}^{L}  e^{\i 2 j \theta_{n+1}} \bigg| 
+\sum_{j=1}^{L} \big| e^{\i2(\vartheta_{n+j,n}-j\theta_{n+1})} -1 \big| \\
& \le \frac{C}{\sin(\theta_{n+1})} \big(1+\tfrac{L^3}{n}\big) 
\end{aligned}\]
for a numerical constant $C>0$.
Here we used that for any $\theta\in(0,\pi)$ and any $L\in\N$, 
\(
\big|
\sum_{j=1}^{L} e^{\i 2 j \theta} \big| \le \frac{2}{\sin \theta}
\)
and
\( 
\sum_{j=1}^{L} \frac{j(j-1)}{4} =\O(L^3) .
\)
Finally, since $\sin \theta_n =\delta_n^{-1}/\sqrt{n}$,  this completes the proof. 
\end{proof}


\paragraph{Control of the $\mathbf{Q}$ terms.}
In Lemma~\ref{lem:linear}, the term $\mathbf{Q}_{n,m}(z)$ and the bracket \eqref{M2} both involves oscillatory sums, so we claim that these quantities are small on an appropriate event of the form \eqref{Asmooth}. 
In particular, this event depends on the spectral parameter $z$ and we need to specify a suitable sequence $\{n_k\}_{k=T}^\infty$.
Fix $N\in\N$ and $z=z(N)\in(-1,1)$. We work with the following blocks: for $k\in\N$, 
\begin{equation} \label{block1}
n_k(z) \coloneqq  
\begin{cases} N_k(z) =\lfloor Nz^2 +  k \mathfrak{L}(z)\rfloor  & k <  \mathfrak{L}(z)^2  \\
\lfloor Nz^2  +  k^{3/2} \rfloor  & k \ge   \L(z)^2  
\end{cases}  \quad\text{if }  z\in\Q , 
\qquad 
n_k(z) \coloneqq  \lfloor   k^{3/2} \rfloor  \quad\text{if } z\notin \Q . 
\end{equation}
We consider two separate regimes because $1)$ we need $n_k=N_k$ in the parabolic stretch after the turning point (in this part of the recursion the deterministic phase grows slowly), $2)$ $n_k$ needs to grow faster than linear for most of the recursion. 
We record that, in both cases, these blocks satisfy the condition of Lemma~\ref{lem:varphase}; 
$\delta_{n_k}^2(z)\cdot  (n_{k+1}-n_k)(z) \lesssim 1/k $ for~$k\in\N_{\ge T}$.

\begin{proposition}[Oscillatory sum 1] \label{prop:osc1}
Fix $N\in\N$ and $z\in(-1,1)$. For $\lambda\in \mathcal{T}$, let $\{q_n(z;\lambda)\}$ be a sequence of (random) coefficients such that (deterministically) for $n> m$, 
\begin{equation}
|q_n(z;\lambda)| \le  \delta_n^2(z),  \qquad\qquad
|q_{n+1}(z;\lambda)-q_n(z;\lambda)| \le  \delta_n^4(z) . 
\end{equation}
Then, for any $T\in\N$ (with $n_T= m$) and $R\ge 1$, on the event $\A_{\chi}(R,T; z)$ with blocks  \eqref{block1}, 
\[
\sup_{\lambda\in \mathcal{T}}\max_{n > m}\bigg| \sum_{m< k \le n} q_k(z;\lambda
) e^{\i2\phi_{k}(z)}  \bigg| 
\lesssim R/\sqrt{T}. 
\] 
\end{proposition}

\begin{proof}
Let $L_k \coloneqq n_{k+1}-n_k$ for $k\in\N$ denotes the block's lengths.
Observe that the blocks \eqref{block1} are designed so that, on top of the condition $\delta_{n_k}^2 L_k \lesssim k^{-1}$, we also have
\begin{equation} \label{block1est}
L_k^3 \lesssim  n_k , \qquad \sqrt{n_k} \delta_{n_k}^{3}(z) \lesssim k^{-3/2} , \qquad
\text{for every $k\in \N$.}
\end{equation} 
Recall that $\phi_{n,m} = \vartheta_{n,m}+ \boldsymbol{\chi}_{n,m}$ for $n>m$ where $\vartheta_{n,m}$ is deterministic. 
We denote $q_n' = q_n e^{2\i\boldsymbol{\chi}_{n,n_k}} $ for $n\in(n_{k}, n_{k+1}]$. 
By splitting the sum into blocks, we have
\[
\max_{n>n_T }\bigg|  \sum_{n_T< k \le n} q_k e^{\i2\phi_{k}} \bigg| 
\le  \sum_{k\ge T}\bigg|  \sum_{n_{k}<j\le n_{k+1}} q_j' e^{\i2\vartheta_{j,n_k}} \bigg|
\]
and we claim that for every block,
\begin{equation} \label{qest}
\max_{n \in [n_k+1,n_{k+1}] }
|q_n' - q_{n_k}'| \lesssim R L_k^{-1}  k^{\frac{\epsilon-3}{2}}/T^{\epsilon/2} .
\end{equation}
Indeed, on the event  \eqref{Asmooth}, it holds for every $n\in(n_k, n_{k+1}]$, 
\[ \begin{aligned}
|q_n' - q_{n_k}'|  &\le |q_n - q_{n_k}| + |q_{n_k}| |e^{2\i \boldsymbol{\chi}_{n,n_k}}- 1 | \\
& \le   \delta_{n_k}^2\big( L_k \delta_{n_k}^2+2 R k^{\frac{\epsilon-1}{2}}/T^{\epsilon/2} \big)  
\end{aligned}\]
using our assumptions on $\{q_n\}_{n>N_0}$.
This gives \eqref{qest} since $\delta_{n_k}^2 L_k \lesssim k^{-1}$, so that
\[
\sum_{n_{k}<j\le n_{k+1}} q_j' e^{\i2\phi_{j,n_k}} 
= q_{n_k} \sum_{n_{k}<j\le n_{k+1}}e^{\i 2 \vartheta_{j,n_k}} 
+ \O\big(R k^{\frac{\epsilon-3}{2}}T^{-\epsilon/2}\big). 
\]

By Lemma~\ref{lem:osc}, using \eqref{block1est} and that  $|q_n| \le \delta_n^2$, this shows  that
\begin{equation} \label{osc:est}
\bigg| \sum_{n_{k}<j\le n_{k+1}} q_j' e^{\i2\phi_{j,n_k}}  \bigg| 
\lesssim  \sqrt{n_k} \delta_{n_k}^{3} + R k^{\frac{\epsilon-3}{2}}T^{-\epsilon/2} \lesssim R k^{\frac{\epsilon-3}{2}}T^{-\epsilon/2} .
\end{equation}
Summing these estimates, we conclude that these sums are uniformly bounded by $\O(R/\sqrt{T})$.
\end{proof}

We will also need the following variant of the previous estimate.

\begin{proposition}[Oscillatory sum 2]\label{prop:osc2}
Fix  $N\in\N$ and $z,x\in(-1,1)$ with $|x|\le |z|$. For $\lambda\in \mathcal{T}$, let $\{q_n(z,x;\lambda)\}$ for  be a sequence of  coefficients such that (deterministically) for $n>m$, 
\[
|q_n(z,x;\lambda)| \le  \delta_n^2(z) |\cos(\ell_n(x,z))| , \qquad\qquad
|q_{n+1}(z,x;\lambda)-q_n(z,x;\lambda)| \le  \delta_n^4(z) 
\]
where  $\ell_n(x,z)= \frac{\theta_{n}(x)+\theta_{n}(z)}{2} \in (0,\pi)$. 
Define the event $\A_{\chi}^2(R, T; x, z) \coloneqq \A_{\chi}(R,T; z) \cap \A_\chi(R,T; x)$ with the same blocks $\{n_k(z)\}_{k\in\N}$ -- \eqref{block1} -- and $T\in\N$ (with $m=n_T$).
Then, for any and $R\ge 1$, on the event $ \A_\chi^2(R,T;x,z)$, 
\[
\sup_{\lambda\in \mathcal{T}}\max_{n> n_T}\bigg| \sum_{n_T< k \le n} q_k e^{\i2(\phi_{k}(x)+\phi_{k}(z))}  \bigg| 
\lesssim R /\sqrt{T}.
\]
\end{proposition}

\begin{proof}The argument is the same as that of the proof of Lemma~\ref{prop:osc1}.
Without loss of generality, $|x|\le |z|$ and the $\{n_k(z)\}_{k\in\N}$ -- \eqref{block1} --
satisfy the required conditions; $\delta_{n_k}^2(x) L_k \le \delta_{n_k}^2(z) L_k \lesssim k^{-1}$ and \eqref{block1est}.

Let $q_n' = q_n e^{2\i (\chi_{n,n_k}(x)+\chi_{n,n_k}(z))} $ for $n\in(n_{k}, n_{k+1}]$.
Then, like in the previous proof, the estimates \eqref{qest} hold on the event $\A_{\chi}^2(R,T;x,z)$ and we can \emph{linearize} $q_n'$ along each block; we obtain
\[
\max_{n>n_T }\bigg|  \sum_{n_T< k \le n} q_k e^{\i2(\phi_{k}(x)+\phi_{k}(z))} \bigg| 
\le  \sum_{k\ge T}\bigg( \bigg|  \sum_{n_{k}<j\le n_{k+1}} q_{n_k} e^{\i2(\vartheta_{j,n_k}(x)+\vartheta_{j,n_k}(z))} \bigg|
+ \O\big(R k^{\frac{\epsilon-3}{2}}T^{-\epsilon/2}\big) \bigg) .
\]
Now, the main difference is that, instead of Lemma~\ref{lem:osc}, it holds for every $n \ge N_0$ and $L\in\N$, 
\begin{equation} \label{oscest1}
\bigg| \sum_{j=n+1}^{n+L} e^{\i2(\vartheta_{j,n}(x)+\vartheta_{j,n}(z))} \bigg|
\le \max\bigg\{ \frac{C}{|\sin 2\ell_{n}|} + \frac{C L^3}{n \sin \theta_n(z)} , L \bigg\}.
\end{equation}
The proof is exactly the same; it relies on  Lemma~\ref{lem:theta} and the fact that $0< \sin \theta_n(z) \le \sin \theta_n(x)$ if $|x| \le |z|$  to control the error term. 
Note that if $\ell_n=\pi/2$ (this corresponds to the case $x=-z$), then the first term on the RHS is $\infty$, which is why we include a truncation. 

Using this bound and the condition $|q_n| \le  \delta_n^2 \cos(\ell_n)$
with   $\sin \theta_n(z) \le \sin(\ell_n)  $, we conclude that for every $k\in \N$, 
\[
\bigg|  \sum_{n_{k}<j\le n_{k+1}} q_{n_k} e^{\i2(\vartheta_{j,n_k}(x)+\vartheta_{j,n_k}(z))} \bigg| \lesssim \delta_{n_k}^2 \frac{1+ L_k^3/n_k}{\sin \theta_{n_k}(z)} .
\]
Since $\sin \theta_n =  n^{-1/2} \delta_n^{-1}$, using the conditions \eqref{block1est}, we obtain for every block an estimate comparable to \eqref{osc:est}. Hence, summing these estimates, this yields the required bound. 
\end{proof}

%

Going back to Lemma~\ref{lem:linear}, we  can use the previous estimates to control the size of  $\mathbf{Q}_{n,m}$ and the bracket $[\M_{n,m}]$, \eqref{M2}.

\begin{proposition} \label{prop:Q}
Fix $N\in\N$, $z\in(-1,1)$, $R\ge 1$ and $T\in\N$ (with $m=n_T$, \eqref{block1}). 
On the event $\A_{\chi}(R,T; z)$, it holds for all $n> m$, 
\[
\mathbf{Q}_{n,m}(z)=   \log\bigg(\frac{n-Nz^2}{m-Nz^2}\bigg)
+ \O\big(R/\sqrt{T}\big) 
\qquad\text{and}\qquad
[\M_{n,m}]= - [\G_{n,m}(z)]+ \O\big(R/\sqrt{T}\big) . 
\]
\end{proposition}

\begin{proof}
For any  $n>N_0$, 
\begin{equation} \label{delta2}
0 \le \delta_{n}-\delta_{n+1} \le  \delta_n^3/2 .
\end{equation}
Thus, a direct application of Lemma~\ref{prop:osc1} yields on the event $\A_\chi^1$,  uniformly for all $n>m$, 
\[
\sum_{m<k\le n} \delta_k^2 e^{2\i \phi_{k-1}} = \O\big(R/\sqrt{T}\big) 
\]
As $\delta_k^2(z)= (k-Nz^2)^{-1}$ for $k>N_0(z)$, computing the harmonic sum, his implies that
\[
\mathbf{Q}_{n,m}= \sum_{m<k\le n} \delta_k^2\big(1 -  e^{-2\i \phi_{k-1}}  \big)
= \log\bigg(\frac{n-Nz^2}{m-Nz^2}\bigg)
+ \O\big(R/\sqrt{T}\big) .
\]
For the other claim, the martingale $\M_{n,m} = \G_{n,m}+\overline{\W_{n,m}}$ with $\W_{n,m} = -\i\sum_{k=m+1}^n \delta_{k}Z_{k} e^{2\i(\theta_k+\phi_{k-1})}$,  $\G_{n,m} = -\i\sum_{k=m+1}^n \delta_{k}Z_{k}$ , so its bracket
\[
[\M_{n,m}]= -[\G_{n,m}] -2  \sum_{m<k\le n} q_{k-1}^1  e^{-2\i\phi_{k-1}} 
- \sum_{m<k\le n} q_{k-1}^2  e^{-4\i\phi_{k-1}} 
\]
where
$q_{k-1}^1 = \delta_k^2  e^{-2\i\theta_k}$ and
$q_{k-1}^2 = \delta_k^2  e^{-3\i\theta_k}\cos\theta_n$. 
The first sum corresponds to the cross-bracket $[\G_{n,m};\overline{\W_{n,m}}]$, while the second sum corresponds to the bracket $[\W_{n,m}]$.

\smallskip

Using \eqref{delta2} and Lemma~\ref{lem:theta}, we verify that for $j\in\{1,2\}$,
$|q_{n+1}^j-q_n^j| \lesssim  \delta_n^4$, 
$|q_n^1| \le \delta_n^2$ and $|q_n^1| \le \delta_n^2  |\cos(\theta_{n+1})|$, so that by  Lemma~\ref{prop:osc1} and  Lemma~\ref{prop:osc2} (with $x=z$ -- in which case  $\A_{\chi}^2(R, T; z, z)=\A_{\chi}^1(R, T; z)$), it holds on this event, uniformly for $n> m$, 
\begin{equation} \label{crossbracket}
[\W_{n,m}]= \O\big(R/\sqrt{T}\big) ,\qquad
[\G_{n,m} ;\overline{\W_{n,m}}]= \O\big(R/\sqrt{T}\big),
\end{equation}
and $[\M_{n,m}]= -[\G_{n,m}]+\O\big(R/\sqrt{T}\big)$. 
\end{proof}

\subsection{Martingale approximation.}
\label{sec:approx}
To conclude, we gather our findings to relate the (complex) phase $\tilde{\psi}_{N,m}$ to the martingales from Definition~\ref{def:GW}.
We formulate two results in different regimes.

\medskip

We first treat the case where $z\in\Q$. We define for $T\ge 1$,  
\begin{equation} \label{Omega2}
\Omega_N^2(z;T) = \tilde{\psi}_{N,m}(z) +\tfrac14\log\big(\tfrac{N-Nz^2}{m-Nz^2}\big)+  \tfrac1{\sqrt\beta} \M_{N,m}(z) + \tfrac{[ \G_{N,m}(z)]}{2\beta} , \qquad m=N_T(z). 
\end{equation}

\begin{proposition}\label{prop:oneray1}
Let $z\in\Q$, $T\ge 1$ and $\A_m$ be as in \eqref{eventA} with $m=N_T(z)$.  
There exists a constants $c>0$ such that for any $R\ge 1$, 
\[\begin{aligned}
\P\big[ \big\{  |\Omega_N^2(z;T)|
\ge R/\sqrt{T} \big\} \cap\A_m \big]  
\lesssim \exp(-c R)  . 
\end{aligned}\]
This implies that the collection of random variables $\big\{ \Omega_N^2(z;T)  ; z\in\Q\big\}_{N\in\N}$ is tight.
Moreover, for $z\in\Q$, the random variable $ \Omega_N^2(z;T)\to 0$ in probability in the limit as $N\to\infty$ and $T\to\infty$. 

\end{proposition}

\begin{proof}
Then, by Lemma~\ref{lem:linear} and Proposition~\ref{prop:Q}, on $\A_{\chi}\cap\A_m$ ($\A_{\chi}=\A_\chi(R,T;z)$ with the appropriate blocks), uniformly for $N>m$, 
\[
\Omega_N^2(z;T) = \tfrac{\mathbf{L}_{n,m}(z)}{2\beta} 
+ \O\big(\tfrac R{\sqrt T}\big) . 
\]
For $z\in\Q$,  $m= N_T(z) \to\infty$ as $N\to\infty$,  so the linearization errors in Lemma~\ref{lem:linear} converge to 0 on $\A_m$.

Moreover, by \eqref{PAchi} and \eqref{Lest}, using that $\delta_{m}^{-2}(z) = T\mathfrak{L}$ for any $R\ge 1$ 
\[
\P[ \A_{\chi}^c \cap \A_m] \lesssim\exp(-cR) , \qquad
\P\left[\max_{n>m}|  \mathbf{L}_{n,m}(z) |   \ge R/T\right] \lesssim\exp(-cR \mathfrak{L}) .  
\]
Adjusting the constants, this proves the estimate. Observe that the events \eqref{eventA} are increasing, so we may replace $m=N_T(z)$ by any fixed $m\in\N$. 

Then, by \eqref{PA}, it follows that
\[
\limsup_{R\to\infty}\limsup_{N\to\infty} \sup_{z\in\Q} \P\big[ |\Omega_N^2(z;1)| \ge R\big] =0 .  
\]
Consequently, the collection of random variables $\big\{ \Omega_N^2(z;1)  ; z\in\Q , N\in\N \big\}$ is tight 

Moreover,  for $\epsilon>0$ and any sequence $z\in\Q$,  choosing $R= \epsilon\sqrt{T}$, if $T$ is sufficiently large, 
\[ 
\P\big[ |\Omega_N^2(z;T)| \ge \epsilon \big]  \le C  e^{-c \epsilon\sqrt{T}} +
\P[\A_m^{\rm c} ]
\]
and then,
\[
\limsup_{T\to\infty}\limsup_{N\to\infty}  \sup_{z\in\Q}\P\big[ |\Omega_N^2(z;T)| \ge \epsilon \big] =0 .\qedhere
\]
\end{proof}

\section{Convergence of $\Omega_N$} \label{sec:struct}
In this section, we prove claim 4 of Theorem~\ref{thm:main}.
Recall that by \eqref{polychar}, $\Phi_{n}(z) = \Re(\exp\bpsi_n(z))$ for $n>N_0(z)$.
Then, to be consistent, the error is defined by; for $N\in \N$ and $z\in(-1,1)$, 
\begin{equation} \label{err0}
\begin{aligned}
\Omega_N(z)&\coloneqq \bpsi_N(z)- \i\pi N F(z) + {\mathrm{c}_\beta}\log(1-z^2) 
+\tfrac{\M_N(z)}{\sqrt{\beta}} , \\
\tfrac{1}{2}\varphi_n(\lambda;z)&\coloneqq  \bpsi_n\big(z + \tfrac{\lambda}{N \varrho(z)}\big)-\bpsi_n(z) , \qquad\qquad \lambda\in\R, \qquad N_0(z)<n\le N.
\end{aligned}
\end{equation}
In particular, the quantity $\Omega_N(z)$ is independent of the local coordinate $\lambda\in\R$, while the asymptotics of $\varphi_n(\lambda;z)$ are expected to be independent of $z$ in the bulk. 

There are two regimes, and they are treated in a slightly different way.
\todo{fix}

\subsection{Asymptotic regime away from 0.}

\begin{proposition} \label{prop:limit}
If $z\in\Q$, there are random variables $\boldsymbol{\mho}_\beta^{\pm}$ (independent of $z$) such that in distribution as $N\to\infty$, 
\[
\Omega_N(z) \to \boldsymbol{\mho}_\beta^{\pm}   - \i2{\mathrm{c}_\beta}\arcsin(z) 
\]
and  $\boldsymbol{\mho}_\beta^+ \overset{\rm law}{=} \overline{\boldsymbol{\mho}_\beta^-}$. 
Moreover, using the notation from Proposition~\ref{cor:mho} (in terms of the stochastic Airy function), one has 
\[
\boldsymbol{\mho}_\beta^{\pm} \overset{\rm law}{=} \hat{\boldsymbol{\mho}}_\beta^{\pm} - \tfrac{\log 2}{\beta} +\tfrac{\mathrm{g}}{\sqrt{\beta}} -\tfrac{\E\mathrm{g}^2}{2\beta} .
\]
\end{proposition}

\begin{proof}
According to Definition~\ref{def:GW}, the martingale term can be decomposed in three part: for $T\ge 1$, 
\[
\M_N(z) = \M_{N,N_T}(z) + \M_{N_T,N_0}(z) + \G_{N_0}(z) ,
\]
coming from the \emph{elliptic}, \emph{parabolic} and \emph{hyperbolic} regimes respectively. 
Recall also the definitions \eqref{psi1} and \eqref{Omega2}; for $z\in\Q$ and $T\ge 1$,  
\[\begin{cases}
\bpsi_N =  \bpsi_{N_T}+  \i\vartheta_{N,N_T}+\tilde{\psi}_{N,N_T} \\
\Omega_N^2(z;T) \coloneqq \tilde{\psi}_{N,N_T}(z) +\tfrac14\log\big(\tfrac{N(1-z^2)}{T\mathfrak{L}(z)}\big)+  \tfrac1{\sqrt\beta} \M_{N,N_T}(z) - \tfrac1{2\beta} \E\big[ \G_{N,N_T}(z) \big] .
\end{cases}\]
This quantity should be compared to \eqref{err0}. Namely, we split  
\begin{align} 
\Omega_N(z) &\notag
= \tilde{\psi}_{N,N_T}(z)+ \bpsi_{N_T}(z) + \i\vartheta_{N,N_T}(z)- \i\pi N F(z) + {\mathrm{c}_\beta}\log(1-z^2) +\tfrac{1}{\sqrt{\beta}}\M_{N}(z) \\
&\notag
= \Omega_N^2(z;T)+\bpsi_{N_T}(z) +\tfrac{\M_{N_T,N_0}(z)}{\sqrt{\beta}}+\Big(\tfrac{\G_{N_0}(z)}{\sqrt{\beta}} + \tfrac{[\G_{N_0}(z)]}{2\beta}\Big)\\
&\notag\qquad
+\i\big(\vartheta_{N,N_T}(z)- \pi N F(z)\big) 
+ {\mathrm{c}_\beta}\log(1-z^2)  -\tfrac14\log\big(\tfrac{N(1-z^2)}{T\L(z)}\big) - \tfrac{[\G_{N_0}(z)]-[\G_{N,N_T}(z)]}{2\beta}
\\
&\label{Omegasplit}
= \Omega_N^1(z;T) + \Omega_N^0(z;T) +\Omega_N^2(z;T) 
\end{align}
where we define for $z\in \Q$, 
\begin{equation} \label{err3}
\begin{aligned}
\Re\Omega_N^0(z;T)&\coloneqq     {\mathrm{c}_\beta}\log(1-z^2)  -\tfrac14\log\big(\tfrac{1-z^2}{T}\big) - \tfrac{[\G_{N_0}(z)]-\Re[\G_{N,N_T}(z)]}{2\beta}\\
\Im \Omega_N^0(z;T)&\coloneqq   \vartheta_{N,N_T}(z)- \pi N F(z) +  \pi N_T(z)\1\{z< 0\} + \tfrac{\Im[\G_{N,N_T}(z)]}{2\beta}, \\
\Omega_N^1(z;T)&\coloneqq \bpsi_{N_T}(z) +\Big(\tfrac{\G_{N_0}(z)}{\sqrt{\beta}} + \tfrac{[\G_{N_0}(z)]}{2\beta}\Big)- \tfrac14\log\big(\tfrac{N}{\L(z)}\big) - \i \pi N_T(z)\1\{z< 0\} + \tfrac{\M_{N_T,N_1}(z)}{\sqrt{\beta}} .
\end{aligned}
\end{equation}
Here,  the \emph{error} $\Omega_N^0$ is deterministic, $\Omega_N^1$ is random and related to the parabolic stretch of the recursion, while $\Omega_N^2$ accounts for the  elliptic part of the error. 
In particular, by Proposition~\ref{prop:oneray1}, for $z\in\Q$, by extracting a subsequence as $N\to\infty$, 
\begin{equation} \label{err2}
\Omega_N^2(z;T) \to  \curlywedge_T^* , \qquad    \curlywedge_T^* = \O_{\P}(T^{-1/2}). 
\end{equation}
The limit \eqref{err2} holds for $T\in\mathds{Q}\cap[1,\infty)$ by a diagonal extraction and $\curlywedge_T  \to 0$ in probability as $T\to\infty$. This limit depends a priori on the subsequence and on $z\in\Q$. 

\smallskip

The \emph{parabolic error} can be handled using the Stochastic Airy machinery from \cite{LambertPaquette03}.
We review the relevant results in Section~\ref{sec:HP} and we have $\Omega_N^1(z;T) =\mho_N^1(0, T;z)+ \M_{N_T,N_1}(z)/\sqrt{\beta}$ so that, by Proposition~\ref{prop:Airy}, in distribution as $N\to\infty$, 
\begin{equation} \label{err5}
\Omega_N^1(z;T) \to \tfrac{\mathrm{g}+ \mathfrak{m}_T^{\pm}}{\sqrt{\beta}} -\tfrac{\E\mathrm{g}^2}{2\beta}+ \boldsymbol{\varpi}^\pm_T(0)  , \qquad \pm=\sgn(z), 
\end{equation}
where $\mathrm{g}$ is a Gaussian variable with mean zero and the law of $ \boldsymbol{\varpi}$ is specified by Definition~\ref{def:SA} in terms of the stochastic Airy function. The limit $ \omega_T^{\pm}$ is a (random) continuous function of $T$ and the convergence holds as processes indexed by $T\in \R_+$ and, besides $\pm = \sgn(z)$, the limit \eqref{err5} is independent of $z$. 

\smallskip

Finally, for the \emph{deterministic error}, by Proposition~\ref{prop:app_G1} below, if $z\in\Q$, 
\[
[\G_{N_0}(z)] - [\G_{N,N_T}(z)]= - \log\bigg(\frac{1-z^2}{T/4}\bigg) + \underset{N\to\infty}{\o(1)} , \qquad
\Im [\G_{N,N_T}(z)] = \pm \pi -2 \arcsin(z)  +\underset{N\to\infty}{\o(1)}
\]
so that with $ {\mathrm{c}_\beta} =\tfrac14-\frac{1}{2\beta}$, 
\[
\Re\Omega_N^0(z;T) =  {\mathrm{c}_\beta} \log T - \tfrac{\log 2}{\beta} + \underset{N\to\infty}{\o(1)}. 
\]
Then, by Proposition~\ref{prop:detphase} below, if $z\in\Q$, with $\pm=\sgn(z)$,
\[
\vartheta_{N,N_T}(z) -  \pi N F(z)  =  - N_T(z)\1\{z<0\} \mp \big( \tfrac23 T^{3/2} - \tfrac\pi4\big)
-\frac{\arcsin(z)}{2} +\underset{N\to\infty}{\o(1)} . 
\]
so that
\[
\Im \Omega_N^0(z;T) = \mp \big( \tfrac23 T^{3/2} - \tfrac\pi4\big)
-\tfrac{\arcsin(z)}{2} - \tfrac{\pm \pi-2\arcsin(z)}{2\beta}   +  \underset{N\to\infty}{\o(1)} .
\]
This implies that  if $z\in\Q$,  
\begin{equation} \label{err4}
\Omega_N^0(z;T) =  {\mathrm{c}_\beta} \log T  - \tfrac{\log 2}{\beta}    \mp\i \big( \tfrac23 T^{3/2} - {\mathrm{c}_\beta}\pi\big) - \i2{\mathrm{c}_\beta}\arcsin(z)  +  \underset{N\to\infty}{\o(1)} .
\end{equation}

Hence, combining \eqref{Omegasplit} with \eqref{err2}, \eqref{err5}, \eqref{err4}, we conclude that in distribution as $N\to\infty$ (along an appropriate subsequence for $T\in\mathds{Q}\cap[1,\infty)$),  
\begin{equation} \label{limitOmega}
\Omega_N(z) \to  {\mathrm{c}_\beta} \log T  - \tfrac{\log 2}{\beta}    \mp\i \big( \tfrac23 T^{3/2} - {\mathrm{c}_\beta}\pi\big) - \i2{\mathrm{c}_\beta}\arcsin(z) +\tfrac{\mathrm{g}}{\sqrt{\beta}} -\tfrac{\E\mathrm{g}^2}{2\beta}+ \tfrac1{\sqrt\beta}\boldsymbol{\varpi}^\pm_T(0) + \mathfrak{m}_T^{\pm} + \curlywedge_T^* . 
\end{equation}
In particular, the RHS of \eqref{limitOmega} is independent of $T$, and since $\curlywedge_T^*\to 0$ as $T\to\infty$, 
the following limit holds in distribution
\begin{equation} \label{limitmho}
\big( \boldsymbol{\varpi}^\pm_T(0) + \tfrac1{\sqrt\beta} \mathfrak{m}_T^{\pm}   \mp\i   \big(\tfrac23 T^{3/2}- {\mathrm{c}_\beta}\pi \big) + {\mathrm{c}_\beta} \log T\big) 
\to \hat{\boldsymbol{\mho}}_\beta^{\pm} .
\end{equation}
These asymptotics are directly relevant to prove Proposition~\ref{cor:mho}. 
Then, by  \eqref{limitOmega}, we also obtain the limit in distribution as $N\to\infty$, 
\[
\Omega_N(z) \to \hat{\boldsymbol{\mho}}_\beta^{\pm} - \tfrac{\log 2}{\beta}  - \i2{\mathrm{c}_\beta}\arcsin(z) +\tfrac{\mathrm{g}}{\sqrt{\beta}} -\tfrac{\E\mathrm{g}^2}{2\beta} .
\]
This limit holds for any $z\in\Q$ and along the full sequence as $N\to\infty$ since limit has the same law along any subsequence. 
This completes the proof.
\end{proof}

\begin{remark}[Hermite polynomial asymptotics]
If $\beta=\infty$, one has $\mathrm{c}_\infty =1/4$ and we obtain the asymptotics for $z\in\Q$, 
\[
\bpsi_N(z)|_{\beta =\infty} =  \i\pi N F(z) - {\mathrm{c}_\infty}\log(1-z^2) + \Omega_N(z)|_{\beta =\infty} 
=   \i\pi N F(z) -\tfrac14\log(1-z^2)  -\tfrac{\log \pi}{2} - \i\tfrac{\arcsin(z)}{2} + \underset{N\to\infty}{\o(1)}
\]
since $\hat{\boldsymbol{\mho}}_\infty = -\tfrac{\log \pi}{2}$ according to Remark~\ref{rk:Airy}. 
Since $h_N(z) = \Re\exp(\bpsi_N(z)|_{\beta =\infty})$, we recover the Hermite polynomial asymptotics \eqref{PR1}. 
\end{remark}

\subsection{Asymptotic regime in a neighborhood of 0.} \label{sec:0} 

As discussed in the introduction, in a $\O(N^{-1/2})$-neighborhood around 0, the whole transfer matrix recursion is elliptic.
In particular, there is no turning point and the characteristic polynomial cannot be approximated using the stochastic Airy function at the start of the recursion.
In fact, by \eqref{def:Phi0},  the characteristic polynomials $\{\widehat\Phi_n\}$ are independent of $N$ in this regime and we consider the complex phase:
\[
\mu\in \big(-\sqrt{4n},\sqrt{4n}\big) \mapsto \widehat{\bpsi}_n(\mu) \coloneqq \bpsi_n(\tfrac{\mu}{2\sqrt N})-\tfrac14\log N .
\]

\begin{lemma} \label{lem:GS}
For every $n\in\N$, the function $\big\{ \widehat{\bpsi}_n(\mu) ; |\mu| <\sqrt{4n} \big\}$
is smooth, independent of $N$, and $ \widehat{\Phi}_n(\mu)  =\Re\big[\exp\widehat{\bpsi}_n(\mu)\big]$. 
\end{lemma}

\begin{proof}
This is a simply a rescaling using \eqref{def:Phi0}; $\{\widehat{\Phi}_n(\mu)\}$ are polynomial of increasing degree $n$ that are  independent of $N$. Then, by \eqref{def:psi}, for $\mu\in\big(-\sqrt{4n},\sqrt{4n}\big)$, 
\[\begin{aligned}
\exp\Big(\bpsi_n(\tfrac{\mu}{2\sqrt N}\big)\big)  & = \Big(\Phi_n(\tfrac{\mu}{2\sqrt N}\big) -\i\Big(2\sqrt{\tfrac{n+1}{4n-\mu^2}} \Phi_{n+1}(\tfrac{\mu}{2\sqrt N}\big)-  \mu \sqrt{\tfrac{n}{4n-\mu^2}} \Phi_n(\tfrac{\mu}{2\sqrt N}\big) \Big)\Big) \\ 
\exp\big(\widehat{\bpsi}_n(\mu)\big) & = \Big(\widehat{\Phi}_n(\mu) -\i\Big(2\sqrt{\tfrac{n+1}{4n-\mu^2}} \widehat{\Phi}_{n+1}(\mu)-  \mu \sqrt{\tfrac{n}{4n-\mu^2}} \widehat{\Phi}_n(\mu) \Big)\Big) \sqrt{e^{-\mu^2/2} /\sqrt{2\pi}} .\qquad \qedhere
\end{aligned}\]
\end{proof}



Let $\widehat{Z}_n \coloneqq Z_n(0)$ for $n\in\N$.
According to \eqref{def:J}, $\{\widehat{Z}_n\}$ is a sequence of i.i.d.~standard Gaussians random variables.
We define for $\mu\in\R$, the martingale sequence 
\begin{equation} \label{M0}
\widehat{\M}_{n,m}(\mu) \coloneqq \sum_{m<k\le n} \frac{-\i}{\sqrt k}\Big(\widehat{Z}_k-\overline{\widehat{Z}_k (-1)^k  e^{2\i\widehat{\phi}_{k-1}(\mu)}}\Big)  , \qquad \sqrt{4m} \ge |\mu| . 
\end{equation}
The last condition guarantees that the phase $\widehat\phi_k = \Im\widehat\psi_k$ is defined for $k\ge m$. 

This martingale is related to the martingale $\{\M_n\}$ from Definition~\ref{def:GW} by the following estimates: if $\mathcal{K} \Subset  \big(-3\sqrt{m},3\sqrt{m}\big)$, 
\[
\sup_{N\ge m}\sup_{\mu \in \mathcal{K}} \big\| \sup_{n>m} \big| \M_{n,m}\big(\tfrac{\mu}{2\sqrt N}\big)  - \widehat{\M}_{n,m}(\mu) \big| \big\|_2^2 
\lesssim \frac{C_{\mathcal{K}}}{m} . 
\]
These sub-Gaussian bounds follow directly from Lemma~\ref{lem:GW0}.

The global/local asymptotic behaviors of $\big\{\widehat{\bpsi}_n(\mu)\big\}$ for $\mu\in\R$ as $n\to\infty$ can be analyzed like that of $\big\{\bpsi_N(z)\big\}$ for $z\in(-1,1)$ as $N\to\infty$, the situation is even simpler because their is no \emph{turning point} and the martingale \eqref{M0} is also less complex. In this regime, we obtain the following result which is a special case of our main Theorem~\ref{thm:main}.
We will review the main steps of the proof to explain the main differences. 

\begin{theorem}[Asymptotics in a neighborhood of 0] \label{thm:0}
Let $\Lambda_n \coloneqq - \tfrac14\log(n) + \frac{\i\pi n}{2}$ for $n\in\N$ and $\mathcal{K} \Subset\R$ be any compact set. 
Fix $m\in\N$ such that $\mathcal{K} \subset  \big(-\sqrt{4m},\sqrt{4m}\big)$.
Then, on $\A_m$, for any $n\ge m$,
\begin{equation} \label{psi0}
\widehat{\bpsi}_{n}(\mu) = \Lambda_n + \i\sqrt{n} \mu   +\tfrac{1}{\sqrt{\beta}}\widehat{\M}_{n,m}(\mu) + \Omega_{n}^{(m)}(\mu)
\end{equation}
and $\Omega_{n}^{(m)}(\mu) \to \Omega_\infty^{(m)}(\mu)$ in probability as $n\to\infty$. 
Moreover, for a fixed $\mu\in\R$, it holds as $N\to\infty$,  
\[
\Big\{ \{\widehat{\bpsi}_{N}(\mu)\}_{2\pi} , \big( \widehat{\bpsi}_{N}\big(\mu + \tfrac{\pi \lambda}{\sqrt N}\big)  - \widehat{\bpsi}_{N}(\mu) \big) : \lambda\in\R \Big\} \to 
\big\{ \boldsymbol{\alpha} , \tfrac12 \omega_1(\lambda) : \lambda\in\R  \big\}
\]
in the sense of of finite dimensional distributions. 
\end{theorem}

\begin{proof}[Proof]
{\bf Pointwise asymptotics.}
We proceed as in Lemma~\ref{lem:linear} to linearize the recurrence equation for $\big\{\widehat{\bpsi}_n(\mu)\big\}$. On the event $\A_m$, it holds for $n\ge m$, 
\begin{equation} \label{lin0}
\widehat{\bpsi}_{n,m}(\mu) =\i\bigg( \frac{\pi (n-m)}{2} - \sum_{k=m+1}^{n} \widehat{\theta}_k(\mu) \bigg) - \frac{\widehat{\mathbf{Q}}_{n,m}(\mu)}{4} - \frac{\widehat{\M}_{n,m}(\mu)}{\sqrt\beta}  + \frac{\widehat{\mathbf{S}}_{n,m}(\mu)}{2\beta} + \O(m^{(3\epsilon-1)/2}) 
\end{equation}
where $\widehat\theta_k(\mu) = \tfrac\pi2-\arccos(\tfrac{\mu}{2\sqrt k}\big)$,
$\widehat{\mathbf{Q}}_{n,m}(\mu)\coloneqq\sum_{k=m+1}^n k^{-1} \big(1 -  e^{-2\i \widehat{\phi}_{k-1}(\mu)}\big)$
and $\widehat{\mathbf{S}}_{n,m} \coloneqq \sum_{m<k\le n} k^{-1} \big( \widehat{Z}_{k}  +\overline{\widehat{Z}_k}  e^{-2\i\widehat{\theta}_k(\mu)} e^{-2\i\widehat{\phi}_{k-1}(\mu)}  \big)^2$.
Moreover, we can replace the deterministic term $\sum_{k=m+1}^{n} \widehat{\theta}_k(\mu)$ by $\mu(\sqrt{n}-\sqrt{m})$, up to a negligible error.
The proof is exactly the same with $z=\tfrac{\mu}{2\sqrt{N}}$, replacing $\{\Delta_n(z),Z_n'(z)\}$ by $\{n^{-1/2}, \widehat{Z}_n \}$ in \eqref{zeta3} instead of $\{\delta_n^2(z),Z_n(z)\}$ by using the estimates \eqref{lin0est}  as in Lemma~\ref{lem:GW0} 
In particular, by \eqref{EL}, the linearization errors satisfy
$\sup_{\lambda\in\mathcal{K}}\sum_{n\ge m}  |\operatorname{EL}_n\big(\tfrac{\mu}{2\sqrt{N}}\big)|  \lesssim m^{(3\epsilon-1)/2}$. Moreover, the oscillatory sums in $\widehat{\mathbf{Q}}_{n,m}$ and  $\widehat{\mathbf{S}}_{n,m}$ are also small when $m\gg1$.

\smallskip

By Propositions~\ref{prop:osc1} and~\ref{prop:osc2} \big(with $z=\tfrac{\mu}{2\sqrt{N}}$, $x=z$ and $m$ fixed using the blocks $n_k \coloneqq  \lfloor   k^{3/2} \rfloor$ for $k\ge T = \lceil m^{2/3}\rceil$\big) we have on the event $\A_\chi =\A_{\chi}(R,T; z)$, 
\[\begin{aligned}
\widehat{\mathbf{Q}}_{n,m}&= \sum_{k=m+1}^n k^{-1} +\O(R m^{-1/3})  = \log\left(\frac nm\right) +\O(R m^{-1/3}) ,  \\
\widehat{\mathbf{S}}_{n,m}&= \sum_{m<k\le n} k^{-1}\big( \widehat{Z}_{k}^2  + 2\big(|\widehat{Z}_{k}|^2 -1\big) e^{-2\i\widehat{\theta}_k} e^{-2\i\widehat{\phi}_{k-1} }  
+ \overline{\widehat{Z}_k^2}  e^{-4\i\widehat{\theta}_k} e^{-4\i\widehat{\phi}_{k-1} }\big)^2 +\O(R m^{-1/3})  ,
\end{aligned}\] 
where the error terms are controlled uniformly for $n\ge m$.
Then, up to errors, $\widehat{\mathbf{S}}_{n,m}$ is a complex martingale whose increments satisfy
$\|\widehat{\mathbf{S}}_{n+1,n}\|_{1} \lesssim n^{-1}$. Then, by Proposition~\ref{lem:conc1}, there is a constant $C\ge 1$ so that for any $R\ge 1$, 
\[
\P\Big[\big\{\sup_{n\ge m} |\widehat{\mathbf{Q}}_{n,m}- \log\left(\tfrac nm\right)| \ge C R m^{-1/3} \big\}\cup \big\{ |\widehat{\mathbf{R}}_{n,m}| \ge CR m^{-1/3} \big\}\cap \A_\chi \Big] \lesssim \exp(-R m^{2/3} ) .
\]
The random sequence $\{\Omega_{n}^{(m)}(\mu)\}_{n\ge m}$ is defined implicitly by \eqref{psi0}. In particular 
$\Omega_{m}^{(m)}(\mu) = \widehat{\bpsi}_{m}(\mu) - \Lambda_m - \i\sqrt{m} \mu $ and the increments of $\{\Omega_{n}^{(m)}(\mu)\}_{n\ge m}$ are controlled using \eqref{lin0}. 
Thus, combining the previous estimate and choosing for instance $R= m^{\epsilon}/C$ with $\epsilon>0$ and fixed, there is a constant $c>0$ such that if $m \ge \mathfrak{K}$,
\begin{equation} \label{Cauchy0}
\P\Big[\sup_{n\ge m} |\widehat{\Omega}_{n,m}^{(m)}(\mu)| \ge C m^{\epsilon-1/3} \Big]  \lesssim \exp(-c m^\epsilon) . 
\end{equation}
Note that we have included $\P[\A_m^c]$ and $\P[\A_\chi^c \cap \A_m]$ on the RHS by \eqref{PA}  and \eqref{PAchi}.

This establishes that 
$\{\Omega_{n}^{(m)}(\mu)\}_{n\ge m}$ is a Cauchy sequence in probability, so it is convergent; that is, for fixed $\mu\in\mathcal{K}$ and $m \ge \mathfrak{K}$, 
\[
\Omega_{n}^{(m)}(\mu) \to \Omega_\infty^{(m)}(\mu)  \qquad\text{in probability as $n\to\infty$}. 
\]

\paragraph{Local asymptotics.}
For the second claim, we consider the relative phase, for $N\gg 1$,  
\[
\widehat{\varphi}_n^{(N)}(\lambda;\mu) \coloneqq 2\bigg(\widehat{\bpsi}_{n}\big(\mu + \tfrac{\pi \lambda}{\sqrt N}\big)  - \widehat{\bpsi}_{n}(\mu) \bigg)  , \qquad  m \le n\le N .
\]
By \eqref{psi0} and \eqref{Cauchy0} to control the error term, on the appropriate event: it holds uniformly for any $n\ge m$, 
\[
\widehat{\varphi}_n^{(N)}(\lambda;\mu)  =  2\pi\i\lambda \sqrt{\frac nN}  +\sqrt{\frac2\beta} \sum_{m \le k < n} \frac{\widehat{W}_{k+1}}{\sqrt{k+1}}{\rm f}\big(\widehat{\varphi}_{k}^{(N)}(\lambda;\mu)\big) +\O\big(m^{\epsilon-1/3}\big) , \qquad \widehat{W}_{k} = \i \sqrt{2} \overline{\widehat{Z}_k (-1)^k  e^{2\i\widehat{\phi}_{k-1}(0)}}, 
\]
where  ${\rm f} : w\in\C \mapsto (1-e^{\i\Im w})$.
This is a discretization of the complex sine equation \eqref{eq:csse} with 
$t = \frac nN \in [ \frac mN,1]$. 
To prove the convergence of this process as $N\to\infty$, we use the scheme from Section~\ref{sec:zeta}.
In particular, this relies on the coupling\footnote{Here, $S_{j+1}\coloneqq  \sqrt{\frac1{N\eta}} \sum_{k=n_{j}+1}^{n_{j+1}}  \widehat{W}_{k} $ with $n_j = j N\eta$ for $j\in[\delta \eta^{-1},\eta^{-1}] $ choosing the parameter $\eta(N) \ll 1$. } 
from Lemma~\ref{lem:coupling} which does not have an explicit rate of convergence, so this requires to consider the above equation starting from $m\leftarrow \lfloor \delta N \rfloor $ for a small $\delta>0$ to apply the stochastic Gr\"onwall inequality (in this case, the initial condition is controlled using~the estimates \eqref{cont0}).
This concludes the proof as a special case of Proposition~\ref{prop:sine}.
\end{proof}

In \eqref{psi0}, the limit $\Omega^{(m)}$ depends on the parameter $m\in\N$ through the truncation of the martingale \eqref{M0}.
Then, using that  $\Omega_{m}^{(m)}(\mu)  = \widehat{\bpsi}_{m}(\mu)- \Lambda_m - \i\sqrt{m} \mu + $ with 
$\Omega_{m}^{(m)}(\mu)  = \Omega_{n}^{(m)}(\mu) - \widehat{\Omega}_{n,m}^{(m)}(\mu)$, taking the limit as $n\to\infty$ using the estimate \eqref{Cauchy0}, we obtain for a fixed $\mu\in\mathcal{K}$
\begin{equation} \label{Omegalim0}
\Omega^{(m)}_\infty(\mu)  =  \widehat{\bpsi}_{m}(\mu) - \Lambda_m - \i\sqrt{m} \mu+ \O_{\P}(m^{\epsilon-1/3}) , \qquad
\text{for any $m\in\N$ with }\mathcal{K} \subset  \big(-\sqrt{4m},\sqrt{4m}\big). 
\end{equation}

Finally, Theorem~\ref{thm:0} implies that the sequence of characteristic polynomials \eqref{def:Phi0} of the random Jacobi matrix $\mathbf{A}$ satisfy for a fixed $\mu\in\mathcal{K}$:
\begin{equation} \label{charpolyasymp}
\det[\mu-\beta^{-1/2}\mathbf{A}]_{n,n}  = \sqrt{n!} n^{-1/4} 
\Re \Big[ \exp\Big( \tfrac{\i\pi n}{2} + \i\sqrt{n} \mu   +\tfrac{1}{\sqrt{\beta}}\widehat{\M}_{n,m}(\mu) +  \Omega_\infty^{(m)}(\mu) + \underset{n\to\infty}{\o_\P(1)} \Big) \Big]. 
\end{equation}

\todo[inline]{Corollary for the eigenfunctions of $\mathbf{A}$?}


\begin{lemma} \label{lem:GW0}
Let $\mathcal{K} \Subset \R$ and $m\in\N$  with $\mathcal{K} \Subset  \big(-3\sqrt{m},3\sqrt{m}\big)$.
Then 
\[
\big\| \sup_{n>m} \big| \G_{n,m}(\tfrac{\lambda}{2\sqrt N}\big)  -  {\textstyle \sum_{k=m}^n} \tfrac{Z_k(0)}{\i\sqrt k} \big| \big\|_2^2  \lesssim  m^{-1} , 
\qquad 
\big\| \sup_{n>m} \big| \W_{n,m}(\tfrac{\lambda}{2\sqrt N}\big)  +{\textstyle \sum_{k=m}^n} \tfrac{Z_k(0)}{\i\sqrt k} (-1)^ke^{2\i\widehat{\phi}_{k-1}(\lambda)} \big| \big\|_2^2  \lesssim m^{-1} , 
\]
where the implied constants depend only on $\mathcal{K}$.
\end{lemma}

\begin{proof}
We use the notations $\widehat{Z}_k(\mu) \coloneqq Z_k(\tfrac{\mu}{2\sqrt N}\big)$ 
and  $\widehat\theta_k(\mu) =  \theta_k(0) -\theta_k(\tfrac{\mu}{2\sqrt N}\big)  = \tfrac\pi2-\arccos(\tfrac{\mu}{2\sqrt k}\big)$
for $k\in\N$ (if defined). 
Then $\widehat{Z}_k(\mu) =  \tfrac{X_k - \i Y_k e^{\i \widehat\theta_k(\mu)}}{\sqrt 2}$,  $\widehat{Z}_k(0) = \widehat{Z}_k$ and
for $k\ge m$,
\begin{equation} \label{lin0est}
\widehat{\theta}_k(\mu) = \tfrac{\mu}{2\sqrt k} + \tfrac{\mu^3}{48 k^{3/2}} +\O(k^{-2})  , \qquad \tfrac1{\sqrt{k-\mu^2/2}} = \tfrac1{\sqrt{k}}- \tfrac{\mu^2}{4 k^{3/2}} +\O(k^{-2}) 
\end{equation}
where the implied constants depend only on $\mathcal{K}$.

According to Definition~\ref{def:GW}, choosing $m$ sufficiently large, one has
\[
\G_{n,m}\big(\tfrac{\mu}{2\sqrt N}\big) = \sum_{k=m+1}^{n} \frac{-\i\widehat{Z}_k(\mu) }{\sqrt{2k-\mu^2/2}} , \qquad 
\W_{n,m}(\tfrac{\mu}{2\sqrt N}\big) =\sum_{k=m+1}^{n} \frac{-\i\widehat{Z}_k(\mu)}{\sqrt{2k-\mu^2/2}} (-1)^k e^{-2\i\widehat\theta_k(\mu)} e^{2\i\widehat\phi_{k-1}(\mu)}
\]
using that $\phi_{k}(\tfrac{\mu}{2\sqrt N}) =\widehat\phi_k(\mu) $ for $k\ge m$. 
Consequently, using that $(\widehat{Z}_k(\mu)-\widehat{Z}_k) = Y_k(1-e^{\i\widehat{\theta}_k(\mu)})/\sqrt{2}$
\[
\G_{n,m}\big(\tfrac{\mu}{2\sqrt N}\big) - \sum_{k=m+1}^{n} \frac{-\i \widehat{Z}_k}{\sqrt k} 
= \i\sum_{k=m+1}^{n} \bigg(\frac1{\sqrt{k}}- \frac1{\sqrt{k-\mu^2/2}}\bigg) \widehat{Z}_k(\mu) 
+ \i \sum_{k=m+1}^{n} Y_k \frac{1-e^{\i\widehat{\theta}_k(\mu)}}{\sqrt{2 k}}
\]

\[
\sup_{\mu\in\mathcal{K}}\bigg| \G_{n,m}\big(\tfrac{\mu}{2\sqrt N}\big) - \sum_{k=m+1}^{n} \frac{\i \widehat{Z}_k}{\i\sqrt k} 
-  \i \mu \sum_{k=m+1}^{n} \frac{Y_k}{\sqrt8 k} \bigg| 
\lesssim \sum_{k=m+1}^{n} \frac{|\widehat{Z}_k|}{k^{3/2}}
\]
and 
\[
\sup_{\mu\in\mathcal{K}}\bigg| \W_{n,m}\big(\tfrac{\mu}{2\sqrt N}\big) - \sum_{k=m+1}^{n} \frac{\widehat{Z}_k (-1)^k  e^{2\i\widehat{\phi}_{k-1}(\mu)}}{\sqrt k} 
-  \i \mu \sum_{k=m+1}^{n} \frac{Y_k  e^{2\i\widehat{\phi}_{k-1}(\mu)}}{\sqrt8 k} \bigg| 
\lesssim \sum_{k=m+1}^{n} \frac{|\widehat{Z}_k|}{k^{3/2}}.
\]

This sum is independent of $N$, it is a $\{\F_n\}$-martingale and, using that  we have 
\[
\sum_{k=m+1}^{n}  \big\| \tfrac{X_n + e^{\i\widehat{\theta}_n(\lambda)}Y_n}{\sqrt{2k-\lambda^2/2}}-  \tfrac{Z_k(0)}{\sqrt k} \big\|_2^2 \lesssim \sum_{k=m+1}^n  \big| \frac1{k-\lambda^2/4} -\frac1k\big|
+  \sum_{k=m+1}^n \frac{|\widehat{\theta}_n(\lambda)-\widehat{\theta}_n(0)|^2}k.  
\]  
Since  $|\widehat{\theta}_n(\lambda)-\widehat{\theta}_n(0)| \le \sqrt{ \frac{\lambda^2}{4k-\lambda^2}}$, this implies that if $2 m \ge \lambda^2$,  
\[
\sum_{k=m+1}^{n}  \big\| \tfrac{X_n + e^{\i\widehat{\theta}_n(\lambda)}Y_n}{\sqrt{k-\lambda/4}} -  \tfrac{Z_k(0)}{\sqrt k} \big\|_2^2 \lesssim \sum_{k=m+1}^n \frac{\lambda^2}{k^2}  \lesssim \frac{\lambda^2}{m} . 
\]  
Using the martingale property, this yields   for any $n\ge m$, 
\[
\big\| \sup_{n>m} \big| \G_{n,m}(\tfrac{\lambda}{2\sqrt N}\big)  -  \tfrac{Z_k(0)}{\sqrt k} \big| \big\|_2^2 \lesssim \frac{\lambda^2}{m} . 
\]
Similarly, using Lemma~\ref{lem:GS},  $\widehat{\phi}_n(\lambda)=\Im\widehat{\bpsi}_n(\lambda)$ and 
\[
\W_{n,m} (\tfrac{\lambda}{2\sqrt N}\big) = \sum_{k=m+1}^{n} \tfrac{(X_n + e^{\i\widehat{\theta}_n(\lambda)}Y_n) }{\sqrt{2k-\lambda^2/2}}e^{2\i(\widehat{\theta}_k(\lambda)+\widehat{\phi}_{k-1}(\lambda))}
\] 
is also a $\{\F_n\}$-martingale, independent of $N$.  Moreover, the same computation as for the $\G$ field, also gives  for any $n\ge m$, 
\[
\sum_{k=m+1}^{n} \big\| \tfrac{(X_n + e^{\i\widehat{\theta}_n(\lambda)}Y_n)}{\sqrt{2k-\lambda^2/2}}e^{2\i(\widehat{\theta}_k(\lambda)+\widehat{\phi}_{k-1}(\lambda))}
+ \tfrac{Z_k(0)}{\sqrt k} e^{2\i\widehat{\phi}_{k-1}(\lambda)} \big\|_{2,k}^2 \lesssim \frac{\lambda^2}{m} . \qedhere
\]
%
\end{proof}

\section{Relative phase}\label{sec:cont}


This section concerns continuity properties of the \emph{complex phase} $z\mapsto \psi_n(z)$ on  short scales. These estimates are important to understand how the phases at different points \emph{decorrelate} in the elliptic regime (branching structure).
In particular, we will give two applications of Proposition~\ref{prop:relcont1}:
\begin{itemize}[leftmargin=*]   \setlength\itemsep{0em}
\item decorrelations of the $\W$ part of the martingale noise (Section~\ref{sec:Wosc}).
\item control of the initial condition in the approximation of the \emph{microscopic} relative phase by the complex sine equation \eqref{eq:csse} (Section~\ref{sec:zeta}).
\end{itemize}

Throughout this section, let $T\ge 2$ ($T$ is fixed independent of $N$), and set
\begin{equation} \label{ell}
m :=  N_T(z) , \qquad\qquad
\Omega(w,z) \coloneqq N^{-1}|z-w|^{-2}. 
\end{equation}
The quantity $\Omega(w,z)$ will be used to control the errors. 
We consider the event  
\begin{equation} \label{B1}
\B =\B(T,\varepsilon; z) \coloneqq 
\big\{ 
|\bpsi_m(w)- \bpsi_m(z)|\leq \big(\mathfrak{L}(z)/\Omega(w,z)\big)^{3/8}  ; \, |z-w| \le \varepsilon/\sqrt{N\mathfrak{L}(z)}
\big\}. 
\end{equation}
This event controls the entrance behavior of the relative phase at the start of the elliptic stretch.
Using the stochastic Airy function machinery, one can prove that if $z\in\Q$, for a fixed $T$, 
$\B(T,\varepsilon; z)$ holds with high probability. 
By Proposition~\ref{prop:entcont} (with $\alpha=8/3$, $c=1$ and $T$ fixed), we have for $z\in\Q$
\begin{equation} \label{B2}
\liminf_{\varepsilon\to0}\liminf_{N\to\infty}\P[\B(T, \varepsilon; z)] =1.
\end{equation}

Throughout this section, we  assume
that for some $0<\varepsilon \le 1$, 
\begin{equation} \label{loc}
|z-w| \le \varepsilon/\sqrt{N\mathfrak{L}(z)} .
\end{equation}
This is the regime where the \emph{turning points are matching}, meaning that  
\[
|N_0(z)- N_0(w)|\le 2\big( N|z| |z-w| + \varepsilon^2\big) \lesssim\varepsilon\, \mathfrak{L}(z)
\]
so that $|\mathfrak{L}(z)-\mathfrak{L}(w)| \lesssim  \varepsilon$ 
and also $\Omega(w,z) \ge  \varepsilon^{-2}\mathfrak{L}(z)$. 




\begin{proposition}[Continuity] \label{prop:relcont1}
Recall that $\phi_n(z) =\Im\bpsi_n(z)$ for $n\ge m$. 
Let $z,w\in(-1,1)$ which satisfy \eqref{loc}. 
Consider the events $\A_m$ \eqref{eventA}, and $\B =\B(T, \varepsilon; z)$ \eqref{B1}. 
There are constants $C , c >0$ (depending only on $\beta$) so that, with $m$ and $\Omega$ as in \eqref{ell}, one has for $R\ge 1$. 
\[\begin{aligned}
\P\Big[\big\{\exists n \in [m,N_0(z)+e^{-C R}\Omega] ;  |\phi_n(w)-\phi_n(z)| > \big(\tfrac{n-N_0(z)}{\Omega(w,z)}\bigr)^{1/4} \big\}\cap \A_m \cap \B \Big]   \lesssim  \exp(-cR).
\end{aligned}\]
\end{proposition}

The proof of Proposition~\ref{prop:relcont1} occupies the remainder of this section. Recall that we decompose the phase $\phi_n =\phi_m+ \vartheta_{n,m}+ \chi_{n,m}$, where the deterministic part  $\{ \vartheta_{n,m}\}_{n\ge m}$ satisfies appropriate estimates and the initial condition $\phi_m=\Im\psi_m$ is controlled by \eqref{B1}.
The proof is divided in the following steps:
\begin{enumerate}[leftmargin=*]   \setlength\itemsep{0em}
\item In Section~\ref{sec:linrel}, based on  Lemma~\ref{lem:linear}, we linearize the recursion equation satisfied by the relative phase $\big\{\partial\chi_n(w,z)\big\}_{n\ge m}$ and establish bounds for the various linearization errors. 
\item In Section~\ref{sec:relrep}, this allows us to express $\big\{\partial\chi_n(w,z)\big\}_{n\ge m}$ in terms of certain  generators $\{P_n\}_{n\ge m}$ and we develop bounds  for these generators in Section~\ref{sec:gen} which allows to control the growth of  $\big\{\partial\chi_n(w,z)\big\}_{n\ge m}$.
\item Finally, we conclude the proof in Section~\ref{sec:st} by using a stopping time argument. 
\end{enumerate}

For Section~\ref{sec:zeta}, we record the following consequence of Proposition~\ref{prop:relcont1}.

\begin{proposition} \label{prop:entrancesmall}
Let $\mathcal{K} \Subset \R$.  
Let $z\in\Q$ and $M = M(\delta;z): = N_0(z)+ \delta N \varrho(z)^2$ for $\delta>0$.
There is a constant $\epsilon>0$ so that
\[
\lim_{\delta\to 0} \limsup_{N\to\infty} \sup_{\lambda\in\mathcal{K}} \P\big[|\partial\bpsi_M\big(z,z+\tfrac{\lambda}{N\varrho(z)}\big)| > \delta^\epsilon \big] = 0 . 
\]
\end{proposition}

\begin{proof}
If $z\in\Q$, then $\sqrt{N\mathfrak{L}(z)} \ll N\varrho(z)$ as $N\to\infty$ so that 
\begin{equation} \label{Bmcond}
\limsup_{N\to\infty} \sqrt{N\mathfrak{L}(z)}\, \sup_{\lambda\in\mathcal{K}}\big|z-\big(z+\tfrac{\lambda}{N\varrho(z)}\big)\big| = 0 .
\end{equation}
Let $w = z+\tfrac{\lambda}{N\varrho(z)}$ (microscopic regime) and $n = N_0+ \delta N \varrho^2$ for a small $\delta>0$.
\eqref{Bmcond} guarantees that for any $\varepsilon>0$,  \eqref{loc} holds if $N$ is sufficiently large.

Let $\a_k = (\delta_k^2\Omega)^{-1/4}$ for $k\ge m = N_T(z)$ with $T$ fixed.  
In particular, $(\delta_M^2\Omega)^{-1} = \frac{M-N_0}{\Omega} = \delta \lambda^2 $
so that if $b<1/8$ and $\delta$ is sufficiently small, $\sqrt{\a_M} \le \delta^{b}$ for $\lambda \in \mathcal{K}$.
Then, by Proposition~\ref{prop:tube} below, 
\[
\P\big[\big\{|\partial\bpsi_{M,m}| > \delta^b \big\}\cap\big\{|\partial \phi_k| \le \sqrt{\a_k} , \forall  k\in[m,M] \big\}\cap\A_m \big] \lesssim \exp\big(-c\delta^{-b} \big) . 
\]
Moreover,  on the event $\B$, we also  $|\partial\bpsi_{m}| \le   \delta^b$ and by Proposition~\ref{prop:relcont1} with $e^{-cR} = C \delta$ for some sufficiently large constant $C(\mathcal{K})$, there is a small constant $a>0$ so that 
\[
\P\big[\big\{|\partial \phi_k| \le \sqrt{\a_k} , \forall  k\in[m,n] \big\}^{\rm c}\cap\A_m \cap \B \big] \lesssim \delta^a . 
\]
We conclude that 
\[
\P\big[\big\{|\partial\bpsi_n| > 2 \delta^b  \big\}\cap\B\cap\A_m \big] \lesssim \delta^\alpha . 
\]
Then by \eqref{B2} (we can take $\varepsilon\to0$  by \eqref{Bmcond}), $\P[\B]\to 0$, and \eqref{PA} ($m=N_T(z) \to\infty$ as $N\to\infty$ for $z\in\Q$), $\P[\A_m] \to 0$,  this completes the proof. 
\end{proof}

\begin{remark} \label{rk:entrance0}
In the regime where the spectral parameter $z_\mu= \frac{\mu}{2\sqrt{N}}$ for $\mu\in\mathcal{K}$, $\mathcal{K}\Subset\R$ is a compact, there is no turning point, so the parameters $m\ge \mathfrak{K}$ and $\mathcal{L}$ are fixed. 
Then, the condition \eqref{loc} is reduced to $2\sqrt{N}|z_\mu-z_{\eta}| =|\mu-\eta| \le \varepsilon$ and 
$\B =\B(m,\varepsilon; \mu) = \big\{|\widehat{\bpsi}_m(\mu) - \widehat{\bpsi}_m(\eta)|^{8/3} \le  2|\mu-\eta|  ; |\mu-\eta| \le \varepsilon \big\}$. 
In this case, \eqref{B2} follows directly from the the smoothness of $\mu \in \mathcal{K}\mapsto \widehat{\bpsi}_m(\mu)$; see Lemma~\ref{lem:GS}, provided that $\mathcal{K} \Subset  \big(-3\sqrt{m},3\sqrt{m}\big)$. 
Then, as in Proposition~\ref{prop:entrancesmall}, we obtain that there is an $\epsilon>0$ such that with $M = \delta N$, 
\begin{equation} \label{cont0}
\lim_{\delta\to 0} \limsup_{N\to\infty} \P\big[|
\widehat{\bpsi}_{M}\big(\mu + \tfrac{\pi \lambda}{\sqrt N}\big)  - \widehat{\bpsi}_{M}(\mu)| > \delta^\epsilon \big] = 0 . 
\end{equation}
The only difference is that one lets $m\to\infty$ at the last step of the proof so that $\P[A_m] \to 1$.  
\end{remark}

\subsection{Linearization.} \label{sec:linrel}
The following basic (deterministic) bounds will be instrumental in the course of the proof.

\begin{lemma} \label{lem:contest}
Let $z,w\in(-1,1)$ satisfying \eqref{loc}.
Recall definition~\ref{def:ell}, \eqref{ell}, and define  
$\Lambda_n(w,z) :=  \big(\delta_n^2(z)\Omega(w,z)\big)^{-1/2}$. 
There are numerical constants so that for any $n\ge m$, 
\begin{enumerate}
\item $|\partial \delta_n(w,z)| \le   \delta_n(z)\Lambda_n(w,z)$.
\item $|\partial \theta_n(w,z)| \le \delta_n(z)/\sqrt{\Omega(w,z)}  = \delta_n^2(z)\Lambda_n(w,z)$.
\item $|\partial_z\Delta_n(w,z)| \le \delta_n^2(z)\Lambda_n(w,z)$.
\item $\|\partial Z_n'(w,z) \|_2 \vee \|\partial(e^{\i \theta_n} Z_n')(w,z)\|_2 \lesssim \delta_n(z)\Lambda_n(w,z)$. 
\end{enumerate}
\end{lemma}

\begin{proof}
One has for $n>N_0(z)$, 
\[
|\partial_z\delta_n(z)| =  2N|z| \delta_n^3(z) , \qquad 
|\partial_z\theta_n(z)| =  \sqrt{N}\delta_n(z) , \qquad 
|\partial_z\Delta_n(z)| \le 2 N|z| \delta_n^4(z) .
\]
Then, using that $ N|z| |w-z| \le \mathfrak{L}(z)^{3/2}\Omega(w,z)^{-1/2}$ and $ N|z| \delta_n^2(z) |w-z| \le T^{-3/2}\Lambda_n(w,z) $, if $T$ is large enough, we obtain  for $|w|\le |z|$, 
\[
|\partial \delta_n(w,z)| \le \delta_n(z)\Lambda_n(w,z) , \qquad 
|\partial \theta_n(w,z)| \le \delta_n(z)/\sqrt{\Omega(w,z)} , \qquad
|\partial_z\Delta_n(w,z)| \le  \delta_n^2(z)\Lambda_n(w,z)
\]
which gives the required estimates.

Recall that $e^{\i \theta_n} Z_n' =  \frac{\i\delta_n}{\sqrt\beta}   \bigg( \sqrt{\frac{n-1}{2n}}  e^{\i\theta_{n-1}} X_n +\frac{1}{\sqrt2} Y_n \bigg) $ so that using the previous estimates
\[
|\partial\big(e^{\i \theta_n} Z_n'\big)| \lesssim \delta_n \Lambda_n \frac{|X_n|+|Y_n|}{\sqrt{2\beta}} . 
\]
Then, using that  $\| X_n\|_2, \|Y_n \|_2\le \mathfrak{S}$, we conclude that $\|\partial(e^{\i \theta_n} Z_n')\|_2 \lesssim \delta_n\Lambda_n$ and similarly  $\|\partial Z_n'\|_2 \lesssim \delta_n\Lambda_n$. 
\end{proof} 

\begin{remark} \label{rk:contest}\normalfont
On microscopic scales, if $\varrho(z)\ge \mathfrak{R}N^{-1/3}$ with $ \mathfrak{R}\ge 1$ and $|z-w| \le \tfrac{C} {N\varrho(z)}$, the same argument (using that $\delta_n(z)\varrho(z) \ge 1$) also shows that
\[
|\partial \delta_n(w,z)| \le  C \delta_n^2(z) , \qquad |\partial \theta_n(w,z)| \le C\delta_n(z) .
\]
This also implies that $\|\partial Z_k(w,z)\|_2 \lesssim C\delta_n(z)$. 
Moreover,
\[
\bigg|\frac{\delta_n^2(w)}{\delta_n^2(z)} -1\bigg|  \le 2N \delta_n(z)^2 |z-w| \le 2C\delta_n(z)
\]
so, if needed, we can replace $\delta_n(z)$ by $\delta_n(w)$ in the previous estimates up to a small multiplicative constant.
\end{remark}

Recall that phase $\{\psi_n(z)\}_{n\ge m}$ is defined by the recursion from Lemma~\ref{lem:rec} and we can linearize the recursion  on the event \eqref{eventA}.
However, one cannot rely directly on Lemma~\ref{lem:linear} to study the relative phase because the errors do not take into account the  improvements due to the condition \eqref{loc}. 
So, we formulate another linearization lemma depending on a stopping time. 

Let $\{\a_n\}_{n\ge m}$ be a non-decreasing deterministic sequence such that $\Lambda_n \le \a_n^2 \le 1$ and define the stopping time
\begin{equation} \label{tube1}
\alpha_n \coloneqq 2\partial \phi_n, \qquad\qquad
\tau_1 \coloneqq \min\big\{ n \ge m :  
|\partial \phi_n| > \a_n \big\} . 
\end{equation}

\begin{lemma}[Linearization] \label{lem:lin2}
On the event $\A_m$ with $m = N_T(z)$, it holds for any $n \ge m$, 
\[
\partial\psi_{n,m} =  {\textstyle\sum_{k=m+1}^n} \big(\alpha_{k-1}  \overline{\Gamma_k} e^{-2\i \phi_{k-1}}+   \operatorname{EM}_{k} +\operatorname{EL}_{k} \big)
\] 
where $\{\Gamma_k(z)\}_{k\ge m}$  is an adapted process $($defined in \eqref{Gamma}$)$, $\operatorname{EM}_{k}$ are martingale increments, and the errors satisfy
\begin{equation} \label{errtube}
\|\1\{k\le \tau_1\}\operatorname{EL}_{k}(w,z)\|_1 \lesssim \delta_k^2 \a_{k}^2,
\qquad\qquad 
\|\1\{k\le \tau_1\}\operatorname{EM}_{k}(w,z)\|_2^2  \lesssim   \delta_k^2  \a_{k}^4 .
\end{equation}
\end{lemma}

\begin{proof}
According to \eqref{ell_rec}, define the ratio
\begin{equation} \label{Upsi}
\Upsilon_n \coloneqq \tfrac{\tilde{\xi}_{n,m}}{\tilde{\xi}_{n,m-1}}-1 = -\Delta_n+  Z_n' + \big( \Delta_n + \overline{Z_n'}e^{-2\i\theta_n}  \big) e^{-2\i \phi_{n-1}} .
\end{equation}
On the event $\A_m$, by \eqref{trunc1}, the complex phase satisfy 
\begin{equation} \label{lin1}
\tilde{\psi}_{n,m} = \log \tilde{\xi}_{n,m} = {\textstyle\sum_{k=m+1}^n \log(1+\Upsilon_k) }
\end{equation}
and we can linearizing $\log(1+\cdot)$;
\[
\partial\log(1+\Upsilon_n) = \frac{\partial\Upsilon_n}{1+\Upsilon_n}  + \O( |\partial\Upsilon_n|^2)
\]
where 
\begin{equation*} 
\partial\Upsilon_n = \big(\Delta_n + \overline{Z_n'}e^{-2\i\theta_n}  \big) e^{-2\i \phi_{n-1}} \big(e^{-\i\alpha_{n-1}}-1\big) + \underbrace{\partial Z_n' +  \partial(\overline{Z_n'}e^{-2\i\theta_n})e^{-2\i \phi_{n-1}}\big(e^{-\i\alpha_{n-1}}-1\big)}_{\operatorname{EM}^1_{n} }  -\partial\Delta_n .
\end{equation*}
We can also linearize $(e^{-\i\,\cdot}-1)$ and rewrite
\[
\partial\Upsilon_n =- \i\alpha_{n-1}\big(\Delta_n + \overline{Z_n'}e^{-2\i\theta_n}  \big) e^{-2\i \phi_{n-1}}
+ \underbrace{\overline{Z_n'}\big(e^{-\i\alpha_{n-1}}-1+\i\alpha_{n-1}\big)e^{-2\i\theta_n-2\i \phi_{n-1}}+   \operatorname{EM}^1_{n} }_{  \operatorname{EM}^2_{n} }
+ \underbrace{\Delta_n\big(e^{-\i\alpha_{n-1}}-1+\i\alpha_{n-1}\big)e^{-2\i \phi_{n-1}}-\partial\Delta_n}_{\operatorname{EL}^2_{n}}
\]
where $\operatorname{EM}_{n}^1, \operatorname{EM}_{n}^2$ are both martingale increments.
Moreover, using Lemma~\ref{lem:contest} and the conditions $\Lambda_n \le \a_n^2 \le 1$, we have 
\[
|\1\{n\le \tau_1\}\operatorname{EL}^2_{n}| \lesssim \delta_n^2\a_n^2, 
\qquad\qquad 
\|\1\{n\le \tau_1\}\operatorname{EM}^j_{n}\|_2^2  \lesssim   \delta_n^2 \a_n^4.
\]
Let $\operatorname{EM}_{n}= \operatorname{EM}_{n}^1+ \operatorname{EM}_{n}^2$.
On $\A_m$, we further expand
\begin{equation} \label{Gamma}
\frac{\partial\Upsilon_n}{1+\Upsilon_n} = \alpha_{n-1} \overline{\Gamma_n} e^{-2\i \phi_{n-1}}+ \operatorname{EM}_{n} +\O(\operatorname{EM}_{n} |\Upsilon_n|)+\O(|\operatorname{EL}^2_{n}|), \qquad  
\Gamma_n \coloneqq \i\frac{\Delta_n + Z_n' e^{2\i\theta_n}}{1+\overline{\Upsilon_n}}. 
\end{equation}
The linearization errors are controlled using that 
\[
\|\Upsilon_n\|_2^2 \lesssim \delta_n^2 , \qquad\qquad
\|\1\{n\le \tau_1\}\partial\Upsilon_{n}\|_{2}^2 \lesssim 
\delta_n^2\a_n^2  
\]
so that
\[
\operatorname{EL}^1_{n} = \O(\operatorname{EM}^1_{n} |\Upsilon_n|) +\O(|\partial\Upsilon_n|^2) , \qquad\qquad
\|\1\{n\le \tau_1\}\operatorname{EL}^j_{n}\|_1\lesssim \delta_n^2\a_n^2. 
\]
Let $\operatorname{EL}_{n}= \operatorname{EL}_{n}^1+ \operatorname{EL}_{n}^2$.
Going back to \eqref{lin1}, we conclude that
\[
\partial\psi_{n,m} =  {\textstyle\sum_{k=m+1}^n} \big( \alpha_{k-1} \overline{\Gamma_k}e^{-2\i \phi_{k-1}}+ \operatorname{EM}_{k} +\operatorname{EL}_{k} \big) 
\]
with the required estimates.
\end{proof}

We record a direct consequence of Lemma~\ref{lem:lin2}.

\begin{proposition} \label{prop:tube}
Under the above assumptions, choosing $\a_n^2= \Lambda_n= \big(\delta_n^2(z)\Omega(w,z)\big)^{-1/2}$, one has for any $n\ge m$,
\[
\P\big[\big\{|\partial\psi_{n,m}(w,z)| > \Lambda_n^{1/4} \big\}\cap\{\tau_1\le n\}\cap\A_m \big] \lesssim \exp(-c\Lambda_n^{-1/4}) . 
\]

\end{proposition}

\begin{proof}
From the previous proof, $\Gamma_n =\i \frac{\Delta_n + Z_n' e^{2\i\theta_n}}{1+\overline{\Upsilon_n}}$ and $\|\Upsilon_n\|_2 \lesssim \delta_n$ (on the event $\A_m$, we also have  $|\Gamma_n| , |\Upsilon_n|\le 1/2$ for all $n\ge m$).
Then, by linearize, we can bound 
\[
\bigg\|\Gamma_n +  \frac{\delta_n Z_n}{\beta^{1/2}}e^{2\i\theta_n}\bigg\|_1 \lesssim \delta_n^2. 
\]
Thus, using the the approximation from Lemma~\ref{lem:lin2}, on the event $\A_m$, it holds for any $n \ge m$,
\[
\partial\psi_{n,m} =  {\textstyle\sum_{k=m+1}^n} \big( -\beta^{-1/2} \underbrace{\alpha_{k-1} \delta_k \overline{Z_k}  e^{-2\i\theta_k-2\i \phi_{k-1}}}_{\operatorname{EM}'_{k}}+  \operatorname{EL}_{k}'+ \operatorname{EM}_{k} +\operatorname{EL}_{k} \big)
\]
where  $\operatorname{EM}_{k}'$ are martingale increments and 
\[
\|\1\{k\le \tau_1\}\operatorname{EL}_{k}'\|_1 \lesssim \delta_k^2 \a_{k},
\qquad\qquad 
\|\1\{k\le \tau_1\}\operatorname{EM}_{k}'\|_2^2  \lesssim   \delta_k^2  \a_{k}^2 .
\]
Then, since $\a_n = \delta_n^{-1}\Omega^{-1/2}$ is increasing, 
\[
\big\|{\textstyle\sum_{k=m+1}^{n\wedge\tau_1}} \operatorname{EM}_{k}'\big\|_{2}^2 \lesssim 
{\textstyle\sum_{k=m+1}^n} \delta_k^2 \a_k^2 \lesssim  \a_n^2 , 
\qquad\qquad 
\big\|{\textstyle\sum_{k=m+1}^{n\wedge\tau_1}} \operatorname{EL}_{k}'\big\|_1 \lesssim 
{\textstyle\sum_{k=m+1}^n} \delta_k^2 \a_k \lesssim  \a_n 
\]
By Lemma~\ref{lem:lin2}, we have a similar control for the contributions of $ \operatorname{EM}_{k}$ and $\operatorname{EL}_{k} $ (in fact the estimates are better since $\a_n\le 1$).
Thus, we conclude the tail-bound: 
\[
\P[\{|\partial\psi_{n,m}| > \a_n^{1/2} \}\cap\{\tau_1\le n\}\cap\A_m ] \lesssim \exp(-c\a_n^{-1/2}) . \qedhere
\]
\end{proof}

\subsection{Representation of $\alpha_n$.} \label{sec:relrep}
Recall that the phase $\phi_n =\phi_m+ \vartheta_{n,m}+ \chi_{n,m}$ with $\chi_{n,m}= \Im\psi_{n,m} $ and $ \vartheta_{n,m} = \sum_{k=m+1}^n\theta_k $.
By Lemma~\ref{lem:contest}, the deterministic phase satisfies
\begin{equation} \label{detgrowth}
|\partial \vartheta_n| \le {\textstyle \sum_{k=m+1}^n} |\partial\theta_k| \le  {\textstyle \sum_{k=m+1}^n}\delta_k/\sqrt{\Omega} \lesssim (\delta_n^{2}\Omega)^{-1/2} .
\end{equation}
Therefore, we must choose a sequence $\{\a_n\}_{n\ge m}$ such that  $|\partial \vartheta_n| \ll \a_n$ for all $n \le  \mathfrak{M} \coloneqq N_0 +e^{-4S}\Omega$ for some large $S>0$.
It will be convenient to choose\footnote{
At the endpoint, $(\delta_{\mathfrak{M}}^{2}\Omega)^{-1}= (\mathfrak{M}-N_0)/\Omega = e^{-4S} \ll 1$.
Thus, $\a_{\mathfrak{M}} \ll 1$  and we indeed have $|\partial \vartheta_n| \ll \a_n$ for all $n \le  \mathfrak{M}$. 
Moreover, on the event $\B$ (definition \eqref{B1}), at the entrance point, we also have  $|\alpha_m| \ll \a_m$.
The power $\frac14$ is arbitrary, any power $< \frac12$ would work.}  $\a_n \coloneqq  (\delta_n^2\Omega)^{-1/4} $ so that $\a_n^2 = \Lambda_n$  as required for using Lemma~\ref{lem:lin2}. 

\medskip

Taking imaginary part in Lemma~\ref{lem:lin2}, we obtain an autonomous equation for $\{\alpha_n\}_{n\ge m}$. It holds for $n \ge m$,
\[
\alpha_n = \alpha_m+  2{\textstyle\sum_{k=m+1}^n} \big(\alpha_{k-1} \Im\big(\Gamma_k  e^{2\i \phi_{k-1}}\big)+  \partial\theta_k+ \operatorname{EM}_{k} +\operatorname{EL}_{k} \big) . 
\] 
This equation has an explicit solution.
Define $P_m =1$ and for $n>m$, 
\begin{equation} \label{def:P}
P_n(z) \coloneqq {\textstyle \prod_{k=m+1}^n\big(1+2\Im\big(\Gamma_k(z)e^{2\i \phi_{k-1}(z)}\big)\big)} . 
\end{equation}
Then we can represent for $n\ge m$, 
\vspace{-.5cm}
\begin{equation} \label{contrec}
\alpha_n =P_n\alpha_m  +2\sum_{k=m+1}^n \frac{P_n}{P_k} 
\big( \partial\theta_k+\Im\big(\operatorname{EM}_{k} +\operatorname{EL}_{k} \big)\big).
\end{equation}
To estimate the growth of $\{\alpha_n\}_{n\ge m}$, we will rely on certain  bounds for $\{P_n/P_k\}_{n\ge k\ge m}$. To Formulate the result, we introduce the \emph{dyadic blocks} 
$n_j = N_{2^j} =N_0 +\mathfrak{L} 2^j$ for $j\ge \kappa$. 

\medskip

Let $J\in\N_{\ge \kappa}$ and $R\ge 1$.
For some constants $C_\beta, c_\beta >0$ and $0<\eta<c_\beta$, we introduce the stopping time
\begin{equation} \label{tube3}
\varsigma_J \coloneqq  \min\big\{ n\ge m : (P_n/P_{n_j})^{\pm 1}   \ge e^{C_\beta R} 2^{\eta(J-j)}
\text{ or }
\max_{i\le j}\big(2^{c_\beta(j-i)} P_{n_j}/P_{n_i} \big)   \ge e^{C_\beta R} 2^{\eta(J-j)} ; n \in [n_j,n_{j+1}) , j\le J \big\} .
\end{equation}

\begin{proposition} \label{prop:prop}
Fix $z\in(-1,1)$ and let $J\in\N_{\ge \kappa}$.
There exists a constant $c=c(\beta)$ such that for any $R\ge 1$, 
\[
\P\big[\{\varsigma_J<n_{J+1}\} \cap \A(2^\kappa,R;z) \big]  
\lesssim  \exp(-cR)
\]
The implied constant depends only on  $(\beta,\eta)$. 
\end{proposition}

Proposition~\ref{prop:prop} will be proved in the next subsection, and for now, we turn to the proof of Proposition~\ref{prop:relcont1}.

\subsection{Proof of Proposition~\ref{prop:relcont1}} \label{sec:st}
Let $S\ge 1$ and let $J\in\N$ be the last integer so that $\mathfrak{L}2^J \le  e^{-4S}\Omega$. 
We introduce the stopping time 
\[
\tau \coloneqq  \tau_1 \wedge \varsigma_J \wedge n_{J+1} .
\]

On the event  $\B$ (definition \eqref{B1}), one has $|\alpha_m| \le \a_m 2^{-\kappa/4}\Omega^{-1/8}$ using that $T=2^\kappa$. 
Then, for the first term in \eqref{contrec}, which arise from the initial condition, we can bound for $n \in [n_j,n_{j+1})$ and $n<\varsigma_J $, 
\[
\frac{|P_n\alpha_m |}{\a_n} \le \frac{|\alpha_m|}{\a_m} \frac{P_n}{P_{n_j}}\frac{P_{n_j}}{P_m}  \le 2^{2CR-\kappa/4} 2^{2\eta J} \Omega^{-1/8}\le \tfrac14 2^{2CR-8\eta S} \Omega^{\eta-1/8}
\]
if $\kappa$ is sufficiently large. Thus, if $\eta\le 1/8$ and $S=LR$ for some large enough constant $L$ (depending on $\eta$ and $C$), then $|P_n\alpha_m |\le \a_n/4$ for all $n<\tau$. 

For the driving term in \eqref{contrec}, using that  $(\a_n \Omega^{1/2})^{-1} = \delta_n\a_n$, we have similarly to \eqref{detgrowth},
\[
\sum_{k=m+1}^n \frac{P_n}{P_k}\frac{|\partial\theta_k|}{\a_n} \le  \frac{P_n}{P_{n_j}} \sum_{\kappa \le i\le j} \frac{P_{n_j}}{P_{n_i}}
\sum_{k=n_i+1}^{n_{i+1}\wedge n} \frac{P_{n_i}}{P_k} \frac{\delta_k}{\a_n\Omega^{1/2}}
\le 2^{3CR+3\eta(J-j)}  \delta_n\a_n\sum_{k=m+1}^{n}   \delta_k   \lesssim 2^{3CR+3\eta(J-j)} \a_{n_{j+1}} .
\]
By construction $ \a_{n_{j+1}}/\a_{n_J+1} =2^{-(J-j)/4}$ and $\a_{n_J+1}\lesssim e^{-S}$ so that if $\eta\le 1/12$ and $S=LR$ for some large constant $L$, we can alslo bound
\[
\sum_{k=m+1}^n \frac{P_n}{P_k}\frac{|\partial\theta_k|}{\a_n} \lesssim 2^{3CR-S} \le \tfrac14.
\]

Now, let 
\[
\hat{\alpha}_n \coloneqq  \a_n^{-1}  \bigg| \sum_{k=m+1}^n \frac{P_n}{P_k} 
\Im\big(\operatorname{EM}_{k} +\operatorname{EL}_{k} \big)\bigg|
\]
and recall that our goal is to estimate $\P\big[\{\tau_1 < n_{J+1}\} \cap \A\cap \B\big] $.

We deduce from \eqref{contrec} and the previous estimates that 
\[
\{\tau_1 < n_{J+1}\} \cap \A\cap \B \cap \{\varsigma_J\ge n_{J+1}\} \subset \{\hat{\alpha}_n > \tfrac14 \text{ for a } m<n<n_{J+1}\} \cap \{\varsigma_J\ge n_{J+1}\}  \subset \{ \hat{\alpha}_\tau>\tfrac14\}.
\]
Hence,
\begin{equation} \label{alphaest}
\P\big[\{\tau_1 < n_{J+1}\} \cap \A\cap \B\big] \le \P\big[\hat{\alpha}_\tau > \tfrac14 \big] + \P\big[\{\varsigma_J<n_{J+1}\} \cap \A\big].
\end{equation}

\medskip

It remains to estimate the first term on the RHS of \eqref{alphaest}.
Using the dyadic blocks, we rewrite 
\begin{equation} \label{alpharep2}
\begin{aligned}
\hat{\alpha}_n&= \a_n^{-1}   \bigg|\sum_{\kappa\le i} \frac{P_{n}}{P_{n_i}} \sum_{k=n_i+1}^{n_{i+1}\wedge n}   \frac{P_{n_i}}{P_{k}}  \Im\big(\operatorname{EM}_{k}+\operatorname{EL}_{k} \big) \bigg| \\
&\le \sum_{\kappa\le i}  \a_{n_i}^{-1}\frac{P_{n}}{P_{n_i}} \bigg| \sum_{k=n_i+1}^{n_{i+1}\wedge n}  \Im\big(\operatorname{EM}_{k}' +\operatorname{EL}_{k}' \big)\bigg|
\end{aligned}
\end{equation}
where 
\[
\operatorname{EM}_{k}'= \frac{P_{n_i}}{P_{k-1}} \operatorname{EM}_{k} ,\qquad\qquad
\operatorname{EL}_{k}' = \frac{P_{n_i}}{P_{k}}\bigg( \frac{P_{k-1}-P_k}{P_{k-1}}\operatorname{EM}_{k}
+  \operatorname{EL}_{k} \bigg). 
\]
In particular $\operatorname{EM}_{k}'$ are still martingale increments and  by \eqref{def:P}, 
\[
\frac{P_{k-1}-P_k}{P_{k-1}} = -2\Re(\Gamma_ke^{2\i \phi_{k-1}}) ,
\qquad\qquad 
\bigg\| \frac{P_{k-1}-P_k}{P_{k-1}}\bigg\|_2  \lesssim \delta_k. 
\]
Then, using the conditions \eqref{tube3} and \eqref{errtube}, for any $k\in[n_i,n_{i+1}]$, 
\[\begin{aligned}
&\|\1\{k\le \tau\}\operatorname{EM}_{k}'\|_2 \le  e^{CR} 2^{\eta(J-i)}  \|\1\{k\le \tau_1\}\operatorname{EM}_{k}\|_2 \lesssim  e^{CR} 2^{\eta(J-i)}   \delta_k  \a_k^2 \\
&\|\1\{k\le \tau\}\operatorname{EL}_{k}'\|_1 \lesssim 
e^{CR} 2^{\eta(J-i)}\big( \delta_k\|\1\{k\le \tau_1\}\operatorname{EM}_{k}\|_2
+\|\1\{k\le \tau_1\}\operatorname{EL}_{k}\|_1\big)
\lesssim e^{CR} 2^{\eta(J-i)}  \delta_k^2  \a_{k}^2 .
\end{aligned}\]
Using these estimates, by Proposition~\ref{lem:conc2},
\[
\bigg\| \max_{n\le n_{i+1}}\bigg|
\sum_{k=n_{i}+1}^{n \wedge \tau } \operatorname{EM}_{k}'  \bigg|\bigg\|_1
\lesssim  e^{CR} 2^{\eta(J-i)} \bigg( \sum_{k=n_{i}+1}^{n_{i+1}}   \delta_k^2  \a_{k}^4\bigg)^{\frac12}
\lesssim e^{CR}  2^{\eta(J-i)}\a_{n_i}^2  
\]
and 
\[
\bigg\| \max_{n\le n_{i+1}}\bigg|
\sum_{k=n_{i}+1}^{n \wedge \tau } \operatorname{EL}_{k}'  \bigg|\bigg\|_1
\lesssim  e^{CR} 2^{\eta(J-i)} \bigg( \sum_{k=n_{i}+1}^{n_{i+1}}   \delta_k^2  \a_{k}^2\bigg)
\lesssim e^{CR} 2^{\eta(J-i)}\a_{n_i}^2  . 
\]

Going back to the representation \eqref{alpharep2}, this implies that  
\[
\hat{\alpha}_{n\wedge\tau} \le\sum_{\kappa\le i} \1\{n_i\le n\} \frac{P_{n}}{P_{n_i}}  \a_{n_i}^{-1}  Q_i , \qquad\qquad Q_{i} =  \max_{\ell \le n_{i+1}}\bigg| \sum_{k=n_{i}+1}^{\ell \wedge \tau }  \Im\big(\operatorname{EM}_{k}' +\operatorname{EL}_{k}' \big)\bigg| . 
\]
Then, evaluating this sum at $\tau$, using the previous estimates, we obtain
\[
\hat{\alpha}_{\tau} 
\le \sum_{\kappa\le j \le J} \1\{n_j <\tau\le n_{j+1}\}  \bigg(\max_{n_j <\ell \le n_{j+1}}\frac{P_\ell}{P_{n_j}} \bigg) \sum_{i\le j} \frac{P_{n_j}}{P_{n_i}} \a_{n_i}^{-1} Q_i  , \qquad\qquad 
\|Q_i\|_1\lesssim  e^{CR} 2^{\eta(J-i)}\a_{n_i}^2 .
\]
Then, for $i\le j$ and $n_j <\tau$, 
\[
\bigg(\max_{n_j <\ell \le n_{j+1}}\frac{P_\ell}{P_{n_j}} \bigg) \frac{P_{n_j}}{P_{n_i}}
\le   e^{2CR} 2^{2\eta(J-j)} 2^{c(i-j)} 
\]
so that
\[
\hat{\alpha}_{\tau} \le e^{2CR}\sum_{\kappa\le j \le J}   2^{\eta(J-j)}  
\sum_{\kappa \le i\le j}  2^{c(i-j)} \a_{n_i}^{-1}  Q_i . 
\]

Consequently, if $\eta\le c$, 
\[
\sum_{\kappa \le i\le j}  2^{c(i-j)} \a_{n_i}^{-1}  \|Q_i\|_1 
\lesssim e^{CR} 2^{\eta J}  \sum_{\kappa \le i\le j}  2^{c(i-j)-\eta i} \a_{n_i}
\lesssim e^{CR}2^{\eta(J-j)}\a_{n_j} 
\]
and summing these bound (with $\eta<1/8$ and $\a_{n_J} \le e^{-S}$), we obtain
\[
\|\alpha_\tau\|_1 \lesssim  e^{3CR}\sum_{\kappa\le j \le J}  2^{2\eta(J-j)} \a_{n_j} 
\lesssim e^{3CR-S} . 
\]
Again, choosing $S=LR$ for some large constant $L$, this quantity is $\O(e^{-R})$ and we conclude that there is constant $c>0$ so that
\[
\P\big[|\hat{\alpha}_\tau|>\tfrac14 \big] \le 2\exp(-c  e^{R}) .
\]
Going back to \eqref{alphaest}, by Proposition~\ref{prop:prop}, this tail bound is negligible and we conclude that
\[
\P\big[\{\tau_1 < n_{J+1}\} \cap \A\cap \B\big] \lesssim  \exp(-cR).
\]
This completes the proof. \qed

\subsection{Proof of Proposition~\ref{prop:prop}.} \label{sec:gen}
The argument is divided in several steps, we first relate the ratios $\{{P_n}/{P_k}\}_{n\ge k\ge m}$ to an \emph{exponential martingale}.

\begin{lemma} \label{lem:lin4}
Recall the martingale $\{\mathbf{W}_{n,m}\}_{n\ge m}$ (Definition~\ref{def:GW}) and the event $\A_{\chi} = \A_{\chi}(T,R; z) $ (Lemma~\ref{lem:varphase}). \\
For all $m\le k \le n$, 
\[
\frac{P_n}{P_k}=\exp\biggl( \tfrac{2}{\sqrt\beta} \Im \mathbf{W}_{n,k} - \tfrac2{\beta} [ \Im \mathbf{W}_{n,k}] + \mathcal{E}_{n,k}\biggr)
\]
and there exist constants $C,c>0$ (depending only on $\beta$) so that for any $ R\ge 1$,
\[
\P\left[\big\{\max_{m\le k \le n } |\mathcal{E}_{n,k}| > CR  \big\}  \cap \A_\chi \right]
\lesssim \exp\left(- c R \sqrt{\mathfrak{L}T}\right).
\]
\end{lemma}

\begin{proof}
Recall that $\Gamma_n =\i \frac{\Delta_n + Z_n' e^{2\i\theta_n}}{1+\overline{\Upsilon_n}}$ according to \eqref{Gamma}, \eqref{Upsi} and on the event $\A \subset \A_\chi$,  we have that  $|\Gamma_n| , |\Upsilon_n|\le 1/2$ for all $n\ge m$.
Then, using Lemma~\ref{lem:Z}, we can linearize 
\[
\Gamma_n =  \frac{\delta_n Z_n}{\beta^{1/2}}e^{2\i\theta_n}(\overline{\Upsilon_n}-1)+ \i\delta_n^2/4 +  \O(R\delta_n^{3-\epsilon})
\]
and
\[
\log\big(1 -2 \Im(\Gamma_ne^{2\i \phi_{n-1}})\big) =  -2\Im(\Gamma_ne^{2\i \phi_{n-1}}) 
- \frac2{\beta}\Im\big(\delta_n Z_n e^{2\i(\theta_n+\phi_{n-1})}\big)^2 +  \O(R\delta_n^{3-\epsilon}). 
\]
Now, since $\overline{\Upsilon_n} =  \frac{\i \delta_n}{\beta^{1/2}}(-\overline{Z_n}+ Z_n e^{2\i(\theta_n+ \phi_{n-1})}\big) +\O(\delta_n^2)$ with a deterministic error, it holds on the event $\A$,
\[\begin{aligned}
\log\big(1-2\Im(\Gamma_ne^{2\i \phi_{n-1}})\big)  &= \tfrac{2}{\sqrt\beta} \Im\big(\delta_n Z_ne^{2\i(\theta_n+ \phi_{n-1})}\big) - \tfrac2{\beta} \Im\big(\delta_n Z_n e^{2\i(\theta_n+\phi_{n-1})}\big)^2 \\
&\quad + \tfrac2\beta \Re\underbrace{\delta_n^2\big( |Z_n|^2 e^{2\i(\theta_n+ \phi_{n-1})}-Z_n^2e^{4\i(\theta_n+ \phi_{n-1})}-\tfrac{\beta}4e^{2\i \phi_{n-1}}\big)}_{\operatorname{EO}_{n}} +  \O(R\delta_n^{3-\epsilon}) \\
& = \tfrac{2}{\sqrt\beta} \Im \mathbf{W}_{n,n-1} - \tfrac{2}{\beta}\E\big[(\Im \mathbf{W}_{n,n-1})^2|\F_{n-1}\big] \\
&\quad -\tfrac2{\beta} \underbrace{\big\{\Im\big(\delta_n Z_n e^{2\i(\theta_n+\phi_{n-1})}\big)^2 
-\E\big[ \Im\big(\delta_n Z_n e^{2\i(\theta_n+\phi_{n-1})}\big)^2 |\F_{n-1}\big]\big\}}_{\operatorname{EM}^1_{n}} +\tfrac2\beta \Re(\operatorname{EO}_{n})  +  \O(R\delta_n^{3-\epsilon}). 
\end{aligned}\]
The terms $\operatorname{EO}_{n}$ can be handled by making a martingale decomposition and using Proposition~\ref{prop:osc1} (and also Proposition~\ref{prop:osc2} with $x=z$); we decompose
\[
\operatorname{EO}_{n} = \operatorname{EM}^2_{n}
+\, q^1_n e^{2\i\phi_{n-1}}+ q^2_ne^{4\i\phi_{n-1}} 
\]
where $ q^1_n = \delta_n^2\big(\E |Z_n|^2e^{2\i\theta_n} -\beta/4)$, $ q^2_n = \delta_n^2\E Z_n^2 e^{4\i\theta_n}$  and the martingale increments $\operatorname{EM}^j_{n}$ satisfy 
$\|\operatorname{EM}^j_{n}\|_1 \lesssim \delta_n^2$ for $j\in\{1,2\}$.
We check that the sequence $\{q^1_n\}_{n\ge m}$ satisfies the assumptions of Proposition~\ref{prop:osc1} and $\{q^2_n\}_{n\ge m}$ that of Proposition~\ref{prop:osc2} (the argument is the same as in the proof of Proposition~\ref{prop:Q}; $\E |Z_n|^2=1$ and $\E Z_n^2 = (\cos\theta_n)e^{-\i\theta_n}$).
Then, it holds on the  event $\A_\chi(T,R; z)$, 
\[
\max_{n> m}\bigg| \sum_{m< k \le n} \big(q^1_n e^{2\i\phi_{n-1}}+ q^2_ne^{4\i\phi_{n-1}}\big) \bigg| 
\lesssim R/ T^{1/3} . 
\]

Let $\operatorname{EM}_{n}= \operatorname{EM}_{n}^2- \operatorname{EM}_{n}^1$ and 
$\operatorname{M}_{n,k}= {\textstyle\sum_{\ell=k+1}^n} \operatorname{EM}_{k}$. 
Then, using that $\sum_{n>k}\|\operatorname{EM}_{n}\|_1^2 \lesssim \delta_k^2$, by Proposition~\ref{lem:conc1},  for any $\lambda>0$
\[
\P\left[\max_{n>k} |\operatorname{M}_{n,k}|  \ge \lambda \right] \le 2 \exp\left(- c \lambda \delta_k^{-1}\right). 
\]
Consequently, by a union bound and using that  $ \delta_m^{-2} = \mathfrak{L}  T $,  
\begin{equation} \label{gammamarterr}
\P\left[\max_{n>k \ge m} |\operatorname{M}_{n,k}|  \ge R \right] \lesssim \exp\left(- c R \sqrt{\mathfrak{L}T}\right). 
\end{equation}

Hence, combining these estimates, we conclude that on the event $\A_\chi$, uniformly for any $n>m$,
\[
\frac{P_n}{P_k}
= \exp\bigg( \sum_{ k< \ell \le n} \log\big(1+ \Re(\Gamma_\ell e^{2\i \phi_{\ell-1}})\big)  \bigg)
= \exp\bigg( \tfrac{2}{\sqrt\beta} \Im \mathbf{W}_{n,k} - \tfrac2{\beta} [ \Im \mathbf{W}_{n,k}]  - \tfrac2{\beta} \operatorname{M}_{n,k}  +  \O(R) \bigg)  
\]
where the error is deterministic and  $\{\operatorname{M}_{n,k}\}_{n\ge m\ge k}$ is controlled by \eqref{gammamarterr}. 
\end{proof}

Dropping the errors for now, define the \emph{exponential martingale}, 
\[
\mathcal{P}_{n,k} \coloneqq \exp\big( \tfrac{2}{\sqrt\beta} \Im \mathbf{W}_{n,k} - \tfrac2{\beta} [ \Im \mathbf{W}_{n,k}]\big)
\]
The next step is to control the variation of $\mathcal{P}_{n,k}$ over the dyadic blocks
$n_j = N_0 +\mathfrak{L} 2^j$ for $j\in \N$. 

\begin{lemma} \label{lem:Pdyadic}
There exists constants $c_i= c_i(\beta)>0$ and $C=C(\beta)$ so that for all $j\ge \kappa$, 
\begin{equation}\label{Wtb}
\P\Big[ \max_{n_j < n \le n_{j+1}}\big(\mathcal{P}_{n,n_j}^{\pm1}\big)\ge e^{R} \Big] 
\lesssim e^{-c_1R^2} 
\end{equation}
and 
\begin{equation}\label{Wdecay}
\P\Big[ \big\{ \max_{\kappa\le i\le j} \big( 2^{c_2(j-i)} \mathcal{P}_{n_j,n_i} \big)\ge e^{CR} \big\}  \cap \A_{\chi}\Big]
\lesssim \exp(- c_1R) .
\end{equation}
\end{lemma}

\begin{proof}
Using  \eqref{def:s}, the martingale $\{ \Im \mathbf{W}_{n,k}\}_{n\ge k}$ 
satisfy $\|\mathbf{W}_{k+1,k} \|_{2}\le \delta_k$ for any $k\ge m$ and its quadratic variation is given by
\[
[\Im \mathbf{W}_{n,k}] = \sum_{k<j\le n} \delta_j^2 \frac{1- \Re(s_j e^{4\i(\theta_j +\phi_{j-1})})}{2} . 
\]
Then,  by Proposition~\ref{lem:conc2}, it holds (uniformly) for $j\ge \kappa$,
\[
\Big\|\max_{n_j < n \le n_{j+1} } |\Im \mathbf{W}_{n,n_j}|\Big\|_2^2
\le \sum_{n_j<k\le n_{j+1}} \delta_j^2  \lesssim 1
\]
by the dyadic construction. 
Moreover, deterministically $ [ \Im \mathbf{W}_{n_{j+1},n_j}] \lesssim 1$ uniformly for all $j\ge \kappa$. 
This yields the tail bound \eqref{Wtb}. 

To prove the second estimate, on the event $\A_{\chi}$ (Lemma~\ref{prop:osc2} with $x=z$), we control the oscillatory part of the quadratic variation; for any $R\ge 1$,
\[
\max_{n> m}\bigg| \sum_{m< j \le n}  \delta_j^2s_j e^{4\i(\theta_j +\phi_{j-1})}  \bigg| 
\lesssim R /T^{1/3} . 
\]
This argument has already been used several times. This implies that on $\A_{\chi}$, for any integer $j>i\ge \kappa$, 
\[
[\Im \mathbf{W}_{n_j,n_i}] \ge \tfrac12  \log(2^{j-i})  +\O(R) .
\]
From this estimate, we expect that $\mathcal{P}_{n_j,n_i} $ decays like $2^{(i-j)/\beta}$.
Then, by Proposition~\ref{lem:conc2} again, it holds for any $\epsilon>0$, 
\[\begin{aligned}
\P\big[ \big\{ \mathcal{P}_{n_j,n_i} \ge 2^{(\beta^{-1}-\epsilon)(i-j)}  e^{C R} \big\} \cap \A_{\chi}\big]
&\le \P\big[ \big\{ |\Im \mathbf{W}_{n_j,n_i} | \ge \sqrt{\beta}(R +  \epsilon \log 2^{j-i})  \big\}\big]\\
&\le 2 \exp\big( -c \tfrac{(R+\epsilon \log 2^{j-i})^2}{\log 2^{j-i}} \big)
= 2^{1+c \epsilon (i-j)} e^{-2c \epsilon R}  . 
\end{aligned}\]
Then, by a union bound, summing these estimates (for $i\le j$),  this yields for $R\ge 1$, 
\[
\P\Big[ \big\{ \max_{\kappa \le i\le j} \big( 2^{(\beta^{-1}-\epsilon)(j-i)} \mathcal{P}_{n_j,n_i} \big)\ge e^{CR} \big\}  \cap \A_{\chi}\Big]
\lesssim  \exp\big(- 2c \epsilon R \big) 
\]
where the implies constant depends only on $\epsilon>0$.
Choosing $\epsilon =1/2\beta$, this completes the proof of \eqref{Wdecay}. 
\end{proof}

We are now ready to complete the proof.

\begin{proof}[Proof of Proposition~\ref{prop:prop}]
By definition \eqref{tube3}, with $C=C_\beta$, 
\[
\{\varsigma_J \ge n_J\} =\Big\{ \max_{n_j<n\le n_{j+1}}\big((P_n/P_{n_j})^{\pm 1}\big)\vee \max_{i\le j}\big(2^{c_\beta(j-i)} P_{n_j}/P_{n_i} \big) \le e^{CR} 2^{\eta(J-j)} ; \forall j\in[\kappa,J]  \Big\} 
\]
and using the notation from Lemma~\ref{lem:lin4},
\[
\Big\{ \max_{n_j<n\le n_{j+1}}\big(\mathcal{P}_{n,n_j}^{\pm1}\big)\vee \max_{i\le j}\big(2^{c_\beta(j-i)}  \mathcal{P}_{n_j,n_i} \big) \le e^{CR/2} 2^{\eta(J-j)} ; \forall j\in[\kappa,J]  \Big\} 
\cap \big\{\max_{m\le k \le n } |\mathcal{E}_{n,k}| \le CR/2  \big\} \subset \big\{\varsigma_J \ge n_J\big\} .
\]
Then, 
\[\begin{aligned}
\P\big[\{\varsigma_J<n_J\} \cap \A_{\chi} \big]  
&\le \P\big[\big\{\exists j\in[\kappa,J]  ;  \max_{i\le j}\big(2^{c_\beta(j-i)}  \mathcal{P}_{n_j,n_i} \big) > e^{CR/2} 2^{\eta(J-j)}\big\} \cap \A_{\chi}\big] \\
&\quad+ \P\big[\exists j\in[\kappa,J]  ; \max_{n_j<n\le n_{j+1}}\big(\mathcal{P}_{n,n_j}^{\pm1}\big)> e^{CR/2} 2^{\eta(J-j)}\big]
+ \exp\big(-c R\sqrt{\mathfrak{L} T}\big).
\end{aligned}\]
These probabilities are controlled using Lemma~\ref{lem:Pdyadic} and a union bound. For instance, by \eqref{Wtb}, 
\[\begin{aligned}
\P\bigg[\max_{j\in[\kappa,J]}\Big( 2^{\eta(j-J)} \max_{n_j < n \le n_{j+1} }\big(\mathcal{P}_{n,n_j}^{\pm1}\big) \Big) \ge e^R\bigg]
& \lesssim  \sum_{j\ge 1}  \exp\big(-c_1R^2 - c_1(\log 2^{\eta j})^2 \big)  \\
&\lesssim e^{-c_1R^2}  
\end{aligned}\]
where the implied constant depends on $\beta,\eta>0$. Similarly, using \eqref{Wdecay} and adjusting constants, we conclude that there is a constant $c=c(\beta)>0$ such that for any $R\ge1$, 
\[
\P\big[\{\varsigma_J<n_J\} \cap \A_{\chi} \big]  
\lesssim   \exp(-cR)
\]
Finally, by Lemma~\ref{lem:varphase}, $\P[ \A_{\chi}^c \cap \A] \lesssim \exp\big( - c R(R \wedge \sqrt{\mathfrak{L} T}) \big) $, which is negligible.
\end{proof}

\section{Log-correlated structure} \label{sec:LCF}

Th goal of this section is to prove Proposition~\ref{prop:M} on the bracket structure or the complex martingale $\{\M_n\}$ and the corresponding claim 3 from Theorem~\ref{thm:main}.
According to Definition~\ref{def:GW}, the martingale has two parts: the $\G$ field which is a sum of independent random variables and the $\W$ field which is a true martingale (meaning that its brackets are stochastic processes). Because of the rapid growth of the phase $\{\phi_n\}$, these two fields are asymptotically uncorrelated and the $\W$ field behaves like a white noise.
The proof is structured as follows: 
\begin{itemize}[leftmargin=*]   \setlength\itemsep{0em}
\item In Section~\ref{sec:G}, we describe the correlation structure of the field $\G$, Proposition~\ref{prop:G}. Since its brackets are deterministic sums, the proof consists of some Riemann sum approximations.
\item In Section~\ref{sec:W}, we  describe the correlation structure of the field $\W$, Proposition~\ref{prop:W}. Its brackets have deterministic equivalents, with errors controlled in probability. Using the techniques introduced in Section~\ref{sec:osc}, one can also obtain tail-bounds for these errors. 
\item In Section~\ref{sec:Wosc}, we prove extra estimates on random oscillatory sums which are instrumental to obtain Proposition~\ref{prop:W}. These estimates in the \emph{merging regime} are based on the continuity properties of the phase obtained in Section~\ref{sec:cont}.
\item Finally, in Section~\ref{sec:M}, we consider the correlation structure between the $\G$ and $\W$ field and combine the previous results to deduce  Proposition~\ref{prop:M}. 
\end{itemize}

Throughout the proof, we abuse the notation from Definition~\ref{def:GW} and let 
\begin{equation}\label{GW} 
\G_n(z) \coloneqq \sum_{0< k\le n} \1\{k\notin \Gamma_T(z)\} \frac{Z_k(z)}{\sqrt{k}\sqrt{Nz^2/k-1}} , \quad
\W_{n}(z)   \coloneqq \sum_{N_0(z)<k \le n}   \1\{k\notin \Gamma_T(z)\}\frac{Z_k(z)e^{2\i(\theta_k(z)+\phi_{k-1}(z))}}{{\sqrt{k}\sqrt{Nz^2/k-1}}} ,
\end{equation}
where the $\sqrt{\cdot}$ is chosen as in \eqref{Jouk}\footnote{For $w\in[-1,1]$, $\sqrt{w^2-1}$ is imaginary and defined by continuity from the upper-half plane.}, for any $T\ge 1$, 
\begin{equation*}
\Gamma_T(z)  \coloneqq \big\{k\in[N] : |k-Nz^2| < T\mathfrak{L}(z)\big\}  
\quad \text{and the process $\{\phi_n(z) : n> N_0(z)\}$ is given by \eqref{def:phase}.} 
\end{equation*} 
$\L(z) =  \lceil Nz^2\rceil^{1/3}$  is the \emph{parabolic time scale} around the turning point, \eqref{para}.
Introducing the parameter $T\ge 1$, independently of $N$, will only affect the $\O(1)$ error terms in the merging regime (in Definition~\ref{def:GW}, $T=1$). In particular, it will be convenient to increase $T$ is necessary for some arguments by using for instance Remark~\ref{rk:cutoff}. 
Recall that  $[z]_N \coloneqq |z| \vee N^{-1/2}$ for $z\in\R$.
Throughout the proof, we also write 
\[
\G(z) = \G_N(z) \qquad \W(z) = \W_N(z) , \qquad \M(z) = \M_N(z) =\G(z) + \overline{\W(z)} 
\]
and we will also distinguish two regimes:
\begin{itemize}\setlength\itemsep{0em}
\item The \emph{global regime}  if  $|z-x| \gg N^{-2/3}[z]_N^{-1/3}$ where the brackets of $\G,\W$ have deterministic equivalents in terms of the map \eqref{Jouk}  (the bracket of $\W$  converges to 0 in probability in this regime).
\item The \emph{local regime}  if  $|z-x| \le S N^{-2/3}[z]_N^{-1/3}$, for some constant $S\ge 1$, 
where the bracket of the $\G$ field is constant and the bracket of $\M$ can be computed up to errors which are tight random variables.
\end{itemize}

\begin{remark}\normalfont \label{rk:parascale}
Observe that the following three conditions are equivalent: $|z-x| = \Theta\big(N^{-2/3}[z]_N^{-1/3}\big)$, 
$ |x-z| = \Theta\big(1/\sqrt{N\mathfrak{L}(z)}\big)$ and 
$|N_0(z)-N_0(x)| = \Theta\big(\mathfrak{L}(z)\big)$. So, the \emph{transition regime} corresponds to the case where \emph{the two  turning points are merging at the parabolic scale}.
It is difficult to obtain information of the brackets of the $\W$ field in this regime since its behavior can be related to the stochastic Airy function.
\end{remark}


\begin{remark}\normalfont \label{rk:cutoff}
The parameter $T$ acts as a cutoff around the turning point. We observe that, since $\E|Z_k(z)|^2=1$,  for any $R\ge T$,
\[
\sum_{k\in\Gamma_R(z)\setminus\Gamma_T(z)} \E\bigg| \frac{Z_k(z)}{\sqrt{k}\sqrt{Nz^2/k-1}}\bigg|^2 
= \sum_{k\in\Gamma_R(z)\setminus\Gamma_T(z)} \frac{1}{Nz^2-k} = \log\bigg(\frac{R}{T}\bigg)  +\underset{N\to\infty}{\o(1)}.
\]
Under the assumptions of Definition~\ref{def:noise}, one has a similar estimate for the $\boldsymbol{\Psi}_2$-norm.
\end{remark}

\subsection{Correlations of the $\G$ field.} \label{sec:G}
The brackets of the $\G$ field are deterministic and so equal to its covariances. 
The goal of this section is to prove the following asymptotics:  

\begin{proposition} \label{prop:G}
The $\G$ field has the following covariance; for $x,z\in\R$,\\
$\bullet[\text{Global regime}]$ If $|x|\le |z|$ and $ |x-z| \gg 1/\sqrt{N\mathfrak{L}(z)}$ or $\big(|z|-1\big)\gg N^{-2/3}$, 
\[
\big[\G(z),\G(x)\big] = - 2\log\big(1-J(z)J(x)\big)  + \underset{N\to\infty}{\o(1)}, \qquad
\big[\G(z),\overline{\G(x)}\big] = - 2 \log\big(1- J(z)\overline{J(x)}\big) +\underset{N\to\infty}{\o(1)}.
\]
$\bullet[\text{Local regime}]$ If $|z|\le 1-N^{-2/3}$ and $|x-z|\le C/\sqrt{N\mathfrak{L}(z)}$ for a constant $C\ge 1$, then 
\[
\big[\G(z),\G(x)\big]  = - 2\log(\varrho(x)) +\O(1), \qquad\qquad
\big[\G(z),\overline{\G(x)}\big]  =\log\big(\varrho(x)^2N\mathfrak{L}(x)) +\O(1).
\]
$\bullet[\text{Edge regime}]$ If $| x\pm 1| , |z\pm 1|  \le C N^{-2/3}$ for some constant $C\ge 1$, 
\[
\big[\G(z),\G(x)\big]  = \log(N^{2/3})+\O(1) \qquad\qquad
\big[\G(z),\overline{\G(x)}\big]  =\log(N^{2/3})+\O(1). 
\]
The error are deterministic and depend only on $(C,T)$.
\end{proposition}

\paragraph{Local estimates.}
We begin the proof by computing the variance of the two parts of the  $\G$ field. 

\begin{proposition}\label{prop:G1}
For $z\in\R$, 
\[\begin{aligned}
&\big[\G^1(z)\big]  = 2\log\big(\mathfrak{L}(z)\big) +\O(1) \qquad\text{if $|z|\le 1+TN^{-2/3}$},\\
&\big[\G^1(z)\big]  = -2\log\big(1-J(z)^2\big)+\O(1) \qquad\text{if $|z|\ge 1+TN^{-2/3}$},\\
&\big[\G^2(z)\big] = 2 \log_+\big(\varrho(z)\mathfrak{L}(z)\big)  + \O(1), \\
&\big[\G^2(z),\overline{\G^2(z)} \big] = \log_+\Big(\tfrac{\varrho(z)^2N}{\mathfrak{L}(z)}\Big)+\O(1)
\end{aligned}\]
where the errors depend on the parameter $T\ge 1$ and are locally uniform in $z$. 
Consequently, it holds uniformly for $z\in[-1,1]$ as $N\to\infty$, 
\[
\big[\Re \G(z)\big]  = \tfrac12\log\big(N\mathfrak{L}(z)\big) +\O(1),
\qquad\qquad
\big[\Im \G(z)\big]  = \tfrac12\log_+\big(\varrho(z)^4N\mathfrak{L}(z)\big) +\O(1).
\]
\end{proposition}

\begin{proof} $\bullet$ Let $m= N\wedge N_{-T}(z)$ and $\G^1 = \G_m^1(z)$ for $z\in\R\setminus\mathfrak{Q}$.  By definition, $\G^1$ is real-valued and 
\[
[\G^1(z)] =  \sum_{k \le m} \frac{1 +J(z\sqrt{N/k})^2}{2k(Nz^2/k-1)}   = \sum_{k \le m} \frac{1}{Nz^2-k} -  \sum_{k \le m} \frac{1 -J(z\sqrt{N/k})^2}{2k(Nz^2/k-1)} .
\]
In terms of \eqref{Jouk}, we have $1-J(w)^2 \lesssim \sqrt{w^2-1}$ for $w\in\R\setminus (-1,1)$ so that the second sum is bounded by
\[
\sum_{k \le m} \frac{1 -J(z\sqrt{N/k})^2}{2k(Nz^2/k-1)} \lesssim \sum_{k < N_0(z)} \frac{1}{2k\sqrt{Nz^2/k-1}}  =\O(1) .
\]
Indeed, this sum is convergent and it can be approximated by the Riemann integral 
\(\displaystyle 
\int_0^{z^2} \frac{dt}{\sqrt{t(z^2-t)}}<\infty. 
\)
Computing the harmonic sum, this shows that for $|z|\le 1+TN^{-2/3}$
\[
[\G^1(z)] = \log\bigg(\frac{N_0(z)}{N_0(z)-N_{-T}(z)}\bigg) +\O(1)
=  \log_+\big(\mathfrak{L}(z)^2/T\big) +\O(1)
\]
where the error is controlled independently of $T$. 
These asymptotics remains true if $z\in\mathfrak{Q}$ (neighborhood of $0$) in which case 
$\mathfrak{L}(z) =1$ and $\G^1$ is a finite sum.\\
Otherwise, if $|z|\ge 1+TN^{-2/3}$, $m=N$ and using that 
$(1-J(z)^2)^2 \sim 4(z^2-1)$ as $z\to\pm 1$,  we obtain
\[
[\G^1(z)] =  -2\log\big(1-J(z)^2\big)+\O(1) .
\]

$\bullet$ Let $z\in\R$ with $|z|\le 1-TN^{-2/3}$, $m= N_{T}(z)$ and $\G^2=\G_{N,N_T(z)}^2(z)$. 
According to Lemma~\ref{lem:Z}, using that  
\[\begin{aligned}
&1+\cos(2\theta_k(z))= 2(\cos \theta_k(z)) = 2 N_0(z)/k \\
&\sin(2\theta_k(z))= \pm 2\sqrt{(k-N_0(z))N_0(z)}/k \qquad \pm = \sgn(z)
\end{aligned}\]
we have $\E Z_k^2(z)= \big(N_0\pm\i\sqrt{(k-N_0)N_0}\big)/ k$ and 
\[
[\G^2(z)] =  \sum_{m<k \le N} \frac{N_0}{k(k-N_0)}\pm  \i   \sum_{m<k \le N}\frac{\sqrt{N_0}}{k\sqrt{k-N_0}} .
\]
As above, the second sum is approximated by the Riemann integral 
\(\displaystyle 
\int_{z^2}^1 \frac{z^2 dt}{t\sqrt{t-z^2}}<\infty. 
\)
Then, computing the harmonic sum,
\[\begin{aligned}
[\G^2(z)]
& =   \sum_{m\le k \le N}\bigg(\frac{1}{k-N_0} - \frac1k\bigg) +\i\O(1) \\
& = \log\bigg(\frac{(1-z^2)N_T(z)}{\mathfrak{L}(z)T}\bigg) + \O(1) \\
& = 2 \log\big(\varrho(z)\mathfrak{L}(z)\big)  + \O(1) . 
\end{aligned}\]
The last estimate follows from the fact that 
$N_T(z) = \mathfrak{L}(z)^3\O(1)$ if $T$ is bounded and the density of states $\varrho(z)= c\sqrt{1-z^2}$. 

Finally, by a similar computation using that $\E |Z_k(z)|^2 =1$ for $k>N_0(z)$, 
\[
\big[\G^2(z),\overline{\G^2}(z)\big] = \sum_{m\le k \le N} \frac{1}{k-Nz^2}
=\log\bigg(\frac{(1-z^2)N}{\mathfrak{L}(z)T}\bigg) +\O(1)
= \log\bigg(\frac{\varrho(z)^2N}{\mathfrak{L}(z)}\bigg)+\O(1). 
\]
If  $|z|\ge 1-TN^{-2/3}$, then the field $\G^2=0$ so that the previous asymptotics remains true for all $z\in\R$ if we replace $\log(\cdot)$ by $\log_+(\cdot)$ where $\log_+(x)=\log(x)\1\{x\ge 1\}$ for $x\in\R_+$. 

To conclude the proof, we use that by definition,
\[
\big[\Re \G\big]  =  [\G^1]  + \tfrac12 \Re\big( \big[\G^2,\overline{\G^2}\big] - \big[\G^2, \G^2\big] \big),
\qquad\qquad
\big[\Im\G\big]  = \tfrac12 \Re\big( \big[\G^2,\overline{\G^2}\big] + \big[\G^2, \G^2\big] \big),
\]
and by combining the previous estimates we obtain for $z\in[-1,1]$,
\[\begin{aligned}
\big[\Re \G(z)\big]  & =  2\log\big(\mathfrak{L}(z)\big) + \tfrac12  \log_+\Big(\tfrac{\varrho(z)^2N}{\mathfrak{L}(z)}\Big) -  \log_+\big(\varrho(z)\mathfrak{L}(z)\big) +\O(1) \\
& = \tfrac12\log\big(N\mathfrak{L}(z)\big) +\O(1)\\
\big[\Im\G(z)\big]  &= \tfrac12  \log_+\Big(\tfrac{\varrho(z)^2N}{\mathfrak{L}(z)}\Big)+  \log_+\big(\varrho(z)\mathfrak{L}(z)\big)  +\O(1) \\
&= \tfrac12\log_+\big(\varrho(z)^4N\mathfrak{L}(z)\big) +\O(1) .\qedhere
\end{aligned}\] 
\end{proof}

\paragraph{Merging regime.}
The correlation structure of the $\G$ field is more complicated to study as it depends whether the turning points are merging. 
The next lemma shows that in the merging regime, the $\G$ field is \emph{continuous}. 

\begin{lemma} \label{lem:G3}
If $x,z\in[-1,1]$ with $N|x-z|^{2} \le C/\mathfrak{L}(z)$ for some $C\ge 2$, then
\[
\| \G(x) -\G(z) \|_2^2 \lesssim \log(C)
\]
\end{lemma}

\begin{proof}
Let $\mathfrak{L}=\mathfrak{L}(z)$.
In this regime, the turning points satisfy
\[
|N_0(z)-N_0(z)| \lesssim \sqrt{N}\mathfrak{L}^{3/2}|x-z| \le \sqrt{C} \mathfrak{L}
\]
and similarly
\[
|\mathfrak{L}(z)-\mathfrak{L}(w)| \lesssim  \sqrt{N}\mathfrak{L}^{-1/2}|x-z| \le 1/\sqrt{C} . 
\]
This implies that the sets $\Gamma_T(x)  \subset \Gamma_{\tau}(z) $ choosing $\tau \ge CT$ if $C$ is sufficiently large. 
Thus, by Remark~\ref{rk:cutoff},
\[
\G(z) = \sum_{k\in N\setminus \Gamma_{\tau}(z)} \frac{Z_k(z)}{\sqrt{k}\sqrt{Nz^2/k-1}} +\O_{\boldsymbol{\Psi}2}(1) ,\qquad 
\G(x) = \sum_{k\in N\setminus \Gamma_{\tau}(z)} \frac{Z_k(x)}{\sqrt{k}\sqrt{Nx^2/k-1}} +\O_{\boldsymbol{\Psi}2}(1) ,
\]
where both errors are of order $\log(C)$. 
Then, we assume (without loss of generality) that $|z|\le |x|$.\\
We claim that
\[
\G(x) = \G(z) + \operatorname{Er}_N^1(x,z) + \operatorname{Er}_N^2(x,z) +\O_{\boldsymbol{\Psi}2}(1), 
\]
where the errors are given by
\[
\operatorname{Er}_N^1(x,z) \coloneqq  \sum_{k\in N\setminus \Gamma} \frac{Z_k(x)}{\sqrt{k}} \bigg(\frac1{\sqrt{Nx^2/k-1}}- \frac1{\sqrt{Nz^2/k-1}}\bigg) , \qquad
\operatorname{Er}_N^2(x,z) \coloneqq  \sum_{k\in N\setminus \Gamma} \frac{Z_k(x)-Z_k(z)}{\sqrt{k}\sqrt{Nz^2/k-1}}  
\]
with $\Gamma= \Gamma_{\tau}(z)$. 

Using that $\|Z_k(x)\|_2^2 \lesssim 1 $ uniformly for $x\in\R$ and that these random variables are independent, if $\Omega(x,z) \ge \varepsilon \mathfrak{L}(z)$,  we obtain 
\[
\|\operatorname{Er}_N^1\|_2^2 
\lesssim  \sum_{k\in N\setminus \Gamma} \frac{(Nz)^2|z-x|^2}{|Nz^2-k|^{3}}
\lesssim C  \sum_{k\in N\setminus \Gamma} \frac{N_0(z)\mathfrak{L}(z)^{-1}}{|Nz^2-k|^{3}} \lesssim T^{-2}
\]
for some numerical constant (since the factor of $C$ cancel and $N_0(z)\le \mathfrak{L}(z)^{3}$). 
Similarly, according to \eqref{Jouk}, $J'(w)  = -J(w)/\sqrt{w^2-1} $ (this holds for $w\in\R\setminus\{\pm1\}$ with the appropriate  $\sqrt{\cdot}$) so that
\[
Z_k(x) = Z_k(z) +\O\big(|Y_k|\sqrt{N/k} |J'(z\sqrt{N/k})|\cdot |z-x| \big)
= Z_k(z) +\O_{\boldsymbol{\Psi}2}\big(\sqrt{N/|Nz^2-k|}\cdot |z-x| \big) .
\]
Hence, if $N|x-z|^{2} \le C/\mathfrak{L}(z)$,
\[
\|\operatorname{Er}_N^2\|_2^2 \le  \sum_{k\in N\setminus \Gamma} \frac{\|Z_k(x)-Z_k(z)\|_2^2}{|Nz^2-k|}   
\le  \sum_{k\in N\setminus \Gamma_{\tau}(z)} \frac{N|z-x|^2}{|Nz^2-k|^{2}}
\lesssim C \sum_{k\in N\setminus \Gamma} \frac{\mathfrak{L}(z)^{-1}}{|Nz^2-k|^{2}}  \lesssim  T^{-1}. 
\]
This is the main error and it concludes the proof.
\end{proof} 

Lemma~\ref{lem:G3} implies that the asymptotics of Proposition~\ref{prop:G1} can be extended on any neighborhood of the diagonal of size $\O(N^{-1/2}\mathfrak{L}(z)^{-1/2}) = \O(N^{-2/3}[z]_N^{-1/3}) $ where $[z]_N = |z| \vee N^{-1/2}$.

\paragraph{Global correlations.}
In the regime where the turning points are sufficiently far apart, we can exactly compute the correlations of  the $\G$ field up to vanishing errors using the properties of the map $J$. 

\begin{proposition}\label{prop:G2}
Let $x,z\in \R$ with $|x| \le |z|$ be such that  $N|z^2-x^2| \gg \mathfrak{L}(z)$. 
Then, we have
\[
\big[\G(z),\G(x)\big] = - 2\log\big(1-J(z)J(x)\big)  + \underset{N\to\infty}{\o(1)}
\qquad\text{and}\qquad
\big[\G(z),\overline{\G(x)}\big] = - 2 \log\big(1- J(z)\overline{J(x)}\big) +\underset{N\to\infty}{\o(1)}.
\]
These asymptotics hold uniformly for $x,z\in \R\setminus [-1,1]$.

\end{proposition}

\begin{proof} 
Note that the condition $|x| \le |z|$  holds without loss of generality and we choose a sequence  $\mho(N) \to\infty$ as $N\to\infty$ such that $N|z^2-x^2| \gg \mho \ge  \mathfrak{L}= \mathfrak{L}(z)$. 
By \eqref{GW}, we have for $(x,z)\in\R^2$, with $\Gamma = \Gamma_T(x)\cup\Gamma_T(z)$, 
\begin{equation}\label{G2}
\big[\G(z),\G(x)\big]  =  \sum_{k\in[N]\setminus \Gamma} \frac{1}{2k} \frac{1+ J(x\sqrt{N/k})J(z\sqrt{N/k})}{\sqrt{Nz^2/k-1}\sqrt{Nx^2/k-1}}  . 
\end{equation}
In this regime, the turning points are separated in the sense that  $\Gamma_T(x)\cap\Gamma_T(z) =\emptyset$. 
Moreover, in the previous sum, we can replace $\Gamma_T(x)\cup\Gamma_T(z)$ by 
\[
\Gamma  \coloneqq \big\{k\in[N] : |Nx^2-k| \vee|Nz^2-k| \le \mho/N \big\} .
\]
Indeed, $|J(w)|\le 1$ for any $w\in\R$ and
\[
\sum_{k\in \Gamma} \frac{1}{\sqrt{|Nz^2-k|Nx^2-k|}} 
\lesssim  \sqrt{\frac{\mho}{N|z^2-x^2|}}  \ll 1. 
\]

Then, we can approximate \eqref{G2} by a Riemann integral using the identity; for $z,x\in\R$ and $t\in (0,1]$, 
\begin{equation}\label{Jid}
\frac{d}{dt}  \log\big(1\pm J(x/\sqrt{t})J(z/\sqrt{t})\big) = - \frac{ \mp  1 +  J(x/\sqrt{t}) J(z/\sqrt{t}) }{4 \sqrt{x^2-t} \sqrt{z^2-t}} . 
\end{equation}
The proof  follows from the definition of the map $J$, see \eqref{Jouk} and \cite[Lemma~A.4]{LambertPaquette02} for details. 
Define $f : [0,1] \to \C$ by $f: t\in  \mapsto  \frac{1 +J(x/\sqrt{t}) J(z/\sqrt{t}) }{\sqrt{x^2-t} \sqrt{z^2-t}}$; we have
\[
\sum_{k\in[N]\setminus \Gamma} \frac{1+ J(x\sqrt{N/k})J(z\sqrt{N/k})}{\sqrt{Nz^2-k}\sqrt{Nx^2-k}}
=  \frac1N \int_{[0,N]\setminus\Gamma} f(t/N) dt +\O\bigg( \frac1{N^2} \int_{[0,N]\setminus\Gamma}|f'(t/N)| dt\bigg) .
\]

Since  $J'(w)  = -J(w)/\sqrt{w^2-1} $ and $|J(w)| \le 1/|w|$ for $w\in\R$, one has  
$|\partial_tJ\big(x/\sqrt{t}\big)|^2 \le  t^{-1} |x^2-t|^{-1}$ for $t\in(0,1)\setminus x^2$.  
Then, it holds for $t\in[0,N]\setminus \Gamma$, 
\[
|f'(t/N)|\lesssim N^2\big\{  F(t;x,z) + F(t;z,x) \big\}, \qquad\qquad
F(t;x,z) = \frac{|Nx^2-t|^{-1/2}+ t^{-1/2}}{|Nx^2-t||Nz^2-t|^{1/2}}
\]
$\bullet$ If $N x^2 \le \mho$,
\[
\int_{[0,N]\setminus \Gamma} F(t;x,z) dt \lesssim \Omega^{-1/2} \int_{\mho}^\infty  t^{-3/2}dt \lesssim \Omega^{-1} . 
\]
$\bullet$ If $N x^2 \ge\mho$, using that $\sqrt{Nx^2} \le \mathfrak{L}^{3/2}$, we have
\[
\int_{[0,N]\setminus \Gamma} F(t;x,z) dt \lesssim \Omega^{-3/2} \int_0^{N x^2}  t^{-1/2}dt 
+ \Omega^{-1/2} \int_{\mho}^\infty  t^{-3/2}dt
\lesssim (\mathfrak{L}/\Omega)^{3/2}+ \Omega^{-1}
\]
$\bullet$ The same computation also shows that 
\[
\int_{[0,N]\setminus \Gamma} F(t;z,x) dt \lesssim (\mathfrak{L}/\Omega)^{3/2}+ \Omega^{-1} . 
\] 
Altogether, the errors are controlled by 
\[
\int_{[0,N]\setminus \Gamma} F(t;x,z) dt \ll  1 , \qquad\qquad
\frac1N \int_\Gamma |f(t/N)| dt  \lesssim \int_{\Gamma} \frac{dt}{\sqrt{|Nz^2-t||Nx^2-t|}}  \lesssim  \sqrt{\frac{\mho}{N|z^2-x^2|}}  \ll 1. 
\]
Going back to \eqref{G2} and \eqref{Jid}, this implies that as $N\to\infty$
\[
\big[\G(z),\G(x)\big]  = -2 \int_{[0,1]} \frac{d}{dt}  \log\big(1 - J(x/\sqrt{t})J(z/\sqrt{t})\big) dt +\o(1).
\]
Since $J(\infty)=0$, this proves the first claim. The second claim follows from the same argument using that according to Remark~\ref{rk:sym0}, we have 
\begin{equation} \label{G6}
\big[\G(z),\overline{\G(-x)}\big]
= \big[\G(z), \G^\dagger(x)\big]
= \sum_{k\in [N]\setminus\Gamma}  \frac{1}{2k}\tfrac{-1+ J(x\sqrt{N/k})J(z\sqrt{N/k})}{\sqrt{Nx^2/k-1}\sqrt{Nz^2/k-1}}  +\underset{N\to\infty}{o(1)} 
\end{equation}
where the error terms are controlled as above. Using \eqref{Jid} again, we obtain in this case, 
\[
\big[\G(z),\overline{\G(-x)}\big]
= - 2  \log\big(1+J(x)J(z)\big)  +\underset{N\to\infty}{\o(1)} . 
\]
Replacing $x \mapsto -x$ using that  $J(x) = -\overline{J(-x)}$ for $x\in\R$, this proves the second claim.

Finally, if $x,z\in\R\setminus[-1,1]$ without any extra assumption, then we can pick $\Gamma = [N-TN^{1/3},N]$ in \eqref{G2}. The Riemann sum approximation remains valid and the errors is controlled in the worst case $x=z=1$ by 
\[
\int_{[0,N]\setminus \Gamma} \frac{dt}{|N-t|^2} \le N^{-1/3} .
\]
This completes the proof.
\end{proof}

We claim that the regimes of Lemma~\ref{lem:G3} and Proposition~\ref{prop:G2} are complementary unless $x$ lies in a small neighborhood of $-z$.
Indeed, because of the symmetry, the turning points are merging in this case. However, we can adapt the proof of Proposition~\ref{prop:G2} to also treat this case. 

\begin{proposition}\label{prop:G4}
Let $x,z\in\R$ and assume that $\mathfrak{L}(z) \to\infty$ (equivalently $|z|\gg N^{-1/2}$) and that $N|z-x|^2 \ll \mathfrak{L}(z)$, then
\[\begin{aligned}
\big[\G(z),\G(-x)\big] &= - 2\log\big(1-J(z)J(-x)\big)   + \underset{N\to\infty}{o(1)}  \qquad
\big[\G(z),\overline{\G(-x)}\big] = - 2\log\big(1+ J(z)J(x)\big) + \underset{N \to\infty}{o(1)} \\
&= - 2\log\big(1+|J(z)|^2\big)   + \underset{N\to\infty}{o(1)}
\hspace{2.9cm}= - 2\log\big(1+ J(z)^2\big) + \underset{N \to\infty}{o(1)} 
\end{aligned}\]
\end{proposition}

\begin{proof}
If $N|x-z|^2 \ll \mathfrak{L}(z)$, the turning point are merging and we can replace $\Gamma = \Gamma_T(z)$ in formula \eqref{G2}, up to an error $\o(1)$ as $N\to\infty$ by Remark~\ref{rk:cutoff}. 
Then, as in the proof of Proposition~\ref{prop:G2}, 
\[\begin{aligned} 
\big[\G(z),\G(-x)\big] & =  \sum_{k\in[N]\setminus \Gamma} \frac{1}{2k} \frac{1+ J(-x\sqrt{N/k})J(z\sqrt{N/k})}{-\sqrt{Nx^2/k-1}\sqrt{Nz^2/k-1}}   \\
&=  \frac{1}{2N} \int_{[0,N]\setminus\Gamma} f(t/N) dt +\O\bigg( \frac1{N^2} \int_{[0,N]\setminus\Gamma}|f'(t/N)| dt\bigg) 
+\underset{N\to\infty}{\o(1)} 
\end{aligned}\]
where $f : [0,1] \to \C$ is given by $f: t\in  \mapsto  \frac{1 +J(-x/\sqrt{t}) J(z/\sqrt{t}) }{-\sqrt{x^2-t} \sqrt{z^2-t}}$. 
We note that $f$ is integrable on $[0,1]$ for every $x,z>0$. In particular, on the diagonal $(x=z)$, since
$J(-w) = -\overline{J(w)}$ for $w\in\R$, one has 
\[
f(t) = \1\{t\le z^2\}\frac{-J(z/\sqrt{t}) }{\sqrt{t}\sqrt{z^2-t}}
\]
where we used the algebraic identity
\(\frac{1-J(w)^2}2 = 1-wJ(w) 
=\sqrt{w^2-1}J(w) .
\)

Again, as in the proof of Proposition~\ref{prop:G2}, the derivative of $f$ satisfies for for every $x,z>0$ and for $t\in[0,N]\setminus \Gamma$, 
\[
|f'(t/N)|\lesssim N^2  \frac{|Nz^2-t|^{-1/2}+ t^{-1/2}}{|Nz^2-t|^{3/2}} 
\]
where we used that the turning points are merging. In particular, as $N_0(z) \to\infty$, 
\[
\frac1{N^2} \int_{[0,N]\setminus\Gamma}|f'(t/N)| dt \lesssim \frac1{\mathfrak{L}^{3/2}} \int_0^1 \frac{dt}{t^{1/2}}
+ \int_{\mathfrak{L}}^\infty \frac{dt}{t^{3/2}} \lesssim \frac1{\mathfrak{L}^{1/2}} \ll 1 .
\]
This shows that the error in the Riemann sum approximation goes to 0 as $N\to\infty$. 
For the main term, we expand for $t\in(0,1]$, 
\[
f(t) =\frac{1 -xz/t -x/\sqrt{z^2-t}-z\sqrt{x^2-t}-\sqrt{z^2-t}\sqrt{x^2-t}}{-\sqrt{x^2-t} \sqrt{z^2-t}}
= \frac{xz-t}{t \sqrt{x^2-t} \sqrt{z^2-t}} + \frac{x}{\sqrt{z^2-t}}+ \frac{z}{\sqrt{x^2-t}} +1 . 
\]
This implies that for every $x,z>0$ (in a compact),
\[
\frac1N \int_{\Gamma} \bigg| f(t/N) - \frac{Nxz-t}{t \sqrt{Nx^2-t} \sqrt{Nz^2-t}}\bigg| dt 
\lesssim \frac1{\sqrt{N}} \int_{\Gamma}  \frac{dt}{|Nz^2-t|^{1/2}} +\O\bigg(\frac{\mathfrak{L}}{N}\bigg)
\lesssim\sqrt{\frac{\mathfrak{L}}{N}} . 
\]
If $x=z$, the main term is exactly $-\displaystyle\int_\Gamma \frac{dt}{t} =\o(1)$ as $\mathfrak{L}\to\infty$. 
If $x\neq z$, we can bound 
\[
\frac1N \int_{\Gamma} \bigg| \frac{Nxz-t}{t \sqrt{Nx^2-t} \sqrt{Nz^2-t}} \bigg| dt
\lesssim \frac1N \int_\Gamma \frac{dt}{t}+ \frac{|x-z|}{|N|z}  \int_\Gamma  \frac{dt}{|Nx^2-t|^{1/2}|Nz^2-t|^{1/2}}
\]
The last integral grows logarithmically,
\[
\frac{|x-z|}{|N|z}  \int_\Gamma  \frac{dt}{|Nx^2-t|^{1/2}|Nz^2-t|^{1/2}}
\lesssim  \frac{|x-z|}{|N|z}  \log\bigg(\frac{\mathfrak{L}}{N|x^2-z^2|}\bigg)
\lesssim \frac{\Delta}{N} \log(\Delta \mathfrak{L}^2)^{-1}
\]
where $\Delta = |x/z-1|$ is small.
Hence, we conclude that in this regime
\[
\frac1N \int_{\Gamma} | f(t/N)|dt  = \underset{N\to\infty}{\o(1)}.
\]
Using the identity \eqref{Jid} again, this proves that 
\[
\big[\G(z),\G(-x)\big] = \frac{1}{2}\int_{[0,1]} f(t)dt +\underset{N\to\infty}{\o(1)}
= -2\log\big(1-J(-x)J(z)\big) +\underset{N\to\infty}{\o(1)}.
\]
We also note that, since $J(-x) = -\overline{J(x)}$ for $x\in\R$ and $J$ is 1/2-H\"older, if $x\to z$,  
\[
\log\big(1-J(-x)J(z)\big) = \log\big(1+|J(z)|^2) +\o(1)
\]

To compute $\big[\G(z),\overline{\G(-x)}\big]$, we use \eqref{G6} and the previous method; the arguments are identical and we obtain if $N|x-z|^2 \ll \mathfrak{L}(z)$, 
\[
\big[\G(z),\overline{\G(-x)}\big] = - 2\log\big(1+ J(z)J(x)\big) + \underset{N \to\infty}{o(1)} .
\]
The main term is singular as $z\to0$ and we have 
\[
\log\big(1+ J(z)J(x)\big) = \log\big(1+ J(z)^2\big)+\O\bigg(\bigg|\frac{J(x)-J(z)}{J(z)+J(z)^{-1}}\bigg|\bigg)
\]
Using that $J(z)+J(z)^{-1}=2z$, the error term converges to 0 away from 0. In a neighborhood of 0, $J$ is smooth and using that $N|x-z|^2 \ll \mathfrak{L}(z)$ and $\sqrt{N}|z= \mathfrak{L}(z)^{3/2}$,  the error is controlled by 
$\O\big(\frac{|x-z|}{|z|}\big) = \O\big(\mathfrak{L}(z)^{-1}\big)$. 
This shows that 
\[
\log\big(1+ J(z)J(x)\big) = \log\big(1+ J(z)^2\big)+\o(1) 
\]
in the regime that we are considering. 
\end{proof}

\paragraph{Proof of Proposition~\ref{prop:G}.}
We now combine the previous estimates to obtain Proposition~\ref{prop:G}. 

\begin{proof}
Let $x,z\in\R$. Without loss of generality, suppose that $|x| \le |z|$.  We split the argument in two regimes (local and global) and we record that the condition $N|z-x|^2 \gg \mathfrak{L}(z)^{-1}$ implies that as $N\to\infty$, 
\begin{equation} \label{loccond}
\mathfrak{L}(z) \to\infty 
\quad\text{and}\quad 
\text{either }\,
i)\, N  |z^2-x^2| \gg \mathfrak{L}(z) \, \text{ or }\, 
ii)\, N|z+x|^2 \ll \mathfrak{L}(z).
\end{equation}
\eqref{loccond} follows from the following case;\\
$\bullet$ if $\mathfrak{L}(x) \le C$, then  $N|z^2-x^2| \simeq \sqrt{Nz^2} \sqrt{N|x-z|^2 } \gg \mathfrak{L}(z)$. \\
$\bullet$ if $\sgn(x) =\sgn(z)$, then $N|z^2-x^2| \ge \sqrt{Nz^2} \sqrt{N|x-z|^2 } \gg \mathfrak{L}(z)$. \\
$\bullet$ otherwise,  $\sgn(x) \neq \sgn(z)$ and  $\sqrt{N|x-z|^2}\ge \sqrt{Nz^2}= \mathfrak{L}(z)^{3/2}$, so that either $N|z+x|^2 \ll \mathfrak{L}(z)$ or 
$N|z+x|^2 \ge c\mathfrak{L}(z)$ for a $c>0$ in which case we also have $N|z^2-x^2| \ge c\mathfrak{L}(z)^2$. 

\smallskip\noindent
$\mathbf{1}.$
In the local regime, $N|z-x|^2 \le C \mathfrak{L}(z)^{-1}$ for a constant $C\ge 1$, by combining Proposition~\ref{prop:G1} and Lemma~\ref{lem:G3}, we obtain 
\[\begin{aligned}
&\big[\G(z),\G(x)\big]  =  \log_+\big(|1-x^2|^{-1}\wedge N^{2/3}\big)  + \O(1) ,\\
&\big[\G(z),\overline{\G(x)}\big]  = \begin{cases} \log\big(\varrho(x)^2N\mathfrak{L}(x)) &\text{if } |x| \le 1-N^{-2/3} \\
\log_+\big(|1-x^2|^{-1}\wedge N^{2/3}\big)  &\text{if } |x| \ge 1-N^{-2/3}
\end{cases}
\,+\O(1).
\end{aligned}\] 
In particular, this covers the case where $\mathfrak{L}(z) \le c$ for some constant $c\ge 1$ ($N^{-1/2}$-neighborhood of 0).  In this special case,  if~$|x-z| \le C N^{-1/2}$, 
\[
\big[\G(z),\G(x)\big]  = \O(1),
\qquad
\big[\G(z),\overline{\G(x)}\big]  = \log(N)+\O(1) . 
\] 

\smallskip\noindent
$\mathbf{2}.$
In the global regime, if  $N|z-x|^2 \gg \mathfrak{L}(z)^{-1}$, by \eqref{loccond}, we can either apply Proposition~\ref{prop:G2}  or Proposition~\ref{prop:G4} (in the special case where $N|z+x|^2 \ll \mathfrak{L}(z)$).
In both cases, we have 
\begin{equation*}
\big[\G(z),\G(x)\big] = - 2\log\big(1-J(z)J(x)\big)  + \underset{N\to\infty}{\o(1)} \qquad
\big[\G(z),\overline{\G(x)}\big] = - 2 \log\big(1- J(z)\overline{J(x)}\big) +\underset{N\to\infty}{\o(1)}. 
\qedhere
\end{equation*}
\end{proof}

\subsection{Oscillatory sums.} \label{sec:Wosc}
To study the bracket of the $\W$ field, we need to refine certain estimates from Section~\ref{sec:osc}.  
Indeed, its bracket is given by certain sums whose oscillations speed is controlled $|\theta_n(x)-\theta_n(z)|$, see formula \eqref{W2} below. 
In this case, we need the following improvement of Lemma~\ref{lem:osc}. 

\begin{lemma}\label{lem:osc3}
Fix $z,x \in [-1,1]$  with $|x|\le |z|$.  
For any $n \ge N_0(z)$ and any $L\in\N$, 
\[
\bigg| \sum_{j=n+1}^{n+L}  e^{\i 2( \vartheta_{j,n}(x) - \vartheta_{j,n}(z))} \bigg|
\lesssim   \frac1{|\sin(2\ell_{n+1}^-(x,z))|}+ |z-x|\sqrt{N}L^3\delta_{n+1}^3(z). 
\]
\end{lemma}

\begin{proof}
Without loss of generality, $z\in[0,1)$.  Then, for $k>n$, 
\[
\big(\theta_k(x)- \theta_n(x)\big) - \big(\theta_k(z)- \theta_n(z)\big) 
=\int_n^k\int_x^z \partial_{t,u}\big(\arccos\big(u\sqrt{N/t}\big)\big)  \d t \d u 
\]
where we compute 
\[
\partial_{t,u}\big(\arccos\big(u\sqrt{N/t}\big)\big)  \big|_{t=n,u=z} = \frac{\sqrt{N}}{2}\delta_n^3(z) . 
\]
Note that  $\delta_n(z) \ge \delta_k(u)$ for $k\ge n > N_0(z)$ and $|u|<z$.
In particular, the quantity $n \mapsto \ell_n^-$ is positive, non-decreasing and 
\[
0 \le \ell^-_k -\ell^-_n
\le (k-n)(z-x)\sqrt{N} \delta_n^3(z)/4 . 
\]
This implies that for any $j\in\N$, 
\[
0\le  \vartheta_{n+j,n}(x) - \vartheta_{n+j,n}(z)- 2 j \ell^-_{n+1}
\le \tfrac{j(j-1)}{4} (z-x) \sqrt{N}\delta_{n+1}^3(z) .
\]
Like in the proof of  Lemma~\ref{lem:osc}, we obtain
\[
\bigg| \sum_{j=n+1}^{n+L}  e^{\i 2 (\vartheta_{j,n}(x)- \vartheta_{j,n}(z))} \bigg|
\le \frac{1}{\sin(2\ell^-_{n+1})} +  L^3\sqrt{N}(z-x) \delta_{n+1}^3(z) . \qedhere
\]
\end{proof}

Throughout this section, we define for $z,x \in (-1,1)$, 
\begin{equation} \label{qW}
\ell_n^{\pm}= \ell_n^\pm(x,z) \coloneqq \frac{\theta_n(x)\pm\theta_n(z)}{2},\qquad\qquad
q_{n}^{\pm} = q_{n}^{\pm} (x,z) \coloneqq \delta_{n+1}(z)\delta_{n+1}(x) \cos(\ell_{n+1}^\pm) e^{\i3 \ell_{n+1}^\pm} .
\end{equation}
These coefficients arise for instance when computing the bracket of the $\W$ field; see \eqref{Z2}--\eqref{W1} below.
We note that the phases of the coefficients $q_{n}^{\pm}$ will not be relevant in the proof. We record two variants of  Proposition~\ref{prop:osc2}. 

\begin{proposition}\label{prop:osc3}
Let $z,x\in(-1,1)$  with $|x|\le |z|$.
If $N|x-z|^2 \mathfrak{L}(z) \gg 1$, then
\[
\max_{n> N_T(z)}\bigg| \sum_{N_T(z)\le k \le n} q_k^-(z,x) e^{\i2(\phi_{k}(x)-\phi_{k}(z))}  \bigg|  \to 0 \qquad\text{in probability as $N\to\infty$}.
\]
\end{proposition}

\begin{proof}
According to \eqref{Asmooth}, we consider the event $\A_\W(R;z,x) \coloneqq \A_\chi(T,R; z) \cap\,  \A_\chi(T,R; x)$ for some sequence of blocks $\big\{n_k = N_0+ L k^{1+\alpha}\big\}_{k\ge K}$ where $L, R\ge 1$ and $\alpha>0$ are to be decided in the course of the proof. Here $n_K=N_T(z)$ with $T\ge 1$ fixed. \\
We proceed as in the proof of Lemma~\ref{prop:osc1} by splitting the sum into blocks,
\[
\max_{n>n_K}\bigg|  \sum_{n_K< k \le n} q_k^-(z,x) e^{\i2(\phi_{k}(x)-\phi_{k}(z))}  \ \bigg| 
\le  \sum_{k\ge K}\bigg|  \sum_{n_{k}<j\le n_{k+1}} q_j'(z,x) e^{\i2(\vartheta_{j,n_k}(x)-\vartheta_{j,n_k}(z))} \bigg|
\]
where  $q_n'(z,x) = q_n^-(z,x) e^{2\i \chi_{n,n_k}(x)} e^{-2\i \chi_{n,n_k}(x)} $ for $n\in(n_{k}, n_{k+1}]$. 

The coefficients $\big\{q_n^-(z,x)\big\}_{n>N_0(z)}$ also satisfy
\[
|q_n^-(z,x)| \le \delta_n(x) \delta_n(z) |\cos(\ell_n^-(z,x))| , \qquad\qquad
|q_{n+1}(z,x)-q_n^-(z,x)| \le \delta_n(x) \delta_n^3(z) . 
\]
The second estimate is a consequence of \eqref{delta2} and Lemma~\ref{lem:theta}
Then on the event $\A_\W$,  for each  block,
\[\begin{aligned}
\bigg|\sum_{n_{k}<n\le n_{k+1}} q_j'(z,x) e^{\i2(\vartheta_{j,n_k}(x)-\vartheta_{j,n_k}(z))}\bigg|
&\le |q_{n_k}(z,x)| \bigg| \sum_{n_{k}<j\le n_{k+1}} e^{\i2(\vartheta_{j,n_k}(x)-\vartheta_{j,n_k}(z))}\bigg|
+ \O\big(R k^{\frac{\epsilon-3}{2}}K^{-\epsilon/2}\big) \\
&\lesssim \frac{\delta_{n_k}(x)\delta_{n_k}(z)|\cos\ell_{n_k}^-|}{|\sin(2\ell_{n_k}^-)|}+ \sqrt{N|z-x|^2}L_k^3\delta_{n_k}^5(z) +R k^{\frac{\epsilon-3}{2}}K^{-\epsilon/2} 
\end{aligned}\] 
where we used Lemma~\ref{lem:osc3} with $L_k = n_{k+1}-n_k$. 

Observe that $\ell_n^- \in [-\pi/2,\pi/2] $, then  using the second claim of Lemma~\ref{lem:theta}, 
\[
\frac{\delta_n(x)|\cos\ell_n^-|}{|\sin 2\ell_n^-|} = \frac{\delta_n(x)}{2|\sin\ell_n^-|} 
\le \frac{2\delta_n(x)}{|\ell_n^-|} \le  \frac{4}{\sqrt{N|z-x| ^2}} . 
\]
By construction, $L_k\simeq L k^\alpha$ and we have $\delta_{n_k}^2 L_k \lesssim k^{-1}$ for every $k\in\N_{\ge K}$, so that 
\[
\frac{\delta_{n_k}(x)\delta_{n_k}(z)|\cos\ell_{n_k}^-|}{|\sin(2\ell_{n_k}^-)|}+ \sqrt{N|z-x|^2}L_k^3\delta_{n_k}^5(z) 
\lesssim \frac{\delta_{n_k}(z)}{\sqrt{N|z-x| ^2}} \bigg(1+ \frac{N|z-x|^2L}{k^{2-\alpha}}\bigg) \lesssim  k^{-3/2}
\]
by choosing  $\alpha=2$ and $L =  N^{-1}|z-x|^{-2}$ (here, $\delta_{n_k}(z) =L^{-1/2}k^{-3/2}$). 

We conclude that on $\A_\W$, 
\begin{equation} \label{sumblock1}
\max_{n>n_K }\bigg|  \sum_{n_K< k \le n} q_k^-(z,x) e^{\i2(\phi_{k}(x)-\phi_{k}(z))}  \ \bigg| 
\lesssim R \sum_{k\ge K}  k^{\frac{\epsilon-3}{2}}K^{-\epsilon/2}  \lesssim RK^{-1/2} 
\end{equation}
where $K= T^{1/3}\Theta^{-1/3}\ge 1$ (by construction, $L K^{3} = T\mathfrak{L}(z)$ and $\Theta =  L/\mathfrak{L}(z) = \mathfrak{L}(z)^{-1}N^{-1}|z-x|^{-2} \ll 1$). 

Hence, we can choose a sequence $R(N)\to\infty$ as $N\to\infty$ in such a way that $R\ll \Theta^{-1/3}$ in which case
$\eqref{sumblock1}\to0$ on $\A_\W$ and $\P[ \A_{\W}] \to 1$ as $N\to\infty$ (see Proposition~\ref{lem:varphase}  and  \eqref{PA} -- $\mathfrak{L}\gg 1$ in this regime). 
This shows that \eqref{sumblock1} converges to 0 in probability. 
\end{proof}

\begin{proposition}\label{prop:osc4}
Let $z,x\in(-1,1)$  with $|x|\le |z|$.
If $N|x-z|^2 \mathfrak{L}(z) \gg 1$, then
\[
\max_{n> N_T(z)}\bigg| \sum_{N_T(z)\le k \le n} q_k^+(z,x) e^{\i2(\phi_{k}(x)+\phi_{k}(z))}  \bigg| \to 0 \qquad\text{in probability as $N\to\infty$}.
\]
\end{proposition}

\begin{proof}
In this regime, \eqref{loccond} holds and, for technical reasons, we treat the two cases separately. Let $\mathfrak{L}=\mathfrak{L}(z)$ and $N_T=N_T(z)$.

\noindent
$\mathbf{1}$.  Using that $\ell_n^{+}(x,z) = \frac\pi2 -\ell_n^{-}(-x,z) $, by Lemma~\ref{lem:theta}, 
\[
|\cos \ell_n^+(x,z)| = |\sin \ell_n^{-}(-x,z)| 
\le \big| \tfrac{\theta_{n+1}(-x)-\theta_{n+1}(z)}{2} \big| \le |x+z| \sqrt{N} \delta_n(z). 
\]
Then, in the case where $x$ lies in a small neighborhood of $-z$, that is if $N|z+x|^2 \ll \mathfrak{L}$, we have the (deterministic) bound,
\[
\sum_{n\ge N_T} |q_n^+(z,x)| 
\le  |x+z| \sqrt{N}  \sum_{n\ge N_T } \delta_n^3(z)
\le \sqrt{ \frac{|x+z|^2N}{\mathfrak{L}}} \ll 1. 
\]

\noindent
$\mathbf{2}$. 
Otherwise, we can choose a sequence  $\mho(N) \to\infty$ as $N\to\infty$
such that $\mho \mathfrak{L}  \ll N|z^2-x^2|$. 
Using that $ |q_n^+(x,z)|\le \delta_n(z)/\sqrt{N(z^2-x^2)}$, we obtain the (deterministic) bound,  with $m=N_0(z)+\mho\mathfrak{L}$, 
\[
\sum_{ N_T\le n\le m} |q_n^+(x,z)|  \lesssim  \sqrt{\frac{\mho\mathfrak{L}}{N(z^2-x^2)}} \ll 1 .
\]

\noindent
$\mathbf{3}$.
The coefficients $\big\{q_n^+(z,x)\big\}_{n\ge N_0(z)}$ also satisfy the conditions;
\begin{equation} \label{qWcont}
|q_{n-1}^+(z,x)| \le \delta_n(z)\delta_n(x) |\cos(\ell_n^+(z,x))| , \qquad\qquad
|q_{n}^+(z,x)-q_{n-1}^+(z,x)| \le  \delta_n^4(z) . 
\end{equation}
Hence, by Proposition~\ref{prop:osc2} (choosing blocks according to \eqref{block1} with $K = \mho\mathfrak{L}$ and $\mathfrak{L}\gg1$ -- in addition, $\delta_n(x) \le \delta_n(z)$), we obtain on the event $\A_\W(R;z,x) \coloneqq \A_\chi(T,R; z) \cap\,  \A_\chi(T,R; x)$,
\[
\max_{n> m}\bigg| \sum_{m<k<n} q_k^+(z,x) e^{\i2(\phi_{k}(x)+\phi_{k}(z))}  \bigg| 
\lesssim R/\sqrt{\mho}. 
\]

Hence, we can choose a sequence $R(N)\to\infty$ as $N\to\infty$ in such a way that $R\ll \sqrt{\mho}$ in which case
\[
\max_{n> N_T}\bigg| \sum_{N_T\le k \le n} q_k^+(z,x) e^{\i2(\phi_{k}(x)+\phi_{k}(z))}  \bigg|  \ll 1 \quad\text{on $\A_\W$}
\]
and $\P[ \A_{\W}] \to 1$ as $N\to\infty$, as in the previous proof.
\end{proof}

The same arguments, replacing  by  yields the following result. 

\begin{proposition}\label{prop:osc5}
Let $z,x\in(-1,1)$  with $|x|\le |z|$.
Suppose that the coefficients $\big\{q_n^\pm(z,x)\big\}_{n\ge N_0(z)}$ satisfy the conditions \eqref{qWcont}. 
Then, if $N|x-z|^2 \mathfrak{L}(z) \gg 1$, 
\[
\sup_{|x|\le |z|}\max_{n> N_T(z)}\bigg| \sum_{N_T(z)\le k \le n} q_k^{\pm}(z,x) e^{\pm\i2\phi_{k}(z)}  \bigg| \to 0 \qquad\text{in probability as $N\to\infty$}.
\]
\end{proposition}

\begin{proof}
Let $\mathfrak{L}=\mathfrak{L}(z)$ and $N_T=N_T(z)$. In this regime $\mathfrak{L} \gg1$ and we need again to the cases in \eqref{loccond} separately.

\noindent
$\bullet$ If $N  |z^2-x^2| \gg \mathfrak{L}$, for $\{q_n^{\pm}\}_{n\ge N_0}$, repeating the steps $\mathbf{2}$--$\mathbf{3}$ from the proof of Proposition~\ref{prop:osc4}, using Proposition~\ref{prop:osc1} (instead of Proposition~\ref{prop:osc2}) at step $\mathbf{3}$, we conclude that 
\[
\sup_{|x|\le |z|}\max_{n> N_T}\bigg| \sum_{N_T\le k \le n} q_k^{\pm}(z,x) e^{\pm\i2\phi_{k}(z)}  \bigg|  \ll 1 \quad\text{on $\A_\chi(T,R; z)$}.
\]

\noindent
$\bullet$ If $N|z+x|^2 \ll \mathfrak{L}$, by step $\mathbf{1}$ of the  proof of Proposition~\ref{prop:osc4},  
$\sum_{n\ge N_T} |q_n^+(z,x)| \ll 1$. 
Then, according to \eqref{qW}\footnote{Observe that \eqref{costheta} is an equality if $x=-z$ as $\ell_n^+(z,-z) =\pi/2-\theta_n(z)$ for $z\ge 0$.}, 
\begin{equation} \label{costheta}
|\cos(\ell_n^-(z,x))| = |\cos(\theta_n(z)-\ell_n^+(z,x))| \le |\cos(\ell_n^+(z,x))|+ |\sin(\theta_n(z))|
\end{equation}
and by \eqref{qWcont},
\[
|q_{n-1}^+(z,x)| \le \delta_n(z)\delta_n(x) |\cos(\ell_n^+(z,x))|  + \delta_n(z)/\sqrt{n}  
\]
using that $\sin(\theta_n(z)) = \delta_n^{-1}(z)/\sqrt{n}$ and $\delta_n(x) \le \delta_n(z)$.
The first term is handled exactly as $\sum_{n\ge N_T} |q_n^+(z,x)| \ll 1$.  For the second term, choosing $m=(1+\mho^{-1}) N_0$  for some $\mho \gg1$, we have the (deterministic) bound
\[
\sum_{ N_T\le n\le m}  \frac{\delta_n(z)}{\sqrt n}  \lesssim  \mho^{-1/2}  \ll 1 . 
\]
Then, using again Proposition~\ref{prop:osc1} (in this regime $\mathfrak{L}\gg 1$ and $K = N_0/\mho\mathfrak{L} = \mathfrak{L} ^2/\mho$), we obtain
\[
\sup_{|x|\le |z|}\max_{n> m}\bigg| \sum_{m<k<n} q_k^-(z,x) e^{\i2\phi_{k}(z)}  \bigg| 
\lesssim R \frac{\mho^{1/2}}{\mathfrak{L}} \ll 1\quad\text{on $\A_\chi(T,R; z)$} 
\]
provided that we choose two sequences $R(N),\mho(N) \to\infty$ in such as way that $R\sqrt{\mho} \ll \mathfrak{L}$ as $N\to\infty$. 

By Proposition~\ref{lem:varphase}  and  \eqref{PA} ($\mathfrak{L} \gg 1$ in this regime), we have  $\P[\A_\chi(T,R; z)] \to 1$ as $N\to\infty$ ($R\gg1$ and $T\ge 1$). This shows that both sums converge to 0 in probability as claimed. 
\end{proof}

Finally, we need to record another variant of the previous propositions in the complementary regime $N|x-z|^2 \mathfrak{L}(z) \le C$ for a constant $C\ge 1$. In this regime, we cannot aim for vanishing errors.

\begin{proposition}\label{prop:osc6}
Let $z,x\in(-1,1)$  with $|x|\le |z|$  and $\Omega = \Omega(w,z)= N^{-1}|w-z|^{-2}$. 
Suppose that $\Omega\ge C\mathfrak{L}$ for a constant $C\ge 1$.  
Let $m_+ = N_T(z)$ and $m_-= N_\Theta(z)$ with $\Theta=\Omega/\mathfrak{L}$. 
For $R\ge 1$, there is an event $\A_\W(R;z,x)$ on which 
\[
\max_{n> m_{\pm}}\bigg| \sum_{m_\pm< k \le n} q_k^\pm(z,x) e^{\i2(\phi_{k}(x)\pm\phi_{k}(z))}  \bigg| \lesssim  R
\]
and
\[
\P\big[ \A_{\W}^c \cap \A(T,R;z) \big] \lesssim\exp( - c R^2) .
\]
\end{proposition}

\begin{proof}
To control the $+$ sum, we apply Proposition~\ref{prop:osc2} as in the proof of Proposition~\ref{prop:osc3}, step $\mathbf{3}$. We obtain on the event $\A_\W(R;z,x)$, 
\[
\max_{n> N_T}\bigg| \sum_{N_T<k<n} q_k^+(z,x) e^{\i2(\phi_{k}(x)+\phi_{k}(z))}  \bigg| 
\lesssim R . 
\]
To control the $-$ sum,  we can apply the same construction as in the proof of Proposition~\ref{prop:osc3}.
Observe that in this regime, the parameter $\Theta  \ge C$ and $K=1$ because we choose $T=\Theta$. 
Then, by \eqref{sumblock1}, on the event  $\A_\W(R;z,x)$, 
\[
\max_{n>n_T }\bigg|  \sum_{n_T< k \le n} q_k^-(z,x) e^{\i2(\phi_{k}(x)-\phi_{k}(z))}  \ \bigg| 
\lesssim R . 
\]
Finally, by Proposition~\ref{lem:varphase}, we obtain 
\[
\P\big[ \A_{\W}^c \cap \A(T,R;z) \big] \lesssim\exp( - c R^2) . \qedhere
\]
\end{proof}

\subsection{$\W$ field.} \label{sec:W}
In terms of Definition~\ref{def:GW},for $z\in(-1,1)$, 
\begin{equation} \label{W0}
\W(z)  = \sum_{N_T(z)<k\le N} \frac{Z_k(z)e^{2\i(\theta_k(z)+\phi_{k-1}(z))}}{\sqrt{k-Nz^2}} .
\end{equation}
where $T\ge 1$ is fixed. The goal of this section is to derive the following asymptotics for the $\W$ field's bracket. 

\begin{proposition}[Correlation structure of the $\W$ field] \label{prop:W}
Define the random fields $\Xi_1, \Xi_2$ for $x,z\in (-1,1)$ by
\[
\big[\W(z),\overline{\W(x)}\big] = 2\log_+\bigg(\frac{|x-z|^{-1}\wedge N\varrho(z)}{\sqrt{N\mathfrak{L}(z)}}\bigg) +\Xi_1(z,x),
\qquad\qquad
\big[\W(z),\W(x)\big] = \Xi_2(z,x) . 
\]
If $|x|\le |z|$, it holds for $i=1,2,$ 
\begin{itemize}[leftmargin=*] \setlength\itemsep{0em}
\item {\normalfont(global regime)}  if  $N|z-x|^2 \gg \mathfrak{L}(z)^{-1}$, $\Xi_i(z,x) \to0$ in probability as $N\to\infty$. 
\item {\normalfont(local regime)} for any $S\ge 1$, 
\[
\lim_{R\to\infty}\sup_{N\in\N}\sup_{|z-x|^{-2} \ge S N\mathfrak{L}(z)}\P\big[ |\Xi_i(z,x)| \le R \big] =0.
\]
\end{itemize}
\end{proposition}

In particular, the $\W$ field behaves like a (complex) \emph{white noise} that is log-correlated on scales $\le \sqrt{N^{-1}\mathfrak{L}(z)} = N^{-2/3}[z]_N^{-1/3}$ where $[z]_N = |z| \vee N^{-1/2}$.

\medskip

The proof of Proposition~\ref{prop:W} relies on the estimates from Section~\ref{sec:Wosc} and the fact that for  $x,z\in(-1,1)$,
\begin{equation} \label{Z2}
\begin{cases}
\E\big[Z_n(x)Z_n(z)\big] = \frac{1+e^{-2\i\ell_n^+(x,z)}}{2} = e^{-\i\ell_n^+(x,z)}\cos(\ell_n^+(x,z)) \\
\E\big[Z_n(x)\overline{Z_n}(z)\big] = \frac{1+e^{-2\i\ell_n^-(x,z)}}{2} = e^{-\i\ell_n^-(x,z)}\cos(\ell_n^-(x,z))
\end{cases}
\qquad\text{with $\ell_n^\pm(x,z) \coloneqq \frac{\theta_n(x)\pm\theta_ n(z)}{2}$.}
\end{equation}
Then, in terms of the notation \eqref{qW} and $N_T= N_T(z)$, one has for $x,z\in(-1,1)$,
\begin{equation} \label{W1}
\big[\W(x), \W(z)\big]  
=  \sum_{N_T\le n< N} q_{n}^+(x,z)  e^{2\i(\phi_{n}(x)+ \phi_{n}(z))}  , \qquad
\big[\W(x), \overline{\W(z)}\big]  
= \sum_{N_T\le n< N} q_{n}^-(z,z)  e^{2\i(\phi_{n}(x)- \phi_{n}(z))} 
\end{equation}
with $q_{n}^{\pm} (x,z) \coloneqq \delta_{n+1}(z)\delta_{n+1}(x) \cos(\ell_{n+1}^\pm) e^{\i3 \ell_{n+1}^\pm}$.

\medskip

We begin by computing the quadratic variation of the $\W$ field.

\begin{lemma} \label{lem:W}
It holds uniformly for $z\in[-1,1]$, with a deterministic error, 
\[
\big[\W_N(z),\overline{\W_{N}(z)}\big] = \log_+\bigg(\frac{\rho(z)^2N}{\mathfrak{L}(z)}\bigg) +  \O(1).
\]
\end{lemma}

\begin{proof}
In this case $q_n^- = \delta_{n+1}^2$ since $\ell_n^- = 0$ for $n\ge N_0(z)$. Thus, by \eqref{W1}, 
\[
\big[\W_N(z),\overline{\W_{N}(z)}\big] = \sum_{N_T<k \le N}  \frac{1}{k-Nz^2} 
= \log_+\bigg(\frac{(1-z^2)N}{T\mathfrak{L}(z)}\bigg) +  \o(1)
\]
where the error goes to zero if $N_0(z) \to\infty$ (it is bounded otherwise). 
\end{proof}

\begin{proof}[Proof of Proposition~\ref{prop:W}]
Let $x,z\in(-1,1)$ with $|x|\le |z|$, $N_T=N_T(z)$ and $\mathfrak{L}= \mathfrak{L}(z)$.\\
$\bullet$ [Global regime] The claim follows directly from \eqref{W1} by combining Propositions~\ref{prop:osc3} and~\ref{prop:osc4}.\\
$\bullet$ [Local regime] Proposition~\ref{prop:osc6} ($+$ case) shows that  if ${|z-x|^{-2} \ge S N\mathfrak{L}(z)}$, then 
\[
[\W(z),\W(x)] = \O(R)\qquad\text{on }\A_\W(R;x,z)
\]
and with $\A= \A(T,R ;z)$, 
\begin{equation} \label{Wtight1}
\lim_{R\to\infty}\sup_{N\in\N}\sup_{|z-x|^{-2} \ge S N\mathfrak{L}(z)}\P\big[ \A_{\W}^c(R;x,z)\cap \A\big] =0. 
\end{equation}
This proves the claim for $\Xi_2$. 

Using the notation from Section~\ref{sec:cont}, let $\delta_{n}=\delta_n(z)$ and $\Omega = \Omega(w,z)= N^{-1}|w-z|^{-2}$. 
By \eqref{qW}, as in the proof of Lemma~\ref{lem:contest},
\[\begin{aligned}
q_{n-1}^{-} = \delta_{n}^2\cos(\partial\theta_n/2)e^{\i3\partial\theta_n/2}+\O(\delta_n\partial\delta_n)
&= \delta_{n}^2 +\O(\delta_n^2\partial\theta_n)+\O(\delta_n\partial\delta_n)\\
&= \delta_{n}^2 +\O\big(\delta_n^3\Omega^{-1/2}\big) +\O\big(\delta_n^4\sqrt{N_0/\Omega}\big) .
\end{aligned}\]
This shows that 
\[
\sum_{n>m} | q_{n-1}^-(x,z) -\delta_n^2(z) | \lesssim  \sqrt{\mathfrak{L}/\Omega} 
\]
and by \eqref{W1}, if $\Omega \ge S \mathfrak{L}$, 
\[
\big[\W_{N}(x), \overline{\W_{N}(z)}\big]  
= \sum_{N_T\le n\le N} \delta_n^2  e^{2\i(\phi_{n}(x)- \phi_{n}(z))} +\O(1)
\] 
with a deterministic error.
Let $\epsilon\in(0,1]$, $M_\epsilon=M_\epsilon(w,z) \coloneqq N_0(z) +\epsilon\Omega(w,z)$, and consider the event
\[
\A_\phi =  \A_\phi(\epsilon;x,z) \coloneqq \big\{|\phi_n(x)-\phi_n(z)| \le \big(\delta_n(z)\Omega\big)^{-1/2} ; \forall  n \in [N_T, M_\epsilon] \big\}.
\]
On $\A_\phi$, one has
\[
\max_{N_T\le n \le M_\epsilon}\bigg|  \sum_{N_T\le k < n} \delta_k^2 \big(  e^{2\i(\phi_k(x)- \phi_k(z))} -1\big) \bigg| 
\le  \sum_{N_T\le n \le M_\epsilon} \frac{\delta_n(z)^{3/2}}{\Omega^{1/2}}  \lesssim \sqrt{\epsilon}. 
\]
Moreover, by Proposition~\ref{prop:osc6} ($-$ case), on $\A_\W$,
\[
\bigg|\sum_{M_1<n < N} q_n^-  e^{2\i(\phi_{n}(x)- \phi_{n}(z))}\bigg| \lesssim R .
\]
For the remaining pieces, we can use the trivial estimates
\[
\sum_{M_\epsilon< k \le M_1} \delta_k^2  \le \log(\epsilon^{-1}) , \qquad\qquad 
\sum_{N_T\le  k \le M_\epsilon} \delta_k^2 =\log\bigg(\frac{\epsilon\Omega}{T\mathfrak{L}}\bigg).
\] 
Choosing $\epsilon= e^{-CR}$, we conclude that on the event $\A_\W(R;x,z)\cap \A_\phi(\epsilon;x,z)$;  
\[
\big[\W_{N}(x), \overline{\W_{N}(z)}\big]  
= \log_+\bigg(\frac{\Omega\wedge N\varrho(z)^2}{\mathfrak{L}}\bigg) +\O(R) .
\]
To conclude the proof, by Propositions~\ref{prop:relcont1}  with $\A= \A(T,R ;z)$  and $\B =\B(T; x,z)$,  one has 
\[
\lim_{R\to\infty}\sup_{N\in\N}\sup_{|z-x|^{-2} \ge S N\mathfrak{L}(z)}\P\big[\A_\phi^c(e^{-CR};x,z) \cap \A \cap \B \big]  =0
\]
Moreover, for a fixed $T\ge 1$,
\(\displaystyle
\lim_{R\to\infty}\sup_{N\in\N}\sup_{z\in[-1,1]}\sup_{|z-x|^{-2} \ge S N\mathfrak{L}(z)}\P\big[ \A(T,R ;z)\cap\B(T; x,z)\big] =1.
\)
Together with \eqref{Wtight1}, this shows that the random fields $\Xi_1,\Xi_2$ are tight. 
\end{proof}

\subsection{Proof of Proposition~\ref{prop:M}.} \label{sec:M}

To finish the proof, it remains to compute the correlations between the two martingale fields $\G$ and $\W$.
Proposition~\ref{prop:W} shows that these fields are \emph{almost uncorrelated}. 

\begin{proposition}[Joint bracket of the $\G,\W$ fields] \label{prop:GW}
Define the random fields $\Xi_3, \Xi_4$  for $x,z\in (-1,1)$,
\[
\Xi_3(z,x) =  [\G(x),\overline{\W(z)}]  
\qquad\qquad
\Xi_4(z,x)= [\G(x),\W(z)] . 
\]
If $|x|\le |z|$, it holds for $i=3,4,$
\begin{itemize}[leftmargin=*] \setlength\itemsep{0em}
\item {\normalfont(global regime)}  if  $N|z-x|^2 \gg \mathfrak{L}(z)^{-1}$, $\Xi_i(z,x) \to0$ in probability as $N\to\infty$. 
\item {\normalfont(local regime)}  
\[
\lim_{R\to\infty}\sup_{N\in\N}\sup_{x,z\in(-1,1)}\P\big[ |\Xi_i(z,x)| \ge R \big] =0.
\]
\end{itemize}
\end{proposition}

\begin{proof}
Without loss of generality, we assume that $|x|\le |z|$ and $N_T=N_T(z)$.
Using \eqref{GW}, \eqref{W0} and \eqref{Z2},  we compute for $x,z\in(-1,1)$
\begin{equation*} \label{W2}
\Xi_4(z,x)= -\i\big[\G^2(x),\W(z)\big] 
=  \sum_{N_T\le n< N} q_{n}^+  e^{2\i\phi_{n}(z)} , \qquad
\Xi_3(z,x)
= -\i\big[\G^2(x), \overline{\W(z)}\big] 
=\sum_{N_T\le n< N} q_{n}^-  e^{-2\i\phi_{n}(z)},
\end{equation*}
where the coefficients satisfy for $n>N_0(z)$, 
\begin{equation} \label{qGW}
|q_{n-1}^\pm(x,z)| \le \delta_n(z)\delta_n(x) \cos(\ell_n^{\pm}(x,z)), \qquad\qquad
|q_{n}^+(x,z)-q_{n-1}^+(x,z)| \le  \delta_n^4(z) . 
\end{equation}

This is exactly the setting of Proposition~\ref{prop:osc1} and~\ref{prop:osc5}.
Hence, in the regime where  $N|x-z|^2 \mathfrak{L}(z) \gg 1$, both 
$[\G(x), \W(z)]$ and $[\G(x), \overline{\W(z)}]$ converge to 0 in probability. 
In general, on the event $\A_{\chi}(T,R; z)$ for $R\ge 1$,  we have
\[
\sup_{|x|\le |z|}\max_{n> N_T}\bigg| \sum_{N_T<k<n} q_k^\pm(z,x) e^{\pm\i2\phi_{k}(z)}  \bigg| 
\lesssim R/T^{1/3}.
\]
By Proposition~\ref{lem:varphase}  and  \eqref{PA},
for any $T\ge 1$,
\(\displaystyle
\lim_{R\to\infty}\sup_{N\in\N}\sup_{z\in[-1,1]}\P\big[ \A_\chi(T,R ;z)\big] =1.
\)
This shows that the random fields $\Xi_3,\Xi_4$ are tight. 
\end{proof}

By \eqref{GW}, combining Proposition~\ref{prop:G}, Proposition~\ref{prop:W} and Proposition~\ref{prop:GW}, we obtain the following asymptotics for $x,z\in\R$ with $|x|\le |z|$; 
\begin{itemize}[leftmargin=*] \setlength\itemsep{0em}
\item (Global regime)\footnote{Observe that the condition $ |x-z| \gg 1/\sqrt{N\mathfrak{L}(z)}$ is equivalent to $ |x-z| \gg N^{-2/3} [z]_N^{-1/3}$.}  
If $z\in[-1,1]$ and  $|x-z| \gg N^{-2/3} [z]_N^{-1/3}$ or $\big(|z|-1\big)\gg N^{-2/3}$ as $N\to\infty$, 
\[\begin{cases}
[\M(z),\M(x)] =  [\G(z),\G(x)] - \Xi_2(z,x) -2\i \Xi_4(z,x) =  - 2\log\big(1-J(z)J(x)\big)  + \o_\P(1) , \\
[\M(z),\overline{\M(x)}] = [\G(z),\overline{\G(x)}] + \Xi_1(z,x) -2\Im \Xi_3(z,x) =- 2 \log\big(1- J(z)\overline{J(x)}\big) + \o_\P(1).
\end{cases}\]
In particular, the main term in $ [\W(z),\overline{\W(x)}] $ vanishes in this case and $\W(z) =0$ if $\big(|z|-1\big)\gg N^{-2/3}$ (that is, outside of the spectrum) -- the errors converge to 0 in probability. 

This shows that in this regime, the bracket of the $\M$ field matches the correlation structure of the Gaussian field $\mathrm{W}$; see \eqref{corrW}. 

\item (Local regime)  For a constant $C\ge 1$, if $|z|\le 1-CN^{-2/3}$ and $ |x-z| \le CN^{-2/3}[z]_N^{-1/3}$ as $N\to\infty$,
\[\begin{cases}
[\M(z),\M(x)] =  [\G(z),\G(x)] - \Xi_2(z,x) -2\i \Xi_4(z,x) =  -2 \log(\varrho(z)) +\O_\P(1) , \\
[\M(z),\overline{\M(x)}] = [\G(z),\overline{\G(x)}] + [\W(z),\overline{\W(x)}]  -2\Im \Xi_3(z,x) = 
2\log\big(\big(\tfrac{|x-z|}{\varrho(z)}\big)^{-1}\wedge\big(N\varrho(z)^2\big)\big)+ \O_\P(1).
\end{cases}\] 
For the second bracket, we used that
\[
\log\big(\tfrac{|x-z|^{-1}\wedge N\varrho(z)}{\sqrt{N\mathfrak{L}(z)}}\big)+ \log\big(\varrho(z)\sqrt{N\mathfrak{L}(z)})  =\log\big(\big(\tfrac{|x-z|}{\varrho(z)}\big)^{-1}\wedge\big(N\varrho(z)^2\big)\big).
\]
In this regime, the $\G$ field is \emph{saturated} and the \emph{extra variance} comes from the $\W$ field. 

\item (Edge regime) If $| x\pm 1| , |z\pm 1|  \le C N^{-2/3}$ for some constant $C\ge 1$, 
\[
[\M(z),\M(x)]  = \log(N^{2/3})+\O_\P(1) \qquad\qquad
[\M(z),\overline{\M(x)}]  =\log(N^{2/3})+\O_\P(1). 
\]
This follows form the fact that at the edge, $\| \W(z)\|_2^2\lesssim 1$, so that the brackets of $\M$ matches that of $\G$, up to order 1 random variables.

\item These three regimes are consistent.
Observe that for $x,z \in (-1,1)$, by \eqref{def:J}, one has 
\[
\log\big(1-J(z)J(x)\big) \underset{x\to z}{\to} \log\big(1-e^{-2\i\arccos(z)}\big)  
= \log\big|\tfrac{\pm 1-z}{\varrho(z)}\big| + \underset{z \to\pm 1}{\O(1)}
= \log(\varrho(z)) +\O(1).
\]
and, if $|x-z| \le C\varrho(z)^2$,
\[
\log\big(1-J(z)\overline{J(x)}\big)  = \log\big(1-e^{\i(\arccos(x)-\arccos(z))}\big)  
= \log\big|\tfrac{x-z}{\varrho(z)}\big| + \underset{x\to z}{\O}\big(\tfrac{|x-z|}{\varrho(z)^2}\vee 1\big)
= \log\big|\tfrac{x-z}{\varrho(z)}\big| +\O(1).
\]
Let $\epsilon_N(z) \coloneqq \big(N\varrho(z)^2 \vee N^{1/3}\big)^{-1}$. 
This implies that for $x,z \in\R$, 
\[\begin{cases}
[\M(z),\M(x)] = - 2\log\big(|1-J(z)J(x)| \vee \epsilon_N(z)\big) + \O_\P(1) , \\
[\M(z),\overline{\M(x)}] =- 2 \log\big(|1- J(z)\overline{J(x)}|\vee \epsilon_N(z)\big) + \O_\P(1).
\end{cases}\]
\todo{Remove, correctly, the absolute value}
This completes the proof. \qed
\end{itemize}

\section{Approximation by the stochastic sine equation}
\label{sec:zeta}

The goal of this section is to study the \emph{microscopic relative phase} and prove the following results (recall that $\phi_{N}(z) = \Im\Psi_N(z)$).
These convergence results are claim 2 of Theorem~\ref{thm:main}

\begin{proposition} \label{prop:sine}
Let $z\in(-1,1)$ with $\varrho(z)\ge \mathfrak{R} N^{-1/3}$ for some sequence $ \mathfrak{R}(N) \to \infty $ as $N\to\infty$. 
Then, in the sense of of finite dimensional distributions, as $N\to\infty$, 
\[
\big\{ \{\phi_{N}(z)\}_{2\pi} , 2\big( \psi_{N}\big(z+\tfrac{\lambda}{N\varrho(z)}\big) -  \psi_{N}(z) \big) : \lambda\in\R \big\}  \to \big\{ \alpha , \omega_1(\lambda) : \lambda\in\R \big\}
\]
where $\{\omega_{t}(\lambda) : t\in\R_+ , \lambda\in\R\}$ is the solution of the complex sine equation \eqref{eq:csse}
with $\omega_{0}=0$ and $\boldsymbol{\alpha}$ is an independent random variable uniformly distributed in $[0,2\pi]$. 
\end{proposition}

\todo{Include the uniform convergence proposition}

\subsection{Linearization and continuity.}
To prove Proposition~\ref{prop:sine}, we first collect our assumptions and some prior results from Sections~\ref{sec:ell} and~\ref{sec:cont}. 

\begin{assumption}\label{ass:bulk}
Let $z(N)\in(-1,1)$ with $\varrho(z)\ge \mathfrak{R} N^{-1/3}$ and $ \mathfrak{R}(N) \to \infty $ as $N\to\infty$.\\
Let $m:= N_0(z)+ \delta N \varrho(z)^2$ for a small $\delta\in(0,1)$ and  $w_\lambda:=z-\tfrac{\lambda}{N\varrho(z)}$ for a fixed  $\lambda\in\R$. \\
Let $\tau\coloneqq (\pi/2)^2 $ and ${\rm c}_\beta \coloneqq \sqrt{2/\beta}$.
\end{assumption}

\begin{lemma}[Linearization] \label{lem:lin3}
Under the Assumptions~\ref{ass:bulk}. 
On has for $n>m$, 
\begin{equation} \label{varphi}
\partial\bpsi_{n,n-1}(w_\lambda,z)
=  \i \delta_n(z) \tfrac{\lambda}{\sqrt{N}\varrho(z)}  + \tfrac{1}{\sqrt\beta}  \overline{\W_{n,n-1}}(z)\big(1-e^{-2\i\partial\phi_{n-1}(w_\lambda,z)}\big)+  \mathscr{E}_{n,n-1}(\lambda;z)
\end{equation}
where $\W_{n,n-1}(z) = \i \delta_n(z) Z_n(z)  e^{2\i\theta_n(z)}e^{2\i\phi_{n-1}(z)} $ and there is an event $\A_\partial(\lambda,\delta;z)$ such that 
\[
\max_{m\le n\le N} |\mathscr{E}_{n,m}(\lambda;z)| \lesssim \mathfrak{R}(N)^{\epsilon-1} \quad\text{on $\A_\partial$} , \qquad 
\limsup_{N\to\infty}\P[\A_\partial^{\rm c}] =1 . 
\]
Moreover, we the relative phase satisfies 
\begin{equation} \label{varphi_ini}
\lim_{\delta\to 0} \limsup_{N\to\infty}  \P\big[|\partial\bpsi_{m}(w_\lambda,z)| > \delta^\epsilon \big] = 0 
\end{equation}
and there is a deterministic sequence $ \Lambda_{N,m}(z)\in\R$ such that the imaginary part of the phase satisfies as $N\to\infty$,
\begin{equation} \label{phi_ini}
\big| \phi_{N,m}(z) - \big( \Lambda_{N,m}(z) -  \tfrac1{\sqrt\beta} \Im\big(\G_{N,m}(z) +  \overline{\W_{N,m}}(z)\big) \big) \big| \overset{\P}{\to} 0. 
\end{equation} 
\end{lemma}

\begin{proof}
We start from the proof of Lemma~\ref{lem:linear}.
From \eqref{zeta3},  on the event $\A_m$, it holds for $n> m$, 
\[
\tilde{\bpsi}_{n,n-1}  
=  - \tfrac14\mathbf{Q}_{n,n-1}^0- \tfrac{1}{\sqrt\beta}  \M_{n,n-1}+ \tfrac{1}{2\beta}\mathbf{L}_{n,n-1}+  \operatorname{EL}_n
\]
where the martingale increments $\M_{n,n-1} = -\i \delta_n Z_n (1+ e^{-2\i\theta_n}e^{-2\i\phi_{n-1}}) $, the linearization errors $|\operatorname{EL}_n| \lesssim N^{\epsilon} \delta_n^3  $ for a small $\epsilon>0$, \eqref{EL}.
This expansion holds at $w_\lambda$ and the errors are controlled uniformly for $\lambda\in\mathcal{K}$ 
where $\mathcal{K} \Subset \R$ is any compact set with $0\in \mathcal{K}$.
Moreover, $\mathbf{Q}_{n,n-1}^0=\delta_n^2(1 -  e^{-2\i \phi_{n-1}})$ and  $\mathbf{L}_{n,n-1} = - (\M_{n,n-1})^2$, so that 
\[
\tfrac14\mathbf{Q}_{n,n-1}^0 -  \tfrac{1}{2\beta}\mathbf{L}_{n,n-1} = \tfrac14 \delta_n^2 -   \tfrac{1}{2\beta} \delta_n^2 Z_n^2 + \mathbf{Q}_{n,n-1}^1
\]
where $\mathbf{Q}_{n,m}^1$ collects oscillatory sums of the types of Proposition~\ref{prop:osc1} and Proposition~\ref{prop:osc2} with $x=z$ (the coefficients are controlled uniformly for $\lambda\in\mathcal{K}$).
We work on the event $\A_\chi = \A_{\chi}(R,T; z)$ with blocks \eqref{block1} with $m=N_T$ so that 
$T \ge \delta \mathfrak{R}^2$ provided that $\varrho(z)\ge \mathfrak{R} N^{1/3}$. 
Thus, on this event ($\delta>0$ is fixed), it holds for $R\ge 1$, 
\[
\sup_{\lambda\in\mathcal{K}} \max_{n\ge m} |\mathbf{Q}_{n,m}^1(w_\lambda)| \lesssim  \frac{R}{\mathfrak{R}}. 
\]
Then,  by Remark~\ref{rk:contest}, $\big\| \partial\big( \tfrac14 \delta_n^2 -   \tfrac{1}{2\beta} \delta_n^2 Z_n^2 \big)\big\|_1 \lesssim \delta_n^3 $,  so that
\[
U_{n,m}= \tfrac14\partial\mathbf{Q}_{n,m} -  \tfrac{1}{2\beta}\partial\mathbf{S}_{n,m} = B_{n,m} +\partial \mathbf{Q}_{n,m}^1 , \qquad 
\|B_{n,n-1}\|_1 \lesssim \delta_n^3 , \qquad
\max_{n\ge m} |\partial\mathbf{Q}_{n,m}^1| \lesssim  \frac{R}{\mathfrak{R}} , 
\]
and $\sup_{n\ge m} \|B_{n,m}\|_1 \le \sum_{n>m}  \|B_{n,n-1}\|_1 \lesssim   \sum_{n>m} \delta_n^3  \lesssim  \delta_m \lesssim N^{-1/6} $. Thus,  by a union bound,
\[
\P\bigg[ \Big\{\sup_{m\le n\le N} |U_{n,m}| \ge \frac{CR}{\mathfrak{R}}  \Big\} \cap \A_\chi  \cap \A_m \bigg]
\le 
\P\bigg[ \sup_{m\le n\le N} |B_{n,m}| \ge \frac{CR}{2\mathfrak{R}}   \bigg] \le N \exp\big(- \mathfrak{R}^{-1} N^{\frac1{6}}\big)
\]
by choosing the constant $C\ge 1$ sufficiently large.
We may assume that $ \mathfrak{R}(N) \ll N^{\frac16}$ as $N\to\infty$. 

Recall \eqref{psi1}, so that on $\A_m$, it holds for $n> m$, 
\[
\partial\bpsi_{n,n-1}  
=  \i \partial\theta_n -U_{n,n-1}-  \tfrac{1}{\sqrt\beta} \partial\M_{n,n-1}+  \partial\operatorname{EL}_n ,
\]
where $|\partial\operatorname{EL}_n| \lesssim N^{\epsilon} \delta_n^3$ and, by Remark~\ref{rk:contest},  the martingale part satisfies
\[
\partial \M_{n,n-1} =  - \underbrace{\i \delta_n \overline{Z_n}  e^{-2\i\theta_n}e^{-2\i\phi_{n-1}}}_{= \overline{\W_{n,n-1}}}(1-e^{-2\i\partial\phi_{n-1}})+ A_{n,n-1}
\]
where $\E_{n-1}[A_{n,n-1}] =0$ and $\|A_{n,n-1}\|_{2,n} \lesssim \delta_n^2$. 
In particular, the process $\{A_{n,m}\}_{n\ge m}$ is a martingale and its quadratic variation is controlled by 
$\sqrt{\sum_{n\ge m} \delta_n^4} \lesssim  \delta_m \lesssim N^{-1/6} $, then by Proposition~\ref{lem:conc2},
\[
\P\bigg[ \sup_{n\ge m} |A_{n,m}| \ge \frac{C}{\mathfrak{R}}      \bigg] \lesssim \exp\big(- \mathfrak{R} N^{-\frac1{6}}\big) .
\]
Finally for the deterministic drift, using that $w_\lambda = z- \tfrac{\lambda}{N\varrho(z)}$ and 
$\partial_{z}\theta_n(z)  = -\sqrt{N}\delta_n(z)$, by a Taylor expansion
\[
\partial\theta_n  = - \tfrac{\lambda}{N\varrho(z)}\partial_z\theta_n(z) + \O\big(\tfrac{\sqrt{N}\delta_n(z)^2}{N^2\varrho(z)^2} \big) 
= \delta_n \tfrac{\lambda}{\sqrt{N}\varrho(z)} + \O\big(\delta_n^2 N^{-5/6}) . 
\]

We conclude that on $\A_m$, 
\[
\partial\bpsi_{n,n-1}  
=  \i \delta_n  \tfrac{\lambda}{\sqrt{N}\varrho(z)}  - \tfrac{1}{\sqrt\beta} V_{n,n-1}(1-e^{-2\i\partial\phi_{n-1}})+  \mathscr{E}_{n,n-1}
\]
where the error $\mathscr{E}_{n,n-1}$ includes the oscillatory terms $U_{n,n-1}$, the martingale part $A_{n,n-1}$ and both deterministic errors $\O(\delta_n^3N^\epsilon)$ and $ \O\big(\delta_n^2 N^{-5/6}) .$
These deterministic errors are summable for $n\in[m,N]$ and their total contribution is $\O\big(N^{\epsilon-1/6})$. 
Consequently, setting  
\[
\A_\partial = \{\sup_{m\le n\le N} |U_{n,m}| \le \tfrac{CR}{\mathfrak{R}}\} \cap \A_\chi \cap \A_m \cap \{\sup_{n\ge m} |A_{n,m}| \le \tfrac{C}{\mathfrak{R}} \} , \qquad \text{choosing } R= \mathfrak{R}^\epsilon ,
\]
we have $\displaystyle\sup_{m\le n\le N} |\mathscr{E}_{n,m}| \lesssim \mathfrak{R}^{\epsilon-1}$ on $\A_\partial $ and, by combining the previous estimates (with \eqref{PAchi}):
\[
\P\big[\A_\partial^{\rm c}\cap \A_\chi  \cap \A_m \big] \le N \exp\big(- \mathfrak{R} N^{-\frac1{6}}\big)  , \qquad 
\P\big[ \A_{\chi}^{\rm c} \cap \A_{m} \big]  \lesssim \exp\big(- c\mathfrak{R}^\epsilon\big) .
\] 
In addition, as $m\ge \delta N$ in this regime, $\P[\A_m^{\rm c}] \lesssim \exp(c N^\epsilon)$, \eqref{PA}. 
This proves the first claims. 

The entrance behavior of the relative phase follows from Proposition~\ref{prop:entrancesmall}; see also \eqref{cont0} for the case where $z$ is in a $\O(N^{-1/2})$-neighborhood of 0. 

If $z\in\Q$, the claim \eqref{phi_ini} is a direct consequence of Proposition~\ref{prop:oneray1};  with our choice of $m$, $T \ge \delta \mathfrak{R}^2 \gg 1$ as $N\to\infty$. 
Otherwise, if $z$ is in a $\O(N^{-1/2})$-neighborhood of 0,  \eqref{phi_ini} with $m= \delta N$ (in this regime $N_0$ is fixed and $\varrho(z)^2 = \tau^{-1}+\O(N^{-1})$)  follows directly from the representation \eqref{psi0} and the estimate \eqref{Cauchy0}. 
\end{proof}

\subsection{Homogenization} \label{sec:homo}
Starting from Lemma~\ref{lem:lin3}, we are going to show that \eqref{varphi} is a discretization of the stochastic sine equation \eqref{Zeq:alpha}.
This will imply that under the Assumptions~\ref{ass:bulk}, after a continuous time change $t\in[\delta,1] \mapsto n_t\in [N_0(z),N]$,
the process $\{\partial\bpsi_{n_t}(w_\lambda,z) : t\in[\delta,1] , \lambda\in\mathcal{K}\}$ converges as $N\to\infty$, in the sense of finite dimensional distributions, to $\{\omega_t(\lambda) :  t\in[\delta,1] , \lambda\in\R\}$. 
This requires to make a series of transformation of the equation \eqref{Zeq:alpha}; 
\begin{itemize}[leftmargin=*] \setlength\itemsep{0em}
\item Step 1: Removing the linearization errors.
\item Step 2: Coarse graining the driving noise using a blocking scheme.
\item Step 3: Replacing the driving noise by i.i.d.~complex Gaussians using a \emph{Wasserstein coupling}.
\item Step 4: Continuum approximation to replace the noise by a stochastic integral. 
\item Step 5: Fixing the initial condition. 
\end{itemize}

Steps 1--4 rely on using a generic \emph{stochastic Gr\"onwall inequality} proved in Section~\ref{sec:Gron}. 
This relies on the fact that the equation \eqref{varphi} is of the type
\begin{equation} \label{Gron}
\Delta_{j+1} = \Delta_j + U_{j+1} + V_{j+1}\overline{{\rm f}(\Delta_j)}, \qquad j\ge j_0, 
\end{equation}
where  ${\rm f} : w\in\C \mapsto (1-e^{\i\Im w})$ is Lipschitz-continuous, uniformly bound with ${\rm f}(0)=0$, and the driving noise $\{V_j\}$ are martingale increments; $\E[V_{j+1}| \F_j] =0$. 
In particular, if the errors $\{U_j\}$ are \emph{small}, then one can uniformly control the size of $\{\Delta_j\}$, see Proposition~\ref{prop:l2logladder}.
Step 5 is a direct consequence of the estimate \eqref{varphi_ini} and the properties of the stochastic sine equation \eqref{Zeq:alpha}; see Proposition~\ref{prop:continuity0}. 

\paragraph{Stochastic Gr\"onwall inequality.}
We start by stating a simplified version of Proposition~\ref{prop:l2logladder} tuned for our applications.

\begin{lemma} \label{lem:Gron}
Suppose that $\{\Delta_j\}$ satisfies \eqref{Gron} where ${\rm f} : \C \to \R$ is $1$-Lipschitz continuous with ${\rm f}(0)=0$ and $\{U_j\}, \{V_j\}$ are two adapted sequences with respect to a filtration $\{\F_j\}$. Suppose that 
\[
\E[V_{j+1}| \F_j] =0 , \qquad \|V_{j+1}\|_{2}^2 \lesssim j^{-1} , \qquad j\ge j_0, 
\]
and we can decompose $U_j = U_j^1 +U_j^2$ where $\{U_j^1\}$ are deterministic errors,  $\E[U_{j+1}| \F_j] =0 $ for $j\ge j_0$ and there is $\varepsilon >0$ such that 
\begin{equation} \label{Gronmartbd}
{\textstyle \sum_{j=j_0}^{j_1-1} U_{j+1}^1} \le \varepsilon , \qquad 
{\textstyle \sum_{j=j_0}^{j_1-1} \E |U_{j+1}^2}|^2 \le \varepsilon^2  . 
\end{equation}
where $1\le j_1/j_0 \le C$ for a constant $C$. 
Then, as $\varepsilon \to 0$, 
\[
\max_{j_0 \leq j \leq j_1} |\Delta_j|  \overset{\P}{\to} 0. 
\]
\end{lemma}

\begin{proof}
The condition on $\{V_j\}$  directly implies that we can apply Proposition~\ref{prop:l2logladder} (with $\delta=0$ and $T\le C$). Moreover, by Doob's maximum inequality, 
\[
\P\Big[\max_{j_0 \leq j \leq j_1} \big| {\textstyle\sum_{k=j_0+1}^{j}} U_k^2 \big| > \sqrt{\varepsilon}\Big]  \le  \sqrt{\varepsilon}  
\] 
and a similar estimate holds for the deterministic part of $\{U_j\}$. Then, there is a constant $0<c\le 1/2$ such 
\[
\P\Big[\max_{j_0 \leq j \leq j_1}  |\Delta_j| \ge \sqrt{\varepsilon \log \varepsilon^{-1} }  \big] \lesssim \varepsilon^{c} .  
\]
This proves the claim. 
\end{proof}

\paragraph{Step 1: Removing the linearization errors.}
Let  $K_0:=  \delta N \varrho(z)^2$ so that $m= N_0 + K_0$ as in~\ref{ass:bulk}.\\
We introduce a new process $\{\varphi^0_k(\lambda;z) \}_{k\ge K_0}$ such that for $k\ge K_0$,
\begin{equation} \label{varphi0}
\varphi^0_k(\lambda;z)  = 2\partial\bpsi_{m}(w_\lambda,z) + 2 \sum_{K_0\le j\le k} \Big( -\i \delta_n(z) \tfrac{\lambda}{\sqrt{N}\varrho(z)}  + \tfrac{1}{\sqrt\beta}  \overline{\W_{n,n-1}}(z)\big(1-e^{-\i \Im\varphi^0_{j-1}(\lambda;z) }\big) \Big)_{n=N_0+j} .
\end{equation}
Thus, modulo a time shift, $\{\varphi^0_k\}_{k\ge K_0}$ follows the same evolution as $\{2\partial\bpsi_{n}\}_{n\ge m}$ without the linearization errors (Lemma~\ref{lem:lin3}) with the same initial condition.
We compare the two processes using by applying a stochastic Gr\"onwall inequality. 
Consider the difference:
\[
\Delta^0_k(\lambda;z) \coloneqq \varphi^0_k(\lambda;z)  - 2\partial\bpsi_{N_0+k}(w_\lambda,z)  , \qquad k\in[K_0,K_1],
\]
with $K_1:=  {\rm c} N \varrho(z)^2$,  $\sqrt{\rm c} = \pi/2$  so that $N_0+K_1 = N$. 
In particular, the ratio $K_1/K_0 = {\rm c}/\delta $ is bounded uniformly in~$N$. 

\begin{proposition} \label{prop:Delta0}
Under the Assumptions~\ref{ass:bulk}, as $N\to\infty$  
\[
\max_{K_0< k \le K_1}|\Delta^0_k(\lambda;z) | \overset{\P}{\to} 0. 
\]
\end{proposition}

\begin{proof}
The process $\{\Delta^0_k\}_{k\ge K_0}$ satisfies $\Delta^0_{K_0} =0$ and the evolution
\[
\Delta^0_k - \Delta^0_{k-1}= -  \tfrac{2}{\sqrt\beta}  V_{k} 
\big(1-e^{-\i \Im\Delta^0_k }\big) - U_k 
\] 
where  $V_k =  \overline{\W_{N_0+k,N_0+k-1}} e^{-\i \Im\varphi^0_{k-1}}$  and the errors satisfy $\sum_{j=K_0+1}^k U_j =  \mathscr{E}_{m+k,m}$.
In particular, the martingale increments satisfy 
$\|V_k \|_2^2 \lesssim \delta_{k+N_0}^2  =  k^{-1}$ and, by Lemma~\ref{lem:lin3},  
\[
\max_{K_0< k \le K_1} |{\textstyle \sum_{j=K_0+1}^k U_j }| \lesssim \mathfrak{R}^{\epsilon-1}  \qquad \text{on the event $\A_\partial(\lambda,\delta;z)$}.
\]
Thus, by Proposition~\ref{prop:l2logladder}, since $\mathfrak{R} \to \infty$ as $N\to\infty$, 
\[
\limsup_{N\to\infty}\P\Big[ \Big\{ \max_{K_0< k \le K_1}|\Delta^0_k| \ge \mathfrak{R}^{2\epsilon-1} \Big\} \cap \A_\partial \Big] =0 . 
\]
Since $\P[\A_\partial ]\to 1$ as $N\to\infty$ (uniformly for $\lambda\in\mathcal{K}$), this completes the proof.
\end{proof}

\paragraph{Step 2: Coarse graining the noise.}
We need to compare  \eqref{varphi0} to an evolution equation driven by Gaussian increments. To achieve this, we first aggregate the noise using a blocking scheme so that the accumulated noise along each blocks can be compared to independent complex Gaussians.

\smallskip

We introduce blocks $n_j \coloneqq N_0(z) + \eta(N) \varrho(z)^2 N j$ for $j\ge \mathfrak{J}_0= \delta \eta^{-1}$ so that $n_{\mathfrak{J}_0} = m$ and $\eta =\eta(N) \ll 1$ is a new parameter that will be fixed later in the course of the proof. 
Let $\mathfrak{J}_1= {\rm c}\eta^{-1}$ so that $n_{\mathfrak{J}_1} = N$, and define the random variables, for $j\ge \mathfrak{J}_0$,
\begin{equation} \label{Srv}
S_{j+1}(z)\coloneqq  \sqrt{\frac2{\mathfrak{D}}}  \sum_{k=n_{j}+1}^{n_{j+1}} Z_{k}(z) e^{2\i \vartheta_{k,n_{j}}(z)} , \qquad  \mathfrak{D} \coloneqq \eta(N) \varrho(z)^2 N . 
\end{equation}

%

Let ${\rm c}_\beta \coloneqq \sqrt{2/\beta}$.
Recall that we decompose the imaginary part of the phase $\phi_{n,m} =\Im(\bpsi_{n,m}) = \vartheta_{n,m}+ \boldsymbol{\chi}_{n,m}$ where $\boldsymbol{\chi}_{n,m}$ is the ``random part''.
Then, we introduce a new process  $\{\varphi^1_j(\lambda;z) \}_{j\ge \mathfrak{J}_0}$ which satisfies the evolution
\begin{equation} \label{varphi1}
\varphi^1_j(\lambda;z)  = 2\partial\bpsi_{m}(w_\lambda,z) +  \sum_{i= \mathfrak{J}_0}^{j-1}  \bigg( \frac{2\i\lambda}{\sqrt{N}\varrho(z)} \bigg(\sum_{k=n_{i} +1}^{n_{i+1}} \hspace{-.1cm} \delta_k(z)  \bigg)
+\i{\rm c}_\beta \big(1-e^{-\i \Im\varphi^1_{i}(\lambda;z) }\big)  \frac{e^{-2\i\phi_{n_{i}}(z)}\overline{S_{i+1}}(z) }{\sqrt i}  \bigg) .
\end{equation}
This should be compared to the evolution  \eqref{varphi0}, the deterministic terms are the same, but the random part of the phase is ``frozen''  along every block.
This is similar to the constructions from Section~\ref{sec:osc}.
We consider the difference 
\[
\Delta_j^1(\lambda;z)   \coloneqq  \varphi^1_j(\lambda;z)- \varphi^0_{n_j}(\lambda;z) , \qquad  \Delta_{\mathfrak{J}_0}^1(\lambda;z) =0 . 
\]
By applying a stochastic Gr\"onwall inequality, we obtain the following estimates:

\begin{proposition} \label{prop:Delta1}
Assume that $\eta(N) \ll 1$ as $N\to\infty$ (with $\mathfrak{D} \in \N$), then as $N\to\infty$,
\[
\max_{j \in[\mathfrak{J}_0, \mathfrak{J}_1)} |\Delta_j^1|  \overset{\P}{\to}   0. 
\]
Moreover, as $N\to\infty$, 
\begin{equation} \label{Wblockbd}
\big[\W_{N,m}(z)-\i{\textstyle \sum_{j = \mathfrak{J}_0}^{\mathfrak{J}_1-1}} \tfrac{e^{2\i\phi_{n_{j}}(z)}S_{j+1}(z)}{\sqrt{2j}} \big]  \overset{\P}{\to}   0.  
\end{equation}
\end{proposition}

\begin{proof}
Both  \eqref{varphi1} and  \eqref{varphi0} are of the type \eqref{Gron} with  ${\rm f}(w)=(1-e^{\i\Im w})$.
In particular, $\overline{{\rm f}(w_1)}-\overline{{\rm f}(w_0)}= -e^{\i\Im w_1} {\rm f}(w_1-w_0)$.
Then, we can decompose
\begin{equation} \label{Delta1}
\Delta_{j+1}^1 - \Delta_{j}^1 =\tfrac2{\sqrt\beta} \overline{V_{j+1}{\rm f}(\Delta_{j}^1)} 
+ \tfrac2{\sqrt\beta}  \big(\overline{U_{j+1}^1 + U_{j+1}^0}\big) , \qquad  V_{j+1}= e^{\i \Im\varphi^1_{j} } \W_{n_{j+1},n_j} ,
\end{equation}
where the errors (replacing $S_{j+1}$)
\[
U_{j+1}^1 = {\rm f}(\varphi^1_{j}) \sum_{k=n_{j}+1}^{n_{j+1}} \big(\W_{k,k-1} - \i\sqrt{\tfrac1{j\mathfrak{D}}} Z_{k} e^{2\i \vartheta_{k,n_{j}}} e^{2\i\phi_{n_j}} \big) , \qquad 
U_{j+1}^0= e^{\i \varphi^0_{n_j}}  \sum_{k=n_{j}+1}^{n_{j+1}} \W_{k,k-1} {\rm f}(\varphi^0_{k-1,n_j}).
\]
In this expansion,  $\{V_{j}\}, \{U_{j}^1\}, \{U_{j}^0\}$ are all  martingale increments with respect to the filtration 
$\{\F_{n_j}\}$ and using that $\W$ is a martingale sum with $\|Z_k\|_2^2 \lesssim 1$,  
\[
\|\W_{n_{j+1},n_j} \|_2^2 \le \delta_{n_j}^2 {\textstyle \sum_{k = n_{j}+1}^{n_{j+1}}}\|Z_k\|_2^2 \lesssim \delta_{n_j}^2 \mathfrak{D} = j^{-1}
\]
since $\W_{k,k-1}(z) = \i \delta_k  Z_k  e^{2\i\theta_k}e^{2\i\phi_{k-1}} $  (see Lemma~\ref{lem:lin3}).
Similarly, using that $\phi_{k-1} = \phi_{n_j}+ \vartheta_{k-1,n_j}+ \boldsymbol{\chi}_{k-1,n_j}$,
we can decompose 
\begin{equation} \label{ErrUblock}
U_{j+1}^1 = \i\, {\rm f}(\varphi^1_{j}) \sum_{k=n_{j}+1}^{n_{j+1}} \Big(
\big( \delta_k - \delta_{n_j}\big) Z_k  e^{2\i\theta_k}e^{2\i\phi_{k-1}} 
+  \delta_{n_j} Z_{k} e^{2\i \vartheta_{k,n_{j}}}e^{2\i\phi_{n_j}} {\rm f}\big(\boldsymbol{\chi}_{k-1,n_j}\big)\Big) 
\end{equation}
This sum is also a martingale with quadratic variation ($Z_k$ are independent random variables with $\E|Z_k|^2=1$ and ${\rm f}$ is 1-Lipschitz continuous and uniformly bounded by 2):
\[
[U_{j+1}^1] \lesssim \mathfrak{D} |\delta_{n_{j+1}} - \delta_{n_j}|^2
+ \mathfrak{D} \delta_{n_j}^2 \max_{k\in(n_j,n_{j+1}]}|{\rm f}\big(\boldsymbol{\chi}_{k-1,n_j}\big)|^2
\le  j^{-2} +  j^{-1}\big(\max_{k\in(n_j,n_{j+1}]}|\boldsymbol{\chi}_{k-1,n_j}| \big)^2
\]
where $\mathfrak{D}$ and using that $\mathfrak{D} \delta_{n_j}^2 = j^{-1}$. 
We now use that the  random phase $\{ \boldsymbol{\chi}_{k,n_j}\}_{k\in[n_j,n_{j-+1}]}$ is slowly varying. 
On the event $\A_\chi = \A_{\chi}(R,\mathfrak{J}_0; z)$ from Lemma~\ref{lem:varphase} with the blocks $\{n_j \}$ and $R \gg 1$, we have  
\[
\max_{j\in[\mathfrak{J}_0,\mathfrak{J}_1]}\big\{ \sqrt{j} \max_{k\in[n_j,n_{j+1}]} |   \boldsymbol{\chi}_{k,n_j} | \big\} \lesssim R
\]
where the implied constant depend only on $\delta$ since $\mathfrak{J}_1/\mathfrak{J}_0 = {\rm c}/\delta$ is independent of $N$. 
Then, on $\A_\chi$, we have  
\[
\big[{\textstyle \sum_{j = \mathfrak{J}_0}^{\mathfrak{J}_1-1}} U_{j+1}^1 \big]
= {\textstyle \sum_{j = \mathfrak{J}_0}^{\mathfrak{J}_1-1}}[U_{j+1}^1] \lesssim R  {\textstyle \sum_{j = \mathfrak{J}_0}^{\mathfrak{J}_1-1}} j^{-2}  \lesssim R \eta . 
\]
This estimate can replace the bound \eqref{Gronmartbd} from Lemma~\ref{lem:Gron}, namely 
\begin{equation} \label{U1}
\E\big( \1\{\A_\chi\} \big[{\textstyle \sum_{j = \mathfrak{J}_0}^{\mathfrak{J}_1-1}} U_{j+1}^1 \big]\big)   \lesssim R \eta  
\end{equation}
will suffice to prove that $\max_{j \in[\mathfrak{J}_0, \mathfrak{J}_1)} |\Delta_j^1| $ converges to 0 in probability since the parameter $\eta \ll 1$. 

\smallskip

We proceed similarly to control $\{U_{j}^0\}$. Using the evolution \eqref{varphi0}, one has for $k\in(n_j, n_{j+1}]$, 
\[
\varphi^0_{k,n_j} = 2 \sum_{n=n_j+1}^{k} \Big(\delta_{n} \tfrac{-\i\lambda}{\sqrt{N\varrho^2}}  + \tfrac{1}{\sqrt\beta}  \overline{\W_{n,n-1}{\rm f}(\varphi^0_{n-1})}\Big) .
\] 
The drift term is $\O\big(\delta_{n_j} \tfrac{\mathfrak{D}}{\sqrt{N\varrho^2}}\big)  = \O\big( \tfrac{\eta}{j}\big)  $
by \eqref{Srv} and the quadratic variation of the martingale 
$[\cdot] \lesssim \delta_{n_j}^2\mathfrak{D} = j^{-1}$
(this is a deterministic bound since $|{\rm f}| \le 2$). 
In particular, the drift is negligible and using a martingale tail-bound, for any $R\ge 1$, 
\[
\P\Big[\max_{k\in[n_j,n_{j+1}]} |  \varphi^0_{k,n_j}| \ge R j^{(\epsilon-1)/2}  \mathfrak{J}_0^{-\epsilon/2} \big] \lesssim \exp\big(-cR^2 (j/ \mathfrak{J}_0)^\epsilon\big)
\]
Then, by a union bound (using that $\mathfrak{J}_1/\mathfrak{J}_0 \le C(\delta)$), for $R\ge 1$
\[
\P\Big[\underbrace{\Big\{\max_{j\in[\mathfrak{J}_0,\mathfrak{J}_1]}\big\{ \sqrt{j} \max_{k\in[n_j,n_{j+1}]}  |  \varphi^0_{k,n_j}| \big\}  \ge C R\Big\}}_{\A_\varphi^{\rm c}} \Big] \lesssim \exp(-cR^2) . 
\]
Exactly as above, 
$\displaystyle\sqrt{[U_{j+1}^2]} \lesssim j^{-1/2} \max_{k\in(n_j,n_{j+1}]}|\boldsymbol{\chi}_{k-1,n_j}| \lesssim R j^{-1}$ on the event ${\A_\varphi}$ and we conclude that 
\begin{equation} \label{U2}
\E\big( \1\{\A_\varphi\} \big[{\textstyle \sum_{j = \mathfrak{J}_0}^{\mathfrak{J}_1-1}} U_{j+1}^2 \big]\big)   \lesssim R \eta . 
\end{equation}
Hence, applying Lemma~\ref{Gron} using \eqref{U1}, \eqref{U2},
since $\P[\A_\chi^{\rm c}],  \P[\A_\varphi^{\rm c}] \to 0$ as $N\to\infty$ followed by $R\to\infty$ (see Lemma~\ref{lem:varphase} with $\P[\A_m^{\rm c}] \lesssim \exp(-c N^\epsilon)$ in this regime) and $\eta \to 0$ as $N\to\infty$, we deduce that 
$\max_{j \in[\mathfrak{J}_0, \mathfrak{J}_1)} |\Delta_j^1|  \to 0$ in probability as $N\to\infty$.

\smallskip

Finally, exactly as in \eqref{ErrUblock},
\[
\left(\W_{n_{j+1}, n_j} -\tfrac{e^{2\i\phi_{n_{j}}}S_{j+1}}{\sqrt{2j}} \right)
\]
are martingales with the same control as $U_{j+1}^1$, then \eqref{Wblockbd} follows as in \eqref{U1}.
\end{proof}

\paragraph{Step3: Gaussian coupling.}
We now proceed to replace the driving noise $\{S_j\}$ in the evolution \eqref{varphi1} by independent complex Gaussians. This relies on the following coupling:

\begin{lemma} \label{lem:coupling}
Assume $\eta(N) \ll 1$ in such a way that $\eta \mathfrak{R}^3 \gg 1 $ (Assumptions~\ref{ass:bulk}) as $N\to\infty$. 
We enlarge our probability space with a sequence of i.i.d.~random variables $\mathscr{Z}_j \sim \boldsymbol{\gamma}_{\C} $ and an independent complex Gaussian random variable $\mathscr{G}_\delta$ with $\E|\mathscr{G}_\delta|^2 =  \log({\rm c}/\delta) +\o(1)$ as $N\to\infty$  such that for any $p\ge 1$, in Wasserstein-$p$ distance, 
\[
\limsup_{N\to\infty}\sup_{j\ge \mathfrak{J}_0(N)} \d_\W^p(S_j,\mathscr{Z}_j ) =0 , \qquad 
\limsup_{N\to\infty} \d_\W^p(\G_{N,m},\mathscr{G}_\delta) =0 .
\]
Moreover, on this enlarged probability space, we consider the filtration for $j\ge \mathfrak{J}_0$, 
\begin{equation} \label{newfiltration}
\widehat{\F}_j \coloneqq \F_{n_j} \vee \sigma\big( \mathscr{Z}_i  : i\le j \big) .
\end{equation}
Then, the processes $\{\phi_{n_j}(z)\}$ and $\{\varphi^1_j(\cdot;z)\}$ are $\{\widehat{\F}_j \}$ adapted and for every $j$, the random variables  $\{S_k,\mathscr{Z}_k\}_{k>j}$ are independent of $\widehat{\F}_j $.
\end{lemma}

The proof of Lemma~\ref{lem:coupling} relies on some estimates in Wasserstein distance and it is postponed at the end of this section. For now, we record its consequence for the evolution \eqref{varphi1}.
Replacing the noise $\{S_j\}$ by $\{\mathscr{Z}_j\}$ and also the drift terms\footnote{
$\sum_{k=n_{j} +1}^{n_{j+1}} \delta_k(z) = \sqrt{\mathfrak{D}/j}\big(1+\O(j^{-1})\big).$
by \eqref{deltablocking}},
we consider a new process  $\{\varphi^2_j(\lambda;z) \}_{j\ge \mathfrak{J}_0}$ which satisfies the evolution
\begin{equation} \label{varphi2}
\varphi^2_j(\lambda;z)  = 2\partial\bpsi_{m}(w_\lambda,z) +  \sum_{i= \mathfrak{J}_0}^{j-1}  \bigg({2\i\lambda}\sqrt{\frac\eta{i}}
+\i{\rm c}_\beta \big(1-e^{-\i \Im\varphi^2_{i}(\lambda;z) }\big)  \frac{e^{-2\i\phi_{n_{i}}(z)}\overline{\mathscr{Z}_{i+1}}}{\sqrt i}  \bigg) .
\end{equation}
By construction, this process is also $\{\widehat{\F}_j \}$  adapted  and the initial data $\partial\bpsi_{m}(\cdot,z)$ is measurable in $\widehat{\F}_{\mathfrak{J}_0}$.
As usual, we control the difference with \eqref{varphi1} using a stochastic Gr\"onwall inequality.
Let 
\[
\Delta_j^2(\lambda;z)   \coloneqq  \varphi^2_j(\lambda;z)-\varphi^1_j(\lambda;z), \qquad  \Delta_{\mathfrak{J}_0}^2(\lambda;z) =0 . 
\]

\begin{proposition} \label{prop:Delta2}
Assume $\eta(N) \ll 1$ in such a way that $\eta \mathfrak{R}^3 \gg 1 $ (Assumptions~\ref{ass:bulk}) as $N\to\infty$, then
\[
\max_{j \in[\mathfrak{J}_0, \mathfrak{J}_1)} |\Delta_j^2|  \overset{\P}{\to}   0. 
\]
Moreover, as $N\to\infty$, 
\[
\big|\W_{N,m}(z)- \i{\textstyle \sum_{j = \mathfrak{J}_0}^{\mathfrak{J}_1-1}} \tfrac{e^{2\i\phi_{n_{j}}(z)}\mathscr{Z}_{j+1}}{\sqrt{2j}} \big| \overset{\P}{\to}   0. 
\]
\end{proposition}

\begin{proof}
The process $\{\Delta^2_j\}_{k\ge \mathfrak{J}_0}$ satisfies the evolution (with $\mathfrak{D} = \eta \varrho^2 N$ )
\[
\Delta_{j+1}^2- \Delta_j^2 = \underbrace{ {2\i\lambda}\sqrt{\tfrac\eta{j}}
\big({\textstyle 1- \sqrt{\tfrac{j}{\mathfrak{D}}}\sum_{k=n_{j} +1}^{n_{j+1}}  \delta_k}  \big)}_{= U^1_{j+1}}
+ \i{\rm c}_\beta \underbrace{\tfrac{(1-e^{-\i \Im\varphi^1_{j}})   e^{-2\i\phi_{n_j}}}{\sqrt j} 
\big(\overline{\mathscr{Z}_{j+1}-S_{j+1}}\big)}_{U^2_{j+1}}
-\i{\rm c}_\beta \overline{V_{j+1}} (1-e^{-\i \Im\Delta^2_{j}})
\] 
where $V_{j+1} = e^{\i \Im\varphi^2_j}\frac{e^{2\i\phi_{n_{j}}(z)}\mathscr{Z}_{j+1}}{\sqrt j} $. 
In particular, $\{U^1_{j+1}\}$ are deterministic errors and $\{V_{j+1}\}$, $\{U^2_{j+1}\}$ are both martingale increments 
$\big( \E[V_{j+1}|\widehat{\F}_j] = \E[U^2_{j+1}|\widehat{\F}_j] =0\big)$ with 
\[
\|V_{j+1}\|_{2,n_j}^2 \lesssim j^{-1} , \qquad  \|U^2_{j+1}\|_{2,n_j}^2 \lesssim \gamma  j^{-1} , 
\]
where $\gamma(N) = \sup_{j\ge \mathfrak{J}_0(N)} \d_\W^2(S_j,\mathscr{Z}_j )$.
Moreover, the deterministic errors satisfy (with $\mathfrak{J}_0= \delta/\eta$),
\[
\sum_{j\ge \mathfrak{J}_0} |U^1_{j+1}| \lesssim \sqrt{\eta} \sum_{j\ge \mathfrak{J}_0} j^{-3/2}
\lesssim  \delta^{-1/2}\eta . 
\]
Then, as $\eta\ll 1$ and  $\gamma \ll 1$ according to Lemma~\ref{lem:coupling},
the first claim follows from Lemma~\ref{lem:Gron}. 

\smallskip

The second claim is a consequence of \eqref{Wblockbd}, since we can (deterministically) bound the bracket
\[
\big[{\textstyle \sum_{j = \mathfrak{J}_0}^{\mathfrak{J}_1-1}} \tfrac{e^{2\i\phi_{n_{j}}}S_{j+1}}{\sqrt{2j}} -{\textstyle \sum_{j = \mathfrak{J}_0}^{\mathfrak{J}_1-1}} \tfrac{e^{2\i\phi_{n_{j}}}\mathscr{Z}_{j+1}}{\sqrt{2j}} \big] 
\le \gamma \log({\rm c}/\delta) .
\]
as $\gamma \ll 1$, this quantity also converges to 0 in probability as $N\to\infty$.
\end{proof}

\begin{proof}[Proof of Lemma~\ref{lem:coupling}]
Recall that $\{Z_k\}$ are independent complex random variables with (Lemma~\ref{lem:Z}), 
\[
\E |Z_{k}|^2 =1 , \qquad \E Z_{k}^2 =  (\cos \theta_k) e^{-\i\theta_k} , \qquad
\]
In addition to \eqref{Srv}, we define for $j\ge \mathfrak{J}_0$, 
\[
G_{j+1}^X := \sqrt{j} \sum_{k=n_{j}+1}^{n_{j+1}} \delta_k  X_k , \qquad    G_{j+1}^Y := \sqrt{j} \sum_{k=n_{j}+1}^{n_{j+1}} \delta_k  Y_k . 
\]
Under the assumptions of Definition~\ref{def:noise}, the random variables $\{G_{j+1}^X ,G_{j+1}^Y\}$ are real-valued, independent with the same variance 
\[
\E (G_{j+1}^{X2}) = \E(G_{j+1}^{Y2}) =   \sum_{k=n_{j}+1}^{n_{j+1}} j\delta_k^2 = 1+ \O(\eta) 
\]
using that $j\ge \delta/\eta$ ($\delta$ is fixed) and $\delta_k$ are decreasing with
\begin{equation} \label{deltablocking}
\delta_{n_{j}}^2 (n_{j+1}-n_j) = j^{-1} , \qquad \delta_{n_{j+1}}^2 (n_{j+1}-n_j) = (1+j)^{-1} .
\end{equation}
The sequence of random variables $\{S_{j+1}, G_{j+1}^X ,G_{j+1}^Y\}$ is  also independent. 

\paragraph{Covariance structure.}
The total variance of $S_{j+1}$ is 
$ \E|S_{j+1}|^2 = \mathfrak{D}^{-1}  \sum_{k=n_j+1}^{n_{j+1}} \E |Z_{k}|^2  = 2 $
and 
\[
\E S_{j+1}^2 =  \frac2{\mathfrak{D}}  \sum_{n_j<k\le n_{j+1}} \E(Z_{k}^2) e^{4\i \vartheta_{k,n_{j}}} 
= \frac{2}{\mathfrak{D}} \E(Z_{n_j}^2)  \sum_{n_j<k\le n_{j+1}} e^{4\i \vartheta_{k,n_{j}}}  
+ \O\bigg( \max_{n_j<k\le n_{j+1}} \big| \E(Z_{k}^2) - \E(Z_{n_j}^2) \big| \bigg).
\]
Using that $\theta_k \mapsto \E(Z_{k}^2) $ is Lipchitz-continuous and $|\theta_{k}-\theta_{n_j}| \le (k-n_j) \delta_{n_j}^2 \le j^{-1}$ for $k\in[n_j,n_{j+1}]$ (Lemma~\ref{lem:theta}), we have 
\[
\max_{n_j<k\le n_{j+1}} \big| \E(Z_{k}^2) - \E(Z_{n_j}^2) \big| \lesssim j^{-1} \lesssim \eta
\]
Then, by \eqref{oscest1} with block-length $\mathfrak{D} \ll N^{1/3}$  (here the condition $\eta \mathfrak{R}^3 \gg 1$ 
implies that $\mathfrak{D} \ll \varrho^{-1} \ll N^{1/3} $ 
and $n_j \ge m \ge \delta N$ with $\delta>0$ fixed), we have
\[
\bigg|\E(Z_{n_j}^2) \sum_{k=n_{j}+1}^{n_{j+1}}  e^{4\i \vartheta_{k,n_{j}}} \bigg| 
\lesssim \frac{1}{\sin \theta_{n_j}} \le \frac{1}{\sin \theta_m} .
\]
By construction, $\tfrac{1}{\sin \theta_m} = \sqrt{\frac{m}{\delta N \varrho^2}} \le \delta^{-1/2} \varrho^{-1}$, then 
\[
\E S_{j}^2 \lesssim \frac{1}{\mathfrak{D}\varrho} +\eta \ll 1 .
\]
Similarly, 
\[
\E(S_{j+1} G_{j+1}^X) =   \sqrt{\frac{j}{\mathfrak{D}}}  \sum_{k=n_{j}+1}^{n_{j+1}} \delta_k   e^{2\i \vartheta_{k,n_{j}}} , \qquad 
\E(S_{j+1} G_{j+1}^X) =   \sqrt{\frac{j}{\mathfrak{D}}}  \sum_{k=n_{j}+1}^{n_{j+1}} \delta_k  e^{-\i \theta_k} e^{2\i \vartheta_{k,n_{j}}}
\]
Using that $\sqrt{j \mathfrak{D}}\delta_k = 1+ \O(j^{-1})$, $\theta_k = \theta_{n_j} + \O(j^{-1})$  for $k\in[n_j,n_{j+1}]$ and  $\big|\sum_{k=n_{j}+1}^{n_{j+1}}  e^{2\i \vartheta_{k,n_{j}}} \big| \lesssim \tfrac{1}{\sin \theta_{n_j}} \le \tfrac{\delta^{-1/2}}{\varrho}$ (Lemma~\ref{lem:osc}) and $j\ge \epsilon/\eta$, we obtain
\[
|\E(S_{j+1} G_{j+1}^X)| , |\E(S_{j+1} G_{j+1}^Y)| \lesssim 
\frac1{\mathfrak{D}} {\textstyle \big|\sum_{k=n_{j}+1}^{n_{j+1}}  e^{2\i \vartheta_{k,n_{j}}} \big|} +\O(\eta) \ll 1. 
\]
So that $\{S_{j+1}, G_{j+1}^X ,G_{j+1}^Y\}$  are asymptotically uncorrelated. 

\paragraph{Central limit theorem \& coupling.}
$\{S_{j}, G_{j}^X ,G_{j}^Y\}$ are normalized linear combinations\footnote{For instance, we can write
$G_{j+1}^X =\frac{1}{\sqrt{ \mathfrak{D}}}  \sum_{n_j<k\le {n_{j+1}}} \sqrt{\gamma_k} X_k$ where 
$\gamma_k = j \mathfrak{D} \delta_k^2 \simeq 1$ in the appropriate range. }
of independent (sub-Gaussian) mean-zero random variables $\{X_k,Y_k\}$, so by the multivariate CLT, 
the above computations show that  
\[
\{\Re S_{j}, \Im S_j,  G_{j}^X ,G_{j}^Y\} \overset{\rm law}{\to} \boldsymbol{\gamma}_{\R^4} \qquad\text{as $N\to\infty$}
\]
where the limit is a standard Gaussian measure on $\R^4$ and all moments also converge (because of the sub-Gaussian condition). 
Then, we claim in Wasserstein-$p$ distance\footnote{Convergence in Wasserstein-$p$ distance is equivalent to convergence in distribution and convergence of the $p^{\rm th}$ moment.
In particular, the collection of probability laws $\{\Re S_{j}, \Im S_j,  G_{j}^X ,G_{j}^Y\}$ lie in a compact set with respect for $\d_\W^p$, so these random variables converge uniformly with respect to $\d_\W^p$.}  (for any $p\ge 1$), 
\[
\limsup_{N\to\infty}\sup_{j\ge \mathfrak{J}_0(N)} \d_\W^p\big(\{\Re S_{j}, \Im S_j,  G_{j}^X ,G_{j}^Y\}, \boldsymbol{\gamma}_{\R^4}\big) =0 . 
\]
Moreover, since $\{S_{j}, G_{j}^X ,G_{j}^Y\}$ are independent for different $j$, the convergence holds jointly, meaning that we can enlarge our probability space with a collection of independent Gaussians
$\mathscr{Z}_j \sim \boldsymbol{\gamma}_{\C} $, $\{\mathscr{G}_j^X, \mathscr{G}_j^Y\} \sim \boldsymbol{\gamma}_{\R^2}$ for $j\in\N$ such that (by definition of the Wasserstein distance), for any $p\ge 1$, 
\[
\limsup_{N\to\infty}\sup_{j\ge \mathfrak{J}_0(N)} \E\big[ {\rm dist}\big(\{S_j,  G_{j}^X ,G_{j}^Y\} ,  \{\mathscr{Z}_j , \mathscr{G}_j^X, \mathscr{G}_j^Y\}\big)^p \big] =0 . 
\]

\paragraph{Convergence of the random variable $\G_{N,m}$.}
We can rewrite
\[
\G_{N,m} = -\i \sum_{m<k\le N} \delta_k Z_k =  -\i \sum_{\mathfrak{J}_0\le j<\mathfrak{J}_1}\frac1{\sqrt{2j}} \bigg(G_{j+1}^X+ e^{-\i \theta_{n_j}} G_{j+1}^Y + \sum_{k=n_{j}+1}^{n_{j+1}}  \delta_k Y_k \big(e^{-\i \theta_{k}} - e^{-\i \theta_{n_j}} \big) \bigg) .
\]
The last sum is an error term, using that $|\theta_{k}-\theta_{n_j}| \le (k-n_j) \delta_{n_j}^2 \le j^{-1}$ for $k\in[n_j,n_{j+1}]$, it second moment is bounded by 
$\delta_m\sum_{j\ge \mathfrak{J}_0} j^{-3/2} \ll 1 $ as $N\to\infty$. 

Within the previous coupling, define the (complex Gaussian) random variable 
\[
\mathscr{G}_\delta  \coloneqq  -\i \sum_{\mathfrak{J}_0\le j<\mathfrak{J}_1}\frac{\mathscr{G}_j^X + e^{-\i \theta_{n_j}} \mathscr{G}_j^Y}{\sqrt{2j}}  
\]
Then, using that $\{G_j^X, G_j^Y,\mathscr{G}_j^X, \mathscr{G}_j^Y\}$ are (mean-zero)  independent and independent for different $j$: 
\[
\E\big[{\rm dist}(\G_{N,m},\mathscr{G}_\delta)^2\big] \le \sum_{\mathfrak{J}_0\le j<\mathfrak{J}_1} \frac{1}{2j}
\E\big[{\rm dist}\big\{G_j^X, G_j^Y\},\{\mathscr{G}_j^X, \mathscr{G}_j^Y\}\big)^2 \big] \ll 1
\]
using that $\mathfrak{J}_1/\mathfrak{J}_0 = {\rm c}/\delta$
so that the previous sum is $\le \log({\rm c}/\delta) \gamma(N)$ where $\gamma(N)\ll 1$ controls the Wasserstein-$2$ distance with the Gaussians uniformly for $j\ge \mathfrak{J}_0$.
Moreover, we immediately verify that since $\mathfrak{J}_0\gg 1$:
\[
\E|\mathscr{G}_\delta|^2  = \sum_{\mathfrak{J}_0\le j<\mathfrak{J}_1} \frac{1}{j}  =  \log({\rm c}/\delta) +\underset{N\to\infty}{o(1)} . \qedhere
\]
\end{proof}

\paragraph{Step 4: Continuum approximation.}
Finally, we can  replace the sum in \eqref{varphi2} by a stochastic integral. To this hand, we let
$t_j \coloneqq \eta(N) j$ so that  $n_j = N_0(z) +\varrho(z)^2 N t_j$  for $j \ge \mathfrak{J}_0$
(in particular $t_{\mathfrak{J}_0} = \delta$ and $t_{\mathfrak{J}_1} ={\rm c}$ with $\sqrt{\rm c} = \pi/2$).  
Then, we make a continuous-time interpolation of \eqref{varphi2}:
\[
\varphi^3_t(\lambda;z)  \coloneqq  \varphi^2_j(\lambda;z) , \qquad t\in[t_j,t_{j+1}).
\]
Since $\{e^{2\i\phi_{n_{j}}(z)} \mathscr{Z}_{j+1}\}$ is a $ \widehat{\F}_j $ adapted  sequence of i.i.d.~complex Gaussians, enlarging our probability space and filtration $\{\widehat{\F}_j\}$, there is a standard complex Brownian motion $\{\zeta_t^z\}_{t\in\R_+}$ 
such that 
\begin{equation}\label{BMinc}
\i \, \overline{e^{2\i\phi_{n_{j}}(z)} \mathscr{Z}_{j+1}} = \eta^{-1/2} \int_{t_j}^{t_{j+1}} \hspace{-.2cm}\d \zeta_t^z ,\qquad\text{for } j\ge \mathfrak{J}_0 ,
\end{equation}
and the sequence $\{\zeta_{t}^z; t\le t_j\}$ is adapted to $\{\widehat{\F}_j\}$.
Now, let $\{\varphi^4_t(\lambda;z) \}_{t\ge\delta}$ be a solution of the stochastic sine equation \eqref{Zeq:alpha} on  driven by  $\{\zeta_t^z\}_{t\in\R_+}$:
\begin{equation} \label{varphi4}
\varphi^4_t(\lambda;z)  = 2\partial\bpsi_{m}(w_\lambda,z) + 2\i \lambda \int_\delta^t  \frac{\d s}{\sqrt{s}} 
+ {\rm c}_\beta  \int_\delta^t  \big(1-e^{-\i \Im\varphi^4_{s}(\lambda;z) }\big) \frac{\d \zeta_s^z}{\sqrt s} . 
\end{equation}

We can compare the two process $\{\varphi^4_t(\lambda;z) \}_{t\ge\delta}$ and $\{\varphi^3_t(\lambda;z) \}_{t\ge\delta}$ using a stochastic Gr\"onwall inequality.
Let 
\[
\Delta_t^3(\lambda;z)   \coloneqq  \varphi^4_t(\lambda;z)-\varphi^3_t(\lambda;z), \qquad  \Delta_{\delta}^3(\lambda;z) =0 . 
\]

\begin{proposition} \label{prop:Delta3}
Assume $\eta(N) \ll 1$ as $N\to\infty$, then
\[
\max_{t \in[\delta,{\rm c}]} |\Delta_t^3|  \overset{\P}{\to}   0. 
\]
Moreover, as $N\to\infty$, 
\begin{equation} \label{Wcvg}
\bigg| \overline{\W_{N,m}(z)}-  \int_\delta^{\rm c} \frac{\d \zeta_s^z}{\sqrt{2s}}  \bigg| \overset{\P}{\to}   0. 
\end{equation}
\end{proposition}

\begin{proof}
By continuity of $t \mapsto \varphi^4_t$ (and using that $t \mapsto \varphi^3_t$ is a step function) it suffices to show that (as $\eta(N) \to 0$) as $N\to\infty$, 
\[
\max_{j \in[\mathfrak{J}_0, \mathfrak{J}_1)} |\Delta_{t_j}^3| \overset{\P}{\to}   0. 
\] 
By construction, $\varphi^3_{t_j} = \varphi^2_j$ along the mesh $t_j =\eta j$, then by \eqref{varphi2} and \eqref{BMinc},
one has
\[\begin{aligned}
\Delta_{t_{j+1}}^3- \Delta_{t_j}^3 
&=  2\i\lambda \bigg(\int_{t_j}^{t_{j+1}} \hspace{-.15cm}\frac{\d s}{\sqrt{s}} -\sqrt{\frac\eta{j}} \bigg)
- {\rm c}_\beta e^{-\i \Im\varphi^4_{t_j}} \int_{t_j}^{t_{j+1}}  \hspace{-.15cm}  \mathrm{f}(\varphi^4_{s}-\varphi^4_{t_j}) \frac{\d \zeta_s^z}{\sqrt s} 
+ {\rm c}_\beta  \overline{\mathrm{f}\big(\varphi^4_{t_j}\big)}  \int_{t_j}^{t_{j+1}}  \hspace{-.15cm}\bigg( \frac{1}{\sqrt{t}} - \frac{1}{\sqrt{t_j}} \bigg)  \d \zeta^z_t \\
&\quad +\i {\rm c}_\beta   \overline{\mathrm{f}\big(-\Delta_{t_j}^3\big) V_{j+1}}
\end{aligned}\]
where $V_{j+1} = e^{\i \Im\varphi^4_{t_j}} \frac{e^{2\i\phi_{n_{j}}(z)}\mathscr{Z}_{j+1}}{\sqrt{j}}$ are independent Gaussians.  Obviously, $\E|V_{j+1}|^2 = j^{-1}$, so this equation is of type \eqref{Gron} with three errors, $U_{j+1} = U_{j+1}^1+U_{j+1}^2+ U_{j+1}^3 $, 
where $\{U^1_{j+1}\}$ are deterministic and $\{U_{j+1}^2\}$, $\{U^3_{j+1}\}$ are both $\{\widehat{\F}_j\}$-martingale increments. 

First, for the deterministic errors:
\[
\bigg|\int_{t_j}^{t_{j+1}} \hspace{-.15cm}\frac{\d s}{\sqrt{s}} -\sqrt{\frac\eta{j}} \bigg| 
\lesssim \frac{\eta^2}{t_j^{3/2}} = \frac{\sqrt{\eta}}{j^{3/2}}
\]
so that $\sum_{j\ge \mathfrak{J}_0}   |U_{j+1}^1|  \lesssim \eta $ using that the first index $\mathfrak{J}_0 = \delta \eta^{-1} $.

Second, using that $|{\rm f}| \le 2$ and $|\sqrt{t} - \sqrt{t_j}| \le t_j^{3/2} \eta$ for $t\in[t_j, t_{j+1}]$
\[
\E |U_{j+1}^3|^2 \le  8 t_j^3 \eta^3 =8 j^{-3}
\]
so that $\sum_{j\ge \mathfrak{J}_0}  \E |U_{j+1}^3|^2  \lesssim \eta^2 $. 

Finally, according to \eqref{varphi4}, for $\lambda\in\mathcal{K}$, for $t\in[t_j,t_{j+1}]$
\[
|\varphi^4_{t}-\varphi^4_{t_j}| \le \frac{C\eta}{\sqrt{t_j}} +   {\rm c}_\beta  \bigg| \int_{t_j}^t {\rm f}(\varphi^4_{s}) \frac{\d \zeta_s^z}{\sqrt s} \bigg|
\]
and the martingale part has quadratic variation $[\cdot] \lesssim \frac{\eta}{t_j} = j^{-1}$
(this is a deterministic bound as $|{\rm f}| \le 2$). 
Then, the drift term is negligible as $\eta \ll 1$ and, using Doob's inequality, we can bound
\[
\E\Big[\max_{t\in[t_j,t_{j+1}]} |\varphi^4_{t}-\varphi^4_{t_j}|^2\Big] \lesssim j^{-1} . 
\]
Consequently, since ${\rm f}$ is Lipchitz continuous,
\[
[|U_{j+1}^2] \lesssim  \frac{\eta}{t_j} \max_{t\in[t_j,t_{j+1}]} |\varphi^4_{t}-\varphi^4_{t_j}|^2
\]
and then
\[
\E |U_{j+1}^2|^2  \lesssim  j^{-2}
\]
so that $\sum_{j\ge \mathfrak{J}_0}  \E |U_{j+1}^2|^2  \lesssim \eta $.  

Altogether, the errors satisfy the conditions \eqref{Gronmartbd} from Lemma~\ref{lem:Gron} with $\varepsilon= \sqrt{\eta} \ll 1$; this proves the first claim.

\smallskip

The second claim is a consequence of Proposition~\ref{prop:Delta2}. By \eqref{BMinc}, 
\[
\int_\delta^{\rm c} \frac{\d \zeta_s^z}{\sqrt{2s}}
- \i \sum_{j = \mathfrak{J}_0}^{\mathfrak{J}_1-1} \tfrac{e^{-2\i\phi_{n_{j}}(z)}\overline{\mathscr{Z}_{j+1}}}{\sqrt{2j}}
= \sum_{j = \mathfrak{J}_0}^{\mathfrak{J}_1-1} \int_{t_j}^{t_{j+1}} \bigg(\frac1{\sqrt{2t}} - \frac1{\sqrt{2t_j}} \bigg) \d \zeta_s^z
\]
and this quantity is similar to ${\textstyle \sum_{j = \mathfrak{J}_0}^{\mathfrak{J}_1-1}} U_{j+1}^3$; its bracket is (deterministically) controlled by $[\cdot] \lesssim \eta^2 \ll 1$ as $N\to\infty$. This proves \eqref{Wcvg}. 
\end{proof}

\paragraph{Step 5: Fixing the initial condition.}
Finally, for $\epsilon>0$, let $\{\omega_t^{(\epsilon)}(\lambda;z)\}$ be the solution of the SDE
\begin{equation} \label{varphiF}
d\omega_t^{(\epsilon)}(\lambda;z) = 2\i \lambda\frac{dt}{\sqrt{t}} + \frac{{\rm c}_\beta}{\sqrt t}  \bigl(1-e^{-\i\Im\omega_t^{(\epsilon)}(\lambda;z)}\bigr)\d\zeta^z_t ,\qquad
t \ge \epsilon ,
\end{equation}
with initial data $\omega_t^{(\epsilon)} =0$ for $t\in[0,\epsilon]$.
Up to a trivial time-change ($t\leftarrow t\tau$), \eqref{varphiF} corresponds to the SDE \eqref{Zeq:alpha}, so it has the same properties (see Section~\ref{app:Sine}). 
In particular, $\{\omega_t^{(\epsilon)}(\lambda;z), t\in\R_+\}$ is continuous and, by Corollary~\ref{cor:continuity0},
for any $\tau>0$, $\{\omega_t^{(\epsilon)}(\lambda;z)\}_{t\in[0,\tau]} \to \{\omega_t(\lambda;z)= \omega_t^{(0)}(\lambda;z)\}_{t\in[0,\tau]}$ in probability as continuous process, where $\{\omega_t(\lambda;z);t\in\R_+\}$ is the unique (strong) solution with initial data $\omega_0(\lambda;z) =0$

We consider the difference with \eqref{varphi4}:
\[
\Delta_t^{4,\delta}(\lambda;z)   \coloneqq  \varphi^{(\delta)}_t(\lambda;z)-\varphi^4_t(\lambda;z), \qquad t \ge \delta  . 
\] 
\begin{proposition} \label{prop:Delta4}
Under the Assumptions~\ref{ass:bulk}, one has as $N\to\infty$,  followed by $\delta\to0$, 
\[
\max_{t \in[\delta,{\rm c}]} |\Delta_t^{4,0}|  \overset{\P}{\to}   0. 
\]
\end{proposition}

\begin{proof}
The processes $\{\varphi_t^4\}_{t\ge \delta}$ and $\{\omega_t^{(\delta)}\}_{t\ge \delta}$ both satisfy the stochastic sine SDE driven by the same Brownian motion $\{\zeta^z_t\}_{t\in\R_+}$, with different initial conditions:
$\Delta_\delta^{4,\delta} = 2\partial\bpsi_{m}(w_\lambda,z)$. 
Then  by Proposition \ref{prop:continuity0} (with $\sqrt{\rm c} =\pi/2$), 
there are a numerical constants $C,c>0$ so that for any small $\varepsilon, \epsilon>0$, 
\[
\P\Big[ \Big\{ \sup_{t\in[\delta,{\rm c}]} |\Delta_t^{4,\delta}|  \ge C \varepsilon \delta^{-\epsilon/2}  \Big\} \cap \big\{ |\partial\bpsi_{m}(w_\lambda,z)| \le \varepsilon \big\} \Big] \lesssim \delta^{c\beta \epsilon} . 
\]
By \eqref{varphi_ini}, choosing $\varepsilon = \delta^\epsilon$, we conclude that 
\[
\lim_{\delta\to 0} \limsup_{N\to\infty}  \P\Big[ \Big\{ \sup_{t\in[\delta,{\rm c}]} |\Delta_t^{4,\delta}|  \ge C \delta^{\epsilon/2}  \Big] =0 . 
\]
Since $\{\omega_t^{(\delta)}\}_{t\ge \delta} \to \{\omega_t\}_{t\ge 0}$  in probability as $\delta\to0$, 
this also implies the claim with $\delta=0$.
\end{proof}

\subsection{Convergence to the stochastic sine equation: Proof of Proposition~\ref{prop:sine}.} \label{sec:sine}
For $z\in(-1,1)$, let $\ell_t(z) \coloneqq N_0(z) +\lfloor \varrho(z)^2 N t \rfloor$ for $t\in [0, {\tau}]$  so that  $\ell_{\tau}(z) = N$ and define the \emph{microscopic relative phase}:
\begin{equation*} \label{varphiN}
\varphi^{(N)}_t(\lambda;z) \coloneqq 2\big( \psi_{\ell_t}\big(z+\tfrac{\lambda}{N\varrho(z)}\big) - \psi_{\ell_t}(z) \big) , 
\qquad t\in [0, {\tau}] .  
\end{equation*}

\paragraph{Convergence in probability.}
For a fixed $\lambda$, we compare $\{\varphi_{t}^{(N)}(\lambda) ; t\in [\delta, {\tau}]\}$ to the solution of the complex sine equation $\{\omega_{t}(\lambda) ; t\in\R_+\}$ with $\omega_{0}=0$.
The starting point is that $\varphi^{(N)}_t(\lambda;z)=2\partial\bpsi_{\ell_t}(w_\lambda,z)$ satisfies an approximate sine equation with a small initial condition as in Lemma~\ref{lem:lin3}.
Based on the approximation scheme described at the beginning of Section~\ref{sec:homo}, we can bound 
for $0<\delta \le \epsilon$,
\[
\max_{t\in[\epsilon,{\tau}]} |\omega_t  - \varphi^{(N)}_t |  
\le  \max_{K_0< k \le K_1}|\Delta^0_k(\lambda;z) |
+\max_{j \in[\mathfrak{J}_0, \mathfrak{J}_1)} |\Delta_j^1| 
+ \max_{j \in[\mathfrak{J}_0, \mathfrak{J}_1)} |\Delta_j^2| 
+ \max_{t \in[\delta,{\tau}]} |\Delta_t^3| 
+\max_{t \in[\delta,{\tau}]} |\Delta_t^4| .
\]
Then, combining  Propositions
\ref{prop:Delta0},
\ref{prop:Delta1},
\ref{prop:Delta2},
\ref{prop:Delta3},
\ref{prop:Delta4},
all these approximation errors converge to 0 in probability as $N\to\infty$ followed by $\delta\to0$ (under the Assumptions~\ref{ass:bulk} and choosing the mesh parameter $\eta(N) \ll 1$ in such a way that $\eta \mathfrak{R}^3 \gg 1$ as $N\to\infty$).
Namely, we construct a Brownian motion $\{\zeta_t^z\}_{t\in\R_+}$  on our probability space, which is driving the SDE \eqref{varphiF} such that for any fixed $\epsilon>0$, as $N\to\infty$, 
\begin{equation} \label{varphiuni}
\max_{t\in[\epsilon,{\tau}]} |\omega_t(\lambda;z)  - \varphi^{(N)}_t(\lambda;z) |   \overset{\P}{\to} 0. 
\end{equation}
This implies convergence of the finite dimensional distributions of the process $\{\varphi^{(N)}_{\tau}(\lambda;z) ; \lambda\in\R\}$.

\smallskip

We also record that by Lemma~\ref{lem:coupling} and \eqref{Wcvg}, as $N\to\infty$, 
\[
\big(\G_{N,m}, \overline{\W_{N,m}}\big)  \overset{\P}{\to} (\mathscr{G}_\delta, \mathscr{W}_\delta)
\]
where the limits $\mathscr{G}_\delta, \mathscr{W}_\delta$ are mean-zero Gaussians, $\displaystyle\mathscr{W}_{\delta} = \int_\delta^{\tau} \frac{\d \zeta_s^z}{\sqrt{2s}}$ and $\mathscr{G}_\delta$ is independent of $\{\zeta_t^z\}_{t\in\R_+}$. \\
Thus, by \eqref{phi_ini}, for some deterministic sequence $\Lambda_{N,m}(z) \in \R$, as $N\to\infty$, 
\begin{equation} \label{GWcvg}
\big(\phi_{N,m} - \Lambda_{N,m}\big)   \overset{\P}{\to} \tfrac1{\sqrt\beta}\Im\big(\mathscr{G}_\delta + \mathscr{W}_\delta \big) .
\end{equation}

\paragraph{Weak convergence.}
We now establish the joint weak convergence of 
$\varphi_{\tau}^{(N)}(\lambda) = 2\big( \psi_{N}\big(z+\tfrac{\lambda}{N\varrho(z)}\big) -  \psi_{N}(z) \big)$ and $\phi_{N}(z)\, [2\pi]$ in the  sense of finite dimensional marginals. 
Let $X_N \coloneqq \big(\varphi_{\tau}^{(N)}(\lambda_j) : 1 \leq j \leq p\big)$ and 
$Y_\epsilon  \coloneqq \big(\omega_{\tau}^{(\epsilon)}(\lambda_j) : 1 \leq j \leq p\big)$ for fixed $\{\lambda_j\} \in\R^p$ and  $\epsilon\ge 0$ in terms of the solutions of \eqref{varphiF}. 

Since the random variable  $\phi_{N}(z)\, [2\pi]$ takes values in $\R/[2\pi]$ and $\boldsymbol{\alpha}$ is uniform in $[0,2\pi]$, by Weyl's equidistribution criterion, it suffices to show that for any function $g :\C^p \to \R$, $1$-Lipchitz continuous with $|g|\le 1$, and for any $k\in\Z$, 
\begin{equation} \label{Weylcvg}
\lim_{N\to\infty}\Exp\left[ e^{\i k\Im \psi_N} g(X_N) \right] = \1\{k=0\} \E\left[ g(Y_0) \right] .
\end{equation}
If $k=0$, the claim follows directly from \eqref{varphiuni}, so we can assume that $k\in\N$.
Then we introduce again the parameter $m=N_0 + \lfloor \delta (N-N_0)\rfloor$ for a fixed $\delta>0$. 
Using the convergence in probability \eqref{varphiuni}, \eqref{GWcvg}, we obtain 
\begin{equation} \label{Weylcond}
\big|\E\left[ e^{\i k\phi_N} g(X_N) | \F_m \right] \big|  =  \big| \E\big[e^{\i k \Im(\mathscr{G}_\delta + \mathscr{W}_\delta)/\sqrt{\beta}} g(Y_0) |\F_m\big] \big| +\underset{N\to\infty}{\o(1)} .
\end{equation}
In particular, the extra phase $ e^{\i k\Im \Lambda_{N,m}+\i k \phi_m }$ cancels while taking modulus.

Here $\mathscr{W}_\delta$ and $Y_0$ are not independent. However, for $\epsilon\ge \delta$, we can replace $Y_0$ by $Y_\epsilon$, up to a small extra error (by Corollary~\ref{cor:continuity0}),  and decompose 
$ \mathscr{W}_\delta =  \mathscr{W}_\epsilon+  \mathscr{W}_{\epsilon,\delta}$ where 
$\displaystyle\mathscr{W}_{\epsilon,\delta} = \int_\delta^\epsilon \tfrac{\d \zeta_s}{\sqrt{2s}}$ is  independent of $(\mathscr{G}_\delta, \mathscr{W}_\epsilon, Y_\epsilon)$  \big($Y_\epsilon$ is $\{\zeta_s\}_{s\ge \epsilon}$ measurable while $\mathscr{W}_{\epsilon,\delta}$ is independent of $\{\zeta_s\}_{s\ge \epsilon}$ \big). 
Hence, 
\[\begin{aligned}
\E\big[e^{\i k \Im(\mathscr{G}_\delta + \mathscr{W}_\delta)/\sqrt{\beta}} g(Y_0) |\F_m\big] 
&=\E\big[e^{\i k \Im(\mathscr{G}_\delta + \mathscr{W}_\delta)/\sqrt{\beta}} g(Y_\epsilon) |\F_m\big] 
+\underset{\epsilon\to0}{\o(1)}  \\
& =\E\big[e^{\i k \Im(\mathscr{W}_{\epsilon,\delta})/\sqrt{\beta}} |\F_m\big] 
\E\big[e^{\i k \Im(\mathscr{G}_\delta )/\sqrt{\beta}} |\F_m\big]
\E\big[e^{\i k \Im(\mathscr{W}_\epsilon)/\sqrt{\beta}} g(Y_\epsilon) |\F_m\big] +\underset{\epsilon\to0}{\o(1)} .  \\
\end{aligned}\]
Here, one cannot use the independent Gaussian $\mathscr{G}_\delta$ for avaeraging since 
$\E(\Im\mathscr{G}_\delta)^2 \simeq \varrho(z_*) \log\big(\tfrac{1}{\delta\varrho(z_*)^2+z_*^2}\big)$ where $z_* = \lim z(N) \in[-1,1]$ and this quantity vanishes in the edge case $z_*\in\{\pm1\}$. 
However,  $\Im\mathscr{W}_{\epsilon,\delta}$ is also Gaussian with variance $\E(\Im\mathscr{W}_{\epsilon,\delta})^2 \simeq \log(\epsilon/\delta) + \O(\epsilon)$. Hence, for $k\in\N$, 
\[
\big| \E\big[e^{\i k \Im(\mathscr{G}_\delta + \mathscr{W}_\delta)/\sqrt{\beta}} g(Y_0) |\F_m\big] \big|
\le \exp\big(- \tfrac{k^2}{4\beta}\big(\log(\epsilon/\delta)+\O(\epsilon)\big) \big)  +\underset{\epsilon\to0}{\o(1)} . 
\]
The LHS of \eqref{Weylcond} is independent of $(\delta,\epsilon)$, so taking the limit as $\delta\to0$, followed by $\epsilon\to0$, we conclude that 
\[
\limsup_{N\to\infty} \big|\E\left[ e^{\i k\phi_N} g(X_N) | \F_m \right] \big| =0
\]
This proves \eqref{Weylcvg}, which completes the proof of proof of Proposition~\ref{prop:sine}.\hfill\qed

\clearpage
\appendix

\section{The complex (stochastic) Sine equation} \label{app:Sine}

In this section, we review the properties of the \emph{log--structure--function} of the Stochastic $\zeta_\beta$ function, and develop some basic properties of it.
Here $\beta>0$ is a fixed parameter. 
This function is the solution of the SDE:
\begin{equation}\label{Zeq:alpha}
d\omega_t(\lambda) = \i2\pi \lambda \d\sqrt{t} + \sqrt{\frac{2}{\beta t}} 
\bigl(1-e^{-\i\Im\omega_t(\lambda)}\bigr) \d Z_t ,
\qquad
t >0,\, \lambda\in\R,
\end{equation}
where $\{Z_t\}_{t\in\R_+}$ is a complex Brownian motion with brackets $[Z_t,Z_{t}]=0$ and $[Z_t,\overline{Z}_t] = 2t$. This equation has a simple structure, $\Im\omega_t$ satisfies an autonomous SDE with a drift proportional to $\lambda\in\R$ and $\Re\omega_t$ is a martingale depending on $\Im\omega_t$. In fact, for a fixed $\lambda\in\R$, there is a standard real Brownian motion $\{X_t\}_{t\in\R_+}$ so that 
\[
\d\Re\omega_t = \frac{4\sin(\Im\omega_t/2)}{\sqrt{2\beta t}} \d X_t , \qquad
t >0. 
\]
The equation \eqref{Zeq:alpha} for $\{\Im \omega_t(\lambda); \lambda\in\R\}_{t\in\R_+}$ first appeared n the seminal work  \cite{KillipStoiciu} to describe the counting function of the sine$_\beta$ point process.  
The equation is singular as $t \to 0$, but there is a unique continuous family of strong solution 
$\{\omega_t(\lambda) ; t\ge 0 , \lambda\in\R \} $ with  the initial condition $\omega_0(\lambda) = 0$ (see \cite{KillipStoiciu} or Proposition \ref{prop:Holdersine}). This SDE can also be considered for $\lambda\in\C$ in which case the solution is analytic for  $\lambda\in\C$, and then the  stochastic $\zeta_\beta$ function can be represented by \eqref{eq:zeta} in terms of the solution. 
This is a direct consequence of the fact that \eqref{Zeq:alpha} only differs from the equation for the \emph{log--structure--function} of $\zeta_\beta$ introduced in \cite{ValkoViragZeta} by a simple time change.  

\begin{lemma}\label{Zlem:tchange}
Let $g : t\in [-\infty,0] \mapsto  e^{\beta t/2}/(2\pi)^2$ and $u : t\in\R_+  \mapsto (t/2\pi)^2 $. \\ 
The process $\{ \omega_{g(t)} ; t\in[-\infty,0] , \lambda\in\R\}$ corresponds to the structure function of the stochastic $\zeta_\beta$ function as defined in \cite{ValkoViragZeta}. 
Moreover, the process $\{\omega_{u(t)} ; t\in\R_+, \lambda\in\R\}$ satisfies the SDE (59) from \cite[Proposition 4.5]{KillipStoiciu}.
\end{lemma}

\begin{proof}
Let $\widehat{\omega}_t \coloneqq  \omega_{g(t)}$ for $t \in [-\infty,\infty)$.
Observe that $\d \sqrt{g(t)} = \tfrac{f(t)}{2\pi} \d t$ for $t\in\R$, with $f(t)\coloneqq \tfrac{\beta}{4}e^{\beta t/4}$ and one has for $t\in\R$, 
\[
d\widehat{\omega}_t =
\i2\pi \lambda \d\sqrt{g(t)}
+{\rm c}_\beta\bigl(1-e^{-\i\Im \hat\omega_t}\bigr) 
\frac{\d Z_{g(t)}}{\sqrt{g(t)}}
\]
with ${\rm c}_\beta^2 = 2/\beta$ and 
and $\d Z_{g(t)} = \i a(t) \d\widehat{Z}_t  $ for  a new complex Brownian motion $\{\widehat{Z}_t\}_{t\in\R}$ where $a(t) = \sqrt{g'(t)}$ (so the brackets match). 
Since $a(t) = {\rm c}_\beta^{-1} \sqrt{g(t)}$, we conclude that  
\[
\d\widehat{\omega}_t = \i \lambda f(t) \d t+\i \bigl(1-e^{-\i\Im \widehat{\omega}_t}\bigr) \d \widehat{Z}_{t},
\]
which is the same SDE as in \cite[Corollary 50 -- Corollary 51]{ValkoViragZeta}. 

Let $\psi_t \coloneqq  \omega_{u(t)}$ for $t\in\R_+$. By a similar computation, $\d \sqrt{u(t)} = \d t /2\pi$
and $\sqrt{(\log u(t))'} = \sqrt{2/t} $, so that 
\[
\d\psi_t = \i\lambda\d t - \frac{2}{\sqrt{\beta t}} 
\bigl(1-e^{-\i\Im\psi_t(\lambda)}\bigr) \d \widehat{Z}_t 
\]
for another complex Brownian motion $\{\widehat{Z}_t\}_{t\in\R_+}$. 
Taking $\Psi_t = \Im \psi_t$, we obtain the SDE \cite[(59)]{KillipStoiciu}.
\end{proof}

We begin by reviewing a few elementary properties of solutions of \eqref{Zeq:alpha}; see also \cite{ValkoVirag} where most of these are developed in greater generality, although in a different time scale.
In particular, $\{\Im \omega_t(\lambda); \lambda\in\R\}_{t\in\R_+}$ satisfies an autonomous SDE with a unique strong continuous solution with initial condition $\omega_0=0$; see also \cite[Proposition 4.5]{KillipStoiciu}.

\begin{lemma}\label{lem:szbasics}
Let $\delta>0$ and let $\{\omega_t^{(\delta)}(\lambda) ; \lambda\in\R \}_{t\in\R_+}$ be the solution of \eqref{Zeq:alpha} with  $\omega_{t} = 0$ for $t\in[0,\delta]$.  Then:
\begin{enumerate}
\item (Positivity) For any $\lambda > 0,$ the function $t\mapsto \Im \omega_t^{(\delta)}(\lambda)$ for $t>\delta$ is almost surely positive.
\item (Symmetry) For any $\lambda \in \R$, $\{-\omega_t^{(\delta)}(\lambda) ; t\in\R_+\} \overset{\rm law}{=}\{\omega_t^{(\delta)}(-\lambda) ; t\in\R_+\}$.
\item (Translation invariance) For $\lambda_1, \lambda_2 \in \R$, $\{ \omega_t^{(\delta)}(\lambda_1)-\omega_t^{(\delta)}(\lambda_2) ; t\in\R_+\}\overset{\rm law}{=}\{\omega_t^{(\delta)}(\lambda_1-\lambda_2); t\in\R_+\}$. 
\item (Monotonicity) Almost surely, $\omega_t^{(\delta)}(0)=0$ and, for any $t> \delta$, $\lambda\in\R \mapsto \Im \omega_t^{(\delta)}(\lambda)$ is increasing. 
\item (Bounded influence) For any $\lambda > 0,$ if $\{\tilde{\omega}_t(\lambda) ; t\ge \delta\}$ is another solution of \eqref{Zeq:alpha}  with $0<\Im \tilde{\omega}_{\delta}(\lambda)<2\pi$, then the difference $\Im (\tilde{\omega}_t-\omega_t^{(\delta)})(\lambda) \in (0,2\pi)$ for all $t \ge \delta.$
Generally, if $\Im \tilde{\omega}_{\delta}(\lambda)>0$, then $\Im \tilde{\omega}_t(\lambda)> \Im \omega_t^{(\delta)}(\lambda)>0$ for all $t \ge \delta.$
\end{enumerate}
\end{lemma}

\begin{proof}
These properties can be verified by elementary manipulations of the SDE.  
Let $ \omega_t(\lambda) \coloneqq \omega_t^{(\delta)}(\lambda)$ for $t\ge 0$ and $\lambda\in\R$.
2 and 3 are direct consequences of the linearity of the drift in the parameter $\lambda$ and the invariance properties of complex Brownian motion. 
1 is a consequence of the drift being positive and the diffusion coefficient vanishes linearly in $\Im \omega_t$ as $\Im \omega_t \to 0$.  
Similarly, for $k\in\N$ if  $\Im \omega_{\tau_k} = 2\pi k$ for a stopping time $\tau_k>0$, then  $\Im \omega_t > 2\pi k$ for $t>\tau_k$. 
In particular $\tau_1<\tau_2,$ etc. 
4 is an immediate consequence of 1 and 3. 
5 also follows by a similar argument; there is a standard real Brownian motion $\{X_t\}$ so that the process $\varpi_t \coloneqq 2 \Im (\tilde{\omega}_t-\omega_t)(\lambda)$ satisfies the autonomous SDE: 
\[
\d \varpi_t  = {\rm c}_\beta \sin(\varpi_t) \frac{\d X_t}{\sqrt t} , \qquad t > \delta
\]
with initial data $\varpi_\delta \in (0,2\pi)$. 
As the diffusion coefficient vanishes linearly as $\varpi_t \to \{0,\pi\}$, this process never hits these values.  Moreover, if $\{\tilde{\omega}_t(\lambda) ; t\ge \delta\}$ is a solution of \eqref{Zeq:alpha} with $\tilde{\omega}_{\delta}(\lambda)=2\pi k$ for a $k\in\N$, then  $\tilde{\omega}_{t}(\lambda) = \omega_t(\lambda)+2\pi k$ for all $t\ge \delta$. 
\end{proof}

Using these properties, we show that  \eqref{Zeq:alpha} has a unique strong solution defined on $[0,1]$ with initial condition $\omega_0(\lambda)=0$ for $\lambda\in\R$. 
Moreover his solution also has the following continuity estimates:

%
%

\begin{lemma} \label{lem:sine0}
The SDE \eqref{Zeq:alpha} has a unique strong solution with $\omega_0(\lambda)=0$ and for any $\lambda\in\R$, $t \in \R_+\mapsto \omega_0(\lambda) $ is continuous. 
This solution also satisfies the properties 1--5 from Lemma~\ref{lem:szbasics} with $\delta=0$ and  
the space-time scaling invariance; for any $\gamma >0$, $\{\omega_{\gamma^2 t}(\gamma \lambda) : t \ge 0, \lambda \in \R\}\overset{\rm law}{=}\{\omega_{t}(\lambda) : t \ge 0, \lambda \in \R\}.$
\end{lemma}

\begin{proof}
We start by constructing the solution, which we show for the case of $\lambda \geq 0$ (a symmetric argument can be used for $\lambda \le 0$ and obviously $\omega_t(0)=0$ for all $t\ge 0$). 
For any $\delta>0$ and $\lambda \le 0$, the initial value problem for 
$\{\omega_t^{(\delta)}(\lambda) ; t\in\R_+\}$ is well-posed since the SDE \eqref{Zeq:alpha} has Lipschitz coefficients for $t\ge \delta$, so there is a unique strong solution, which is continuous for $t\in\R_+$. 
Then, almost surely, the function $(t,\lambda) \mapsto \Im \omega_t^{(\delta)}(\lambda)$ is non-negative and non-decreasing in $\lambda\ge 0$.  
By Property 5, for any $\lambda\ge 0$, $\{\Im \omega_t^{(\delta)}(\lambda), t\ge 0\}$ are also non-decreasing in $\delta>0$ ($\Im\omega_t^{(\delta)}(\lambda)\ge  \Im\omega_\epsilon^{(\epsilon)}(\lambda) =0$ for $t\in[0,\epsilon]$ if $\epsilon\ge \delta$). 
Then, we can define 
\[
\alpha_t(\lambda) = \sup_{\delta>0} \Im\omega_t^{(\delta)}(\lambda) , \qquad \lambda ,t \geq 0.
\]
To ensure that this supremum is finite, by  \eqref{Zeq:alpha},  we observe that for any $\delta>0$,
\[
\Exp\Im\omega_t^{(\delta)}(\lambda) \le 2\pi \lambda \sqrt{t}.
\]
Hence, by Fatou's lemma,
\begin{equation} \label{Ealphabd}
\Exp \alpha_t(\lambda)  \leq 2\pi\lambda\sqrt{t},
\end{equation}
so that almost surely, $\alpha_t(0)=0$ for $t\ge 0$ and, if $\lambda>0$, $0<\alpha_t(\lambda) < \infty$ for $t>0$ with $\alpha_0(\lambda) = 0$.  By dominated convergence for stochastic integrals (the diffusion coefficient is bounded away from 0), it holds for $t>s>0$, 
\[
\alpha_t(\lambda) = \alpha_s(\lambda) 
+\int_s^t\biggl\{ 2\pi\lambda \d\sqrt{u} +  
{\rm c}_\beta   \frac{\Im\bigl((1-e^{-\i\alpha_u})\d Z_u \bigr)}{\sqrt u}  
\biggr\}.
\]
In particular, from the existence of strong solutions, for a fixed $\lambda\ge 0$, $t\mapsto \alpha_t(\lambda)$ is continuous on $[0,\infty)$ and 
$t\mapsto \alpha_t(\lambda)$ is a positive submartingale (the drift is non-negative). Thus, by Doob's maximal inequality and \eqref{Ealphabd}, for any $c > 0 $
\[
\P\Big[ \sup_{0 \leq u \leq t} \alpha_u(\lambda) > c\Big]= \lim_{s\to 0} \P\Big[ \sup_{s \leq u \leq t} \alpha_u(\lambda) > c\Big] \le c^{-1} \lambda\sqrt t , \qquad t>0, \lambda\ge0 .
\]
The limit follows from monotone convergence. 
Then, for $0<\gamma<1/2$, by a Borel--Cantelli argument, 
\[
2^{k \gamma}  \sup_{0 \leq u \leq 2^{-k}} \alpha_u(\lambda) \to0 \qquad\text{talmost surely as $k\to\infty$}, 
\]
hence $\alpha_t(\lambda)/t^{\gamma} \to 0$ almost surely, locally uniformly in $\lambda$ (by monotonicity again).
This allows us to define the stochastic integral for any  $\lambda, t \geq 0$, 
\begin{equation} \label{psisdeint} 
\omega_t(\lambda)
\coloneqq
\int_0^t \biggl\{ \i 2\pi\lambda  \d\sqrt{u} +  
{\rm c}_\beta   \frac{\Im\bigl((1-e^{-\i\alpha_u})\d Z_u \bigr)}{\sqrt u}  \biggr\} 
\end{equation}
and $\alpha_t(\lambda)$ is the imaginary part of both sides. 
Then, the  properties 1--5  of Lemma~\ref{lem:szbasics} (with the same proof with $\delta=0$) follow for the process $\{\omega_t(\lambda) ;\lambda\in\R \}_{t\ge 0}$ and 
\[
\alpha_t(\lambda) = \Im \omega_t(\lambda) = \lim_{\delta\to0} \Im\omega_t^{(\delta)}(\lambda) \qquad\text{for $t\in\R_+$ and  $\lambda\in\R$}.
\]
Finally, the space-time distributional scaling invariance of the solution follows from the scaling law for Brownian motion  and the invariance of the drift $(\lambda ,t) \mapsto  \lambda \d\sqrt{t} $ by rescaling $(\lambda ,t)\leftarrow (\gamma^2 t, \gamma \lambda)$ for  any $\gamma>0$
\end{proof}

Lemma~\ref{lem:sine0} does not ensure continuity in $\lambda$ of the process $\{\lambda\mapsto\omega_t(\lambda); \lambda\in\R\}_{t\in\R_+}$. 
We now prove that the process is (almost surely) H\"older continuous in both variables.

\begin{proposition}\label{prop:Holdersine}
The (strong) solution $\{\omega_t(\lambda); \lambda\in\R\}_{t\in\R_+}$ of \eqref{Zeq:alpha} with $\omega_0 = 0$ satisfies, for any $0 < \delta < \min\{ \frac{\beta}{2+\beta}, \frac12\}$, there is an $\epsilon > 0$ so that for any compact $\mathcal{K} \subset \R$,
\begin{equation} \label{psiestcont}
\E\biggl(\sup_{0 \leq s \leq 1} \sup_{\lambda_1,\lambda_2 \in \mathcal{K}} 
\frac{|\omega_s(\lambda_1) -\omega_s(\lambda_2)|}{s^{1/2}\log(\tfrac{1}{s+1})|\lambda_1 - \lambda_2|^{\delta}}\biggr)^{1+\epsilon}<\infty.
\end{equation}
\end{proposition}

\begin{proof}
Without loss of generality, we assume that $\lambda\ge 0$. 
Let $\alpha_t(\lambda) \coloneqq \Im \omega_t(\lambda) $ for $\lambda\ge 0$ and $t\ge 0$.
By It\^o's rule, for $\gamma>1$ and $t > 0$
\[
\d \alpha_t^\gamma
=\gamma \alpha_t^{\gamma-1}
\biggl( 2\pi \lambda \d\sqrt{t} + \sqrt{\frac{2}{\beta t}}
\Im \bigl( \bigl(1-e^{-\i\alpha_t}\bigr) dZ_t \bigr) \biggr)
+
\frac{4}{\beta t}
{\gamma(\gamma-1)}\alpha_t^{\gamma-2}
\sin( \alpha_t/2)^2 \d t.
\]
Taking expectation and using that $4\sin( \alpha_t/2)^2 \le \alpha_t^2$,  we obtain the inequality
\[
\Exp\alpha_t^\gamma
\leq \gamma
\int_0^t
\biggl( 
\frac{\pi \lambda}{\sqrt{s}} \Exp \alpha_s^{\gamma-1}+
\frac{\gamma-1}{\beta s}\Exp \alpha_s^{\gamma}
\biggr)\,ds
\]
Strictly speaking, we should first derive this inequality for the process $\alpha_t^{(\delta)}$ with $\delta>0$, in case all moments exist, and then take a limit as $\delta\to0$ using monotone convergence (by Lemma~\ref{lem:sine0}). 

By Jensen's inequality and \eqref{Ealphabd}, we obtain for $\gamma\le 2$, 
\[
\Exp \alpha_t^{\gamma-1}(\lambda) \le \lambda^{\gamma-1} t^{(\gamma-1)/2}.
\]
so that by comparison: $\Exp\alpha_t^\gamma \le w(t)$ where $w(t)$ solves the ODE 
\[
w(t)= \pi \lambda^\gamma t^{\frac\gamma2}  + \frac{\gamma(\gamma-1)}{\beta} \int_0^t
\frac{w(s)}{s} ds \qquad\text{with $w(0)=0$.}
\]
This equation can be solved (uniquely) by $c_\gamma \lambda^\gamma  t^{\frac\gamma2}$ with $c_\gamma = \frac{1}{4} -\frac{\gamma-1}{2\beta}$ provided that  $\gamma < 1+\frac\beta2$. 
This shows that for $1\le \gamma < \min(1+\frac\beta2,2)$, 
\[
\Exp \alpha_t^{\gamma}(\lambda) \le c_\gamma\lambda^{\gamma} t^{\frac\gamma2}.
\]

By Doob's maximal inequality ($t\mapsto \alpha_t^{\gamma}(\lambda)$ is a submartingale for $\gamma\ge 1$), one has under the same conditions; 
\begin{equation} \label{Doobmaxalpha}
\E\big( \max_{0 \leq u \leq t}\alpha_u^\gamma(\lambda) \big)
\lesssim \lambda^{\gamma} t^{\frac\gamma2}.
\end{equation}
Since for $\lambda_1 \ge \lambda_2\ge 0$, $\{ \alpha_t(\lambda_1)-\alpha_t(\lambda_2) ; t\ge 0 \} \overset{\rm law}{=} \{ \alpha_t(\lambda_1-\lambda_2) ; t\ge 0 \} $, using  Kolmogorov continuity criterion, for a given $t>0$, 
$\lambda \mapsto  t^{-\frac12}\max_{0 \leq u \leq t}\alpha_u(\lambda)$ is $\delta$--H\"older continuous with $\delta<\min\{\frac{\beta}{2+\beta},\frac12\}$ and there is $\epsilon(\delta)>0$ sufficiently small so that for any $t>0$, 
\begin{equation} \label{alphaestcont}
\Exp
\biggl(
\max_{0 \leq s \leq t} \sup_{\lambda_1,\lambda_2 \in \mathcal{K}}
\frac{
|\alpha_s(\lambda_1) -\alpha_s(\lambda_2)|}
{
t^{1/2}
|\lambda_1 - \lambda_2|^{\delta}
}
\biggr)^{1+\epsilon}
\lesssim 1
\end{equation}
for some constant depending only on $(\beta, \delta, \mathcal{K})$. 

To deduce \eqref{psiestcont} for $\alpha=\Im\psi$ using  \eqref{alphaestcont}, we 
break the maximum in dyadic scales (replacing $\max_{k\ge 0}$ by $\sum_{k\ge 0}$)
\[\begin{aligned}
\Exp
\biggl(\max_{0 \leq s \leq 1} 
\sup_{\lambda_1,\lambda_2 \in \mathcal{K}} 
\frac{
|\alpha_s(\lambda_1) -\alpha_s(\lambda_2)|}{s^{1/2}\log(\tfrac{1}{s+1})
|\lambda_1 - \lambda_2|^{\delta}}\biggr)^{1+\epsilon}
&\lesssim \Exp
\bigg( \max_{k\ge 0}  \max_{e^{-k-1} \leq s \leq e^{-k}} 
\sup_{\lambda_1,\lambda_2 \in \mathcal{K}}
\frac{
|\alpha_s(\lambda_1) -\alpha_s(\lambda_2)|}
{
ke^{-k/2}
|\lambda_1 - \lambda_2|^{\delta}}\bigg)^{1+\epsilon}\\
&\lesssim \sum_{k\ge 0} k^{-1-\epsilon} <\infty
\end{aligned}\]

It remains to prove \eqref{psiestcont} for $\rho  = \Re\omega $.
This process is defined through the SDE \eqref{psisdeint} and it is a martingale so, using the Burkholder--Davis--Gundy inequality with $\gamma\ge 1$, 
\[
\Exp \sup_{0 \leq s \leq t} |\rho_s(\lambda_1)-\rho_s(\lambda_2)|^{\gamma}
\leq \frac{C_\gamma}\beta
\Exp\bigg[
\biggl(\int_0^t 
\frac{|\alpha_s(\lambda_1)-\alpha_s(\lambda_2)|^2}{s} ds\biggr)^{\gamma/2}\bigg] .
\]
By the previous estimates, this expectation is finite  for $\gamma \le 1+\epsilon$. 

We break this integral into $(te^{-k},te^{1-k}]$ for $k\ge1$, using subadditivity of $x\in\R_+ \mapsto x^{\gamma/2}$ for $\gamma\le 2$, we obtain 
\[\begin{aligned}
\Exp \sup_{0 \leq s \leq t} |\rho_s(\lambda_1)-\rho_s(\lambda_2)|^{\gamma}
& \lesssim \sum_{k\ge 1} 
\Exp\Big[\Big( \max_{te^{-k} \leq s \leq te^{1-k}}
|\alpha_s(\lambda_1)-\alpha_s(\lambda_2)| \Big)^\gamma\Big]
\biggl(\underbrace{\int_{te^{-k}}^{te^{1-k}} \frac{ds}{s}}_{=1}\biggr)^{\gamma/2}
\end{aligned}\]
Now, if $1\le \gamma < \min(1+\frac\beta2,2)$, using \eqref{Doobmaxalpha} and translation-invariance, we obtain
\[
\Exp \sup_{0 \leq s \leq t} |\rho_s(\lambda_1)-\rho_s(\lambda_2)|^{\gamma}
\lesssim |\lambda_1-\lambda_2|^{\gamma}
\sum_{k\ge1}  t^{\frac\gamma2} e^{-\frac\gamma2 k}
\lesssim
|\lambda_1-\lambda_2|^{\gamma} t^{\gamma/2}
\]
where the implied constants depend only on $(\beta,\gamma)$. 
Thus, using  Kolmogorov continuity criterion again, we conclude that \eqref{alphaestcont} also holds for the process  $\rho  = \Re\omega$.
Just as above, we can upgrade this estimate using a dyadic decomposition scheme to obtain \eqref{psiestcont}. 
\end{proof}

\begin{proposition} \label{prop:continuity0}
Let $0<\delta<1$ and let $\{\Delta_t(\lambda)\}_{t\ge \delta}$ be the difference of two solutions of \eqref{Zeq:alpha} with $\lambda\in\R$ fixed for $t\ge \delta$ with different initial conditions at time $\delta$. One has for any $\tau \ge 1$,  $\varepsilon>0$ and ${\rm c}>0$, 
\[
\P\bigg[ \Big\{ \sup_{t\in[\delta,\tau]} |\Delta_t|  \ge 2\varepsilon  \tau^{\rm c}\delta^{-\rm c}  \Big\} \cap \big\{ |\Delta_\delta| \le \varepsilon \big\} \bigg] \lesssim \delta^{\tfrac{\rm c^2 \beta}{4}} . 
\]
\end{proposition}

\begin{proof}
We consider two solutions with different initial conditions at time $\delta>0$, so the difference $\Delta_t$ satisfies the SDE, 
\[
\d\Delta_t = \sqrt{\frac{2}{\beta t}} \bigl(1- e^{-\i\Im \Delta_t}\bigr) e^{-\i\Im\psi_t}\d Z_t , \qquad t\ge \delta
\]
and we assume that $|\Delta_\delta| \le \varepsilon$. 
Write writing $\Delta_t/2 =  \rho_t + \i\alpha_t$, since $\{Z_t\}$ is a complex Brownian motion, introducing a new complex Brownian motion $\{W_t\}$ with $\d W_t  =  e^{-\i \alpha_t} e^{-\i\Im\psi_t}\d Z_t$, we obtain 
\[
\frac{\d \Delta_t}{2} = \i \sqrt{\frac{2}{\beta t}} \sin(\alpha_t) \d W_t , \qquad t \ge \delta.
\]
This yields an autonomous equation for the imaginary part  $\{\alpha_t\}$, writing $\{W_t =X_t- \i Y_t\}$, 
\begin{equation} \label{ODEalpha}
\d \alpha_t = \sqrt{\frac{2}{\beta t}} \sin(\alpha_t) \d X_t , \qquad \d\rho_t =  \sqrt{\frac{2}{\beta t}} \sin(\alpha_t) \d Y_t . 
\end{equation}

Consider the exponential martingale:
\[
M_t := 
\exp \biggl( \int_{\delta}^t 
\sqrt{\frac{2}{\beta s}} 
\frac{\sin(\alpha_s)}{\alpha_s} \d X_s - \frac12 \int_{\delta}^t 
\frac{2}{\beta s}
\bigg(\frac{\sin(\alpha_s)}{\alpha_s}\bigg)^2 \d s
\biggr) , \qquad t \ge \delta.
\]
By It\^o's formula ($\{X_t\}$ is a standard Brownian motion), 
\[
\d M_t^{-1} = -\bigg(\sqrt{\frac{2}{\beta t}}  \frac{\sin(\alpha_t)}{\alpha_t} \d X_t
- \frac{2}{\beta t} \bigg(  \frac{\sin(\alpha_t)}{\alpha_t}\bigg)^2 \d t \bigg) M_t^{-1}
\]
In particular, the bracket $\d  \langle \alpha_t,  M_t^{-1} \rangle = -\frac{2}{\beta t}  \frac{\sin(\alpha_t)^2}{\alpha_t} \d t $, so that
\[
\d( \alpha_t M_t^{-1} ) =
M_t^{-1} \d\alpha_t + \alpha_t \d M_t^{-1} +\d \langle \alpha_t,  M_t^{-1} \rangle = 0. 
\]
Then, since $M_\delta =1$,  we have $\alpha_t = \alpha_\delta M_t$ for $t\in[\delta,1]$. 
In particular, $\alpha_t\neq 0$ almost surely (if $\alpha_\delta \neq 0$). 

Let $S : = \sup_{t\in[\delta,\tau]} M_t $ and define the martingale, 
\[
R_t : =  \int_{\delta}^t 
\sqrt{\frac{2}{\beta s}} \frac{\sin(\alpha_s)}{\alpha_s} \d X_s ,\qquad
[R_t] = \int_{\delta}^t \frac{2}{\beta s} \bigg(\frac{\sin(\alpha_s)}{\alpha_s}\bigg)^2 \d s \le \frac{2}{\beta} \log(\tau\delta^{-1}) , \qquad t \in [\delta,\tau] .
\]
Then, $M_t \le \exp R_t$   and using a martingale tail-bound, for any ${\rm c}>0$
\[
\P\left[S\ge  \tau^{\rm c}\delta^{-\rm c}\right]
\le \P\left[\sup_{t\in[\delta,\tau]} R_t \ge  {\rm c} \log(\tau\delta^{-1}) \right]
\lesssim \exp\left( -  \frac{{\rm c}^2 \log(\tau\delta^{-1})}{4/\beta} \right) \le \delta^{\tfrac{\rm c^2 \beta}{4}} . 
\]
Similarly, by \eqref{ODEalpha}, $\rho_t = \rho_\delta + S_t$ where the martingale 
\[
S_t : =  \int_{\delta}^t 
\sqrt{\frac{2}{\beta s}} \sin(\alpha_s) \d Y_s ,\qquad
[S_t] = \int_{\delta}^t \frac{2}{\beta s} \sin(\alpha_s)^2 \d s \le \frac{2}{\beta} \alpha_\delta^2 S  \log \delta^{-1}  , \qquad t \ge \delta.
\]
Then, using a martingale tail-bound, 
\[\begin{aligned}
\P\left[ \Big\{ \sup_{t\in[\delta,\tau]} |\rho_t| \ge 2\varepsilon \Big\} \cap \{S\le  \tau^{\rm c}\delta^{-\rm c}\} \cap  \cap \big\{ |\Delta_\delta| \le \varepsilon \big\}  \right]
& \le \P\left[ \Big\{ \sup_{t\in[\delta,\tau]} |S_t| \ge \varepsilon \Big\} \cap \Big\{\sup_{t\in[\delta,\tau]} [S_t] \le   \varepsilon^2  \delta^{-\rm c/2}  \Big\}  \right] \\
&\lesssim  \exp\left( - \delta^{-\rm c/2}  \right) .
\end{aligned}\]
This probability is negligible, so we conclude that 
\[
\P\bigg[ \Big\{ \sup_{t\in[\delta,\tau]} |\Delta_t|  \ge 2\varepsilon \tau^{\rm c} \delta^{-\rm c}  \big\} \cap \big\{ |\Delta_\delta| \le \varepsilon \big\} \bigg] \le \P\left[S\ge  \tau^{\rm c}\delta^{-\rm c}\right] + \O\left(\exp(- \delta^{-\rm c/2}) \right)
\lesssim \delta^{\tfrac{\rm c^2 \beta}{4}} . \qedhere
\]
\end{proof}

\begin{corollary} \label{cor:continuity0}
Let $\delta>0$ and let $\{\omega_t^{(\delta)}(\lambda) ; \lambda\in\R \}_{t\in\R_+}$ be the solution of \eqref{Zeq:alpha} with  $\omega_{t} = 0$ for $t\in[0,\delta]$.
For any fixed $\lambda\in\R$ and $\tau>0$, one has as $\delta\to0$, 
\[
\max_{t\in[0,\tau]}\big| \omega_t^{(\delta)}(\lambda)- \omega_t(\lambda)\big|  \overset{\P}{\to}   0. 
\]
\end{corollary}

\begin{proof}
Note that in the proof of Lemma~\ref{lem:sine0}, we have already established that for a fixed $\lambda\in\R$,  $ \Im\omega_t^{(\delta)}(\lambda) \to \Im\omega_t(\lambda)$ as continuous processes on $\R_+$, almost surely as $\delta\to0$, so the statement is in fact about $\Re\omega_t^{(\delta)}(\lambda)$. 
It can be proved directly using Propositions~\ref{prop:Holdersine} and~\ref{prop:continuity0}. Let $\Delta_t \coloneqq \omega_t^{(\delta)}(\lambda)- \omega_t(\lambda)$ for $t\ge 0$ (with $\lambda>0$ fixed). One has 
\[
\max_{t\in[0,\tau]}|\Delta_t| \le  \max_{t\in[0,\delta]}|\omega_t(\lambda)| + \max_{t\in[\delta,\tau]}|\Delta_t|  . 
\]
Proposition~\ref{prop:Holdersine} (with $\lambda_2=0$ so that $ \omega_t(\lambda_2) =0$ for $t\ge 0$) implies that, using Markov's inequality, for a small $c>0$,  for any $\varepsilon>0$, 
\[
\P\Big[ \max_{t\in[0,\delta]}|\omega_t(\lambda)| \ge \varepsilon \Big]  \lesssim  \varepsilon^{-1} \delta^{1-c} .
\]
Then, by Proposition~\ref{prop:continuity0}, we conclude that there is a constant $c>0$ such that if $\delta\tau \ll 1 $, 
\[\begin{aligned}
\P\Big[\max_{t\in[0,\tau]}|\Delta_t| \ge 2\varepsilon \Big] 
&\le \P\Big[\big\{\max_{t\in[\delta,\tau]}|\Delta_t|  \ge \varepsilon\big\} \cap \big\{ |\Delta_\delta| \le \varepsilon \big\} \Big]
+ \P\Big[ \max_{t\in[0,\delta]}|\omega_t(\lambda)| \ge \varepsilon \Big] \\
&\lesssim \delta^c(1+ \varepsilon^{-1}) .
\end{aligned}\]
This proves the claim.
\end{proof}

\section{Pr\"ufer phase for the characteristic polynomials} \label{app:charpoly}

\todo[inline]{Fix notations}

The monic characteristic polynomials of the tridiagonal matrix model \eqref{def:trimatrix} are the sequence
\begin{equation*} 
\widehat{\Phi}_n(z) \coloneqq \det [z-({4}{N\beta})^{-1/2}\mathbf{A}]_{n} , \qquad n\in\N, \, z\in\C .
\end{equation*}
With this normalization, the zeros of $\E\widehat{\Phi}_{n}$ (a rescaled Hermite polynomial of degree $n$) lie in the interval $\mathcal{I}_n \coloneqq \big(-\sqrt{n/N}, \sqrt{n/N}\big)$ with an asymptotically semicircular density. The goal of this section is to introduce a polar representation, or \emph{Pr\"ufer phase}, for the characteristic polynomials that holds in the \emph{elliptic regime}, and which will be the basis for the study of the characteristic polynomials.

The starting point for this representation is the 3-term recurrence, which we can represent via \emph{transfer matrices}. 
By a cofactor expansion, we obtain the following recursion;  for any $n \in \N_0$, 
\begin{equation} \label{rec1}
\begin{pmatrix}\widehat{\Phi}_{n+1}(z) \\ \widehat{\Phi}_{n}(z) \end{pmatrix}=  T_{n}^\beta(z) 
\begin{pmatrix}\widehat{\Phi}_{n}(z) \\ \widehat{\Phi}_{n-1}(z) \end{pmatrix} , \qquad T_{n}^\beta(z) \coloneqq \begin{pmatrix}z-\frac{b_{n+1}}{{2}\sqrt{N\beta}} & -\frac{ a_{n}^2}{{4}N\beta}  \\
1 & 0 \end{pmatrix}
\end{equation}
with initial condition $\big(\begin{smallmatrix} 1 \\ 0\end{smallmatrix}\big)$. 
Under the conditions of Definition~\ref{def:noise} (and for G$\beta$E), $\E \tfrac{ a_{n}^2}{{4}N\beta} = \tfrac{n}{4N}$ and 
\begin{equation} \label{mnoise}
T_{n}^\beta(z) = T_{n}^\infty(z)  -  \frac{1}{\sqrt{2\beta N}}\begin{pmatrix} X_n & \sqrt{\frac{n}{4N}} Y_n\\  0 & 0  \end{pmatrix} , \qquad 
T_{n}^\infty(z)  =\E T_{n}^\beta(z) =  \begin{pmatrix} z &- \frac{n}{{4}N} \\  1 & 0  \end{pmatrix} .
\end{equation}

The main behavior of this recursion is governed by the deterministic matrices $\{T^\infty_{n}(z)\}_{n\ge 0}$ and their eigenvalues.
In particular, the eigenvalues are real if $n<N_0(z) = Nz^2$ (and complex conjugates otherwise), so that if $z\in [-1,1]$, there is a turning point\footnote{In fact, if $|z|\le \mathfrak{R} N^{-1/2}$ for some constant $ \mathfrak{R} \ge 1$, there is no tuning point and the whole behavior of the recursion is \emph{elliptic}.
From this viewpoint, 0 is a special point in the spectrum (with extra symmetries) and the recurrence \eqref{rec1} in this case has already been study in \cite{TaoVu}.} where the qualitative behavior of the recursion \eqref{rec1} changes from \emph{hyperbolic} to \emph{elliptic}, which is to say the eigenvalues of $\Exp T_n^\beta(z)$ change from real to complex conjugate pairs. 
This identifies $[-1,1]$ as the \emph{support of the spectrum} of the truncated matrix  
$[\mathbf{A}]_{N}/\sqrt{4\beta N}$ and the noise is diffusive away from the turning point. 

One can explicitly diagonalize the matrix $T_{n}^\infty(z)$; according to  \cite[Lemma~1]{LambertPaquette02}, one has for $n\in \N$, 
\begin{equation} \label{conj}
T_{n}^\infty =  V_n \Lambda_n V_{n}^{-1}  ,  \qquad 
\Lambda_n \coloneqq  \sqrt{\frac{n}{4N}} \begin{pmatrix}  \lambda_{n}  &0 \\  0 & \lambda_{n}^{-1}\end{pmatrix} , \quad
V_n \coloneqq  \begin{pmatrix}  \sqrt{\frac{n}{4N}} \lambda_{n} &   \sqrt{\frac{n}{4N}} \lambda_{n}^{-1} \\ 1 &1  \end{pmatrix} 
\end{equation}
where for $z\in\R$, 
\vspace{-.3cm}
\begin{equation*} 
\lambda_{n}(z) = J\big(z\sqrt{N/n}\big)^{-1}, 
\qquad\qquad
J(w) = \begin{cases}
w \mp \sqrt{w^2-1} , &\pm w \ge 1\\
e^{-\i\arccos(w)} , & w\in [-1,1] 
\end{cases} . 
\end{equation*}

In the elliptic regime $n>N_0(z)$, it is convenient to convert the recursion \eqref{mnoise} into a (complex) scalar recursion by using the matrix 
\begin{equation} \label{V-1}
V_n^{-1}(z) = -\i\sqrt{N}\delta_n(z) \begin{pmatrix}  1 &   -\sqrt{\frac{n}{4N}}e^{-\i\theta_n(z)}\\ -1 & \sqrt{\frac{n}{4N}} e^{\i\theta_n(z)}\end{pmatrix} 
\end{equation}
where $\lambda_{n}(z) =e^{-\i\theta_n(z)}$, 
$\theta_n(z)= \arccos\big(z\sqrt{N/n}\big)$ and we used that
$\big(\sin\theta_n(z)\big)^{-1} =\sqrt{n}\delta_n(z)$.  

Then, with $\xi_n(z) = e^{\bpsi_n(z)}$ for $z\in(-1,1)$ and $n>N_0(z)$, 
\begin{equation} \label{xi1}
\begin{pmatrix}
\xi_n \\
\overline{\xi_n}
\end{pmatrix}
=  2V_n^{-1} \begin{pmatrix} \sqrt{\frac{n+1}{4N}}  \Phi_{n+1} \\  \Phi_{n} \end{pmatrix}
\end{equation}
and we recover the characteristic polynomial taking $\Phi_n = \Re\xi_n$. 
In addition, we deduce from the recursion \eqref{rec1} a scalar recursion for the process 
$\{\xi_n(z)\}_{n>N_0(z)}$. These calculations are collected in Lemma~\ref{lem:rec}, and they will be the main recurrence studied in this paper.  In this Appendix, we will develop some basic properties of these phases.

Before doing so, we note there is another way to represent the characteristic polynomials, which is to use the \emph{Pr\"ufer phases} introduced in \cite[Section 1.9.9]{Forrester}:
$\{\tphi_n(\mu); \mu\in \R\}_{n\ge 0}$ by setting $R_n(\mu) e^{\i\tphi_{n-1}(\mu)} \coloneqq {\rm x}_n(\mu)+ \i a_{n-1} {\rm x}_{n-1}(\mu)$ for $n\in\N$ and $\mu \in \R$ where $\{{\rm x}_n(\mu); \mu\in \R\}_{n\ge 0}$ are the solutions of the symmetric 3-term recursion associated with the matrix \eqref{def:trimatrix}.
We slightly modify this definition by replacing $a_k$ by $\E a_k  = \sqrt{\tfrac{k}{4N}}$. 
The next lemma is the counterpart of \cite[Proposition~1.9.10]{Forrester} in this case. 

\begin{lemma}\label{lem:PP}
Define for $n\in\N$, 
\begin{equation*} 
\widehat{\Phi}_{n+1}(z)+ \i \sqrt{\tfrac n{4N}}\widehat{\Phi}_n(z) 
\coloneqq \widehat{R}_n(z) e^{\i\tphi_n(z)}, \qquad\qquad z\in\R,
\end{equation*}
where $\widehat{R}_n(z)>0$, $\tphi_n(z)\in\R$ and $\tphi_n(+\infty) =\displaystyle\lim_{z\to+\infty} \tphi_n(z) =0$.
The phases $\{\tphi_n(z); z\in\R\}_{n\ge 0}$ are smooth on $\R$, decreasing and it holds for $n\in\N$, 
\[
\big\{z: \widehat{\Phi}_n(z) =0 \big\} = \big\{ z: \tphi_n(z) = k\pi  , k\in[n]\big\} \qquad\qquad 
\tphi_n(-\infty) =\displaystyle\lim_{z\to-\infty} \tphi_n(z) = (n+1)\pi. 
\]
Moreover, one has $| \tphi_{n+1}(z)- \tphi_n(z) |<3\pi/2 $ for all $n\ge 1$ and $z\in\R$.

\end{lemma}

\begin{proof}
For $n\ge 1$, since the zeros of the polynomial $\widehat{\Phi}_{n}(z),\widehat{\Phi}_{n+1}(z)$ interlace on $\R$, so $\widehat{R}_n(z)>0$ and the phase $\tphi_n(z)$ is determined for all $z\in\R$ by the condition  $\tphi_n(+\infty)=0$ (one has $\widehat{R}_n(z) \sim z^{n+1}$ as $n\to\infty$). 
Moreover, it follows that the phase is smooth on $\R$ with 
$\widehat{\Phi}_{n+1}(z)=\widehat{R}_n(z)\cos{\tphi_n(z)},$
$\widehat{\Phi}_{n}(z)= \sqrt{\tfrac{4N}n} \widehat{R}_n(z)\sin{\tphi_n(z)},$ for all $n\ge 1$. 
Using the 3-term recursion \eqref{rec1}, one has for $n\ge 2$, 
\[
\frac1{\tan(\tphi_n)}= \sqrt{\tfrac{4N}n} \frac{\widehat{\Phi}_{n+1}}{\widehat{\Phi}_{n}}
= \sqrt{\tfrac{4N}n}\bigg( z -\frac{b_{n+1}}{\sqrt{4 N\beta}} -\frac{a_{n}^2}{4N\beta}  \frac{\widehat{\Phi}_{n-1}}{\widehat{\Phi}_{n}}\bigg)
= \sqrt{\tfrac{4N}n}\bigg(z -\frac{b_{n+1}}{\sqrt{4N\beta}} - \frac{a_{n}^2\tan(\tphi_{n-1})}{2\sqrt{N(n-1)}\beta}  \bigg).
\]
Then, if we differentiate this equation with respect to $z\in\R$, we obtain 
\begin{equation} \label{PP3}
\frac{\tphi_n'}{\sin^2(\tphi_n)} = - \sqrt{\tfrac{4N}n}\bigg(1 - \frac{a_n^2\tphi_{n-1}'}{2\sqrt{N(n-1)}\beta\cos(\tphi_{n-1})^2} \bigg)
\end{equation}
or equivalently
\[
\tphi_n' = - \sqrt{\tfrac n{4N}}\widehat{R}_n^{-2}\big( \widehat{\Phi}_{n}^2 -  \tfrac{a_n^2\tphi_{n-1}'}{2\sqrt{N(n-1)}\beta}\widehat{R}_{n-1}^2\big)
\]
In particular, if $\tphi_{n-1}'(z)< 0$ for all $z\in\R$, then $\tphi_n'(z) <0$ as well.
At initialization ($n=1$), the above computation gives instead
\[
\frac1{\tan(\tphi_1)} = \sqrt{4N}\bigg(z -\frac{b_{2}}{\sqrt{4N\beta}} - \frac{a_1^2}{4N\beta \widehat{\Phi}_1}\bigg)
\]
and using that $\widehat{\Phi}_1'=1$, by differentiating  with respect to $z\in\R$, we obtain
\[
\tphi_1' = -\sqrt{4N}{\sin^2(\tphi_1)}\bigg(1+ \frac{a_1^2}{4N\beta\widehat{\Phi}_1^2}\bigg)
=-\frac{\widehat{\Phi}_1^2+a_1^2/4N\beta}{\widehat{R}_1^{2}\sqrt{4N}} <0
\]
This completes the proof that $\tphi_n'(z) < 0$ for all $z\in\R$ and $n\in\N$.
Since the phase $\tphi_n$ is decreasing, we easily obtain that the zeros of the polynomial $\widehat{\Phi}_n$ are the (unique) solution of the equations $\tphi_n(z) = k\pi$  for  $k\in[n]$ and also that $\tphi_n(-\infty) = (n+1)\pi$.
Finally the zeros of $\widehat{\Phi}_{n+1}$ are also solutions of $\tphi_{n}(z)= k\pi -\pi/2$ for $k\in[n]$, so  $| \tphi_{n+1}(z)- \tphi_n(z) |<3\pi/2 $ for all $z\in\R$  (since both functions are decreasing)
\end{proof} 

\begin{remark}
The \emph{Pr\"ufer phase} is related to the eigenvalue of counting function as follows;
$\lfloor\tphi_N(z)/\pi\rfloor = \#\big\{\lambda_j\ge z\big\}$ where $\{\lambda_j\}_{j=1}^N$ denotes the eigenvalues of the matrix $[({4}{N\beta})^{-1/2}\mathbf{A}]_{N}$.
\end{remark}

\vspace{0.5cm}
In what remains, we return to the phases $\phi_n$, to develop some of their basic properties, which mirror those of the $\hat{\phi}_n$.  The next proposition collects some deterministic properties of this process. 

\begin{proposition}\label{prop:prufer}
For any $n\ge 1$, there are smooth functions $\rho_n, \phi_n : \mathcal{I}_n \to \R$ such that
\begin{equation} \label{PP1}
\xi_n(z) = \exp(\bpsi_n(z)) \coloneqq \exp\big(\rho_n(z) + \i \phi_n(z)\big)
=\i\sqrt{n}\delta_n(z) 
\biggl(e^{-\i\theta_n(z)}\Phi_{n}(z)-
\sqrt{\frac{n+1}{n}}
\Phi_{n+1}(z)
\biggr).
\end{equation}
Furthermore, these functions $\phi_n$ satisfy the following properties: 
\begin{enumerate}
\item $\phi_n(0) = \tphi_n(0)+\pi/2 $ for all $n\ge 1$;  hence $\phi_1(0) \in[\pi/2,3\pi/2]$ and $| \phi_{n+1}(0)- \phi_n(0) |<3\pi/2 $ for all $n\ge 1$.
\item The zeros of $\Phi_n$ in $\mathcal{I}_n$ are exactly the solutions of the equations $\phi_n(z)-\pi/2 \in \pi \Z$, and moreover for all $z\in \mathcal{I}_n$,
\[
	\lfloor\phi_n(z)-\tfrac\pi2\rfloor_{\pi}= \pi\mathrm{N}_n([z,\infty)),
\]
where $\lfloor\cdot\rfloor_\pi$ denotes the integer part modulo $\pi$ and $\mathrm{N}_n(I)$ for $I \subset \R$ is the number of eigenvalues of the matrix $[({4}{N\beta})^{-1/2}\mathbf{A}]_{n}$ in the set $I$.
\item The map 
\[
z \in \mathcal{I}_n  \mapsto 
\phi_n(z) - \frac{3}{\tan\theta_n(z)} = \phi_n(z) - \frac{3z\sqrt{N}}{\sqrt{n-Nz^2}} 
\]
is decreasing on $\mathcal{I}_n$.
\end{enumerate}
\end{proposition}

\begin{proof}
Since the zeros of $\Phi_n,\Phi_{n+1}$ interlace, by \eqref{xi1}, the function $z\in  \mathcal{I}_n \mapsto \xi_n(z)$ does not vanish so we can define its $\psi_n =\log \xi_n $ as a smooth complex-valued function on $\mathcal{I}_n.$
Then, we  define the phases $\phi_n= \Im \psi_n$ by fixing the values of $\phi_n(0)$ for $n\ge 1$ using the relationship to the \emph{Pr\"ufer phases} for tridiagonal matrix models introduced in \cite[Section 1.9.4]{Forrester}.
As a consequence, we also obtain some monotonicity properties of the phases $\{\phi_n(z) : z\in\mathcal{I}_n\}$. 
Moreover, formula \eqref{PP1} follows immediately from \eqref{xi1}.

By \eqref{PP1} and using Lemma~\ref{lem:PP}, for $n\ge 1$, 
\begin{equation} \label{PP2}
\Big(\Phi_n(z) - \i\sqrt{\tfrac{n+1}n}\Phi_{n+1}(z)\Big)= \Big( \widehat{\Phi}_n(z) - \i\sqrt{\tfrac{4N}n} \widehat{\Phi}_{n+1}(z)\Big)\Big(  {\textstyle \prod_{k=1}^{n} \sqrt{\tfrac{k}{4N}}} \Big)^{-1} (\tfrac{2N}{\pi})^{1/4}e^{-N z^2}  = -\i R_n(z) e^{-\i\tphi_n(z)}
\end{equation}
with $R_n(z) \coloneqq \Big(  {\textstyle \prod_{k=1}^{n-1} \sqrt{\tfrac{k}{4N}}} \Big)^{-1}
\widehat{R}_n(z) (\tfrac{2N}{\pi})^{1/4}e^{-N z^2}$.
In particular, $R_n(z)>0$ for all $z\in\R$ and at $z=0$,
\[
\xi_n(0) = \Phi_n(0) - \i\sqrt{\tfrac{n+1}n}\Phi_{n+1}(0) = R_n(0) e^{-\i(\tphi_n(0)+\pi/2)}
\]
This allows us to define the phase \eqref{def:psi} using the convention that for all $n\ge1$,
\[
\phi_n(0) = \tphi_n(0)+\pi/2 . 
\]
This definition is consistent in the sense that by Lemma~\ref{lem:PP}, 
$| \phi_{n+1}(0)- \phi_n(0) |<3\pi/2 $ for all $n\ge 1$. 





\medskip  
We start by deriving a differential identity for $\phi_n$ which connects the two Pr\"ufer phases $\tphi_n$ and $\phi_n$.
Matching the real/imaginary part of \eqref{PP1} and \eqref{PP2}, we obtain
\[
\xi_n
= \frac{-\i R_n}{\sin\theta_n} \big( e^{\i\theta_n}\sin(\tphi_n) +   \cos(\tphi_n)\big)
\eqqcolon \frac{-\i R_n}{\sin\theta_n} Q_n
\]
after replacing $(\sin\theta_n)^{-1} =\sqrt{n}\delta_n$ and $\overline{\lambda_n}= e^{\i\theta_n}$. 
Since $\xi_n$ and $\frac{R_n}{\sin\theta_n}$ are smooth $0$-free functions for $z \in \mathcal{I}_n$, taking the logarithmic derivative with respect to $z$ and imaginary part to recover $\phi_n'$, we obtain 
\begin{equation} \label{app:phi_prime}
\phi_n' =\frac{\Im(Q_n'\overline{Q_n})}{|Q_n|^2}
= \frac{\sin\theta_n \cdot\tphi_n'+ \theta_n' \sin(\tphi_n) \cdot(\cos\theta_n \sin(\tphi_n) +\cos(\tphi_n))}{|Q_n|^2} .
\end{equation}
Note that $\sin(\theta_n) > 0$ for all $z \in \mathcal{I}_n$ and that 
\begin{equation} \label{app:Q_n}
|Q_n|^2 = 1-\cos\theta_n \cdot \sin(2\tphi_n) \ge 1-\cos\theta_n > 0,
\end{equation}
so that $\phi_n'$ is well-defined for all $z \in \mathcal{I}_n$.

To make the link between $\phi_n$ and the counting function, we start by recalling \eqref{PP1}; at any zero of $\Phi_n(z)$, we have 
\[
\xi_n(z) = -\i \sqrt{n} \delta_n(z) \sqrt{\tfrac{n+1}{n}} \Phi_{n+1}(z)
\]
which is on the imaginary axis, and hence $\phi_n(z)-\pi/2 \in \pi \Z$.  Conversely at any $z \in \mathcal{I}_n$ for which $\phi_n(z)-\pi/2 \in \pi \Z$, we have that $\Im e^{-\i \theta_n(z)} \neq 0$ and hence $\Phi_n(z) = 0$.  Thus the solution set of $\phi_n(z)-\pi/2 \in \pi \Z$ is exactly the set of zeros of $\Phi_n(z)$, when restricting both sets to $\mathcal{I}_n$.

Now at any $z$ for which $\tphi_n(z) \in \pi \Z$ (which is equivalent to $\phi_n(z)-\pi/2 \in \pi \Z$), we have that 
\[
\phi_n'(z) = \frac{\sin( \theta_n(z) ) \cdot \tphi_n'(z)}{|Q_n|^2} < 0.
\]
Hence it follows that the integer part $[\tfrac{1}{\pi}\phi_n(z) - \tfrac12]$ is non-increasing, and moreover it jumps by $1$ at each zero of $\Phi_n(z)$.  Since at $0$, we have $\phi_n(0) = \tphi_n(0)+\pi/2$ and since $[\tfrac1{\pi}\tphi_n(0)] = N_n([0,\infty))$, we therefore conclude that 
\[
[\tfrac{1}{\pi}\phi_n(z)-\tfrac12] = N_n([z,\infty))
\]
at all $z \in \mathcal{I}_n$.

Finally, turning to the monotonicity of $\phi_n(z) - \tfrac{3}{\tan\theta_n(z)}$, recall that $\theta_n(z)= \arccos\big(z\sqrt{N/n}\big)$ so that $\theta_n :z \in \mathcal{I}_n \mapsto [0,\pi]$ is decreasing.
Then, the first term of \eqref{app:phi_prime} is negative according to Lemma~\ref{lem:PP}, so we can bound for $z \in \mathcal{I}_n$ using \eqref{app:Q_n},
\[
\phi_n' \le  \frac{-3\theta_n'}{2|Q_n|^2} \le  \frac{-3\theta_n'}{2(1-\cos\theta_n)} 
\]
So using that $\big(\tfrac1{\tan\theta_n}\big)' =  \frac{-\theta_n'}{2(1-\cos\theta_n )} $, this concludes the proof.
\end{proof}

\begin{remark}
Since the term added to is smooth on a different scale from $\phi_n$, the phase $\phi_n$ will be in practical terms monotone on its smoothness scale.  
\end{remark}

\section{Concentration \& Martingale CLT} \label{sec:concentration}

We rely on concentration results for martingales with sub--exponential and/or   sub--Gaussian entries. 
We refer to  \cite[Chapter 2]{Vershynin} for a comprehensive introduction and we  briefly overview in this section the results that we need.

Define, for any $p >0$ and any complex valued random variable $X$, 
\[
\| X \|_p = \inf\left\{ t \geq 0 : \Exp\big( e^{|X|^p/t^p}\big) \leq 2\right\}
\asymp \sup_{k \in \N} \frac{\left(\Exp |X|^k\right)^{1/k}}{k^{1/p}}
\]
If $p\ge 1$, $X \mapsto \| X \|_p < \infty$ defines a norm on our probability space. In particular, by the triangle inequality, if $ \|X_k \|_p <\infty$ for $k\in\N$, then for any $n\in\N$, 
\[
\left\| {\textstyle \sum_{k=1}^n} X_k \right\|_p \le  {\textstyle \sum_{k=1}^n} \|X_k \|_p . 
\]

Other important properties include;
\begin{itemize}
\item
$\| \cdot \|_p$ is essentially monotone in $p\geq 1$, that is for any random variable $X$,
\[
\| X \|_p \lesssim \| X \|_q
\]
where the implied constants depend only on $(p,q)$. 
\item If $\| X \|_p < \infty,$ by Markov's inequality, for all $t \geq 0$, 
\begin{equation*}
\P \big( |X| \geq t \big)\leq 2\exp( -t^p/ \| X \|_p^p ). 
\end{equation*}
This is equivalent to the finiteness of $\|\cdot\|_p$ and the  infimum in the definition of $\| \cdot \|_p$ is attained. 
\item Control of $\| \cdot \|_p$ can also be formulated in terms of moments. For any $p,q \geq 1$, 
\begin{equation}\label{eq:Xmoment}
\E \big( |X|^q  \big) \lesssim  \| X\|_p^q 
\end{equation}
where the implied constants depend only on $(p,q)$. 
\item There is a version of Young's inequality, that is for any $p,q \geq 1$ satisfying $1/p + 1/q = 1,$
for any two random variables $X$ and $Y$, 
\[
\| XY \|_{1} \leq  \| X\|_p \| Y \|_q.
\]
See \cite[Lemma 2.7.7]{Vershynin} for details.
\end{itemize}

We now recall some important concentration inequalities for sums of random variables, which we will formulate in terms of the $\| \cdot \|_p$ norms for $p\in\{1,2\}$.   
We begin with a version of Hoeffding's inequality; 

\begin{proposition}[{\cite[Proposition 2.6.1]{Vershynin}}]\label{thm:hoeffding}
If $(X_k)_{k\in\N}$ are independent sub-Gaussian random variables (i.e.~$ \|X_k \|_2 <\infty$ for $k\in\N$), then for any $n\in\N$, 
\[
\| \textstyle{\sum_{i=1}^n (X_i-\Exp X_i }) \|_2 \lesssim \sum_{i=1}^n \| X_i \|_2^2.
\]
\end{proposition}

We can also upgrade this inequality for martingale differences. 
Let $(\F_n)_{n\ge 0}$ be a filtration of our probability space $(\Omega,\F,\P)$. Define for $p\ge 1$, 
\[
\| X \|_{p,n} = \left\| \inf\left\{ t \geq 0 :  \Exp\big( e^{|X|^p/t^p} | \F_n\big) \leq 2\right\} \right\|_{L^\infty(\P)} . 
\]
In particular, $\| X \|_{p} = \| X \|_{p,0}$ with this definition. 

\begin{proposition} \label{lem:conc2}
Let $(M_n)_{n\ge 0}$ be a $(\F_n)_{n\ge 0}$--martingale such that $M_0 = 0$.
Suppose that for any $n\in\N_0$, 
\[
\| M_{n+1}-M_n \|_{2,n}  \le \sigma_n <\infty. 
\]
Then, for any $n\in\N$, 
\[
\big\| \max_{k\le n}|M_k| \big\|_2 \lesssim \sqrt{\textstyle \sum_{k=1}^n \sigma_k^2}  . 
\]
In particular, there is a numerical constant $c>0$, so that for any $n\in\N$ and $t>0$, 
\[
\P\left[\max_{k\le n}|M_k|   \ge t\right] \le 2 \exp\left(-\frac{ct^2}{\sum_{k=1}^n \sigma_k^2} \right). 
\]
\end{proposition}

\begin{proposition} \label{lem:conc1}
Let $(M_n)_{n\ge 0}$ be a $(\F_n)_{n\ge 0}$--martingale such that $M_0 = 0$. 
Suppose that for any $n\in\N_0$, 
\[
\| M_{n+1}-M_n \|_{1,n}  \le \sigma_n <\infty. 
\]
Then, there is a numerical constant $c>0$, so that for any $n\in\N$ and $t>0$,
\[
\P\left[\max_{k\le n}|M_k|   \ge t\right] \le 2 \exp\left(-\frac{ct^2}{ \sum_{k=1}^n \sigma_k^2+ t \max_{k\le n} \sigma_k} \right). 
\]
\end{proposition} 

\section{Stochastic Gr\"onwall inequality.} \label{sec:Gron}

In this section, we prove a tail bound for comparing solutions of some equations driven by some martingale noise scaling  geometrically,  which can be viewed as a type of \emph{stochastic Gr\"onwall inequality}. 
We apply this bound several times in Section~\ref{sec:zeta} to compare solutions of different approximations of the stochastic sine equation. 

\begin{proposition} \label{prop:l2logladder}
Let  $\{U_j\}, \{V_j\}$ be two adapted sequences of (complex) random variables with respect to a filtration $\{\F_j\}$ and assume that $\{V_j\}$ are martingale increments \big($\E[V_{j+1}|\F_j] =0$\big) and for a $\delta \in[0,\frac16) $, 
\[
\Exp( |V_j|^2~| \mathscr{G}_{j-1}) \leq C_V j^{-1}, \qquad
\|V_j\|_{1} \leq j^{-1/2+\delta}, \qquad 
\text{for } j\ge j_0, 
\]
with $j_0\in\N$. 
Let ${\rm f} : \C \to \R$ be a $1$-Lipschitz continuous function with ${\rm f}(0)=0$. 

Let $\{A_j\}$ be a solution of 
\[
A_{j+1} = A_j + U_{j+1} + V_{j+1}F(A_j),
\quad\text{for}
\quad j \geq j_0
\]
with $A_{j_0} = 0$. 
Let $j_1 \in \N$ with $j_1^2 \leq e^{j_0^\delta}$ and $T\coloneqq\max\{1, \log(j_1/j_0)\}$. 
Then, there are numerical constant $c>0$ and $C_\delta\ge 1$, 
\[
\P\biggl(
\biggl\{
	\max_{j_0 \leq j \leq j_1} |A_j| \geq a
\biggr\}
\cap
\biggl\{
	\max_{j_0 \leq j \leq j_1} \biggl| \sum_{k=j_0+1}^j U_k \biggr| \leq u
\biggr\}
\biggr)
\leq C\exp\left(-c \min\left\{ j_0^{\delta}, \frac{\left(\log (a/u)\right)^2}{C_V T + C_\delta} \right\}\right).
\]
\end{proposition}

\begin{proof}
We can create a sequence of cutoff martingale increments $\tilde{V}_j$  which are also adapted to $\mathscr{G}$, which have $|\tilde{V}_j| \leq j^{-1/2+2\delta}$ almost surely and 
\[
\P( V_j \neq \tilde{V}_j ) 
\leq \exp( -j_0^{\delta}).
\]
This can also be done in such a way that the variance of $\tilde{V}_j$ increases no more than a factor of $2$.
Let $\mathscr{E}$ be the event that all $V_j = \tilde{V}_j$ for $j_0 \leq j \leq j_1$ and that $\max_{j_0 \leq j \leq j_1} \biggl| \sum_{k=j_0+1}^j U_k \biggr| \leq u$.  Under the setup of the proposition, it suffices to work on the event $\mathscr{E}$.

Define for $j \geq j_0$ (with the product empty in the case $j=j_0$)
\[
P_{j} \coloneqq \prod_{k=j_0+1}^{j} \bigl(1+ \tilde{V}_{k} F(A_{k-1})/A_{k-1}\bigr).
\]
Then $P_{j+1}/P_j = \bigl(1+ \tilde{V}_{j+1} F(A_j)/A_j\bigr)$ and so we can express
on $\mathscr{E}$,
\[
A_{j+1} = \frac{P_{j+1}}{P_j}A_j + {U}_{j+1},
\quad\text{for}
\quad j \geq j_0.
\]
Dividing through by $P_{j+1}$, this can therefore be solved explicitly to give the representation
\begin{equation}\label{eq:A_j_rep}
\frac{A_{j+1}}{P_{j+1}} = 
\sum_{\ell=j_0+1}^{j+1} \frac{U_\ell}{P_\ell}
=
\sum_{k=j_0+1}^{j+1} \left(\frac{-1}{P_k}+\frac{1}{P_{k-1}}\right)
\left(\sum_{\ell=j_0+1}^{k-1} U_\ell\right)
+ \frac{1}{P_{j+1}} \sum_{\ell=j_0+1}^{j+1} U_\ell.
\end{equation}

Now we introduce the event $\mathscr{P}$, and let $p$ be a parameter to be chosen later, and we show that for any $p$ greater than some constant depending only on $\delta$ and for some absolute constant $c>0$,
\begin{equation}
	\mathscr{P} \coloneqq 
	\biggl\{
		\max_{j_0 \leq j \leq j_1} \bigl|\log |P_j|\bigr| \leq p
	\biggr\}
	\quad\text{we have}\quad
\P\bigl( \mathscr{P}^c\bigr) \leq 4\exp\left(-c \min\left\{ p j_0^{1/2-2\delta}, \frac{p^2}{2C_V T} \right\}\right).
\label{eq:logPtail}
\end{equation}
Using the truncation, we have that
\[
\log\biggl(\frac{P_{j+1}}{P_j} \biggr) = 
\tilde{V}_{j+1} F(A_j)/A_j
- \tilde{V}^2_{j+1} (F(A_j)/A_j)^2/2 + \mathcal{O}( j^{-3/2+3\delta}).
\]
With the $\mathcal{O}$ term deterministically bounded, and its absolute sum bounded by a constant that depends only on $\delta$.  Hence we have
\[
\max_{j_0 \leq j \leq j_1} \bigl|\log |P_j| \bigr|
\leq
\max_{j_0 \leq j \leq j_1}
\biggl|
\sum_{k=j_0+1}^j
\tilde{V}_{k+1} F(A_k)/A_k
\biggr|
+
\max_{j_0 \leq j \leq j_1}
\biggl|
\sum_{k=j_0+1}^j
\tilde{V}^2_{k+1} (F(A_k)/A_k)^2
\biggr|
+\mathcal{O}(1).
\]
Thus the tail bound on $\log|P_j|$ in \eqref{eq:logPtail} follows from Freedman's inequality, noting that the sum of variances is bounded by 
\[
C_V\left(\log (j_1/j_0)+\gamma\right)
\leq 2C_V T
\] with $\gamma$ the Euler-Mascheroni constant; the mean term in the second sum is bounded by the same; and the sum of variances of the second is bounded by a constant that depends only on $\delta$.  Thus we conclude that there is an aboslute constant $c>0$ and a constant $C(\delta)>0$ so that for all $p \geq C(\delta) + C_V T$ 
\[
	\P\bigl( \mathscr{P}^c\bigr)
	\leq 2\exp\left(-c \min\left\{ p j_0^{1/2-2\delta}, p^2/(2C_V T) \right\}\right)
	+2\exp\left(-c \min\left\{ p j_0^{1-4\delta}, p^2/C(\delta) \right\}\right)
\]
We note the first of these always dominates the second, provided $p$ is greater than some constant depending on $\delta,$ which completes the claim of \eqref{eq:logPtail}.

We can represent the difference
\[
\frac{1}{P_k} - \frac{1}{P_{k-1}} = \frac{1}{P_{k-1}^2}\left(
\tilde{V}_{k} F(A_{k-1})/A_{k-1}
-
\tilde{V}_{k}^2 (F(A_{k-1})/A_{k-1})^2
+\mathcal{O}(j^{-3/2+3\delta})
\right)
\]
On the event $\mathscr{E} \cap \mathscr{P}$, we have that the process 
\[
M^{(1)} \coloneqq j\mapsto \sum_{k=j_0+1}^{j} \frac{\tilde{V}_{k} F(A_{k-1})/A_{k-1}}{P_{k-1}^2} \left(\sum_{\ell=j_0+1}^{k-1} U_\ell\right)
\]
is a martingale whose increments are predictably bounded by $e^{2p}u j_0^{-1/2+\delta}$ and whose bracket process is bounded by $2C_V e^{4p} u^2 T$.  Hence we can apply Freedman's inequality to conclude that for some absolute constant $c>0$
\[
\P\biggl(\biggl\{\max_{j_0 \leq j \leq j_1} |M^{(1)}_j| \geq x\biggr\}
\cap
\mathscr{E} \cap \mathscr{P}
\biggr) 
\leq 2\exp\left(-c \min\left\{ \frac{x j_0^{1/2-2\delta}}{e^{2p} u},
\frac{x^2}{2C_V e^{4p} u^2 T} \right\}\right)
\]
A similar argument bounds the same process with the square:
\[
M^{(2)} \coloneqq j\mapsto \sum_{k=j_0+1}^{j} \frac{\bigl(\tilde{V}_{k} F(A_{k-1})/A_{k-1}\bigr)^2}{P_{k-1}^2} \left(\sum_{\ell=j_0+1}^{k-1} U_\ell\right),
\]
which now has a mean bounded by $2C_V e^{2p} u T$, and has bracket bounded by $C(\delta) e^{4p} u^2$ for some constant depending only on $\delta$.  Thus, for $x > C_V e^{2p} u T$
\[
\P\biggl(\biggl\{\max_{j_0 \leq j \leq j_1} |M^{(2)}_j| \geq x\biggr\}
\cap
\mathscr{E} \cap \mathscr{P}
\biggr) 
\leq 2\exp\left(-c \min\left\{ \frac{x j_0^{1-4\delta}}{e^{2p} u},
\frac{x^2}{C(\delta) e^{4p} u^2} \right\}\right)
\]
Once more, provided that $x$ is larger than some constant depending on $\delta$, the $M^{(1)}$ tail bound dominates the $M^{(2)}$ tail bound.  

Returning to \eqref{eq:A_j_rep}
and we conclude that if $|A_j| \geq a$, on the even $\mathscr{E} \cap \mathscr{P}$, one of $|M^{(1)}_j|$ or $|M^{(2)}_j|$ is larger than $a/e^p$.  Hence, for all $p,a$ such that $a/e^p > \max\{C(\delta), C_V e^{2p} u T\}$ and $p > \max\{C(\delta), C_V T\}$,
\[
	\begin{gathered}
	\P\biggl(
	\biggl\{
		\max_{j_0 \leq j \leq j_1} |A_j| \geq a
	\biggr\}
	\cap
	\biggl\{
		\max_{j_0 \leq j \leq j_1} \biggl| \sum_{k=j_0}^j U_k \biggr| \leq u
	\biggr\}
	\biggr)
	\leq 4\exp\left(-c \min\left\{ p j_0^{1/2-2\delta}, \frac{p^2}{2C_V T} \right\}\right)
	\\
	+4\exp\left(-c \min\left\{ \frac{a j_0^{1/2-2\delta}}{e^{3p} u},
	\frac{a^2}{2C_V e^{6p} u^2 T} \right\}\right)
	+j_1 \exp(-j_0^{\delta}).
	\end{gathered}
\]

We optimize this quantity by choosing $p = \frac{1}{3}\left(\log a/u - \log\log a/u\right)$, which is feasible provided that $\log a/u > 3\max\{C(\delta), C_V T\}$ for some constant $C(\delta)>0.$  Hence we conclude that for all such $a/u$, all $j_1 \leq e^{j_0^\delta/2}$, we can appropriately shrink the constant $c>0$ to abosrb the $j_1 \exp(-j_0^{\delta})$ term and absorb the other absolute constants, to conclude that
\[
	\P\biggl(
	\biggl\{
		\max_{j_0 \leq j \leq j_1} |A_j| \geq a
	\biggr\}
	\cap
	\biggl\{
		\max_{j_0 \leq j \leq j_1} \biggl| \sum_{k=j_0}^j U_k \biggr| \leq u
	\biggr\}
	\biggr)
	\leq C\exp\left(-c \min\left\{ j_0^{\delta}, \frac{\left(\log (a/u)\right)^2}{C_V T} \right\}\right). \qedhere
\]
\end{proof}

\section{Asymptotics for the deterministic part of the phase.}

\begin{proposition} \label{prop:detphase}
Let $z\in\Q$ and $\pm = \sgn(z)$,  it holds locally uniformly for $T>0$, 
\[
\vartheta_{N,N_T}(z) -  \pi N F(z)  =  - N_T(z)\1\{z<0\} \mp \big( \tfrac23 T^{3/2} - \tfrac\pi4\big)
-\frac{\arcsin(z)}{2} +\underset{N\to\infty}{\o(1)} . 
\] 
\end{proposition}

\begin{proof}
Using McLaurin formula, if $f'(u) \ge 0 $ and decreasing for $u\ge m$, with $f''(u)$ integrable, then 
\[
\sum_{k=m+1}^N f(k/N_0) = \int_{m+1}^N f(t/N_0) \d t + \frac{f(N/N_0)- f(m/N_0)}{2}
+\O\bigg(\frac{f'(m/N_0)}{N_0} \bigg)
\]
We apply this formula with $f : u \in [1,\infty] \mapsto \arccos(u^{-1/2})$,  $N_0 = Nz^2$ (with $z>0$), $m=N_0+T\L$ and $\L^3 =N_0$. 
We have for $u>1$
\[
f'(u) = \frac{1/2}{u\sqrt{u-1}} , \qquad\qquad   \frac{f'(m/N_0)}{N_0} \le \frac{1}{\sqrt{N_0(m-N_0)}} =\frac{1}{\L^2\sqrt{T}} .
\]
Using that $f(1+\epsilon)\le \sqrt{\epsilon}$, we obtain
\[
\vartheta_{N,m}(z)  = \int_{m+1}^N \arccos(z\sqrt{N/t})\, \d t + \frac{\arccos(z)}{2} + \O\bigg( \frac{\sqrt{T}}{\L} + \frac{1}{\L^2\sqrt{T}} \bigg) .
\]
For the leading term, we have for $z>0$
\[
-\partial_z\bigg( \int_{N_0}^N \arccos(z\sqrt{N/t})\, \d t\bigg)
=  \int_{N_0}^N \frac{\sqrt{N}}{\sqrt{t-N_0}} \d t   =2 N\sqrt{1-z^2} 
\]
and then 
\[
\int_{N_0}^N \arccos(z\sqrt{N/t})\, \d t = 2 N \int_z^1 \sqrt{1-u^2} \d u  = \pi N F(z) .
\]
Moreover, using that $ f(1+\epsilon) =  \sqrt{\epsilon}  -\epsilon^{3/2}/3 +\O(\epsilon^{5/2})$, 
\[\begin{aligned}
\int_{N_0}^m \arccos(z\sqrt{N/t})\, \d t  &= \L  \int_0^{T}   f(1+u/\L^2) \d u \\
&=  \int_0^{T}  \sqrt{\epsilon} \d \epsilon +\O\big(T^{5/2}/\L^2\big) \\
&\simeq  \tfrac23 T^{3/2}  
\end{aligned}\]
Using that $\arccos(z) = \pi/2-\arcsin(z)$, we conclude that for $z\ge 0$,
\[
\vartheta_{N,m}(z) =  \pi N F(z) - \tfrac23 T^{3/2} - \frac{\arcsin(z)-\pi/2}{2} +\underset{N\to\infty}{\o(1)}
\]

We have $\arccos(-z) = \pi-\arccos(z)$ and $F(-z) = 1-F(z)$ for $z\in[-1,1]$, so that  $\vartheta_{N,m}(-z)-\pi NF(-z) = -\pi m - (\vartheta_{N,m}(z)-\pi NF(z))$ and for $z<0$, 
\[\begin{aligned}
\vartheta_{N,m}(z)-\pi NF(z) &= -\pi m + \tfrac23 T^{3/2} - \frac{\arcsin(z)+\pi/2}{2} +\underset{N\to\infty}{\o(1)} . 
\end{aligned}\]
Combining these asymptotics, this completes the proof. 
\end{proof}

Finally, we also need precise asymptotics for the bracket of the $\G$ field. 

\begin{proposition}\label{prop:app_G1}
The $\G$ field satisfies for $z\in\Q$ and $T\ge 1$, 
\[
[\G_{N_0}(z)]   =  2 \log(\L/2) + \underset{N\to\infty}{\o(1)} , \qquad
\big[\G_{N,N_T}(z)\big]  = \log_+\bigg(\frac{\L^2 (1-z^2)}{T}\bigg) \pm \i\pi -2\i \arcsin(z)  +\underset{N\to\infty}{\o(1)} .
\]
\end{proposition}

\begin{proof} 
Using the properties of the map $J$, \eqref{Jouk},  one has for $z\in\R$ and $0<t<z^2$, 
\[
\frac{d}{dt}  \log\big(1+ J(z/\sqrt{t})^2\big) =  \frac{J(z/\sqrt{t}) }{2t\sqrt{z^2/t-1}} 
= \frac{1- J(z/\sqrt{t})^2}{4(z^2-t)} .
\]
This follows from the fact that $J'(w)  = -J(w)/\sqrt{w^2-1} $ and the fact that $J$ satisfies the quadratic equation 
\[
1+J(w)^2 =  2wJ(w)  , \qquad\qquad   1-J(w)^2 = 2\sqrt{w^2-1}J(w) .
\]
Under Definition~\ref{def:noise} and Definition~\ref{def:GW},
with $\L=\L(z)$ and $\L^3= Nz^2$, the field $\G^1$ is real-valued and with $m =Nz^2-\L$, 
\[
[\G_{N_0}(z)]  =  \sum_{k \le m} \frac{1 +J(z\sqrt{N/k})^2}{2(Nz^2-k)} 
=  \sum_{k \le m} \frac{1/2}{Nz^2-k} -\sum_{k \le m} \frac{1 -J(z\sqrt{N/k})^2}{2(Nz^2-k)} .
\]
The main term has the asymptotics
\[
\sum_{k \le m} \frac{1}{Nz^2-k}  = \log\bigg(\frac{Nz^2}{\L}\bigg) +\underset{\L\to\infty}{\o(1)}. 
\]
For the second term, we use a Riemann sum approximation with $f: t\in \R_+\mapsto \frac{1- J(z\sqrt{N/t})^2}{4(Nz^2-t)}$. 
This function is monotone for $t<m$, so that 
\[\begin{aligned}
\sum_{k \le m} \frac{1 -J(z\sqrt{N/k})^2}{Nz^2-k}  & = \int_{1}^m \frac{1 -J(z\sqrt{N/t})^2}{Nz^2-t}  \d t +  \underset{\L\to\infty}{\o(1)} \\
&= 4\log\big(1+ J(z/\sqrt{N/m})^2\big) +  \underset{\L\to\infty}{\o(1)}  \\
&= 4 \log 2 +  \underset{\L\to\infty}{\o(1)}
\end{aligned}\]
since $J(w) \to 0$ as $w\to\infty$ and $J(w) \to \pm1$ as $w\to\pm1$. 

We conclude that if $z\in\Q$, 
\[
[\G_{N_0}(z)]   =  2 \log(\L/2) + \underset{N\to\infty}{\o(1)} .
\]

By definition, we also have for $z\in [-1,1]$ and $T\ge 1$, 
\[
\big[\G_{N,N_T}(z)\big]= \sum_{N_T < k \le N} \frac{1 +J(z\sqrt{N/k})^2}{2(k-Nz^2)} 
= \sum_{N_T < k \le N} \frac{1 + e^{2\i \arccos(z\sqrt{N/k})}}{2(k-Nz^2)} 
\]
where this sum is 0 if $|z|^2 \ge 1-T/\L^2$. 
This implies that 
\[\begin{aligned}
\Re\big[\G_{N,N_T}(z)\big]  
&= \sum_{N_T < k \le N} \frac{\cos(\arccos(z\sqrt{N/k}))^2}{k-Nz^2} 
= \sum_{N_T < k \le N} \frac{N z^2}{k(k-Nz^2)}  \\
\Im\big[\G_{N,N_T}(z)\big]
&= \sum_{N_T < k \le N}  \frac{\sin(2\arccos(z\sqrt{N/k})) }{2(k-Nz^2)}
= \sum_{N_T < k \le N} \frac{\pm \sqrt{N z^2} }{k\sqrt{k-Nz^2}} 
\end{aligned}\]
where $\pm =\sgn(z)$.
These two sums are convergent and   
\[
\sum_{N_T < k \le N} \frac{N z^2}{k(k-Nz^2)}  =  \sum_{N_T < k \le N} \bigg( \frac{1}{k-Nz^2}  -\frac1k \bigg) =  \log\bigg(\frac{N(1-z^2)}{T\L}\bigg) -\log\bigg(\frac{N}{N_T}\bigg)  +  \underset{\L\to\infty}{\o(1)}
\]
If $z\in\Q$, as $N_T \simeq \L^3$, this implies that 
\[
\Re\big[\G_{N,N_T}(z)\big]  = \log_+\bigg(\frac{\L^2 (1-z^2)}{T}\bigg) + \underset{N\to\infty}{\o(1)} .
\]
By a Riemann sum approximation,
\[
\sum_{N_T < k \le N} \frac{\sqrt{N z^2} }{k\sqrt{k-Nz^2}}   =  \int_{N_T}^N \frac{\sqrt{N z^2} \, \d t}{t^{3/2}\sqrt{1-Nz^2/t}} + \underset{N\to\infty}{\o(1)} .
\]
We make a change of variable $u = \sqrt{Nz^2 / t}$ $(\d u = - \sqrt{Nz^2} \d t/ 2t^{3/2})$ so that 
\[
\sum_{N_T < k \le N} \frac{\sqrt{N z^2} }{k\sqrt{k-Nz^2}}   =  \int_{|z|}^1 \frac{2\d u}{\sqrt{1-u^2}} + \underset{N\to\infty}{\o(1)} .
\]
using that $\sqrt{Nz^2 / N_T} \simeq 1$. This implies that for $z\in[-1,1]$
\[
\Im\big[\G_{N,N_T}(z)\big] = \pm 2 \big(\pi/2-\arcsin(|z|) \big) + \underset{N\to\infty}{\o(1)} .
\]
We conclude that if $z\in\Q$ for $T\ge 1$, 
\[
\big[\G_{N,N_T}(z)\big]  = \log_+\bigg(\frac{\L^2 (1-z^2)}{T}\bigg) \pm \i\pi -2\i \arcsin(z)  +\underset{N\to\infty}{\o(1)} . \qedhere
\]
\end{proof}

\todo{Include the sketch of proof}

\printbibliography

\end{document}